\documentclass{amsart}

\usepackage{latexsym,amsfonts,amssymb,exscale,enumerate,comment}
\usepackage{amsmath,amsthm,amscd}

\newcommand{\wt}{\tilde{w}}
\newcommand{\com}{\color{green}  }

\usepackage{url}
\usepackage[bookmarks=true,%
    colorlinks=true,%
    linkcolor=blue,%
    citecolor=blue,%
    filecolor=blue,%
    menucolor=blue,%
    urlcolor=blue,%
    breaklinks=true]{hyperref}

\input xy
\usepackage[all]{xy}
\xyoption{line}
\xyoption{arrow}
\xyoption{color}
\SelectTips{cm}{}

\usepackage{tikz}
\usetikzlibrary{shapes,snakes}
\usetikzlibrary{decorations.markings}
\usetikzlibrary{decorations.pathreplacing}
\tikzstyle directed=[postaction={decorate,decoration={markings,
    mark=at position #1 with {\arrow{>}}}}]

\newcommand{\hackcenter}[1]{
 \xy (0,0)*{#1}; \endxy}

\tikzset{->-/.style={decoration={
  markings,
  mark=at position #1 with {\arrow{>}}},postaction={decorate}}}

\tikzset{middlearrow/.style={
        decoration={markings,
            mark= at position 0.5 with {\arrow{#1}} ,
        },
        postaction={decorate}
    }
}
\newcommand{\bbullet}{
\begin{tikzpicture}
  \draw[fill=black] circle (0.55ex);
\end{tikzpicture}
}

\usepackage{graphicx}
\usepackage{color}

\newcommand{\sE}{\cal{E}}
\newcommand{\sF}{\cal{F}}
\newcommand{\cE}{\cal{E}}
\newcommand{\cF}{\cal{F}}
\newcommand{\uep}{\und{\varepsilon}}
\newcommand{\onel}{\1_{\lambda}}
\newcommand{\onenn}{\1_{n}}
\newcommand{\oneml}{\1_{-\lambda}}

\newcommand{\tsigma}{\sigma}
\newcommand{\tomega}{\omega}
\newcommand{\tpsi}{\psi}

\newcommand{\Ucas}{\cal{C}}
\newcommand{\UupD}{\cal{E}}
\newcommand{\UdownD}{\cal{F}}

\def\P{\mathsf{P}}
\def\Q{\mathsf{Q}}

\newcommand{\ogam}{\xi^+}
\newcommand{\odelt}{\hat{\xi}^+}
\newcommand{\gam}{\xi^{-}}
\newcommand{\delt}{\hat{\xi}^{-}}
\newcommand{\gams}{\xi^{-}_{\tsigma}}
\newcommand{\delts}{\hat{\xi}^{-}_{\tsigma}}
\newcommand{\ogams}{\xi^+_{\tsigma}}
\newcommand{\odelts}{\hat{\xi}^+_{\tsigma}}
\newcommand{\gamp}{\xi'^{-}}

\newcommand{\SL}{\mathrm{SL}}

\def\W{W_{1+\infty}}
\def\tW{\widetilde{W}_{1+\infty}}
\def\id{\mathrm{id}}
\def\Id{\mathrm{Id}}
\def\adj{\mathrm{adj}}

\newcommand{\END}{{\rm END}}
\newcommand{\Gr}{\cat{Flag}_{N}}
\newcommand{\Grn}[1]{\cat{Flag}_{#1}}
\theoremstyle{plain}
\newtheorem{theorem}{Theorem}

\newtheorem{corollary}[theorem]{Corollary}
\newtheorem{proposition}[theorem]{Proposition}
\newtheorem{lemma}[theorem]{Lemma}

\theoremstyle{definition}
\newtheorem{example}[theorem]{Example}
\newtheorem{definition}[theorem]{Definition}

\theoremstyle{definition}
\newtheorem{remark}[theorem]{Remark}

\numberwithin{equation}{section}
\numberwithin{theorem}{section}

\newcommand{\sym}{{\rm Sym}}
\newcommand{\maps}{\colon}
\newcommand{\und}[1]{\underline{#1}}

\newcommand{\xsum}[2]{
  \xy
  (0,.4)*{\sum};
  (0,3.7)*{\scs #2};
  (0,-2.9)*{\scs #1};
  \endxy
}

\newcommand{\refequal}[1]{\xy {\ar@{=}^{#1}
(-1,0)*{};(1,0)*{}};
\endxy}

\newcommand{\cat}[1]{\ensuremath{\mbox{\bfseries {\upshape {#1}}}}}
\newcommand{\numroman}{\renewcommand{\labelenumi}{\roman{enumi})}}
\newcommand{\numarabic}{\renewcommand{\labelenumi}{\arabic{enumi})}}
\newcommand{\numAlph}{\renewcommand{\labelenumi}{\Alph{enumi}.}}

\newcommand{\To}{\Rightarrow}
\newcommand{\TO}{\Rrightarrow}
\newcommand{\Hom}{{\rm Hom}}
\newcommand{\HOM}{{\rm HOM}}
\renewcommand{\to}{\rightarrow}
\newcommand{\nh}{\mathrm{NH}}

\def\Res{{\mathrm{Res}}}
\def\Ind{{\mathrm{Ind}}}
\def\Lad{\mathrm{Lad}}
\def\lra{{\longrightarrow}}
\def\dmod{{\mathrm{-mod}}}   
\def\fmod{{\mathrm{-fmod}}}   
\def\pmod{{\mathrm{-pmod}}}  
\def\rk{{\mathrm{rk}}}
\def\Id{\mathrm{Id}}
\def\mc{\mathcal}
\def\mf{\mathfrak}
\def\Bim{{\mathrm{Bim}}}
\def\Br{{\mathrm{Br}}}

\numberwithin{equation}{section}
\def\YY#1{\textcolor[rgb]{1.00,0.00,0.50}{[YY: #1]}}%
\def\AL#1{\textcolor[rgb]{1.00,0.00,0.00}{[AL: #1]}}%
\def\PS#1{\textcolor[rgb]{0.00,0.49,0.25}{[MK: #1]}}%
\def\JS#1{\textcolor[rgb]{0.40,0.00,0.90}{[JS: #1]}}%

\def\b{$\blacktriangleright$}
\def\e{$\blacktriangleleft$}
\def\new#1{\b #1\e}%

\let\hat=\widehat
\let\tilde=\widetilde

\let\epsilon=\varepsilon

\usepackage{bbm}
\def\C{{\mathbb{C}}}
\def\N{{\mathbbm N}}
\def\R{{\mathbbm R}}
\def\Z{{\mathbbm Z}}
\def\H{{\mathcal{H}}}

\def\cal#1{\mathcal{#1}}%
\def\1{\mathbbm{1}}%
\def\ev{\mathrm{ev}}%
\def\coev{\mathrm{coev}}%
\def\tr{\mathrm{tr}}%
\def\st{\mathrm{st}}%
\def\nn{\notag}
\newcommand{\ontop}[2]{\genfrac{}{}{0pt}{2}{\scriptstyle #1}{\scriptstyle #2}}

\def\la{\langle}
\def\ra{\rangle}

\renewcommand{\l}{\lambda}

\usepackage{bbm}

\def\cal#1{\mathcal{#1}}

\newcommand\nc{\newcommand}
\nc\rnc{\renewcommand}
\nc\Kar{\operatorname{Kar}}
\nc\End{\operatorname{End}}

\newcommand{\scs}{\scriptstyle}

\newcommand{\Ucat}{\cal{U}}
\newcommand{\UcatD}{\dot{\cal{U}}}

\nc\Sym{\operatorname{Sym}}
\newcommand{\bigb}[1]{
\begin{tikzpicture}
\node[draw,  fill=white,rounded corners=4pt,inner sep=3pt] (X) at (0,.75) {$\scs #1$};
\end{tikzpicture}}
\newcommand{\Sq}{{\rm Sq}}

\newcommand{\eqr}[1]{{\overset{#1}{=}}}

\allowdisplaybreaks

\title[Braid group actions from deformed Webster algebras]{Braid group actions from categorical symmetric Howe duality on
deformed Webster algebras}

\begin{document}
\setcounter{tocdepth}{1}

\author{Mikhail Khovanov}
\email{khovanov@math.columbia.edu}
\address{Department of Mathematics \\ Columbia University \\ New York, NY}

\author{Aaron D. Lauda}
\email{lauda@usc.edu}
\address{Department of Mathematics\\ University of Southern California \\ Los Angeles, CA}

\author{Joshua Sussan}
\email{jsussan@mec.cuny.edu}
\address{Department of Mathematics \\ CUNY Medgar Evers \\ Brooklyn, NY}
\date{\today}

\author{Yasuyoshi Yonezawa}
\email{yasuyoshi.yonezawa@math.nagoya-u.ac.jp}
\address{Graduate school of Mathematics \\
Nagoya University \\ Nagoya, Japan}

\maketitle

\begin{abstract}
We construct a 2-representation  categorifying the symmetric Howe representation  of $\mathfrak{gl}_m$ using a deformation of an algebra introduced by Webster.  As a consequence, we obtain a categorical braid group action taking values in a homotopy category.
\end{abstract}
\tableofcontents

\section{Introduction}

\subsection{Howe pairs and link homology}
Beginning with the algebro-geometric constructions of link homology theories by Cautis and Kamnitzer~\cite{CaKa-slm} and further developed in \cite{CKL-skew}, categorical skew Howe duality has proven to be a powerful tool in higher representation theory and link homology theory.
The key idea originates in the decategorified context where Toledano Laredo~\cite{LT} used Howe duality to show that the
 braid group action on an $m$-fold tensor product of $\mf{sl}_n$-representations can be realized via the deformed Weyl group action associated with an auxiliary quantum group for $\mf{sl}_m$.  This perspective was fundamental in the resolution of several important conjectures in link homology theory that have resisted proof by other techniques.  These include:
\begin{itemize}
\item the first complete definition of $\mf{sl}_n$-foams with a combinatorial evaluation for closed foams not requiring the use of the Kapustin-Li formula~\cite{LQR,QR},
\item a completely integral formulation of $\mf{sl}_n$-link homology~\cite{QR},
\item the study of $\mf{sl}_n$ deformations analogous to Lee homology in the $\mf{sl}_2$-case~\cite{RoseW},
\item the discovery of direct connections between quantum groups and Chern-Simon's gauge theory~\cite{CGR},
\item the definition of $\mf{sl}_n$-web algebras generalizing Khovanov's arc algebra~\cite{MPT,MY},
\item a proof of the exponential growth of coloured  HOMFLYPT homology~\cite{Wed}, and
\item a unification of $\mathfrak{sl}_n$ link homologies constructed using wildly different techniques (e.g. category $\mathcal{O}$, matrix factorizations, Webster's categorification of tensor products, and coherent sheaves on orbits in the affine Grassmannian)~\cite{Cautis,MW}.
\end{itemize}

Starting from the work of Cautis, Kamnitzer, and Morrison~\cite{CKM}, these methods have proven to be incredibly powerful even in the decategorified setting.
They gave a complete diagrammatic description of the category generated by tensor products of fundamental quantum $\mf{sl}_n$-representations using $\mf{sl}_n$-webs.
The Karoubi envelope of this category is the full category of quantum $\mf{sl}_n$-representations, since every irreducible arises as a summand of a tensor product of fundamental representations.
In this framework, relations on $\mf{sl}_n$ webs arise from relations in quantum $\mf{gl}_m$ under a natural action on so-called ladder webs.   The proof that these relations suffice uses Howe duality and the commuting actions of quantum $\mf{gl}_m$ and $\mf{sl}_n$ on the space $\bigwedge^N(\C^n_q \otimes \C^m_q)$.
This idea has been generalized in several directions~\cite{CW, ES,Grant,QS,TVW}.
This approach was also instrumental in the second author's work with Garoufalidis and L\^{e} resolving the $q$-holonomicity conjecture for the HOMFLYPT polynomial~\cite{Lau-q}.

In the classical theory of Howe duality,
one considers the commuting actions of $\mf{gl}_m$ and $\mf{sl}_n$ on the space ${\rm Sym}^N(\C^n \otimes \C^m)$.  This duality was extended to the quantum setting in~\cite{RT}, extending the work~\cite{BZ}.  Quantum symmetric Howe duality for $(\mf{gl}_m, \mf{sl}_2)$ was used in \cite{RT} to provide a generators and relations description for the `symmetric web category', that is, a generators and relations description of the full space of intertwiners between tensor powers of symmetric powers of fundamental $\mf{sl}_2$-representations, not requiring a passage to the Karoubi envelope.   This was later extended to symmetric powers of quantum $\mf{sl}_n$-representations in \cite{TVW}.


In this paper we utilize the commuting actions of quantum $ \mf{gl}_m$ and $\mf{sl}_2$ on the symmetric representation ${\rm Sym}^N(\C^2_q \otimes \C^m_q)$ in the categorical setting.  Our framework introduces deformations of Webster's tensor product algebras associated to symmetric powers of $\mf{sl}_2$.   Closely related deformations were studied for fundamental representations by Mackaay and Webster \cite{MW}.  Our framework allows us to define a categorical coloured braid invariant where strands are coloured by arbitrary irreducibles of $\mf{sl}_2$ (viewed as symmetric powers of the defining representation).
Our approach also suggests a natural generalization to deformations of Webster's algebras for tensor products of fundamental $\mf{sl}_n$ representations.

\subsection{Motivation for symmetric Howe duality from link homology theory}
The first author's original categorification of the Jones polynomial~\cite{Kh1,Kh2} can be defined entirely within the homotopy category of complexes over an additive category.  In this approach the homologies associated with knots and links are finite dimensional.  An approach to coloured $\mf{sl}_2$ link invariants using cabling was proposed in ~\cite{Kh05}.  An alternative approach was proposed by Cooper and Krushkal~\cite{CoopK0} using categorified Jones-Wenzl projectors where the homologies are infinite dimensional  for coloured unknots.  Hogancamp gave a refinement of this invariant where the homology of the unknot is finite dimensional over a larger ground ring~\cite{Hog}.      An alternative approach to categorification of Jones-Wenzl projectors was given by Rozansky~\cite{Roz}.  A Lie theoretic approach was outlined in ~\cite{FSS}.

Utilizing the higher representation theory framework arising from the theory of categorified quantum groups~\cite{KL1,KL2,KL3,Rou2}, Webster  gave a general combinatorial categorification of Reshetikhin-Turaev invariants~\cite{Web} for links coloured by arbitrary irreducible representations of the quantum group associated to a semisimple Lie algebra $\mf{g}$.  The invariants are constructed using algebras that are defined in an elementary way. However, the associated tangle invariants are difficult to compute.  For example, in order to associate a functor to a braid, one must work in the derived category rather than the homotopy category as in~\cite{Kh2}. Furthermore, invariants for the unknot coloured by non-miniscule representations are infinite dimensional in this approach.

Our interest in the symmetric Howe 2-representation ${\rm Sym}^N(\C^2_q \otimes \C^m_q)$ stems from the goal of modifying Webster's approach to obtain coloured $\mf{sl}_2$-link invariants using finite complexes in the homotopy category, rather than the derived category.
This work can be viewed as providing a potential algebraic categorification of the symmetric web category  from \cite{RT}.    We expect that this work should be closely related to the geometric categorification of symmetric Howe duality given by Cautis and Kamnitzer using the derived category of coherent sheaves on the  Beilinson-Drinfeld Grassmannian~\cite{CK-symm}.
A foam setting for symmetric Howe duality was recently studied in \cite{RWag,QRS}.

Cautis defines a link homology theory (which could be rephrased in the language of Soergel bimodules) where each component of the link is coloured by an arbitrary partition ~\cite{Cautis-Rem}.  The resulting homology categorifies the coloured HOMFLYPT polynomial.  Cautis also defines differentials $d_N$ leading to a categorified $\mf{sl}_N$ coloured link homology theory.  When the components are coloured by symmetric powers of the defining representation, the homology of the link is finite dimensional.  Both of our constructions have ingredients coming from the theory of Soergel bimodules and it is enticing to understand how our braid group actions (and optimistically coloured link homologies built from them) are related.



\subsection{The redotted Webster algebra and categorical braid group action}

Webster defined a family of algebras categorifying an arbitrary tensor product $V_{\l_1} \otimes V_{\l_2} \otimes \dots \otimes V_{\l_k}$ of irreducible representations of the quantum group associated to a semisimple Lie algebra $\mf{g}$.  These algebras admit a diagrammatic interpretation extending the diagrammatic categorifications of quantum groups and their irreducibles from \cite{Lau1,KL1,KL2,KL3}.  In this presentation, generators of the algebra correspond to planar diagrams consisting of red strands labelled by the dominant weights $\l_i$ of irreducibles and black strands labelled by the simple roots of the Lie algebra $\mf{g}$ that are governed by the KLR-algebra of type $\mf{g}$. Here we will be primarily interested in the case when $\mf{g}=\mf{sl}_2$, so that the black strands are governed by the nilHecke algebra.
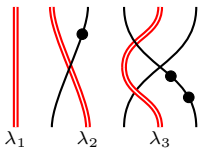
\[
\hackcenter{\begin{tikzpicture}[scale=0.8]
   \draw[thick,red, double, ] (-1.8,0) to (-1.8,2.0);
    \draw[thick,red, double, ] (-.6,0) .. controls ++(0,.5) and ++(0,-.5) .. (-1.2,2.0);
   \draw[thick,] (-1.2,0) .. controls ++(0,.5) and ++(0,-.5) .. (-.6,2.0)node[pos=.75, shape=coordinate](DOT){};
    \draw[thick,  ] (0,0) .. controls ++(0,1) and ++(0,-1) .. (1.2,2);
    \draw[thick,red, double, ] (.6,0) .. controls ++(0,.5) and ++(0,-.5) .. (0,1.0);
    \draw[thick,red, double, ] (0,1.0) .. controls ++(0,.5) and ++(0,-.5) .. (0.6,2) ;
    \draw[thick, ] (1.2,0) .. controls ++(0,1) and ++(0,-1) .. (0,2) node[pos=.2, shape=coordinate](DOT2){}
     node[pos=.4, shape=coordinate](DOT3){};
    \filldraw[black]  (DOT) circle (2.5pt);
    \filldraw[black]  (DOT2) circle (2.5pt);
    \filldraw[black]  (DOT3) circle (2.5pt);
    \node at (.6,-.2) {$\scs \lambda_3$};
    \node at (-.6,-.2) {$\scs \lambda_2$};
    \node at (-1.8,-.2) {$\scs \lambda_1$};
\end{tikzpicture}}
\]

Webster's tensor product algebras admit deformations, not only for $\mf{sl}_2$, but in vast generality (see for example ~\cite{WebgradedHecke} and \cite{MW}).
The first and third authors independently studied these deformations in the context of $\mf{sl}_2$ where all the red strands are labelled by the fundamental $\mf{sl}_2$ representation.  In this context, the deformation led to additional generators where red strands are also allowed to carry dots and additional diagrammatic relations are required.  The resulting algebras were called the {\em redotted Webster algebra} $W(1^m,n)$ in \cite{KS}, where there are $n$ black strands and $1^m$ indicates $m$ red strands labelled by the fundamental $\mf{sl}_2$-representation.  These algebras were studied with the aim of simplifying Webster's braid invariants.

The algebras $W(1^m,0)$, corresponding to no black strands, are directly related to Soergel bimodules.  In this language, early work by Rouquier~\cite{Rou-Soergel} defines a braid group action on the homotopy category of modules for $W(1^m,0)$.  This categorical braid group action was shown to be a strong action by  Elias and Krasner ~\cite{EKras}, meaning that it extends to braid cobordisms.  This categorical braid group action was extended to a link invariant in ~\cite{Kho-triple} based on earlier work with Rozansky ~\cite{KhR2}.

More recently, the case of a single black strand was studied, where a strong braid group action was constructed~\cite{KS} on the homotopy category of modules for $W(1^m,1)$ and it was conjectured that a categorical braid group action can be defined on the homotopy category of $W(1^m,n)$-modules for all $n$.  In this article we prove this conjecture in much greater generality.
We extend the algebra $W(1^m,n)$ to $W({\bf s},n)$ where ${\bf s}=(s_1, \ldots, s_m)$ is a tuple of natural numbers corresponding to red strands coloured by symmetric powers of the fundamental representation of $\mf{sl}_2$.
We then extend the braid group action to $W({\bf s},n)$.
In \cite{MW} the authors use formal arguments and connections with category $\cal{O}$ to place deformations of Webster algebras into the \emph{skew} Howe framework.  The authors work in the derived category.  Here we work in the context of \emph{symmetric} Howe duality and our explicit constructions allow us to use elementary techniques and to stay in the homotopy category.

The technical work involved is to show that a version of the $\mathfrak{gl}_m$ $2$-category  $\cal{U}=\cal{U}(\mf{gl}_m)$ ~\cite{KL3,MSV} acts on direct sums of categories of bimodules for $W({\bf s},n)$, so that the categorical quantum Weyl group action of $\mathfrak{gl}_m$ can be used to induce the braiding.
The computationally most difficult relations to verify are the $\mathcal{E} \mathcal{F}$ and
$\mathcal{F} \mathcal{E}$ decompositions.
A key observation provided to us by Cautis allowed us to avoid these calculations.
The braid group action for $W(1^m,1)$ constructed in \cite{KS} avoided the use of categorified $\mathfrak{gl}_m$ but fits into the framework of this paper.

The categories of modules over Webster's $\mf{sl}_2$ algebras were shown to admit the structure of a 2-representation of the categorified quantum group for $\mf{sl}_2$~\cite{Web}.  It would be interesting to extend that action to our redotted context, as our framework is well suited to understanding the full categorified Reshetikhin-Turaev invariant set up, including the commuting quantum group action.
A related action of categorified $\mf{sl}_2$ (and its generalizations) on cyclotomic nilHecke algebras (and its generalizations) was studied in \cite{KK}.

\subsection{Connections with Soergel bimodules}
A diagrammatic category $ \cal{S}\cal{C}_1(m)$ describing Soergel bimodules in type $A$ was introduced in ~\cite{EK}. In \cite{MSV}, the authors construct a functor from the diagrammatic Soergel category $\cal{S}\cal{C}_1(m)$ into endomorphisms of the $(1^m)$ weight space of a {\it Schur quotient} of $\cal{U}(\mf{gl}_m)$
\begin{equation*}
 \Sigma_{m,m} \maps
 \cal{S}\cal{C}_1(m) \to \cal{U}(1^m,1^m)/ \la \lambda \notin \Lambda(m,m)  \ra,
\end{equation*}
where
\begin{equation*}
\Lambda(m,m)=\{\lambda=(\lambda_1,\ldots,\lambda_m) | \lambda_i \in \Z_{\geq 0}, \lambda_1+\cdots + \lambda_m=m \}.
\end{equation*}

Under the 2-functor $\Sigma_{m,m}$, Rouquier's braiding on Soergel bimodules is sent to the braid group action on the homotopy category of the Schur quotient coming from two term Rickard complexes (see \eqref{complexes1} for more details).  Elias and Krasner's proof that Rouquier complexes give a strong braid group action~\cite{EKras} can then be used to show that two term Rickard complexes also give rise to a strong braid group action (functorial under braid cobordisms) in the context of the Schur quotient $\cal{U}(1^m,1^m)/\la \lambda \notin \Lambda(m,m) \ra$.  Thus the homotopy category of a category which is the image of
$ \cal{U}(1^m,1^m)/\la \lambda \notin \Lambda(m,m) \ra$ under a categorical action also has the structure of a \emph{strong} braid group action.

The categorical $\cal{U}(\mf{gl}_m)$ action we define on the reddotted Webster algebras factors through the Schur quotient $ \cal{U}(1^m,1^m)/\la \lambda \notin \Lambda(m,m) \ra$ when the sequence ${\bf s} = (1,1, \dots, 1)$.    This allows us to deduce that the braid group action we construct is a strong action in this case.  Note that the Rickard complexes for more general of weights ${\bf s}=(s_1, \ldots, s_m)$ in $\cal{U}(\mf{gl}_m)$ (with $s_i$ not all equal to 1) will not usually be two term complexes, so the results of Elias and Krasner do not apply.

\subsection{Connections with singular Soergel bimodules}
In ~\cite{KS} it was shown that a {\it cyclotomic quotient} of the algebra $W(1^m,1)$ is isomorphic to the endomorphism algebra of a direct sum of all the unique indecomposable objects of the category of singular Soergel bimodules ${}^{(1^m)}\mathcal{R}^{(1,m-1)} $ for the polynomial ring $\Bbbk[Y_1,\ldots, Y_m]$.  We conjecture that the cyclotomic quotient of the algebra
$W({\bf s},n)$ is Morita equivalent to the endomorphism algebra of a direct sum of all the indecomposable objects for the category of singular Soergel bimodules
${}^{(\mathbf{s})} \mathcal{R}^{(n,|{\mathbf{s}|-n)}} $ where $|\mathbf{s}|=s_1+\cdots+s_m$.

\subsection{Roots of unity}

Part of the motivation in reformulating Webster's link invariants in the context of the redotted Webster algebra is that we anticipate this framework will be useful in the program of categorification at a root of unity.  If one hopes to obtain a categorification of the Witten-Reshetikhin-Turaev 3-manifold invariant, it is expected that we must be able to define categorifications of coloured $\mf{sl}_2$ link invariants specialized at a root of unity.

Categorification at an $N$-th root of unity requires categorifying the base ring of Laurent polynomials in $q$ modulo the ideal generated by the $N$-th cyclotomic polynomial $\Phi_N(q)$. A framework for such a theory, known as hopfological algebra, was proposed in ~\cite{Kh4} when $N$ is a power of a prime $p$, via the stable category of $p$-complexes: generalized complexes
over a field of characteristic $p$ with the $p$-th power of the differential rather than
the second power being $0$. Recently, $p$-complexes were successfully
used as the ground category to categorify small and big quantum $\mf{sl}_2$ at a $p$-th root of
unity \cite{KQ,Q1,EQ,EQ2}.  A categorification of the Burau representation was given in ~\cite{QSus}.
There, a braid group action was constructed on the compact derived category of a special case of a $p$-DG Webster algebra.  It would be interesting to reformulate that work in the context of the redotted Webster algebra.
The paper \cite{KQ} contains an explicit proposal
about extending this categorification from $\mf{sl}_2$ to other simply-laced Lie algebras
as well as extending it to Webster's categorification ~\cite{Web} of tensor
products of integrable representations of quantum groups.  We anticipate that  redotted Webster algebras will provide a framework for defining $p$-DG analogs of coloured Khovanov homology.

\subsection{Towards link homology}
It is a natural question how to extend the braid invariant we construct here to a link invariant.  When $m=0$, it was shown by Khovanov that taking Hochschild cohomology yields a link invariant categorifying the HOMFLYPT polynomial.  It is not clear that this procedure will work for $m>0$.  Even when $m=0$, taking Hochschild cohomology seems conceptually wrong here since we expect to categorify the Jones polynomial.  The Reshetikhin-Turaev procedure for recovering the Jones polynomial requires one to work in the middle weight space of $V_1^{\otimes m}$ ($m$ should be even).  So again, it is not clear that the braid invariant existing for arbitrary $m$ and $n$ should give rise to a non-trivial link invariant.

\subsection{Acknowledgements}
The authors would like to thank Sabin Cautis for graciously providing us the key argument in Section ~\ref{sec:EF}.

M.K.~is partially supported by the NSF grants DMS-1406065 and DMS-1664240.
A.D.L.~ is partially supported by the NSF grants DMS-1255334 and DMS-1664240
J.S.~is partially supported by the NSF grant DMS-1407394, PSC-CUNY Award 67144-0045, and Simons Foundation Collaboration Grant 516673.

\section{The nilHecke algebra}
In subsequent sections, the nilHecke algebra will play two roles.  It is an ingredient in the definition of the categorified quantum group for $\mf{gl}_n$.  It also part of the definition of the algebra $W(\mathbf{s},n)$.
\subsection{Definition}
Let $\nh_n$ be the nilHecke algebra of rank $n$ over $\Bbbk$.   It is generated by commuting degree $2$ elements $x_1, \ldots, x_n$ and elements $\partial_1, \ldots, \partial_{n-1}$ of degree $-2$.  The generators satisfy the following relations.
\begin{enumerate}
\item $\partial_i^2=0$,
\item $\partial_i \partial_j \partial_i = \partial_j \partial_i \partial_j$ if $|i-j|=1$,
\item $\partial_i \partial_j = \partial_j \partial_i$ if $|i-j|>1$,
\item $x_i \partial_i - \partial_i x_{i+1}=1=\partial_i x_i - x_{i+1} \partial_i$.
\end{enumerate}

In general, for a $\Z$-graded algebra $A$ and a graded module $M$ over $A$, let $M_i$ be the subset of $M$ contained in degree $i$.  Let $M \langle r \rangle$ be the shift of $M$ up by $r$.  That is, $(M \langle r \rangle)_i = M_{i-r}$.
We will use the following notation for a direct sum of shifts of $M$:
\begin{equation*}
[r]M = M \langle r-1 \rangle \oplus M \langle r-3 \rangle
\oplus \cdots \oplus M \langle 3-r \rangle \oplus M \langle 1-r \rangle.
\end{equation*}

We now present this algebra in a diagrammatic fashion.
Consider collections of smooth arcs in the plane connecting $ n $ points on one horizontal line with $n$ points on another horizontal line.
Arcs are assumed to have no critical points (in other words no cups or caps).
Arcs are allowed to intersect, but no triple intersections are allowed.
Arcs can carry dots.
Two diagrams that are related by an isotopy that does not change the combinatorial types of the diagrams or the relative position of crossings are taken to be equal.
The elements of the vector space $ \nh_n$ are formal linear combinations of these diagrams modulo the local relations given below.
We give $ \nh_n $ the structure of an algebra by concatenating diagrams vertically.

Dots on strands correspond to generators $x_i$ given earlier.  Strands which cross correspond to generators $\psi_i$.
The generating elements of $\nh_n$ and their degrees are given below.
\[
\deg
\left( \;
\hackcenter{\begin{tikzpicture}[scale=0.6]
    \draw[thick, ] (0,0) -- (0,1.5)  node[pos=.35, shape=coordinate](DOT){};
    \filldraw  (DOT) circle (2.5pt);
    \node at (0,-.2) {$\scs $};
\end{tikzpicture}}\;
\right)
= 2,
\quad
\deg\left( \; \hackcenter{\begin{tikzpicture}[scale=0.6]
    \draw[thick, ] (0,0) .. controls (0,.75) and (.75,.75) .. (.75,1.5);
    \draw[thick, ] (.75,0) .. controls (.75,.75) and (0,.75) .. (0,1.5);
\end{tikzpicture}} \; \right) = -2.
\]

The diagrams satisfy the relations given below.
\[
\hackcenter{\begin{tikzpicture}[scale=0.8]
    \draw[thick,] (0,0) .. controls ++(0,1) and ++(0,-1) .. (1.2,2);
    \draw[thick, ] (.6,0) .. controls ++(0,.5) and ++(0,-.5) .. (0,1.0);
    \draw[thick, ] (0,1.0) .. controls ++(0,.5) and ++(0,-.5) .. (0.6,2);
    \draw[thick, ] (1.2,0) .. controls ++(0,1) and ++(0,-1) .. (0,2);
\end{tikzpicture}}
\;\; = \;\;
\hackcenter{\begin{tikzpicture}[scale=0.8]
    \draw[thick, ] (0,0) .. controls ++(0,1) and ++(0,-1) .. (1.2,2);
    \draw[thick, ] (.6,0) .. controls ++(0,.5) and ++(0,-.5) .. (1.2,1.0);
    \draw[thick, ] (1.2,1.0) .. controls ++(0,.5) and ++(0,-.5) .. (0.6,2.0);
    \draw[thick, ] (1.2,0) .. controls ++(0,1) and ++(0,-1) .. (0,2.0);
\end{tikzpicture}}
\qquad \qquad
\hackcenter{\begin{tikzpicture}[scale=0.8]
    \draw[thick] (0,0) .. controls ++(0,.5) and ++(0,-.5) .. (.75,1);
    \draw[thick] (.75,0) .. controls ++(0,.5) and ++(0,-.5) .. (0,1);
    \draw[thick,] (0,1 ) .. controls ++(0,.5) and ++(0,-.5) .. (.75,2);
    \draw[thick, ] (.75,1) .. controls ++(0,.5) and ++(0,-.5) .. (0,2);
    \node at (0,-.25) {$\;$};
    \node at (0,2.25) {$\;$};
\end{tikzpicture}}
 \;\; = \;\;
0
\]

\[
\hackcenter{\begin{tikzpicture}[scale=0.8]
    \draw[thick, ] (0,0) .. controls (0,.75) and (.75,.75) .. (.75,1.5)
        node[pos=.25, shape=coordinate](DOT){};
    \draw[thick,] (.75,0) .. controls (.75,.75) and (0,.75) .. (0,1.5);
    \filldraw  (DOT) circle (2.5pt);
\end{tikzpicture}}
\quad - \quad
\hackcenter{\begin{tikzpicture}[scale=0.8]
    \draw[thick, ] (0,0) .. controls (0,.75) and (.75,.75) .. (.75,1.5)
        node[pos=.75, shape=coordinate](DOT){};
    \draw[thick, ] (.75,0) .. controls (.75,.75) and (0,.75) .. (0,1.5);
    \filldraw  (DOT) circle (2.5pt);
\end{tikzpicture}}
\quad = \quad
\hackcenter{\begin{tikzpicture}[scale=0.8]
    \draw[thick, ] (0,0) -- (0,1.5);
        \draw[thick, ] (.75,0) -- (.75,1.5);
\end{tikzpicture}} \;
\qquad = \quad
\hackcenter{\begin{tikzpicture}[scale=0.8]
    \draw[thick,  ] (0,0) .. controls (0,.75) and (.75,.75) .. (.75,1.5);
    \draw[thick,  ] (.75,0) .. controls (.75,.75) and (0,.75) .. (0,1.5)
        node[pos=.75, shape=coordinate](DOT){};
    \filldraw  (DOT) circle (2.75pt);
\end{tikzpicture}}
\quad - \quad
\hackcenter{\begin{tikzpicture}[scale=0.8]
    \draw[thick,  ] (0,0) .. controls (0,.75) and (.75,.75) .. (.75,1.5);
    \draw[thick, ] (.75,0) .. controls (.75,.75) and (0,.75) .. (0,1.5)
        node[pos=.25, shape=coordinate](DOT){};
      \filldraw  (DOT) circle (2.75pt);
\end{tikzpicture}}
\]

It is not hard to show that these relations imply the following.

\begin{lemma}
\label{diff1}
The following equalities hold in the nilHecke algebra.
\[
\hackcenter{\begin{tikzpicture}[scale=0.8]
    \draw[thick] (0,0) .. controls ++(0,.5) and ++(0,-.5) .. (.75,1);
    \draw[thick] (.75,0) .. controls ++(0,.5) and ++(0,-.5) .. (0,1)  node[pos=1, shape=coordinate](DOT1){};
    \draw[thick,] (0,1 ) .. controls ++(0,.5) and ++(0,-.5) .. (.75,2);
    \draw[thick, ] (.75,1) .. controls ++(0,.5) and ++(0,-.5) .. (0,2);
    \node at (0,-.25) {$\;$};
    \node at (0,2.25) {$\;$};
    \filldraw  (DOT1) circle (3.5pt);
    \node at (-.25,1) {$\scs d$};
\end{tikzpicture}}
 \;\; = \;\;
\sum_{A+B=d-1}
\hackcenter{\begin{tikzpicture}[scale=0.8]
    \draw[thick,  ] (0,0) .. controls (0,.75) and (.75,.75) .. (.75,1.5) node[pos=.25, shape=coordinate](DOT1){};
    \draw[thick, ] (.75,0) .. controls (.75,.75) and (0,.75) .. (0,1.5) node[pos=.25, shape=coordinate](DOT2){};
      \filldraw  (DOT1) circle (2.75pt);
      \filldraw  (DOT2) circle (2.75pt);
    \node at (-.25,.5) {$\scs A$};
    \node at (1,.5) {$\scs B$};
\end{tikzpicture}}
\qquad
\hackcenter{\begin{tikzpicture}[scale=0.8]
    \draw[thick] (0,0) .. controls ++(0,.5) and ++(0,-.5) .. (.75,1)  node[pos=1, shape=coordinate](DOT1){};
    \draw[thick] (.75,0) .. controls ++(0,.5) and ++(0,-.5) .. (0,1);
    \draw[thick,] (0,1 ) .. controls ++(0,.5) and ++(0,-.5) .. (.75,2);
    \draw[thick, ] (.75,1) .. controls ++(0,.5) and ++(0,-.5) .. (0,2);
    \node at (0,-.25) {$\;$};
    \node at (0,2.25) {$\;$};
    \filldraw  (DOT1) circle (3.5pt);
    \node at (1.125,1) {$\scs d$};
\end{tikzpicture}}
 \;\; = \;\;
-
\sum_{A+B=d-1}
\hackcenter{\begin{tikzpicture}[scale=0.8]
    \draw[thick,  ] (0,0) .. controls (0,.75) and (.75,.75) .. (.75,1.5) node[pos=.25, shape=coordinate](DOT1){};
    \draw[thick, ] (.75,0) .. controls (.75,.75) and (0,.75) .. (0,1.5) node[pos=.25, shape=coordinate](DOT2){};
      \filldraw  (DOT1) circle (2.75pt);
      \filldraw  (DOT2) circle (2.75pt);
    \node at (-.25,.5) {$\scs A$};
    \node at (1,.5) {$\scs B$};
\end{tikzpicture}}
\]
\end{lemma}

\subsection{Thick calculus}

It is convenient to introduce an enhanced diagrammatic calculus for the nilHecke algebra -- the so-called `thick calculus'.  This notation is helpful for providing a graphical calculus for the category of finitely generated projective $\nh_n$-modules.
We are primarily interested in the category of graded projective modules over the nilHecke algebra.  We can access this category diagrammatically by viewing $\nh_n$ as a one object $\Bbbk$-linear category where the morphisms are the elements of $\nh_n$.  We then pass to the Karoubi envelope allowing us to split idempotents in the diagrammatic language.
We review some of the relevant details here, but the reader is referred to \cite{KLMS} for more details.

\subsubsection{Box diagrams}
For any composition $\mu=(\mu_1,\dots, \mu_n)$ write $\und{x}^{\mu}:= x_1^{\mu_1} x_2^{\mu_2} \dots x_n^{\mu_n}$.  We depict these diagrammatically as
\begin{equation}
  \und{x}^{\mu}  \;\; = \;\;
  \hackcenter{\begin{tikzpicture}
    \draw[thick, ] (-1.2,0) -- (-1.2,1.5);
    \draw[thick, ] (-.6,0) -- (-.6,1.5);
    \draw[thick,  ] (.6,0) -- (.6,1.5);
        \draw[thick,  ] (1.2,0) -- (1.2,1.5);
    \node at (-1.2,.75) {$\bullet$};
    \node at (-1.43,.9) {$\scs \mu_1$};
    \node at (-.6,.75) {$\bullet$};
    \node at (-.83,.9) {$\scs \mu_2$};
    \node at (.6,.75) {$\bullet$};
    \node at (.27,.9) {$\scs \mu_{n-1}$};
    \node at (1.2,.75) {$\bullet$};
    \node at (1.5,.9) {$\scs \mu_n$};
    \node at (0, .35) {$\cdots$};
\end{tikzpicture}}
\;\; = \;\;
  \hackcenter{\begin{tikzpicture}
    \draw[thick, ] (-1.2,0) -- (-1.2,1.5);
    \draw[thick, ] (-.6,0) -- (-.6,1.5);
    \draw[thick,  ] (.6,0) -- (.6,1.5);
        \draw[thick, ] (1.2,0) -- (1.2,1.5);
    \node[draw, fill=white!20 ,rounded corners ] at (0,.75) {$ \qquad \quad \und{x}^{\mu} \quad \qquad$};
    \node at (0, .35) {$\cdots$};
\end{tikzpicture}}
\end{equation}

Since the center of the nilHecke algebra $Z(\nh_n)$ is isomorphic to the ring of symmetric polynomials $\Z[x_1,\dots, x_n]^{S_n}$ it is convenient to set notation for certain bases of symmetric polynomials.
The box labelled ${\sf h}_d$ and ${\sf e}_d$ represents the $d$-th complete and elementary symmetric function in $n$ variables, respectively:
\[
\hackcenter{
\begin{tikzpicture}[scale=.9]
\draw (0,0) -- (0,1.5)[thick, ];
\draw (1,0) -- (1,1.5)[thick, ];
    \filldraw[white] (-.25,.5) -- (1.25,.5) -- (1.25,1) -- (-.25,1) -- (-.25,.5);
    \draw (-.25,.5) -- (1.25,.5) -- (1.25,1) -- (-.25,1) -- (-.25,.5);
        \node at (.5,.75) {$\scs {\sf h}_{d}$};
        \node at (.5,.25) {$\cdots$};
        \node at (.5,-.35) {$\underbrace{\hspace{1cm}}_{n}$};
        \node at (.5,2) {};
\end{tikzpicture}}
=
\sum_{d_1+\cdots+ d_{n}=d}\hackcenter{
\begin{tikzpicture}[scale=.9]
\draw (0,0) -- (0,1.25)[thick, ] node[pos=.5, shape=coordinate](DOT1){};
\draw (0.75,0) -- (0.75,1.25)[thick, ] node[pos=.5, shape=coordinate](DOT3){};
\draw (1.75,0) -- (1.75,1.25)[thick, ] node[pos=.5, shape=coordinate](DOT2){};
    \filldraw  (DOT3) circle (2.5pt);
    \filldraw  (DOT1) circle (2.5pt);
    \filldraw  (DOT2) circle (2.5pt);
        \node at (-.25,.675) {$\scs d_1$};
        \node at (0.5,.675) {$\scs d_2$};
        \node at (2.1,.675) {$\scs d_{n}$};
        \node at (1.25,.675) {$\scs \cdots$};
\end{tikzpicture}}
\qquad \qquad
\hackcenter{
\begin{tikzpicture}[scale=.9]
\draw (0,0) -- (0,1.5)[thick, ];
\draw (1,0) -- (1,1.5)[thick, ];
    \filldraw[white] (-.25,.5) -- (1.25,.5) -- (1.25,1) -- (-.25,1) -- (-.25,.5);
    \draw (-.25,.5) -- (1.25,.5) -- (1.25,1) -- (-.25,1) -- (-.25,.5);
        \node at (.5,.75) {$\scs {\sf e}_{d}$};
        \node at (.5,.25) {$\cdots$};
        \node at (.5,-.35) {$\underbrace{\hspace{1cm}}_{n}$};
        \node at (.5,2) {};
\end{tikzpicture}}
=
\sum_{\overset{d_1+\cdots +d_{n}=d}{ \scs 0 \leq d_i\leq 1}}\hackcenter{
\begin{tikzpicture}[scale=.9]
\draw (0,0) -- (0,1.25)[thick, ] node[pos=.5, shape=coordinate](DOT1){};
\draw (0.75,0) -- (0.75,1.25)[thick, ] node[pos=.5, shape=coordinate](DOT3){};
\draw (1.75,0) -- (1.75,1.25)[thick, ] node[pos=.5, shape=coordinate](DOT2){};
    \filldraw  (DOT3) circle (2.5pt);
    \filldraw  (DOT1) circle (2.5pt);
    \filldraw  (DOT2) circle (2.5pt);
        \node at (-.25,.675) {$\scs d_1$};
        \node at (0.5,.675) {$\scs d_2$};
        \node at (2.1,.675) {$\scs d_{n}$};
        \node at (1.25,.675) {$\scs \cdots$};
\end{tikzpicture}}
\]
More generally, we denote the Schur polynomial associated to a partition $\mu = (\mu_1, \dots, \mu_n)$ by a box labelled $s_{\mu}$.

We will make use of the following crossing diagrams.
\begin{equation}
\label{diagramsCk1}
\hackcenter{
\begin{tikzpicture}[scale=.75]
\draw[thick,] (0,0) to (0,2);
\draw[fill=white!20,] (-.4,1.5) rectangle (.4,.5);
\node at (0,1) {$C_1$};
\end{tikzpicture}}
=
\hackcenter{
\begin{tikzpicture}[scale=.75]
\draw[thick,] (0,0) to (0,2);
\draw[fill=white!20,] (-.4,1.5) rectangle (.4,.5);
\node at (0,1) {$\rotatebox{180}{C}_1$};
\end{tikzpicture}}
=
\hackcenter{
\begin{tikzpicture}[scale=.75]
\draw[thick,] (0,0) to (0,2);
\end{tikzpicture}}
~,
\qquad
\hackcenter{
\begin{tikzpicture}[scale=.75]
\draw[thick,] (-.9,2) -- (-.9,4);
\node at (-.3,2.25) {$\cdots$};
\node at (-.3,3.85) {$\cdots$};
\draw[thick,] (.3,2) -- (.3, 4);
\draw[thick,] (.9,2) -- (.9, 4);
\draw[fill=white!20,] (-1.1,2.5) rectangle (1.1,3.5);
\node at (0,3) {$C_k$};
\end{tikzpicture}}
\; =\;
\hackcenter{
\begin{tikzpicture}[scale=.75]
\draw[thick,] (1.2,0) .. controls ++(0,1) and ++(0,-1) ..  (-1.2,2);
\draw[thick,] (-1.2,0) .. controls ++(0,1) and ++(0,-1) .. (-.6,2);
\draw[thick,] (-.6,0) .. controls ++(0,1) and ++(0,-1) ..  (0,2);
\draw[thick,] (.6,0) .. controls ++(0,1) and ++(0,-1) ..  (1.2,2);
\node at (.6,1.5) {$\cdots$};
\node at (.15,.5) {$\cdots$};
\end{tikzpicture}}
~,
\quad
\hackcenter{
\begin{tikzpicture}[scale=.75]
\draw[thick,] (-.9,2) -- (-.9,4);
\node at (-.3,2.25) {$\cdots$};
\node at (-.3,3.85) {$\cdots$};
\draw[thick,] (.3,2) -- (.3, 4);
\draw[thick,] (.9,2) -- (.9, 4);
\draw[fill=white!20,] (-1.1,2.5) rectangle (1.1,3.5);
\node at (0,3) {$\rotatebox{180}{C}_k$};
\end{tikzpicture}}
=
\hackcenter{
\begin{tikzpicture}[scale=.75]
\draw[thick,] (-1.2,0) .. controls ++(0,1) and ++(0,-1) ..  (1.2,2);
\draw[thick,] (1.2,0) .. controls ++(0,1) and ++(0,-1) .. (.6,2);
\draw[thick,] (.6,0) .. controls ++(0,1) and ++(0,-1) ..  (0,2);
\draw[thick,] (-.6,0) .. controls ++(0,1) and ++(0,-1) ..  (-1.2,2);
\node at (-.6,1.5) {$\cdots$};
\node at (-.15,.5) {$\cdots$};
\end{tikzpicture}}
\quad
\hbox{for $k>1$,}
\end{equation}

\begin{equation}
\label{diagramsCk2}
\hackcenter{
\begin{tikzpicture}[scale=.75]
\draw[thick,] (0,0) to (0,2);
\draw[fill=white!20,] (-.4,1.5) rectangle (.4,.5);
\node at (0,1) {$D_1$};
\end{tikzpicture}}
=
\hackcenter{
\begin{tikzpicture}[scale=.75]
\draw[thick,] (0,0) to (0,2);
\end{tikzpicture}}
~,
\qquad
\hackcenter{
\begin{tikzpicture}[scale=.75]
\draw[thick,] (-.9,2) -- (-.9,4);
\node at (-.3,2.25) {$\cdots$};
\node at (-.3,3.85) {$\cdots$};
\draw[thick,] (.3,2) -- (.3, 4);
\draw[thick,] (.9,2) -- (.9, 4);
\draw[fill=white!20,] (-1.1,2.5) rectangle (1.1,3.5);
\node at (0,3) {$D_k$};
\end{tikzpicture}}
\;\; = \;\;
\hackcenter{
\begin{tikzpicture}[scale=.75]
\draw[thick,] (-.9,2) -- (-.9,4);
\node at (-.3,2.25) {$\cdots$};
\node at (-.3,4) {$\cdots$};
\draw[thick,] (.3,2) -- (.3, 4);
\draw[thick,] (.9,2) -- (.9, 4);
\draw[fill=white!20,] (-1.1,3.05) rectangle (.5,3.75);
\node at (-.3,3.4) {$D_{k-1}$};
\draw[fill=white!20,] (-1.1,2.95) rectangle (1.1,2.25);
\node at (0,2.6) {$\rotatebox{180}{C}_k$};
\end{tikzpicture}}
\quad
\hbox{for $k>1$,}
\end{equation}
which are used to define a minimal idempotent $e_k$ in $\nh_k$ via
\begin{equation}
\label{diagramsCk3}
\hackcenter{
\begin{tikzpicture}[scale=.75]
\draw[thick,] (-.9,2) -- (-.9,4);
\node at (-.3,2.25) {$\cdots$};
\node at (-.3,3.85) {$\cdots$};
\draw[thick,] (.3,2) -- (.3, 4);
\draw[thick,] (.9,2) -- (.9, 4);
\draw[fill=white!20,] (-1.1,2.5) rectangle (1.1,3.5);
\node at (0,3) {$e_k$};
\end{tikzpicture}}
=
\hackcenter{
\begin{tikzpicture}[scale=.75]
\draw[thick,] (-.9,2) -- (-.9,4);
\node at (-.3,2.15) {$\cdots$};
\node at (-.3,3.85) {$\cdots$};
\draw[thick,] (.3,2) -- (.3, 4);
\draw[thick,] (.9,2) -- (.9, 4);
\draw[fill=white!20,] (-1.1,2.4) rectangle (1.1,3.4);
\node at (0,3) {$D_k$};
\node at (-1.4,3.65) {\tiny $k-1$};
    \filldraw[black]  (-0.9,3.65) circle (2pt);
    \filldraw[black]  (0.3,3.65) circle (2pt);
\end{tikzpicture}}
\end{equation}
Note that for any reduced expression $w=s_{i_1}s_{i_2} \dots s_{i_{\ell}} \in S_k$ we write
\[
\partial_w := \partial_{i_1} \partial_{i_2} \dots \partial_{i_{\ell}}.
\]
The relations in the nilHecke algebra imply that $\partial_{w}$ is independent of the reduced expression for $w$.  In the language the element $D_k$ corresponds to $\partial_{w_0}$ where $w_0$ denotes the longest word in $S_k$.
Hence, $D_k e_k = \partial_{w_0}(x_1^{k-1}x_2^{k-2} \dots x_k^0) D_k = D_k$, which implies the idempotence of $e_k$.

\subsubsection{Thick calculus definitions}

We introduce a thick line carrying a label corresponding to the idempotent $e_k$. Then a for a map between thick labelled strands to be well defined it must be invariant under pre and post composing with the relevant idempotents.  In particular, the endomorphisms of a thick strand are given by multiplication by symmetric functions $x\in \sym_k$ since these commute with the idempotent $e_k$:
\[
\hackcenter{
\begin{tikzpicture} [scale=.75]
\draw[thick,  double, ](0,1.15) to(0,-1.15);
 \node at (0,1.35) {$\scs k$};
\node at (0,-1.35) {$\scs k $};
\end{tikzpicture}}
\;\; := \;\;
\hackcenter{\begin{tikzpicture} [scale=0.75]
\draw[thick,  ] (-.6,-.5) to (-.6,1.85);
\draw[thick,  ] (.6,-.5) to (.6,1.85);
\draw[thick,  ] (-1.2,-.5) to (-1.2,1.85);
\draw[thick,  ] (1.2,-.5) to (1.2,1.85);
\draw[fill=white!20,] (-1.3,.1) rectangle (1.3,1.25);
\node at (0,-.2) {$\dots$};
\node at (0,1.65) {$\dots$};
 \node at (0,.75) {$e_k$};
\end{tikzpicture}}~,
\qquad \qquad \quad
\hackcenter{
\begin{tikzpicture} [scale=.75]
\draw[thick,  double, ](0,1.15) to(0,-1.15);
\filldraw  (0,0) circle (3.0pt);
\node at (.35,0) {$x$};
 \node at (0,1.35) {$\scs k$};
\node at (0,-1.35) {$\scs k $};
\end{tikzpicture}} \; \; \text{for $x \in \sym_k$, }
\qquad \qquad \quad
\hackcenter{
\begin{tikzpicture} [scale=.75]
\draw[thick,  double, ](0,1.15) to(0,-1.15);
\filldraw  (0,.5) circle (3.0pt);
\node at (.35,.55) {$x$};
\filldraw  (0,-.6) circle (3.0pt);
\node at (.35,-.5) {$y$};
 \node at (0,1.35) {$\scs k$};
\node at (0,-1.35) {$\scs k $};
\end{tikzpicture}}
\;\; = \;\;
\hackcenter{
\begin{tikzpicture} [scale=.75]
\draw[thick,  double, ](0,1.15) to(0,-1.15);
\filldraw  (0,0) circle (3.0pt);
\node at (.45,.05) {$xy$};
 \node at (0,1.35) {$\scs k$};
\node at (0,-1.35) {$\scs k $};
\end{tikzpicture}}
\]
where the product $xy$ is well defined since $xe_kye_k = xy e_k$.  We have splitter maps
\[
\hackcenter{ \begin{tikzpicture} [scale=.75]
\draw[thick,  double, ](0,2).. controls ++(0,-.75) and ++(0,.3) ..(.6,1);
\draw[thick,  double,  ](1.2,2).. controls ++(0,-.75) and ++(0,.3) ..(.6,1) to (.6,0);;
 \node at (0,2.2) {$\scs a$};
 \node at (1.2,2.2) {$\scs b$};
\node at (.6,-.2) {$\scs a+b$};
\end{tikzpicture}}
\;\; := \quad
\hackcenter{\begin{tikzpicture} [scale=0.75]
\draw[thick,  ] (.4,-1.2) .. controls ++(0,.5) and ++(0,-.5) ..(-1.4,.1) to (-1.4,1.85);
\draw[thick,  ] (1.4,-1.2) .. controls ++(0,.6) and ++(0,-.5) .. (-.4,.1) to (-.4,1.85);
\draw[thick,  ] (-.4,-1.2) .. controls ++(0,.5) and ++(0,-.5) ..(1.4,.1) to (1.4,1.85);
\draw[thick,  ] (-1.4,-1.2) .. controls ++(0,.6) and ++(0,-.5) ..(.4,.1) to (.4,1.85);
\draw[fill=white!20,] (-1.6,.3) rectangle (-.2,1.25);
\draw[fill=white!20,] (1.6,.3) rectangle (.2,1.25);
\node at (-.9,1.5) {$\dots$};
\node at (.9,1.5) {$\dots$};
\node at (-.8,-1.1) {$\dots$};
\node at (.9,-1.1) {$\dots$};
 \node at (-.9,.75) {$e_a$};
 \node at (.9,.75) {$e_b$};
\end{tikzpicture}}
\qquad \qquad
\hackcenter{ \begin{tikzpicture} [scale=.75]
\draw[thick,  double  ](0,-2).. controls ++(0,.75) and ++(0,-.3) ..(.6,-1);
\draw[thick,  double, ](1.2,-2).. controls ++(0,.75) and ++(0,-.3) ..(.6,-1) to (.6,0);;
 \node at (0,-2.2) {$\scs b$};
 \node at (1.2,-2.2) {$\scs a$};
\node at (.6,.2) {$\scs a+b$};
\end{tikzpicture}}
\;\; := \quad
\hackcenter{\begin{tikzpicture} [scale=0.75]
\draw[thick,  ] (-.6,-.5) to (-.6,1.85);
\draw[thick,  ] (.6,-.5) to (.6,1.85);
\draw[thick,  ] (-1.2,-.5) to (-1.2,1.85);
\draw[thick,  ] (1.2,-.5) to (1.2,1.85);
\draw[fill=white!20,] (-1.3,.1) rectangle (1.3,1.25);
\node at (0,-.2) {$\dots$};
\node at (0,1.65) {$\dots$};
 \node at (0,.75) {$e_{a+b}$};
\end{tikzpicture}}
\]
These maps satisfy (co)associativity relations making it possible to define
\[
\hackcenter{
\begin{tikzpicture} [scale=0.75]
\draw[thick,  double] (0,1.15) to (0,1.85);
\node at (0,3.2) {$\dots$};
\draw[thick,  ] (0,1.85) .. controls ++(.25,.1) and ++(0,-.5) .. (.6,3.5);
\draw[thick,  ] (0,1.85) .. controls ++(-.25,.1) and ++(0,-.5) .. (-.6,3.5);
\draw[thick,  ] (0,1.85) .. controls ++(.25,.1) and ++(0,-.7) .. (1.2,3.5);
\draw[thick,  ] (0,1.85) .. controls ++(-.25,.1) and ++(0,-.7) .. (-1.2,3.5);
%
  \node at (0,.95) {$\scs k$};
\end{tikzpicture}}
\;\; := \quad
\hackcenter{\begin{tikzpicture} [scale=0.75]
\draw[thick,  ] (-.6,-.5) to (-.6,1.85);
\draw[thick,  ] (.6,-.5) to (.6,1.85);
\draw[thick,  ] (-1.2,-.5) to (-1.2,1.85);
\draw[thick,  ] (1.2,-.5) to (1.2,1.85);
\draw[fill=white!20,] (-1.3,.1) rectangle (1.3,1.25);
\node at (0,-.2) {$\dots$};
\node at (0,1.65) {$\dots$};
 \node at (0,.75) {$D_k$};
\end{tikzpicture}}
\qquad \qquad \quad
\hackcenter{
\begin{tikzpicture} [scale=0.75]
\draw[thick,  ] (-.6,-.5) .. controls ++(0,.5) and ++(-.4,-.1) .. (0,1.15);
\draw[thick,  ] (.6,-.5) .. controls ++(0,.5) and ++(.4,-.1) .. (0,1.15);
\draw[thick,  ] (-1.2,-.5) .. controls ++(0,.7) and ++(-.4,-.1) .. (0,1.15);
\draw[thick,  ] (1.2,-.5) .. controls ++(0,.7) and ++(.4,-.1) .. (0,1.15);
\draw[thick,  double, ] (0,1.15) to (0,1.85);
\node at (0,0) {$\dots$};
  \node at (0,2.05) {$\scs k$};
\end{tikzpicture}} \;\; := \quad
\hackcenter{\begin{tikzpicture} [scale=0.75]
\draw[thick,  ] (-.6,-.5) to (-.6,1.85);
\draw[thick,  ] (.6,-.5) to (.6,1.85);
\draw[thick,  ] (-1.2,-.5) to (-1.2,1.85);
\draw[thick,  ] (1.2,-.5) to (1.2,1.85);
\draw[fill=white!20,] (-1.3,.1) rectangle (1.3,1.25);
\node at (0,-.2) {$\dots$};
\node at (0,1.65) {$\dots$};
 \node at (0,.75) {$e_k$};
\end{tikzpicture}}
\]
unambiguously.  For any Schur polynomial $s_{\mu}$ corresponding to the partition $\mu = (\mu_1, \dots , \mu_k)$  one can show that
\[
\hackcenter{
\begin{tikzpicture} [scale=.75]
\draw[thick,  double, ](0,1.15) to(0,-1.15);
\filldraw  (0,0) circle (3.0pt);
\node at (.5,.05) {$s_{\mu}$};
 \node at (0,1.35) {$\scs k$};
\node at (0,-1.35) {$\scs k $};
\end{tikzpicture}}
\; \; =\;\;
\hackcenter{
\begin{tikzpicture} [scale=0.65]
\draw[thick,  ] (0,-1.15).. controls++(-.4,.1)and ++(0,-.5).. (-.6,0) .. controls ++(0,.5) and ++(-.4,-.1) .. (0,1.15);
\draw[thick,  ] (0,-1.15).. controls ++(.4,.1) and ++(0,-.5).. (.6,0) .. controls ++(0,.5) and ++(.4,-.1) .. (0,1.15);
\draw[thick,  ]  (0,-1.15).. controls ++(-.4,.1)and ++(0,-.7) .. (-1.2,0) .. controls ++(0,.7) and ++(-.4,-.1) .. (0,1.15);
\draw[thick,  ] (0,-1.15) .. controls ++(.4,.1) and ++(0,-.7) .. (1.2,0) .. controls ++(0,.7) and ++(.4,-.1) .. (0,1.15);
\node[draw, fill=white!20 ,rounded corners ] at (0,.3 ) {$ \qquad \und{x}^{\mu+\delta}\qquad $};
\draw[thick,  double, ] (0,1.15) to (0,1.85);
\draw[thick,  double ] (0,-1.15) to (0,-1.65);
\node at (0,-.4) {$\dots$};
  \node at (0,2.05) {$\scs k$};
\end{tikzpicture}}
\]
where $\mu+\delta$ is the partition $(\mu_1+k-1, \mu_2+k-2, \dots, \mu_k+ 0)$.  Symmetric functions can be slid through splitters via the relations
\begin{equation} \label{eq:splitters}
\hackcenter{ \begin{tikzpicture} [scale=.75]
\draw[thick,  double, ](0,2).. controls ++(0,-.75) and ++(0,.3) ..(.6,1);
\draw[thick,  double,  ](1.2,2).. controls ++(0,-.75) and ++(0,.3) ..(.6,1) to (.6,0);;
 \node at (0,2.2) {$\scs b$};
 \node at (1.2,2.2) {$\scs a$};
\node at (.6,-.2) {$\scs a+b$};
\filldraw  (.6,.55) circle (3.0pt);
    \node at (.15,.55) {${\sf e}_{j}$};
\end{tikzpicture}}
\;\; = \;\; \sum_{x+y=j}
\hackcenter{ \begin{tikzpicture} [scale=.75]
\draw[thick,  double, ](0,2).. controls ++(0,-.75) and ++(0,.3) ..(.6,.55);
\draw[thick,  double,  ](1.2,2).. controls ++(0,-.75) and ++(0,.3) ..(.6,.55) to (.6,0);;
 \node at (0,2.2) {$\scs b$};
 \node at (1.2,2.2) {$\scs a$};
\node at (.6,-.2) {$\scs a+b$};
\filldraw  (.15,1.35 ) circle (3.0pt);
    \node at (-.3,1.35 ) {${\sf e}_{x}$};
\filldraw  (1.05,1.35 ) circle (3.0pt);
    \node at (1.5,1.35) {${\sf e}_{y}$};
\end{tikzpicture}}
\qquad \qquad
\hackcenter{ \begin{tikzpicture} [scale=.75]
\draw[thick,  double  ](0,-2).. controls ++(0,.75) and ++(0,-.3) ..(.6,-1);
\draw[thick,  double, ](1.2,-2).. controls ++(0,.75) and ++(0,-.3) ..(.6,-1) to (.6,0);;
 \node at (0,-2.2) {$\scs b$};
 \node at (1.2,-2.2) {$\scs a$};
\node at (.6,.2) {$\scs a+b$};
\filldraw  ( .6,-.65 ) circle (3.0pt);
    \node at ( .15,-.65 ) {${\sf e}_{j}$};
\end{tikzpicture}}
\;\; = \;\; \sum_{x+y=j}
\hackcenter{ \begin{tikzpicture} [scale=.75]
\draw[thick,  double,  ](0,-2).. controls ++(0,.75) and ++(0,-.3) ..(.6,-.55);
\draw[thick,  double,   ](1.2,-2).. controls ++(0,.75) and ++(0,-.3) ..(.6,-.55) to (.6,0);;
 \node at (0,-2.2) {$\scs b$};
 \node at (1.2,-2.2) {$\scs a$};
\node at (.6,.2) {$\scs a+b$};
\filldraw  (.1,-1.35 ) circle (3.0pt);
    \node at (-.3,-1.35 ) {${\sf e}_{x}$};
\filldraw   (1.1,-1.35 ) circle (3.0pt);
    \node at  (1.5,-1.35 ) {${\sf e}_{y}$};
\end{tikzpicture}}
\end{equation}
and more generally for any $x\in \sym_k$
\begin{equation}\label{eq:elem_slide}
\hackcenter{
\begin{tikzpicture} [scale=0.75]
\draw[thick,  double] (0,1.15) to (0,1.85);
\node at (0,3.2) {$\dots$};
\draw[thick,  ] (0,1.85) .. controls ++(.25,.1) and ++(0,-.5) .. (.6,3.5);
\draw[thick,  ] (0,1.85) .. controls ++(-.25,.1) and ++(0,-.5) .. (-.6,3.5);
\draw[thick,  ] (0,1.85) .. controls ++(.25,.1) and ++(0,-.7) .. (1.2,3.5);
\draw[thick,  ] (0,1.85) .. controls ++(-.25,.1) and ++(0,-.7) .. (-1.2,3.5);
\node[draw, fill=white!20 ,rounded corners ] at (0,2.7) {$ \qquad x\qquad $};
  \node at (0,.95) {$\scs k$};
\end{tikzpicture}}
\;\; = \;\;
\hackcenter{
\begin{tikzpicture} [scale=0.75]
\draw[thick,  double] (0,1.15) to (0,2.3);
\node at (0,3.2) {$\dots$};
\draw[thick,  ] (0,2.3) .. controls ++(.25,.1) and ++(0,-.5) .. (.6,3.5);
\draw[thick,  ] (0,2.3) .. controls ++(-.25,.1) and ++(0,-.5) .. (-.6,3.5);
\draw[thick,  ] (0,2.3) .. controls ++(.25,.1) and ++(0,-.7) .. (1.2,3.5);
\draw[thick,  ] (0,2.3) .. controls ++(-.25,.1) and ++(0,-.7) .. (-1.2,3.5);
\filldraw   (0,1.75) circle (3.0pt);
    \node at  (-.4,1.75) {${\sf e}_{x}$};
  \node at (0,.95) {$\scs k$};
\end{tikzpicture}}
\qquad \qquad
\hackcenter{
\begin{tikzpicture} [scale=0.75]
\draw[thick,  ] (-.6,-.5) .. controls ++(0,.5) and ++(-.4,-.1) .. (0,1.15);
\draw[thick,  ] (.6,-.5) .. controls ++(0,.5) and ++(.4,-.1) .. (0,1.15);
\draw[thick,  ] (-1.2,-.5) .. controls ++(0,.7) and ++(-.4,-.1) .. (0,1.15);
\draw[thick,  ] (1.2,-.5) .. controls ++(0,.7) and ++(.4,-.1) .. (0,1.15);
\node[draw, fill=white!20 ,rounded corners ] at (0,.3 ) {$ \qquad x \qquad $};
\draw[thick,  double, ] (0,1.15) to (0,1.85);
\node at (0,-.3) {$\dots$};
  \node at (0,2.05) {$\scs k$};
\end{tikzpicture}}
\;\; = \;\;
\hackcenter{
\begin{tikzpicture} [scale=0.75]
\draw[thick,  ] (-.6,-.5) .. controls ++(0,.5) and ++(-.4,-.1) .. (0,.7);
\draw[thick,  ] (.6,-.5) .. controls ++(0,.5) and ++(.4,-.1) .. (0,.7);
\draw[thick,  ] (-1.2,-.5) .. controls ++(0,.7) and ++(-.4,-.1) .. (0,.7);
\draw[thick,  ] (1.2,-.5) .. controls ++(0,.7) and ++(.4,-.1) .. (0,.7);
\draw[thick,  double, ] (0,.7) to (0,1.85);
\filldraw   (0,1.15) circle (3.0pt);
    \node at  (-.4,1.15) {${\sf e}_{x}$};
\node at (0,-.3) {$\dots$};
  \node at (0,2.05) {$\scs k$};
\end{tikzpicture}}
\end{equation}

\subsubsection{Primitive idempotents}

The set of sequences
\begin{equation}
  \Sq(n) := \{
  \und{\ell} = (\ell_1, \dots, \ell_{n-1}) \mid 0 \leq \ell_{\nu} \leq \nu, \;\; \nu =  1,2, \dots n-1
  \}
\end{equation}
has size $|\Sq(n)|=n!$.  Let $|\und{\ell}|=\sum_{\nu} \ell_{\nu}$, and set $\hat{\ell_j}=j-\ell_j$.
Define a composition with $n$-parts by
\begin{equation}
  \hat{\und{\ell}}=(0, \hat{\ell}_1,\dots,\hat{\ell}_{n-1})=  \left(0, 1-\ell_1, 2-\ell_2, \cdots,   n-1-\ell_{n-1} \right).
\end{equation}

Let ${\sf e}_r^{(a)}$ denote the $r$-th elementary symmetric polynomial in $a$ variables.   The {\em standard elementary monomials} are given  by
\begin{equation}
  {\sf e}_{\und{\ell}} := {\sf e}_{\ell_1}^{(1)}{\sf e}_{\ell_2}^{(2)} \dots {\sf e}_{\ell_{a-1}}^{(a-1)}.
\end{equation}
We depict these diagrammatically as
\begin{equation}
  {\sf e}_{\und{\ell}} \;\; = \;\;
  \hackcenter{\begin{tikzpicture}
    \draw[thick, ] (-1.2,-.5) -- (-1.2,2);
    \draw[thick, ] (-.6,-.5) -- (-.6,2);
    \draw[thick,  ] (.3,-.5) -- (.3,2);
        \draw[thick, ] (.9,-.5) -- (.9,2);
    \node at (-1.2,1.75) {$\bullet$};
    \node at (-1.4,1.8) {$\scs \ell_1$};
    \node[draw, fill=white!20 ,rounded corners ] at (-.9,1.3) {$ \; {\sf e}_{\ell_2} \; $};
    \node[draw, fill=white!20 ,rounded corners ] at (-.45,0) {$ \quad \;{\sf e}_{\ell_{n-1}} \quad\; $};
    \node at (-.4, .8) {$\vdots$};
    \node at (-.1, 1.8) {$\cdots$};
\end{tikzpicture}}
\;\; =: \;\;
  \hackcenter{\begin{tikzpicture}
    \draw[thick, ] (-1.2,-.5) -- (-1.2,2);
    \draw[thick, ] (-.6,-.5) -- (-.6,2);
    \draw[thick,  ] (.3,-.5) -- (.3,2);
        \draw[thick,  ] (.9,-.5) -- (.9,2);
    \node[draw, fill=white!20 ,rounded corners ] at (-.15,.75) {$ \qquad \;\;\;{\sf e}_{\und{\ell}} \qquad\;\;\; $};
    \node at (-.1, 1.8) {$\cdots$};
\end{tikzpicture}}
\end{equation}
These polynomials can be used to provide a complete set of primitive orthogonal idempotents decomposing the identity of $\nh_n$.
\begin{lemma}[\cite{KLMS} Proposition 2.5.3] \label{thm_Ea}
\begin{align}
  \hackcenter{\begin{tikzpicture}
    \draw[thick, ] (-1.2,0) -- (-1.2,1.5);
    \draw[thick, ] (-.6,0) -- (-.6,1.5);
    \draw[thick,  ] (.6,0) -- (.6,1.5);
        \draw[thick, ] (1.2,0) -- (1.2,1.5);
    \node at (0, .35) {$\cdots$};
\end{tikzpicture}}
\;\;
&=
\;\;
\sum_{\und{\ell} \in \Sq(n)}(-1)^{|\hat{\und{\ell}}|}\;\;
\hackcenter{
\begin{tikzpicture} [scale=0.75]
\draw[thick,  ] (-.6,-.5) .. controls ++(0,.5) and ++(-.4,-.1) .. (0,1.15);
\draw[thick,  ] (.6,-.5) .. controls ++(0,.5) and ++(.4,-.1) .. (0,1.15);
\draw[thick,  ] (-1.2,-.5) .. controls ++(0,.7) and ++(-.4,-.1) .. (0,1.15);
\draw[thick,  ] (1.2,-.5) .. controls ++(0,.7) and ++(.4,-.1) .. (0,1.15);
\draw[thick,  double] (0,1.15) to (0,1.85);
\draw[thick,  ] (0,1.85) .. controls ++(.25,.1) and ++(0,-.5) .. (.6,3.5);
\draw[thick,  ] (0,1.85) .. controls ++(-.25,.1) and ++(0,-.5) .. (-.6,3.5);
\draw[thick,  ] (0,1.85) .. controls ++(.25,.1) and ++(0,-.7) .. (1.2,3.5);
\draw[thick,  ] (0,1.85) .. controls ++(-.25,.1) and ++(0,-.7) .. (-1.2,3.5);
\node[draw, fill=white!20 ,rounded corners ] at (0,2.7) {$ \qquad {\sf e}_{\und{\ell}} \qquad $};
\node[draw, fill=white!20 ,rounded corners ] at (0,.35) {$ \qquad   \und{x}^{\hat{\und{\ell}}} \qquad $};
%
\end{tikzpicture}}
\\
  \hackcenter{
\begin{tikzpicture} [scale=0.75]
\draw[thick,  ] (0,-1.15).. controls++(-.4,.1)and ++(0,-.5).. (-.6,0) .. controls ++(0,.5) and ++(-.4,-.1) .. (0,1.15);
\draw[thick,  ] (0,-1.15).. controls ++(.4,.1) and ++(0,-.5).. (.6,0) .. controls ++(0,.5) and ++(.4,-.1) .. (0,1.15);
\draw[thick,  ]  (0,-1.15).. controls ++(-.4,.1)and ++(0,-.7) .. (-1.2,0) .. controls ++(0,.7) and ++(-.4,-.1) .. (0,1.15);
\draw[thick,  ] (0,-1.15) .. controls ++(.4,.1) and ++(0,-.7) .. (1.2,0) .. controls ++(0,.7) and ++(.4,-.1) .. (0,1.15);
\node[draw, fill=white!20 ,rounded corners ] at (0,.55 ) {$ \qquad \scs  {\sf e}_{\und{\ell'}}\qquad $};
\node[draw, fill=white!20 ,rounded corners ] at (0,-.5 ) {$ \qquad\; \scs  \und{x}^{\hat{\und{\ell}}}\qquad $};
\draw[thick,  double, ] (0,1.15) to (0,1.85);
\draw[thick,  double ] (0,-1.15) to (0,-1.65);
\node at (0,0) {$\dots$};
  \node at (0,2.05) {$\scs n$};\;
\end{tikzpicture}}
\;\; &= \;\; \delta_{\und{\ell},\und{\ell'}}
\hackcenter{
\begin{tikzpicture} [scale=0.75]
\draw[thick,  double, ] (0,-1.65) to (0,1.85);
  \node at (0,2.05) {$\scs n$};\;
\end{tikzpicture}}
\end{align}
\end{lemma}

\begin{remark} \label{rem:dual-bases}
The polynomial ring $\Z[x_1,\dots,x_n]$ is a free module of rank $n!$ over the ring $\Lambda_n=\Z[x_1,\dots,x_n]^{S_n}$ of symmetric functions \cite[Proposition 2.5.5]{Man}.  There is a $\Lambda_n$ bilinear form on $\Z[x_1, \dots, x_n]$ defined by defined by $(x,y)=\partial_{w_0}(xy)$, for $w_0$ the longest element in the symmetric group.
The significance of the sets $\{{\sf e}_{\und{\ell}} \mid \und{\ell} \in \Sq(n) \}$ and $\{ \und{x}^{\hat{\und{\ell}}} \mid \und{\ell} \in \Sq(n)\}$ are that they are dual bases for the polynomial ring $\Z[x_1,\dots,x_n]$ as a $\Lambda_n$-module with respect to this form.  Another dual set of bases are given by Schubert polynomials and dual Schubert polynomials  \cite[Proposition 2.5.7]{Man}.
\end{remark}

\subsubsection{Helpful thick calculus lemmas}
We collect here several lemmas that will be useful in establishing our main results in later sections.

\begin{lemma}[\cite{KLMS} Lemma 2.2.3]\label{diff}
\[
\hackcenter{
\begin{tikzpicture} [scale=.75]
\draw[thick, double ] (0,.25) to (0,.65);
\draw[thick, double] (0,.65) .. controls ++(-.7,0) and ++(-.7,0) .. (0,1.45);
\draw[thick] (0,.65) .. controls ++(.7,0) and ++(.7,0) .. (0,1.45) node[pos=.3, shape=coordinate](DOT1){};
\draw[thick, double  ] (0,1.45) to (0,1.75);
    \filldraw  (DOT1) circle (3.5pt);
\draw[thick,  ] (0,1.75) .. controls ++(.25,.1) and ++(0,-.5) .. (.3,2.5);
\draw[thick,  ] (0,1.75) .. controls ++(-.25,.1) and ++(0,-.5) .. (-.3,2.5);
\draw[thick,  ] (0,1.75) .. controls ++(.25,.1) and ++(0,-.7) .. (.9,2.5);
\draw[thick,  ] (0,1.75) .. controls ++(-.25,.1) and ++(0,-.7) .. (-.9,2.5);
\draw[thick,  ] (-.3,-.5) .. controls ++(0,.5) and ++(-.4,-.1) .. (0,.25);
\draw[thick,  ] (.3,-.5) .. controls ++(0,.5) and ++(.4,-.1) .. (0,.25);
\draw[thick,  ] (-.9,-.5) .. controls ++(0,.7) and ++(-.4,-.1) .. (0,.25);
\draw[thick,  ] (.9,-.5) .. controls ++(0,.7) and ++(.4,-.1) .. (0,.25);
    \node at (.5,1.45) {$\scs 1$};
    \node at (-.8,1.45) {$\scs k-1$};
    \node at (.7,.65) {$\scs d$};
\end{tikzpicture}}
\;\; = \;\;
\hackcenter{
\begin{tikzpicture} [scale=.75]
\draw[thick,] (-.9,-.5) to (-.9,1.5) .. controls ++(0,.5) and ++(0,-.5) .. (-.3,2.5);
\draw[thick,] (-.3,-.5) to (-.3,1.5) .. controls ++(0,.5) and ++(0,-.5) .. (.3,2.5);
\draw[thick,] (.3,-.5) to (.3,1.5) .. controls ++(0,.5) and ++(0,-.5) .. (.9,2.5);
\draw[thick,] (.9,-.5) to (.9,1.5) .. controls ++(0,.5) and ++(0,-.5) .. (-.9,2.5)  node[pos=0, shape=coordinate](DOT1){};
\draw[fill=white!20,] (-1.1,0) rectangle (1.1,1);
    \filldraw  (DOT1) circle (3.5pt);
    \node at (1.25,1.5) {$\scs d$};
 \node at (0,.5) {${D}_k$};
\end{tikzpicture}}
\;\; = \;\;
\begin{cases}
(-1)^{k-1}\hackcenter{
\begin{tikzpicture}[scale=.9]
\draw (0,0) -- (0,2.5)[thick, ];
\draw (1,0) -- (1,2.5)[thick, ];
    \filldraw[white] (-.25,.5) -- (1.25,.5) -- (1.25,1) -- (-.25,1) -- (-.25,.5);
    \draw (-.25,.5) -- (1.25,.5) -- (1.25,1) -- (-.25,1) -- (-.25,.5);
    \filldraw[white] (-.25,1.5) -- (1.25,1.5) -- (1.25,2) -- (-.25,2) -- (-.25,1.5);
    \draw (-.25,1.5) -- (1.25,1.5) -- (1.25,2) -- (-.25,2) -- (-.25,1.5);
        \node at (.5,1.675) {$\scs {\sf h}_{d-k+1}$};
        \node at (.5,.675) {$\scs D_{k}$};
\end{tikzpicture}}
&\hbox{if $d\geq k-1$,}
\\
0 &\hbox{otherwise.}
\end{cases}
\]
\end{lemma}

\begin{lemma}\label{C}
We have the following equalities involving the diagrams defined in \eqref{diagramsCk1}, \eqref{diagramsCk2}, and \eqref{diagramsCk3}.

\begin{equation}\label{dot-C}
\hackcenter{
\begin{tikzpicture}[scale=.75]
\draw[thick,] (-.9,2) -- (-.9,4);
\draw[thick,] (-.3,2) -- (-.3,4);
\draw[thick,] (.3,2) -- (.3,4);
\draw[thick,] (.9,2) -- (.9,4);
\draw[fill=white!20,] (-1.1,2.5) rectangle (1.1,3.5);
    \filldraw (-.9,3.75) circle (3pt);
\node at (0,3) {$C_k$};
\end{tikzpicture}}
\;\; - \;\;
\hackcenter{
\begin{tikzpicture}[scale=.75]
\draw[thick,] (-.9,2) -- (-.9,4);
\draw[thick,] (-.3,2) -- (-.3,4);
\draw[thick,] (.3,2) -- (.3,4);
\draw[thick,] (.9,2) -- (.9,4);
    \filldraw (.9,2.25) circle (3pt);
\draw[fill=white!20,] (-1.1,2.5) rectangle (1.1,3.5);
\node at (0,3) {$C_k$};
\end{tikzpicture}}
\;\; = \;\;
\sum_{\ell=1}^{k-1}\;
\hackcenter{
\begin{tikzpicture}[scale=.75]
\draw[thick,] (-1.5,2) -- (-1.5,4);
\draw[thick,] (-.3,2) -- (-.3,4);
\draw[thick,] (.3,2) -- (.3,4);
\draw[thick,] (1.5,2) -- (1.5,4);
\draw[fill=white!20,] (-1.7,2.5) rectangle (-.1,3.5);
\draw[fill=white!20,] (.1,2.5) rectangle (1.7,3.5);
\node at (-.9,3) {$C_{\ell}$};
\node at (.9,3) {$C_{k-\ell}$};
\end{tikzpicture}}
\end{equation}

\begin{equation}\label{cross-D}
\hackcenter{
\begin{tikzpicture}[scale=.75]
\draw[thick,] (-1.2,2) -- (-1.2,4.5);
\draw[thick,] (-.3,2) -- (-.3,3.5);
\draw[thick,] (.3,2) -- (.3,3.5);
\draw[thick,] (1.2,2) -- (1.2,4.5);
\node at (.75,4) {$\cdots$};
\node at (-.75,4) {$\cdots$};
\draw[thick,   ](-.3,3.5) .. controls ++(0,.3) and ++(0,-.3) .. (.3,4.5);
\draw[thick,   ](.3,3.5) .. controls ++(0,.3) and ++(0,-.3) .. (-.3,4.5);
\draw[fill=white!20,] (-1.4,2.5) rectangle (1.4,3.5);
\node at (0,3) {$D_k$};
\end{tikzpicture}}
\;\; =\;\;0,
\qquad \quad
\hackcenter{
\begin{tikzpicture}[scale=.75]
\draw[thick,] (-1.2,.5) -- (-1.2,4);
\draw[thick,] (-.3,.5) -- (-.3,4);
\draw[thick,] (.3,.5) -- (.3,4);
\draw[thick,] (1.2,.5) -- (1.2,4);
\node at (.75,3) {$\cdots$};
\draw[fill=white!20,] (-1.4,2.5) rectangle (-.1,3.5);
\node at (-.75,3) {$C_{\ell}$};
\draw[fill=white!20,] (-1.4,1) rectangle (1.4,2);
\node at (0,1.5) {$D_k$};
\end{tikzpicture}}
\;\; = \;\;0
\quad\hbox{for $2\leq \ell\leq k$.}
\end{equation}
\end{lemma}

\begin{proof}
The equations in \eqref{dot-C} and \eqref{cross-D} follow easily from the defining relations of the nilHecke algebra.
\end{proof}

\begin{lemma}\label{id-dots-C}
\[
\hackcenter{
\begin{tikzpicture} [scale=.75]
\draw[thick,  ] (1.2,0) -- (1.2,3);
\draw[thick,  ] (.6,0) -- (.6,3);
\draw[thick,  ] (-.6,0) -- (-.6,3);
\draw[thick,  ] (-1.2,0) -- (-1.2,3);
\node at (0,2.5) {$\cdots$};
\node[draw, fill=white!20] at (.3,1.675) {$ \quad\;   D_{k-1} \;\quad $};
\end{tikzpicture}}
\;\; = \;\;
\sum_{d=0}^{k-1}(-1)^d
\hackcenter{
\begin{tikzpicture} [scale=.75]
\draw[thick,  ] (1.2,0) -- (1.2,3);
\draw[thick,  ] (.6,0) -- (.6,3);
\draw[thick,  ] (-.6,0) -- (-.6,3);
\draw[thick,  ] (-1.2,0) -- (-1.2,3);
\node[draw, fill=white!20 ,rounded corners ] at (.3,.75) {$ \quad\;\; {\sf e}_{d} \quad\;\; $};
\node[draw, fill=white!20] at (0,1.675) {$ \qquad\;   D_k \;\qquad $};
\filldraw (-1.2,2.5) circle (3pt);
\node at (-2.,2.5) {$\scs k-1-d$};
\node at (0,2.5) {$\cdots$};
\end{tikzpicture}}
\]
\end{lemma}

\begin{proof}
By a variant of \cite[Lemma 2.4.5]{KLMS}  we have that
\[
\hackcenter{
\begin{tikzpicture} [scale=.75]
\draw[thick, double, ] (.6,0) -- (.6,3);
\draw[thick,  ] (-.6,0) -- (-.6,3);
\node at (.6,-.2) {$\scs k-1$};
\node at (-.6,-.2) {$\scs 1$};
\end{tikzpicture}}
\;\; = \;\;
\sum_{d=0}^{k-1}(-1)^d\;
\hackcenter{ \begin{tikzpicture} [scale=.75]
\draw[thick,   ](0,2).. controls ++(0,-.75) and ++(0,.3) ..(.6,.8)  node[pos=0.35, shape=coordinate](DOT1){};;
\draw[thick,  double,  ](1.2,2).. controls ++(0,-.75) and ++(0,.3) ..(.6,.8) to (.6,.3);
\draw[thick, ](0,-1).. controls ++(0,.75) and ++(0,-.3) ..(.6,.3);
\draw[thick,  double, ](1.2,-1).. controls ++(0,.75) and ++(0,-.3) ..(.6,.3);
 \node at (0,-1.2) {$\scs 1$};
 \node at (1.2,-1.2) {$\scs k-1$};
\node[draw, fill=white!20 ,rounded corners ] at ( 1.1,-.5 ) {$\scs  {\sf e}_{d}$};
    \filldraw  (DOT1) circle (2.5pt);
\node at (-.65,1.6) {$\scs k-1-d$};
\end{tikzpicture}}
\;\; =\;\;
\sum_{d=0}^{k-1}(-1)^d
\hackcenter{
\begin{tikzpicture} [scale=.75]
\draw[thick,  ] (1.2,-.1) -- (1.2,3.25);
\draw[thick,  ] (.6,-.1) -- (.6,3.25);
\draw[thick,  ] (-.6,-.1) -- (-.6,3.25);
\draw[thick,  ] (-1.2,-.1) -- (-1.2,3.25);
\node[draw, fill=white!20 ,rounded corners ] at (.3,.4) {$ \quad\;\; {\sf e}_{d} \quad\;\; $};
\node[draw, fill=white!20] at (0,1.1) {$ \qquad\;   e_k \;\qquad $};
\node[draw, fill=white!20] at (0,1.9) {$ \qquad\;   C_k \;\qquad $};
\node[draw, fill=white!20] at (.3,2.7) {$ \quad   e_{k-1} \quad $};
\filldraw (-1.2,2.75) circle (3pt);
\node at (-2.,2.75) {$\scs k-1-d$};
\end{tikzpicture}}
\]

Applying the diagram $D_{k-1}$ to the top of the diagrams on both sides, we have
\begin{eqnarray*}
\hackcenter{
\begin{tikzpicture} [scale=.75]
\draw[thick,  ] (1.2,-.1) -- (1.2,3.25);
\draw[thick,  ] (.6,-.1) -- (.6,3.25);
\draw[thick,  ] (-.6,-.1) -- (-.6,3.25);
\draw[thick,  ] (-1.2,-.1) -- (-1.2,3.25);
\node[draw, fill=white!20] at (.3,1.5) {$ \quad   D_{k-1} \quad $};
\node at (0,2.5) {$\cdots$};
\end{tikzpicture}}
&=&
\sum_{d=0}^{k-1}(-1)^d
\hackcenter{
\begin{tikzpicture} [scale=.75]
\draw[thick,  ] (1.2,-.1) -- (1.2,4.25);
\draw[thick,  ] (.6,-.1) -- (.6,4.25);
\draw[thick,  ] (-.6,-.1) -- (-.6,4.25);
\draw[thick,  ] (-1.2,-.1) -- (-1.2,4.25);
\node[draw, fill=white!20 ,rounded corners ] at (.3,.4) {$ \quad\;\; {\sf e}_{d} \quad\;\; $};
\node[draw, fill=white!20] at (0,1.1) {$ \qquad\;   e_k \;\qquad $};
\node[draw, fill=white!20] at (0,1.9) {$ \qquad\;   C_k \;\qquad $};
\node[draw, fill=white!20] at (.3,2.7) {$ \quad   e_{k-1} \quad $};
\node[draw, fill=white!20] at (.3,3.5) {$ \quad   D_{k-1} \quad $};
\filldraw (-1.2,2.75) circle (3pt);
\node at (-2.,2.75) {$\scs k-1-d$};
\end{tikzpicture}}
=
\sum_{d=0}^{k-1}(-1)^d
\hackcenter{
\begin{tikzpicture} [scale=.75]
\draw[thick,  ] (1.2,-.1) -- (1.2,3.5);
\draw[thick,  ] (.6,-.1) -- (.6,3.5);
\draw[thick,  ] (-.6,-.1) -- (-.6,3.5);
\draw[thick,  ] (-1.2,-.1) -- (-1.2,3.5);
\node[draw, fill=white!20 ,rounded corners ] at (.3,.4) {$ \quad\;\; {\sf e}_{d} \quad\;\; $};
\node[draw, fill=white!20] at (0,1.1) {$ \qquad\;   e_k \;\qquad $};
\node[draw, fill=white!20] at (0,1.9) {$ \qquad\;   C_k \;\qquad $};
\node[draw, fill=white!20] at (.3,2.7) {$ \quad   D_{k-1} \quad $};
\filldraw (-1.2,2.75) circle (3pt);
\node at (-2.,2.75) {$\scs k-1-d$};
\end{tikzpicture}}
\\&=&
\sum_{d=0}^{k-1}(-1)^d
\hackcenter{
\begin{tikzpicture} [scale=.75]
\draw[thick,  ] (1.2,-.1) -- (1.2,3);
\draw[thick,  ] (.6,-.1) -- (.6,3);
\draw[thick,  ] (-.6,-.1) -- (-.6,3);
\draw[thick,  ] (-1.2,-.1) -- (-1.2,3);
\node[draw, fill=white!20 ,rounded corners ] at (.3,.4) {$ \quad\;\; {\sf e}_{d} \quad\;\; $};
\node[draw, fill=white!20] at (0,1.1) {$ \qquad\;   e_k \;\qquad $};
\node[draw, fill=white!20] at (0,1.9) {$ \qquad\;   D_k \;\qquad $};
\filldraw (-1.2,2.5) circle (3pt);
\node at (-2.,2.5) {$\scs k-1-d$};
\node at (0,2.5) {$\cdots$};
\end{tikzpicture}}
=
\sum_{d=0}^{k-1}(-1)^d
\hackcenter{
\begin{tikzpicture} [scale=.75]
\draw[thick,  ] (1.2,0) -- (1.2,3);
\draw[thick,  ] (.6,0) -- (.6,3);
\draw[thick,  ] (-.6,0) -- (-.6,3);
\draw[thick,  ] (-1.2,0) -- (-1.2,3);
\node[draw, fill=white!20 ,rounded corners ] at (.3,.75) {$ \quad\;\; {\sf e}_{d} \quad\;\; $};
\node[draw, fill=white!20] at (0,1.675) {$ \qquad\;   D_k \;\qquad $};
\filldraw (-1.2,2.5) circle (3pt);
\node at (-2.,2.5) {$\scs k-1-d$};
\node at (0,2.5) {$\cdots$};
\end{tikzpicture}}
\end{eqnarray*}
\end{proof}

The following lemma  holds in a context more general the than the nilHecke algebra $\nh_k$.  We use it to establish results for the redotted Webster algebras defined in section~\ref{sec:reddotted}. We denote by a box labelled $X$ a region with appropriate inputs and outputs that we make absolutely no assumption about whatsoever.  The identity holds completely externally to any assumption about the content of $X$.

\begin{lemma}\label{reduction-EX} For an arbitrary diagram $X$, the identity
\[
\sum_{\und{\ell} \in \Sq(k)}(-1)^{|\hat{\und{\ell}}|}\;\;
\hackcenter{
\begin{tikzpicture} [scale=.75]
\draw[thick,  ] (1.2,0) -- (1.2,3);
\draw[thick,  ] (.6,0) -- (.6,3);
\draw[thick,  ] (-.6,0) -- (-.6,3);
\draw[thick,  ] (-1.2,0) -- (-1.2,3);
\node[draw, fill=white!20 ,rounded corners ] at (0,2.45) {$ \qquad\; {\sf e}_{\und{\ell}} \qquad\; $};
\node[draw, fill=white!20 ,rounded corners ] at (0,.7) {$ \qquad\;   \und{x}^{\hat{\und{\ell}}} \;\qquad $};
\node[draw, fill=white!20] at (0,1.675) {$ \qquad\;   X \;\qquad $};
\end{tikzpicture}}
\;\; =\;\;
\sum_{\und{\ell} \in \Sq(k-1)}(-1)^{|\hat{\und{\ell}}|}\;\;
\sum_{d=0}^{k-1}(-1)^d
\hackcenter{
\begin{tikzpicture} [scale=.75]
\draw[thick,  ] (1.2,-.75) -- (1.2,3) ;
\draw[thick,  ] (.6,-.75) -- (.6,3);
\draw[thick,  ] (-.6,-.75) -- (-.6,3);
\draw[thick,  ] (-1.2,-.75) -- (-1.2,3);
\node[draw, fill=white!20 ,rounded corners ] at (.3,2.45) {$ \quad\;\; {\sf e}_{\und{\ell}} \quad\;\; $};
\node[draw, fill=white!20 ,rounded corners ] at (.3,.75) {$ \quad\;\; {\sf e}_{d} \quad\;\; $};
\node[draw, fill=white!20 ,rounded corners ] at (.3,-.15) {$ \quad\;\;   \und{x}^{\hat{\und{\ell}}} \;\;\quad $};
\node[draw, fill=white!20] at (0,1.6) {$ \qquad\;   X \;\qquad $};
\filldraw (-1.2,2.5) circle (3pt);
\node at (-2.,2.5) {$\scs k-1-d$};
\end{tikzpicture}}
\]
holds.
\end{lemma}

\begin{proof}
For $\und{\ell}=(\ell_1,...,\ell_{k-1})\in \Sq(k)$, we have
\begin{eqnarray*}
{\sf e}_{\und{\ell}}(x_1,x_2,...,x_k)&=&\prod_{i=1}^{k-1}{\sf e}_{\ell_i}(x_1,x_2,...,x_i)
\\
&=&x_1^{\ell_1}\prod_{i=2}^{k-1}\left((1-\delta_{\ell_i i}){\sf e}_{\ell_i}(x_2,...,x_i)+(1-\delta_{\ell_i 0})\; x_1 {\sf e}_{\ell_i-1}(x_2,...,x_i)\right)
\\
&=&\sum_{\und{j}=(j_1,...,j_{k-2})\atop \in\{0,1\}^{k-2}}
x_1^{\ell_1+|\und{j}|}\sum_{\und{\ell}'\in \Sq(k-1)\atop \und{\ell}'+\und{j}=(\ell_2,...,\ell_{k-1})}
{\sf e}_{\und{\ell}'}(x_2,...,x_k).
\end{eqnarray*}
The second equality above follows from the standard relation of elementary symmetric functions
\[
{\sf e}_{\ell_i}(x_1,x_2,...,x_i) = {\sf e}_{\ell_i}(x_2,...,x_i) + x_1 {\sf e}_{\ell_i-1}(x_2,...,x_i).
\]

Letting $\und{j}=(j_1,...,j_{k-2}) \in\{0,1\}^{k-2}$ and $\und{\ell}'\in \Sq(k-1)$ such that $\und{\ell}'+\und{j}=(\ell_2,...,\ell_{k-1})$
we have
\begin{eqnarray*}
\und{x}^{\hat{\und{\ell}}}
\;=\; \prod _{i=1}^{k-1} x_{i+1}^{i-\ell_i}
\;=\; x_2^{1-\ell_1}\prod _{i=2}^{k-1} x_{i+1}^{i-\ell_{i-1}'-j_{i-1}}
\;=\; x_2^{1-\ell_1}x_3^{1-j_{1}}x_4^{1-j_{2}}\cdots x_k^{1-j_{k-2}} \prod _{i=2}^{k-1} x_{i+1}^{i-1-\ell_{i-1}'}.
\end{eqnarray*}

Therefore, we have
\begin{eqnarray*}
\sum_{\und{\ell} \in \Sq(k)}(-1)^{|\hat{\und{\ell}}|}\;\;
\hackcenter{
\begin{tikzpicture} [scale=.75]
\draw[thick,  ] (1.2,0) -- (1.2,3);
\draw[thick,  ] (.6,0) -- (.6,3);
\draw[thick,  ] (-.6,0) -- (-.6,3);
\draw[thick,  ] (-1.2,0) -- (-1.2,3);
\node[draw, fill=white!20 ,rounded corners ] at (0,2.45) {$ \qquad\; {\sf e}_{\und{\ell}} \qquad\; $};
\node[draw, fill=white!20 ,rounded corners ] at (0,.7) {$ \qquad\;   \und{x}^{\hat{\und{\ell}}} \;\qquad $};
\node[draw, fill=white!20] at (0,1.6) {$ \qquad\;   X \;\qquad $};
\end{tikzpicture}}
&=&
\sum_{\und{\ell} \in \Sq(k)}(-1)^{|\hat{\und{\ell}}|}
\sum_{\und{j}=(j_1,...,j_{k-2})\atop \in\{0,1\}^{k-2}}
\sum_{\und{\ell}'\in \Sq(k-1)\atop \und{\ell}'+\und{j}=(\ell_2,...,\ell_{k-1})}
\hackcenter{
\begin{tikzpicture} [scale=.75]
\draw[thick,  ] (1.5,-.5) -- (1.5,3);
\draw[thick,  ] (0,-.5) -- (0,3);
\draw[thick,  ] (-.4,-.5) -- (-.4,3);
\draw[thick,  ] (-1.5,-.5) -- (-1.5,3);
\filldraw (-1.5,2.5) circle (3pt);
\filldraw (1.5,1) circle (3pt);
\filldraw (-.4,1) circle (3pt);
\filldraw (0,1) circle (3pt);
\node[draw, fill=white!20 ,rounded corners ] at (.55,2.45) {$ \quad\;\; {\sf e}_{\und{\ell}'} \quad\;\; $};
\node[draw, fill=white!20 ,rounded corners ] at (.55,.2) {$ \quad\;\; \und{x}^{\hat{\und{\ell}'}} \quad\;\; $};
\node[draw, fill=white!20] at (0,1.675) {$ \qquad\;\;\;   X \qquad \;\;\;$};
\node at (-2.3,2.5) {$\scs \ell_1+|\und{j}|$};
\node at (-1,1) {$\scs 1-\ell_1$};
\node at (.6,1) {$\scs 1-j_1$};
\node at (2.3,1) {$\scs 1-j_{k-2}$};
\end{tikzpicture}}
\\&=&
\sum_{\und{\ell}'\in \Sq(k-1)}
\sum_{\und{j}=(j_1,...,j_{k-2})\atop \in\{0,1\}^{k-2}}
\sum_{\und{\ell} \in \Sq(k)\atop \und{\ell}'+\und{j}=(\ell_2,...,\ell_{k-1})}
(-1)^{|\hat{\und{\ell}}|}
\hackcenter{
\begin{tikzpicture} [scale=.75]
\draw[thick,  ] (1.5,-.5) -- (1.5,3);
\draw[thick,  ] (0,-.5) -- (0,3);
\draw[thick,  ] (-.4,-.5) -- (-.4,3);
\draw[thick,  ] (-1.5,-.5) -- (-1.5,3);
\filldraw (-1.5,2.5) circle (3pt);
\filldraw (1.5,1) circle (3pt);
\filldraw (-.4,1) circle (3pt);
\filldraw (0,1) circle (3pt);
\node[draw, fill=white!20 ,rounded corners ] at (.55,2.45) {$ \quad\;\; {\sf e}_{\und{\ell}'} \quad\;\; $};
\node[draw, fill=white!20 ,rounded corners ] at (.55,.2) {$ \quad\;\; \und{x}^{\hat{\und{\ell}'}} \quad\;\; $};
\node[draw, fill=white!20] at (0,1.675) {$ \qquad\;\;\; X \qquad\;\;\;$};
\node at (-2.3,2.5) {$\scs \ell_1+|\und{j}|$};
\node at (-1,1) {$\scs 1-\ell_1$};
\node at (.6,1) {$\scs 1-j_1$};
\node at (2.3,1) {$\scs 1-j_{k-2}$};
\end{tikzpicture}}
\end{eqnarray*}

We find that when $(\ell_1,|\und{j}|)=(0,0)$, this summation equals
\[
\sum_{\und{\ell}'\in \Sq(k-1)}
(-1)^{|\hat{\und{\ell}}'|-k+1}
\hackcenter{
\begin{tikzpicture} [scale=.75]
\draw[thick,  ] (1.5,-.75) -- (1.5,3);
\draw[thick,  ] (0,-.75) -- (0,3);
\draw[thick,  ] (-.4,-.75) -- (-.4,3);
\draw[thick,  ] (-1.5,-.75) -- (-1.5,3);
\node[draw, fill=white!20 ,rounded corners ] at (.55,2.45) {$ \quad\;\; {\sf e}_{\und{\ell}'} \quad\;\; $};
\node[draw, fill=white!20] at (0,1.675) {$ \qquad\;\;\; X \qquad \;\;\;$};
\node[draw, fill=white!20 ,rounded corners ] at (.55,.85) {$ \quad\; {\sf e}_{k-1} \quad\; $};
\node[draw, fill=white!20 ,rounded corners ] at (.55,-.15) {$ \quad\;\; \und{x}^{\hat{\und{\ell}'}} \quad\;\; $};
\end{tikzpicture}}
\]

We find that when $(\ell_1,|\und{j}|)=(0,d),(1,d-1)$ for $1\leq d\leq k-2$ the summation equals
\begin{eqnarray*}
&&
\sum_{\und{\ell}'\in \Sq(k-1)}
\sum_{\und{j}=(j_1,...,j_{k-2}) \in\{0,1\}^{k-2}\atop |\und{j}|=d-1}
\sum_{\und{\ell} \in \Sq(k)\atop \und{\ell}'+\und{j}=(\ell_2,...,\ell_{k-1})}
(-1)^{|\hat{\und{\ell}'}|-k+1+d}
\hackcenter{
\begin{tikzpicture} [scale=.75]
\draw[thick,  ] (1.5,-.5) -- (1.5,3);
\draw[thick,  ] (0,-.5) -- (0,3);
\draw[thick,  ] (-.4,-.5) -- (-.4,3);
\draw[thick,  ] (-1.5,-.5) -- (-1.5,3);
\filldraw (-1.5,2.5) circle (3pt);
\filldraw (1.5,1) circle (3pt);
\filldraw (-.4,1) circle (3pt);
\filldraw (0,1) circle (3pt);
\node[draw, fill=white!20 ,rounded corners ] at (.55,2.45) {$ \quad\;\; {\sf e}_{\und{\ell}'} \quad\;\; $};
\node[draw, fill=white!20 ,rounded corners ] at (.55,.2) {$ \quad\;\; \und{x}^{\hat{\und{\ell}'}} \quad\;\; $};
\node[draw, fill=white!20] at (0,1.625) {$ \qquad\;\;\; X \qquad\;\;\;$};
\node at (-1.8,2.5) {$\scs d$};
\node at (.6,1) {$\scs 1-j_1$};
\node at (2.3,1) {$\scs 1-j_{k-2}$};
\end{tikzpicture}}
\\
&+&\sum_{\und{\ell}'\in \Sq(k-1)}
\sum_{\und{j}=(j_1,...,j_{k-2}) \in\{0,1\}^{k-2}\atop |\und{j}|=d}
\sum_{\und{\ell} \in \Sq(k)\atop \und{\ell}'+\und{j}=(\ell_2,...,\ell_{k-1})}
(-1)^{|\hat{\und{\ell}'}|-k+1+d}
\hackcenter{
\begin{tikzpicture} [scale=.75]
\draw[thick,  ] (1.5,-.5) -- (1.5,3);
\draw[thick,  ] (0,-.5) -- (0,3);
\draw[thick,  ] (-.4,-.5) -- (-.4,3);
\draw[thick,  ] (-1.5,-.5) -- (-1.5,3);
\filldraw (-1.5,2.5) circle (3pt);
\filldraw (1.5,1) circle (3pt);
\filldraw (0,1) circle (3pt);
\node[draw, fill=white!20 ,rounded corners ] at (.55,2.45) {$ \quad\;\; {\sf e}_{\und{\ell}'} \quad\;\; $};
\node[draw, fill=white!20 ,rounded corners ] at (.55,.2) {$ \quad\;\; \und{x}^{\hat{\und{\ell}'}} \quad\;\; $};
\node[draw, fill=white!20] at (0,1.625) {$ \qquad\;\;\; X \qquad\;\;\;$};
\node at (-1.8,2.5) {$\scs d$};
\node at (.6,1) {$\scs 1-j_1$};
\node at (2.3,1) {$\scs 1-j_{k-2}$};
\end{tikzpicture}}
\\
&=&
\sum_{\und{\ell}'\in \Sq(k-1)}
(-1)^{|\hat{\und{\ell}}'|-k+1+d}
\hackcenter{
\begin{tikzpicture} [scale=.75]
\draw[thick,  ] (1.5,-.75) -- (1.5,3);
\draw[thick,  ] (0,-.75) -- (0,3);
\draw[thick,  ] (-.4,-.75) -- (-.4,3);
\draw[thick,  ] (-1.5,-.75) -- (-1.5,3);
\node[draw, fill=white!20 ,rounded corners ] at (.55,2.45) {$ \quad\;\; {\sf e}_{\und{\ell}'} \quad\;\; $};
\node[draw, fill=white!20] at (0,1.625) {$ \qquad\;\;\; X \qquad \;\;\;$};
\node[draw, fill=white!20 ,rounded corners ] at (.55,.85) {$ \quad\; {\sf e}_{k-1-d} \quad\; $};
\node[draw, fill=white!20 ,rounded corners ] at (.55,-.15) {$ \quad\;\; \und{x}^{\hat{\und{\ell}'}} \quad\;\; $};
\filldraw (-1.5,2.5) circle (3pt);
\node at (-2,2.5) {$\scs d$};
\end{tikzpicture}}
\end{eqnarray*}

We find that when $(\ell_1,|\und{j}|)=(1,k-2)$ this summation equals
\[
\sum_{\und{\ell}'\in \Sq(k-1)}
(-1)^{|\hat{\und{\ell}}'|}
\hackcenter{
\begin{tikzpicture} [scale=.75]
\draw[thick,  ] (1.5,0) -- (1.5,3);
\draw[thick,  ] (0,0) -- (0,3);
\draw[thick,  ] (-.4,0) -- (-.4,3);
\draw[thick,  ] (-1.5,0) -- (-1.5,3);
\node[draw, fill=white!20 ,rounded corners ] at (.55,2.45) {$ \quad\;\; {\sf e}_{\und{\ell}'} \quad\;\; $};
\node[draw, fill=white!20] at (0,1.625) {$ \qquad\;\;\; X \qquad \;\;\;$};
\node[draw, fill=white!20 ,rounded corners ] at (.55,.7) {$ \quad\;\; \und{x}^{\hat{\und{\ell}'}} \quad\;\; $};
\filldraw (-1.5,2.5) circle (3pt);
\node at (-2.3,2.5) {$\scs k-1$};
\end{tikzpicture}}
\]
\end{proof}

\section{Quantum $\mf{gl}_m$ $2$-category}
\label{2categorysection}

We recall here the $\mf{gl}_m$ version of the categorified quantum group from~\cite{MSV}.
 The 2-category $\cal{U}=\cal{U}(\mf{gl}_m)$ is obtained
from the corresponding $\mf{sl}_m$ category defined in \cite{KL3,CLau} by switching from $\mf{sl}_m$ weights to $\mf{gl}_m$ weights~\cite{MSV}, however, the coefficients in the defining relations for the 2-category take on subtle changes.  This is further studied in \cite{BHLW2,Lau-param}.  A minimal presentation of this category can be determined from \cite{Brundan2}.

Here we identify the weight lattice $X$ of $\mf{gl}_m$ with $\Z^m$ and write $\epsilon_i = (0, \dots, 0, 1, 0 \dots ,0) \in X$, with a single 1 on the $i$th coordinate.  Let $I = \{ 1,2 ,\dots, m-1\}$ and define $\alpha_i = \epsilon_i - \epsilon_{i+1} \in X$ for $i \in I$.  There is a bilinear form on $\Z[I]$ define by
\[
 i\cdot j =
 \left\{
   \begin{array}{ll}
     2, & \hbox{if $i=j$;} \\
     -1, & \hbox{if $j=i\pm1$;} \\
     0, & \hbox{otherwise.}
   \end{array}
 \right.
\]
For $\lambda \in \Z^m$ let $\lambda_i$ denote the $i$th component of $\lambda$. Define the associated $\mf{sl}_m$ weight $\bar{\lambda} \in \Z^{m-1}$ by setting the $i$-th component $\bar{\lambda}_i = \lambda_i -\lambda_{i+1}$.

\begin{definition} \label{defU_cat}
The 2-category $\cal{U}$ is the graded additive $\Bbbk$-linear 2-category consisting of:
\begin{itemize}
\item \textbf{Objects} $\lambda$ for $\lambda \in X$.
\item \textbf{1-morphisms} are formal direct sums of (shifts of) compositions of
$$\onel, \quad \1_{\lambda+\alpha_i} \sE_i= \1_{\lambda+\alpha_i} \sE_i\onel, \quad \text{ and }\quad
\1_{\lambda-\alpha_i} \sF_i= \1_{\lambda-\alpha_i} \sF_i\onel$$
for $i \in I$ and $\lambda \in X$.  We denote the grading shift by $\la 1 \ra$, so that for each 1-morphism $x$ in $\cal{U}$ and $t\in \Z$ we a 1-morphism $x\la t\ra$.

\item \textbf{2-morphisms} are $\Bbbk$-vector spaces spanned by compositions of decorated tangle-like diagrams  coloured by $i\in I$ illustrated below.
\begin{align}
\hackcenter{\begin{tikzpicture}[scale=0.8]
    \draw[thick, ->] (0,0) -- (0,1.5)
        node[pos=.5, shape=coordinate](DOT){};
    \filldraw  (DOT) circle (2.5pt);
    \node at (-.85,.85) {\tiny $\lambda +\alpha_i$};
    \node at (.5,.85) {\tiny $\lambda$};
    \node at (-.2,.1) {\tiny $i$};
\end{tikzpicture}} &\maps \cal{E}_i\onel \to \cal{E}_i\onel \la i\cdot i \ra  & \quad
 &
\hackcenter{\begin{tikzpicture}[scale=0.8]
    \draw[thick, <-] (0,0) -- (0,1.5)
        node[pos=.5, shape=coordinate](DOT){};
    \filldraw  (DOT) circle (2.5pt);
    \node at (-.85,.85) {\tiny $\lambda -\alpha_i$};
    \node at (.5,.85) {\tiny $\lambda$};
    \node at (-.2,1.4) {\tiny $i$};
\end{tikzpicture}}
\maps \cal{F}_i\onel \to \cal{F}_i\onel\la i\cdot i \ra  \nn \\
   & & & \nn \\
  \hackcenter{\begin{tikzpicture}[scale=0.8]
    \draw[thick, ->] (0,0) .. controls (0,.5) and (.75,.5) .. (.75,1.0);
    \draw[thick, ->] (.75,0) .. controls (.75,.5) and (0,.5) .. (0,1.0);
    \node at (1.1,.55) {\tiny $\lambda$};
    \node at (-.2,.1) {\tiny $i$};
    \node at (.95,.1) {\tiny $j$};
\end{tikzpicture}} \;\;&\maps \cal{E}_i\cal{E}_j\onel  \to \cal{E}_j\cal{E}_i\onel\la -i\cdot j \ra  &
  &
  \hackcenter{\begin{tikzpicture}[scale=0.8]
    \draw[thick, <-] (0,0) .. controls (0,.5) and (.75,.5) .. (.75,1.0);
    \draw[thick, <-] (.75,0) .. controls (.75,.5) and (0,.5) .. (0,1.0);
    \node at (1.1,.55) {\tiny $\lambda$};
    \node at (-.25,.1) {\tiny $i$};
    \node at (1,.1) {\tiny $j$};
\end{tikzpicture}}\;\; \maps \cal{F}_i\cal{F}_j\onel  \to \cal{F}_j\cal{F}_i\onel\la -i\cdot j \ra  \nn \\
  & & & \nn \\
\hackcenter{\begin{tikzpicture}[scale=0.8]
    \draw[thick, <-] (.75,2) .. controls ++(0,-.75) and ++(0,-.75) .. (0,2);
    \node at (.4,1.2) {\tiny $\lambda$};
    \node at (-.2,1.9) {\tiny $i$};
\end{tikzpicture}} \;\; &\maps \onel  \to \cal{F}_i\cal{E}_i\onel\la  1 + \bar{\lambda}_i  \ra   &
    &
\hackcenter{\begin{tikzpicture}[scale=0.8]
    \draw[thick, ->] (.75,2) .. controls ++(0,-.75) and ++(0,-.75) .. (0,2);
    \node at (.4,1.2) {\tiny $\lambda$};
    \node at (.95,1.9) {\tiny $i$};
\end{tikzpicture}} \;\; \maps \onel  \to\cal{E}_i\cal{F}_i\onel\la  1 - \bar{\lambda}_i  \ra   \nn \\
      & & & \nn \\
\hackcenter{\begin{tikzpicture}[scale=0.8]
    \draw[thick, ->] (.75,-2) .. controls ++(0,.75) and ++(0,.75) .. (0,-2);
    \node at (.4,-1.2) {\tiny $\lambda$};
    \node at (.95,-1.9) {\tiny $i$};
\end{tikzpicture}} \;\; & \maps \cal{F}_i\cal{E}_i\onel \to\onel\la  1 + \bar{\lambda}_i  \ra   &
    &
\hackcenter{\begin{tikzpicture}[scale=0.8]
    \draw[thick, <-] (.75,-2) .. controls ++(0,.75) and ++(0,.75) .. (0,-2);
    \node at (.4,-1.2) {\tiny $\lambda$};
    \node at (-.2,-1.9) {\tiny $i$};
\end{tikzpicture}} \;\;\maps\cal{E}_i\cal{F}_i\onel  \to\onel\la  1 - \bar{\lambda}_i  \ra  \nn
\end{align}
\end{itemize}
In this $2$-category (and those throughout the paper) we
read diagrams from right to left and bottom to top.  The identity 2-morphism of the 1-morphism
$\cal{E}_i \onel$ is
represented by an upward oriented line labelled by $i$ and the identity 2-morphism of $\cal{F}_i \onel$ is
represented by a downward such line.

The 2-morphisms satisfy the following relations:
\begin{enumerate}
\item \label{item_cycbiadjoint-cyc} The 1-morphisms $\cal{E}_i \onel$ and $\cal{F}_i \onel$ are biadjoint (up to a specified degree shift). These conditions are expressed diagrammatically as
\begin{equation} \label{eq_biadjoint1-cyc}
    \hackcenter{
}
\end{equation}

  \item The 2-morphisms are cyclic with respect to this biadjoint structure.
\begin{equation}\label{eq_cyclic_dot-cyc}
\hackcenter{%
}
\end{equation}

The cyclic relations for crossings are given by
\begin{equation} \label{eq_cyclic}
\hackcenter{
%
}
\end{equation}

Sideways crossings are equivalently defined by the following identities:
\begin{equation} \label{eq_crossl-gen-cyc}
\hackcenter{
%
}
\end{equation}

\item The $\cal{E}$'s (respectively $\cal{F}$'s) carry an action of the KLR algebra for a fixed choice of parameters $Q$. We find it convenient to take a particular choice of parameters in this paper.  Set
\[
 t_{ij} =
\left\{
  \begin{array}{ll}
    0, & \hbox{if $i=j$;} \\
    -1, & \hbox{if $j =i+1$;} \\
    1, & \hbox{otherwise.}
  \end{array}
\right.
\]
The KLR algebra $R$ associated to a fixed set of parameters $Q$ is defined by finite $\Bbbk$-linear combinations of braid--like diagrams in the plane, where each strand is labelled by a vertex $i \in I$.  Strands can intersect and can carry dots, but triple intersections are not allowed.  Diagrams are considered up to planar isotopy that do not change the combinatorial type of the diagram. We recall the local relations:

\begin{enumerate}[i)]

\item The quadratic KLR relations are
\begin{equation}
\hackcenter{
}
 \qquad = \qquad
 \left\{
 \begin{array}{ccc}
  t_{ij}
  \;      \hackcenter{
\begin{tikzpicture}[scale=0.8]
    \draw[thick, ->] (0,0) to (0,1.6);
    \draw[thick, ->] (.75,0) to (.75,1.6);
    \node at (1.1,1.25) { $\lambda$}; \filldraw  (0,.8) circle (2.75pt);
    \node at (-.2,.1) {\tiny $i$};
    \node at (.95,.1) {\tiny $j$};
\end{tikzpicture}}
  \;\; + \;\; t_{ji} \;
 \hackcenter{
\begin{tikzpicture}[scale=0.8]
    \draw[thick, ->] (0,0) to (0,1.6);
    \draw[thick, ->] (.75,0) to (.75,1.6);
    \node at (1.1,1.25) { $\lambda$}; \filldraw  (.75,.8) circle (2.75pt);
    \node at (-.2,.1) {\tiny $i$};
    \node at (.95,.1) {\tiny $j$};
\end{tikzpicture}}
   &  & \text{if $i \cdot j <0$,}\\ \\
     t_{ij}\;
     \hackcenter{
\begin{tikzpicture}[scale=0.8]
    \draw[thick, ->] (0,0) to (0,1.6);
    \draw[thick, ->] (.75,0) to (.75,1.6);
    \node at (1.1,1.25) { $\lambda$};
    \node at (-.2,.1) {\tiny $i$};
    \node at (.95,.1) {\tiny $j$};
\end{tikzpicture}}&  &  \text{otherwise.}
 \end{array}
 \right. \label{eq_r2_ij-gen-cyc}
\end{equation}

\item The nilHecke dot sliding relations
\begin{align} \label{eq:nil-dot}
\hackcenter{%
}
\end{align}

\item For $i \neq j$ the dot sliding relations
\begin{align} \label{eq:nil-dot}
\hackcenter{%
}
\end{equation}
\end{enumerate}

\item When $i \ne j$ one has the mixed relations  relating $\cal{E}_i \cal{F}_j$ and $\cal{F}_j \cal{E}_i$:
\begin{equation}  \label{mixed_rel-cyc}
 \hackcenter{%
}
\end{equation}

\item Negative degree bubbles are zero.  That is for all $m \in \Z_{>0}$ one has
\begin{equation}
 \hackcenter{ %
 } \;  = 0 \quad  \text{if $m < -\bar{\lambda}_i -1$}
\end{equation}
Furthermore, dotted bubbles of degree zero are a scalar multiple of the identity 2-morphism
\begin{equation} \label{eq:degreezero}
 \hackcenter{ %
 } \;\; = (-1)^{\lambda_{i+1}-1} \quad \text{if $   \bar{\lambda}_i \leq -1$}
\end{equation}

For the final relations it is convenient to add additional generating diagrams called \textit{fake bubbles}.  These fake bubbles are represented by dotted bubbles of positive total degree, but with a negative label for the dots.  These symbols make it possible to express the relations independent of the weight $\lambda$.

\begin{itemize}
  \item Degree zero fake bubbles are equal to
  \begin{equation}
 \hackcenter{ %
 } \;\; = (-1)^{\lambda_{i+1}-1} \quad \text{if $   \bar{\lambda}_i > -1$}
\end{equation}
(Compare with \eqref{eq:degreezero}.)
\item Higher degree fake bubbles for $\bar{\lambda}_i <0$ are defined inductively as
\begin{equation}
  \hackcenter{  } \;\; = \;\;
\left\{
  \begin{array}{ll}
    - (-1)^{\lambda_{i+1}}
\displaystyle \sum_{\stackrel{\scs x+y=j}{\scs y\geq 1}} \hackcenter{ \begin{tikzpicture}[scale=.8]
 \draw (-.15,.35) node { $\scs i$};
 \draw  (0,0) arc (180:360:0.5cm) [thick];
 \draw[,<-](1,0) arc (0:180:0.5cm) [thick];
\filldraw  [black] (.1,-.25) circle (2.5pt);
 \node at (-.35,-.45) {\tiny $\overset{\bar{\lambda}_i-1}{+x}$};
 \node at (.85,1) { $\lambda$};
\end{tikzpicture}  \;\;
\begin{tikzpicture}[scale=.8]
 \draw (-.15,.35) node { $\scs i$};
 \draw  (0,0) arc (180:360:0.5cm) [thick];
 \draw[->](1,0) arc (0:180:0.5cm) [thick];
\filldraw  [black] (.9,-.25) circle (2.5pt);
 \node at (1.45,-.5) {\tiny $\overset{-\bar{\lambda}_i-1}{+y}$};
 \node at (1.15,1.1) { $\;$};
\end{tikzpicture}  }
  & \hbox{if $0 < j < -\bar{\lambda}_i+1$;} \\
    0, & \hbox{if $j<0$.}
  \end{array}
\right.
\end{equation}

\item Higher degree fake bubbles for $\bar{\lambda}_i >0$ are defined inductively as
\begin{equation}
\hackcenter{  } \;\; = \;\;
\left\{
  \begin{array}{ll}
    - (-1)^{\lambda_{i+1}}
\displaystyle\sum_{\stackrel{\scs x+y=j}{\scs x\geq 1}} \hackcenter{ \begin{tikzpicture}[scale=.8]
 \draw (-.15,.35) node { $\scs i$};
 \draw  (0,0) arc (180:360:0.5cm) [thick];
 \draw[,<-](1,0) arc (0:180:0.5cm) [thick];
\filldraw  [black] (.1,-.25) circle (2.5pt);
 \node at (-.35,-.45) {\tiny $\overset{\bar{\lambda}_i-1}{+x}$};
 \node at (.85,1) { $\lambda$};
\end{tikzpicture}  \;\;
\begin{tikzpicture}[scale=.8]
 \draw (.3,.125) node {};
 \draw  (0,0) arc (180:360:0.5cm) [thick];
 \draw[->](1,0) arc (0:180:0.5cm) [thick];
\filldraw  [black] (.9,-.25) circle (2.5pt);
 \node at (1.45,-.5) {\tiny $\overset{-\bar{\lambda}_i-1}{+y}$};
 \node at (1.15,1.1) { $\;$};
\end{tikzpicture}  }
  & \hbox{if $0 < j < \bar{\lambda}_i+1$;} \\
    0, & \hbox{if $j<0$.}
  \end{array}
\right.
\end{equation}
\end{itemize}
The above relations are sometimes referred to as the \textit{infinite Grassmannian relations}.

\item The $\mf{sl}_2$ relations (which we also refer to as the $\sE \sF$ and $\sF \sE$ decompositions) are:
\begin{equation}
 \hackcenter{%
 }
\end{equation}
\end{enumerate}
\end{definition}

\section{Redotted Webster algebra for $\mf{sl}_2$} \label{sec:reddotted}
In this section we generalize a deformation of Webster's algebra considered in ~\cite{KS} to allow for the possibility of certain objects to be labelled by all natural numbers rather than just $1$.
Related deformations were also introduced by Webster (see for example \cite{WebgradedHecke}).
An algebraic and graphical description of this algebra are provided.  We construct a faithful representation of this algebra allowing us to write a basis.  The section is concluded with certain symmetries of the algebra.

\subsection{The definition of $W(\mathbf{s},n)$}

Let $m\geq 0$ be an integer.
For a sequence of non-negative integers $\mathbf{s}=(s_{1},...,s_{m})$, let $\mathrm{Seq}(\mathbf{s},n)$ be the set of all sequences $\mathbf{i}=(i_1,...,i_{m+n})$ in which $n$ of the entries are $\mf{b}$ and ${s}_{i}$ appears exactly once and in the order in which it appears in $\mathbf{s}$.
Denote by $\mathbf{i}_j$ the $j$-th entry of $\mathbf{i}$ and denote by $I_\mathbf{s}$ the set $\{{s}_{i}|1\leq i\leq m\}$.

Let $S_{n+m}$ be the symmetric group on $n+m$ letters generated by simple transpositions $\sigma_1, \ldots, \sigma_{n+m-1}$.  Each transposition $\sigma_j$ naturally acts on a sequence $\mathbf{i}$.  Note that if
$\mathbf{i} \in \mathrm{Seq}(\mathbf{s},n)$ it is not always the case that $\sigma_j.\mathbf{i} \in \mathrm{Seq}(\mathbf{s},n)$.

Let $W(\mathbf{s},n)$ be the algebra over $\Bbbk$ generated by $e(\mathbf{i})$, where $\mathbf{i} \in \mathrm{Seq}(\mathbf{s},n)$, $x_j$, $E(d)_j$, where $1\leq j\leq m+n$, $d\geq 1$, and $\psi_j$, where $1\leq j \leq m+n-1$, satisfying the relations below. For convenience, we use the notation $E(0)_j=1$.

\begin{minipage}{0.4\textwidth}
\begin{align}
& e(\mathbf{i})e(\mathbf{j}) = \delta_{\mathbf{i},\mathbf{j}}e(\mathbf{i})
\\
& E(d)_j e(\mathbf{i})=e(\mathbf{i})E(d)_j
\\
& E(d)_j e(\mathbf{i})=0 \quad \text{if $\mathbf{i}_j=\mathfrak{b}$}
\\ \label{eq:Ered}
&E(d)_j e({\bf i})=0 \quad \text{if $d > {\bf i}_j \in I_{\mathbf{s}}$}
\\
& x_j e(\mathbf{i})=e(\mathbf{i})x_j%
\\
& x_j e(\mathbf{i} )=0 \quad \text{if $\mathbf{i}_j\in I_{\mathbf{s}}$}
\\
& \psi_j e(\mathbf{i}) = e(\sigma_j(\mathbf{i}))\psi_j
\\
& \psi_j e(\mathbf{i}) = 0 \quad \text{if $\mathbf{i}_j,\mathbf{i}_{j+1}\in I_\mathbf{s}$}
\end{align}
\end{minipage}
\begin{minipage}{0.6\textwidth}
\begin{align}
& x_j x_\ell=x_\ell x_j
\\
& E(d)_j x_\ell=x_\ell E(d)_j
\\
& E(d)_j E(d')_\ell=E(d')_\ell E(d)_j
\\
&\psi_j x_\ell = x_\ell\psi_j  \quad \text{if $\ell\not=j$ and $\ell\not=j+1$}
\\
&\psi_j E(d)_\ell = E(d)_\ell\psi_j  \quad \text{if $\ell\not=j$ and $\ell\not=j+1$}
\\
& \psi_j E(d)_j  = E(d)_{j+1} \psi_j
\\
& E(d)_j \psi_j = \psi_j E(d)_{j+1}
\\
& \psi_j\psi_\ell = \psi_\ell\psi_j \quad \text{if $|j - \ell| > 1$}
\end{align}
\end{minipage}
\begin{minipage}{\textwidth}
\begin{align}
&x_j \psi_j e(\mathbf{i}) - \psi_j x_{j+1} e(\mathbf{i})= \delta_{\mathbf{i}_j,\mathbf{i}_{j+1}}e(\mathbf{i}) \quad \text{unless $\mathbf{i}_j\in I_\mathbf{s}$ and $\mathbf{i}_{j+1}\in  I_\mathbf{s}$}
\\
&\psi_j x_j e(\mathbf{i}) - x_{j+1} \psi_j e(\mathbf{i})= \delta_{\mathbf{i}_j,\mathbf{i}_{j+1}}e(\mathbf{i}) \quad  \text{unless $\mathbf{i}_j\in I_\mathbf{s}$ and $\mathbf{i}_{j+1}\in  I_\mathbf{s}$}
\\
\label{rel:r2}
&\psi_j^2 e(\mathbf{i})=\left\{
	\begin{array}{ll}
		0
		&\text{if } \mathbf{i}_j=\mathbf{i}_{j+1}={\mathfrak{b}},\\
		\sum_{d=0}^{\mathbf{i}_j}(-1)^{d}E(d)_j x_{j+1}^{\mathbf{i}_j-d} e(\mathbf{i})
		&\text{if } \mathbf{i}_j\in I_{\mathbf{s}}, \mathbf{i}_{j+1}={\mathfrak{b}},\\
		\sum_{d=0}^{\mathbf{i}_{j+1}}(-1)^{d} x_j^{\mathbf{i}_{j+1}-d} E(d)_{j+1} e(\mathbf{i})
		&\text{if } \mathbf{i}_j={\mathfrak{b}}, \mathbf{i}_{j+1} \in I_{\mathbf{s}},
	\end{array}\right.
\\
\label{rel:r3}
& (\psi_j\psi_{j+1} \psi_j-\psi_{j+1}\psi_j \psi_{j+1}) e(\mathbf{i}) \\
& \qquad
=\left\{
	\begin{array}{ll}
	 \sum_{d_1+d_2+d_3=\mathbf{i}_{j+1}-1}(-1)^{d_3}x_j^{d_1}E(d_3)_{j+1}x_{j+2}^{d_2}e(\mathbf{i})
	&\text{if } \mathbf{i}_{j+1}\in I_{s}, \mathbf{i}_{j}=\mathbf{i}_{j+2}={\mathfrak{b}},\\
	0
	&\text{otherwise},
	\end{array}\right. \nn
\end{align}
\end{minipage}

\begin{remark} \label{rem:cyclotomic}
Denote by $\overline{W}(\mathbf{s},n) $ the quotient of $W(\mathbf{s},n)$ by the extra \it{cyclotomic relation} $e(\mathbf{i}) = 0$ if $\mathbf{i}_1=\mf{b}$.
This cyclotomic quotient (when ${\bf s}=(1^m)$ and $n=1$) was the main object of study in ~\cite{KS}.  As mentioned there, the braid group action works whether or not we work in the cyclotomic quotient.  It is more difficult to find a basis for the cyclotomic quotient (except in the special case in ~\cite{KS}) which is why we prefer to work with the larger algebra here.  We expect that this cyclotomic quotient  categorifies a weight space of a tensor product of finite dimensional irreducible representations of quantum $\mathfrak{sl}_2$.
\end{remark}

$W(\mathbf{s},n)$ is a graded algebra with the degrees of generators:
\begin{equation} \label{eq:grading}
\begin{split}
\deg(e(\mathbf{i}))=0,\quad
\deg(x_j)=2,\quad
\deg(E(d)_j)=2d.
\\
\deg(\psi_j e(\mathbf{i}))=\left\{
	\begin{array}{ll}
	-2
	&\text{if}\quad \mathbf{i}_j=\mathbf{i}_{j+1}=\mf{b},\\
	\mathbf{i}_j
	&\text{if}\quad \mathbf{i}_j\in I_{s}, \mathbf{i}_{j+1}=\mf{b},\\
	\mathbf{i}_{j+1}
	&\text{if}\quad \mathbf{i}_j=\mf{b}, \mathbf{i}_{j+1}\in I_{\mathbf{s}}.\\
	\end{array}\right.
\end{split}
\end{equation}

\begin{remark} \label{rem:Redsymm}
Relation~\ref{eq:Ered}, together with the grading convention above, allow us to identify $E(d)_je(\mathbf{i})$ (where $\mathbf{i}_j \in I_{\mathbf{s}}$) with the elementary symmetric function ${\sf e}_d$ in the ring of symmetric functions. 
This identification is further justified by the basis result below
in Proposition~\ref{prop:basis} and the relations on bimodules studied in Section~\ref{subsec:splitter}.
\end{remark}

We describe a diagrammatic presentation of $W(\mathbf{s},n)$ in next subsection.

\begin{remark}
$W((0^m),n)$ is just the nilHecke algebra $\nh_n$.
\end{remark}

\subsection{Diagrammatic presentation of $W(\mathbf{s},n)$}
\label{sectiondiagdescription}
There is a graphical presentation of $W(\mathbf{s},n)$ similar to the algebra introduced in ~\cite{KS}.
We consider collections of smooth arcs composed of $n$ black arcs corresponding to $\mf{b}$ and $m$ red arcs with $I_{\mathbf{s}}$-labelling.
Arcs are assumed to have no critical points (in other words no cups or caps).
Arcs are allowed to intersect (as long as they are both not red), but no triple intersections are allowed.
Arcs can carry dots.
Two diagrams that are related by an isotopy that does not change the combinatorial types of the diagrams or the relative position of crossings and or dots are taken to be equal.
The elements of the vector space $W(\mathbf{s},n)$ are formal linear combinations of these diagrams modulo the local relations given below.
We give $W(\mathbf{s},n)$ the structure of an algebra by concatenating diagrams vertically as long as the colors of the endpoints match.  If they do not, the product of two diagrams is taken to be zero.

The generators $e(\mathbf{i})$ for $\mathbf{i} \in \mathrm{Seq}(\mathbf{s},n)$ represent the diagram consisting entirely of vertical strands whose $j$-th strand (counting from left for each $j$) is black if $\mathbf{i}_j$ is $\mf{b}$ or thick red with $\mathbf{i}_j$-labelling if $\mathbf{i}_j$ is in $I_\mathbf{s}$.
The red strand with $1$-labelling will often be represented by  a thin red strand without $1$-labelling.
Dots on black strands correspond to generators $x_j$ given earlier.
Dots on red strands correspond to generators $ E(d)_j$ defined in the previous subsection.
A crossing of two strands corresponds to a generator $\psi_j$.

\begin{example}
In the case of $\mathbf{s}=(2,1)$ and $n=2$, the set $\mathrm{Seq}((2,1),2)$ is
\[
\{ (2,1,\mf{b},\mf{b}),(2,\mf{b},1,\mf{b}),(\mf{b},2,1,\mf{b}),(2,\mf{b},\mf{b},1),(\mf{b},2,\mf{b},1),(\mf{b},\mf{b},2,1)\}
\]
The diagrammatic presentation of $e(\mathbf{i})$ corresponding to sequences of $\mathrm{Seq}((2,1),2)$ are
\[
\hackcenter{\begin{tikzpicture}[scale=0.6]
    \draw[red,double,thick, ] (0,0) -- (0,1.5)  node[pos=.35, shape=coordinate](DOT){};
    \draw[red,double,thick, ] (.5,0) -- (.5,1.5)  node[pos=.35, shape=coordinate](DOT){};
    \draw[thick, ] (1,0) -- (1,1.5)  node[pos=.35, shape=coordinate](DOT){};
    \draw[thick, ] (1.5,0) -- (1.5,1.5)  node[pos=.35, shape=coordinate](DOT){};
    \node at (0,-.25) {$\scs 2$};
    \node at (.5,-.25) {$\scs 1$};
\end{tikzpicture}}
,\quad\hackcenter{\begin{tikzpicture}[scale=0.6]
    \draw[red,double,thick, ] (0,0) -- (0,1.5)  node[pos=.35, shape=coordinate](DOT){};
    \draw[thick, ] (.5,0) -- (.5,1.5)  node[pos=.35, shape=coordinate](DOT){};
    \draw[red,double,thick, ] (1,0) -- (1,1.5)  node[pos=.35, shape=coordinate](DOT){};
    \draw[thick, ] (1.5,0) -- (1.5,1.5)  node[pos=.35, shape=coordinate](DOT){};
    \node at (0,-.25) {$\scs 2$};
    \node at (1,-.25) {$\scs 1$};
\end{tikzpicture}}
,\quad\hackcenter{\begin{tikzpicture}[scale=0.6]
    \draw[thick, ] (0,0) -- (0,1.5)  node[pos=.35, shape=coordinate](DOT){};
    \draw[red,double,thick, ] (.5,0) -- (.5,1.5)  node[pos=.35, shape=coordinate](DOT){};
    \draw[red,double,thick, ] (1,0) -- (1,1.5)  node[pos=.35, shape=coordinate](DOT){};
    \draw[thick, ] (1.5,0) -- (1.5,1.5)  node[pos=.35, shape=coordinate](DOT){};
    \node at (1,-.25) {$\scs 1$};
    \node at (.5,-.25) {$\scs 2$};
\end{tikzpicture}}
,\quad\hackcenter{\begin{tikzpicture}[scale=0.6]
    \draw[red,double,thick, ] (0,0) -- (0,1.5)  node[pos=.35, shape=coordinate](DOT){};
    \draw[thick, ] (.5,0) -- (.5,1.5)  node[pos=.35, shape=coordinate](DOT){};
    \draw[thick, ] (1,0) -- (1,1.5)  node[pos=.35, shape=coordinate](DOT){};
    \draw[red,double,thick, ] (1.5,0) -- (1.5,1.5)  node[pos=.35, shape=coordinate](DOT){};
    \node at (0,-.25) {$\scs 2$};
    \node at (1.5,-.25) {$\scs 1$};
\end{tikzpicture}}
,\quad\hackcenter{\begin{tikzpicture}[scale=0.6]
    \draw[thick, ] (0,0) -- (0,1.5)  node[pos=.35, shape=coordinate](DOT){};
    \draw[red,double,thick, ] (.5,0) -- (.5,1.5)  node[pos=.35, shape=coordinate](DOT){};
    \draw[thick, ] (1,0) -- (1,1.5)  node[pos=.35, shape=coordinate](DOT){};
    \draw[red,double,thick, ] (1.5,0) -- (1.5,1.5)  node[pos=.35, shape=coordinate](DOT){};
    \node at (1.5,-.25) {$\scs 1$};
    \node at (.5,-.25) {$\scs 2$};
\end{tikzpicture}}
,\quad\hackcenter{\begin{tikzpicture}[scale=0.6]
    \draw[thick, ] (0,0) -- (0,1.5)  node[pos=.35, shape=coordinate](DOT){};
    \draw[thick, ] (.5,0) -- (.5,1.5)  node[pos=.35, shape=coordinate](DOT){};
    \draw[red,double,thick, ] (1,0) -- (1,1.5)  node[pos=.35, shape=coordinate](DOT){};
    \draw[red,double,thick, ] (1.5,0) -- (1.5,1.5)  node[pos=.35, shape=coordinate](DOT){};
    \node at (1,-.25) {$\scs 2$};
    \node at (1.5,-.25) {$\scs 1$};
\end{tikzpicture}}
\]
\end{example}

The generators $x_j e(\mathbf{i})$ and $E(d)_j e(\mathbf{i})$ represent the dot on the $j$-th strand counting from left of the $(m+n)$ strand diagram:
\begin{eqnarray*}
x_j e(\mathbf{i})&=&\hackcenter{\begin{tikzpicture}[scale=0.6]
    \draw[thick, ] (0,0) -- (0,1.5)  node[pos=.45, shape=coordinate](DOT){};
    \filldraw  (DOT) circle (2.5pt);
\end{tikzpicture}}\quad  \hbox{if $\mathbf{i}_j=\mf{b}$,} \\
E(d)_j e(\mathbf{i})&=&
\hackcenter{\begin{tikzpicture}[scale=0.6]
    \draw[thick,red, double, ] (0,0) to (0,1.5);
 \filldraw[red]  (0,.65) circle (2.75pt);
    \node at (-.5,.65) {$\scs {\sf e}_d$};
    \node at (0,-.25) {$\scs \mathbf{i}_j$};
\end{tikzpicture}}\quad  \hbox{if $\mathbf{i}_j\in I_{\mathbf{s}}$ and $1\leq d \leq \mathbf{i}_j$.}
\end{eqnarray*}
where we make use of the identification from Remark~\ref{rem:Redsymm} to identify red dots with elementary symmetric functions.  In particular, we have that $E_d={\sf e}_d$ on the red arc with $i$-labelling is zero if $d$ is bigger than $i$.  Note that on thin red strands, we only have dots labelled by ${\sf e}_1$.  For simplicity in this case, we will abuse notation and label a dot $d$ when we really mean ${\sf e}_1^d$.

The generator $\psi_j e(\mathbf{i})$ represents the crossing diagram as follows.
\[
\psi_j e(\mathbf{i})=
 \left\{
   \begin{array}{ll}
     \;
\hackcenter{\begin{tikzpicture}[scale=0.6]
    \draw[thick, ] (0,0) .. controls (0,.75) and (1.5,.75) .. (1.5,1.5);
    \draw[thick, ] (1.5,0) .. controls (1.5,.75) and (0,.75) .. (0,1.5);
\end{tikzpicture}}, & \hbox{if $\mathbf{i}_j=\mf{b}$, $\mathbf{i}_{j+1}=\mf{b}$;} \\
\hackcenter{\begin{tikzpicture}[scale=0.6]
    \draw[thick, ] (0,0) .. controls (0,.75) and  (1.5,.75) .. (1.5,1.5);
    \draw[thick,red, double, ] (1.5,0) .. controls (1.5,.75) and (0,.75) .. (0,1.5);
    \node at (1.5,-.25) {$\scs \mathbf{i}_{j+1}$};
\end{tikzpicture}}, & \hbox{if $\mathbf{i}_j=\mf{b}$, $\mathbf{i}_{j+1}\in I_{\mathbf{s}}$;} \\
\hackcenter{\begin{tikzpicture}[scale=0.6]
    \draw[thick,red, double, ] (0,0) .. controls (0,.75) and  (1.5,.75) .. (1.5,1.5);
    \draw[thick, ] (1.5,0) .. controls (1.5,.75) and (0,.75) .. (0,1.5);
    \node at (0,-.25) {$\scs \mathbf{i}_j$};
\end{tikzpicture}}, & \hbox{if $\mathbf{i}_j\in I_{s}$, $\mathbf{i}_{j+1}=\mf{b}$.}
   \end{array}
 \right.
\]
The degrees of the generating diagrams are
\[
\deg
\left( \;
\hackcenter{\begin{tikzpicture}[scale=0.6]
    \draw[thick, ] (0,0) -- (0,1.5)  node[pos=.35, shape=coordinate](DOT){};
    \filldraw  (DOT) circle (2.5pt);
    \node at (0,-.2) {$\scs $};
\end{tikzpicture}}\;
\right)
= 2,
\quad
\deg
\left( \;
\hackcenter{\begin{tikzpicture}[scale=0.6]
    \draw[thick,red, double, ] (0,0) to (0,1.5);
 \filldraw[red]  (0,.5) circle (2.75pt);
    \node at (-.5,.5) {$\scs {\sf e}_d$};
    \node at (0,-.2) {$\scs i$};
\end{tikzpicture}}\;
\right)
= 2d,
\quad
\deg\left( \; \hackcenter{\begin{tikzpicture}[scale=0.6]
    \draw[thick, ] (0,0) .. controls (0,.75) and (.75,.75) .. (.75,1.5);
    \draw[thick, ] (.75,0) .. controls (.75,.75) and (0,.75) .. (0,1.5);
\end{tikzpicture}} \; \right) = -2,
\quad
\deg\left( \; \hackcenter{\begin{tikzpicture}[scale=0.6]
    \draw[thick, ] (0,0) .. controls (0,.75) and (.75,.75) .. (.75,1.5);
    \draw[thick,red, double, ] (.75,0) .. controls (.75,.75) and (0,.75) .. (0,1.5);
    \node at (.8,-.2) {$\scs i$};
\end{tikzpicture}} \; \right) =
\deg\left( \; \hackcenter{\begin{tikzpicture}[scale=0.6]
    \draw[thick,red, double, ] (0,0) .. controls (0,.75) and (.75,.75) .. (.75,1.5);
    \draw[thick, ] (.75,0) .. controls (.75,.75) and (0,.75) .. (0,1.5);
    \node at (0,-.2) {$\scs i$};
\end{tikzpicture}} \; \right) = i.
\]

The elements of the vector space $W(\mathbf{s},n)$ are formal linear combinations of these diagrams modulo isotopy, nilHecke relations on the black strands, and the following local relations.
We give the algebra structure of $W(\mathbf{s},n)$ by concatenating diagrams vertically when the colors on the endpoints of two diagrams match.
\begin{equation}\label{RBr2-rel}
\hackcenter{\begin{tikzpicture}[scale=0.8]
    \draw[thick] (0,0) .. controls ++(0,.5) and ++(0,-.5) .. (.75,1);
    \draw[thick, red, double, ] (.75,0) .. controls ++(0,.5) and ++(0,-.5) .. (0,1);
    \draw[thick,red, double, ] (0,1 ) .. controls ++(0,.5) and ++(0,-.5) .. (.75,2);
    \draw[thick, ] (.75,1) .. controls ++(0,.5) and ++(0,-.5) .. (0,2);
    \node at (0,-.25) {$\;$};
    \node at (0,2.25) {$\;$};
    \node at (.8,-.2) {$\scs i$};
\end{tikzpicture}}
 \;\; = \;\;
\sum_{d_1+d_2=i}(-1)^{d_2}
\hackcenter{\begin{tikzpicture}[scale=0.8]
    \draw[thick, ] (0,0) -- (0,1.5)  node[pos=.55, shape=coordinate](DOT){};
    \filldraw  (DOT) circle (2.5pt);
    \draw[thick,red, double, ] (.75,0) -- (.75,1.5) node[pos=.55, shape=coordinate](DOT2){};
    \filldraw[red]  (DOT2) circle (2.75pt);
    \node at (.8,-.2) {$\scs i$};
    \node at (-.3,.75) {$\scs d_1$};
    \node at (1.25,.75) {$\scs {\sf e}_{d_2}$};
\end{tikzpicture}}
%
\qquad \quad
\hackcenter{\begin{tikzpicture}[scale=0.8]
    \draw[thick, red, double] (0,0) .. controls ++(0,.5) and ++(0,-.5) .. (.75,1);
    \draw[thick ] (.75,0) .. controls ++(0,.5) and ++(0,-.5) .. (0,1);
    \draw[ thick, ] (0,1 ) .. controls ++(0,.5) and ++(0,-.5) .. (.75,2);
    \draw[thick,red, double, ] (.75,1) .. controls ++(0,.5) and ++(0,-.5) .. (0,2);
    \node at (0,-.25) {$\;$};
    \node at (0,2.25) {$\;$};
    \node at (0,-.2) {$\scs i$};
\end{tikzpicture}}
 \;\; = \;\;
\sum_{d_1+d_2=i}(-1)^{d_1}
\hackcenter{\begin{tikzpicture}[scale=0.8]
    \draw[thick, ] (.75,0) -- (.75,1.5)  node[pos=.55, shape=coordinate](DOT){};
    \filldraw  (DOT) circle (2.5pt);
    \draw[thick,red, double, ] (0,0) -- (0,1.5) node[pos=.55, shape=coordinate](DOT2){};
    \filldraw[red]  (DOT2) circle (2.75pt);
    \node at (.05,-.2) {$\scs i$};
    \node at (1.15,.75) {$\scs d_2$};
    \node at (-.4,.75) {$\scs {\sf e}_{d_1}$};
\end{tikzpicture}}
\end{equation}

\begin{align}
\label{blackdot}
\hackcenter{\begin{tikzpicture}[scale=0.8]
    \draw[thick, ] (0,0) .. controls (0,.75) and (.75,.75) .. (.75,1.5)
        node[pos=.25, shape=coordinate](DOT){};
    \draw[thick,red, double, ] (.75,0) .. controls (.75,.75) and (0,.75) .. (0,1.5);
    \filldraw  (DOT) circle (2.5pt);
    \node at (.8,-.2) {$\scs i$};
\end{tikzpicture}}
\quad =\quad
\hackcenter{\begin{tikzpicture}[scale=0.8]
    \draw[thick, ] (0,0) .. controls (0,.75) and (.75,.75) .. (.75,1.5)
        node[pos=.75, shape=coordinate](DOT){};
    \draw[thick,red, double, ] (.75,0) .. controls (.75,.75) and (0,.75) .. (0,1.5);
    \filldraw  (DOT) circle (2.5pt);
    \node at (.8,-.2) {$\scs i$};
\end{tikzpicture}}
\qquad \quad
\hackcenter{\begin{tikzpicture}[scale=0.8]
    \draw[thick,red, double, ] (0,0) .. controls (0,.75) and (.75,.75) .. (.75,1.5);
    \draw[thick,  ] (.75,0) .. controls (.75,.75) and (0,.75) .. (0,1.5)
        node[pos=.75, shape=coordinate](DOT){};
    \filldraw  (DOT) circle (2.75pt);
    \node at (.05,-.2) {$\scs i$};
\end{tikzpicture}}
\quad=\quad
\hackcenter{\begin{tikzpicture}[scale=0.8]
    \draw[thick,red, double, ] (0,0) .. controls (0,.75) and (.75,.75) .. (.75,1.5);
    \draw[thick, ] (.75,0) .. controls (.75,.75) and (0,.75) .. (0,1.5)
        node[pos=.25, shape=coordinate](DOT){};
      \filldraw  (DOT) circle (2.75pt);
    \node at (.05,-.2) {$\scs i$};
\end{tikzpicture}}
\end{align}

\begin{equation}
\label{reddot}
\hackcenter{\begin{tikzpicture}[scale=0.8]
    \draw[thick,red, double, ] (0,0) .. controls (0,.75) and (.75,.75) .. (.75,1.5)
        node[pos=.25, shape=coordinate](DOT){};
    \draw[thick, ] (.75,0) .. controls (.75,.75) and (0,.75) .. (0,1.5);
    \filldraw[red]  (DOT) circle (2.5pt);
    \node at (.05,-.2) {$\scs i$};
    \node at (-.25,.5) {$\scs {\sf e}_d$};
\end{tikzpicture}}
\quad =\quad
\hackcenter{\begin{tikzpicture}[scale=0.8]
    \draw[thick,red, double, ] (0,0) .. controls (0,.75) and (.75,.75) .. (.75,1.5)
        node[pos=.75, shape=coordinate](DOT){};
    \draw[thick, ] (.75,0) .. controls (.75,.75) and (0,.75) .. (0,1.5);
    \filldraw[red]  (DOT) circle (2.5pt);
    \node at (.05,-.2) {$\scs i$};
    \node at (1,1) {$\scs {\sf e}_d$};
\end{tikzpicture}}
\qquad \quad
\hackcenter{\begin{tikzpicture}[scale=0.8]
    \draw[thick, ] (0,0) .. controls (0,.75) and (.75,.75) .. (.75,1.5);
    \draw[thick,red, double, ] (.75,0) .. controls (.75,.75) and (0,.75) .. (0,1.5)
        node[pos=.75, shape=coordinate](DOT){};
    \filldraw[red]  (DOT) circle (2.75pt);
    \node at (.8,-.2) {$\scs i$};
    \node at (-.25,1) {$\scs {\sf e}_d$};
\end{tikzpicture}}
\quad=\quad
\hackcenter{\begin{tikzpicture}[scale=0.8]
    \draw[thick, ] (0,0) .. controls (0,.75) and (.75,.75) .. (.75,1.5);
    \draw[thick,red, double, ] (.75,0) .. controls (.75,.75) and (0,.75) .. (0,1.5)
        node[pos=.25, shape=coordinate](DOT){};
      \filldraw[red]  (DOT) circle (2.75pt);
    \node at (.8,-.2) {$\scs i$};
    \node at (1.,.5) {$\scs {\sf e}_d$};
\end{tikzpicture}}
\end{equation}

\begin{align}
\hackcenter{\begin{tikzpicture}[scale=0.8]
    \draw[thick,red, double, ] (0,0) .. controls ++(0,1) and ++(0,-1) .. (1.2,2);
    \draw[thick, ] (.6,0) .. controls ++(0,.5) and ++(0,-.5) .. (0,1.0);
    \draw[thick, ] (0,1.0) .. controls ++(0,.5) and ++(0,-.5) .. (0.6,2);
    \draw[thick, ] (1.2,0) .. controls ++(0,1) and ++(0,-1) .. (0,2);
    \node at (0,-.2) {$\scs i$};
\end{tikzpicture}}
\;\; = \;\;
\hackcenter{\begin{tikzpicture}[scale=0.8]
    \draw[thick,red, double, ] (0,0) .. controls ++(0,1) and ++(0,-1) .. (1.2,2);
    \draw[thick, ] (.6,0) .. controls ++(0,.5) and ++(0,-.5) .. (1.2,1.0);
    \draw[thick, ] (1.2,1.0) .. controls ++(0,.5) and ++(0,-.5) .. (0.6,2.0);
    \draw[thick, ] (1.2,0) .. controls ++(0,1) and ++(0,-1) .. (0,2.0);
    \node at (0,-.2) {$\scs i$};
\end{tikzpicture}}
\qquad \qquad
\hackcenter{\begin{tikzpicture}[scale=0.8]
    \draw[thick,  ] (0,0) .. controls ++(0,1) and ++(0,-1) .. (1.2,2);
    \draw[thick, ] (.6,0) .. controls ++(0,.5) and ++(0,-.5) .. (0,1.0);
    \draw[thick, ] (0,1.0) .. controls ++(0,.5) and ++(0,-.5) .. (0.6,2);
    \draw[thick,red, double, ] (1.2,0) .. controls ++(0,1) and ++(0,-1) .. (0,2);
    \node at (1.2,-.2) {$\scs i$};
\end{tikzpicture}}
\;\; = \;\;
\hackcenter{\begin{tikzpicture}[scale=0.8]
    \draw[thick,  ] (0,0) .. controls ++(0,1) and ++(0,-1) .. (1.2,2);
    \draw[thick, ] (.6,0) .. controls ++(0,.5) and ++(0,-.5) .. (1.2,1.0);
    \draw[thick, ] (1.2,1.0) .. controls ++(0,.5) and ++(0,-.5) .. (0.6,2.0);
    \draw[thick,red, double, ] (1.2,0) .. controls ++(0,1) and ++(0,-1) .. (0,2.0);
    \node at (1.2,-.2) {$\scs i$};
\end{tikzpicture}}
\end{align}

\begin{equation}\label{r3-2}
\hackcenter{\begin{tikzpicture}[scale=0.8]
    \draw[thick,  ] (0,0) .. controls ++(0,1) and ++(0,-1) .. (1.2,2);
    \draw[thick,red, double, ] (.6,0) .. controls ++(0,.5) and ++(0,-.5) .. (0,1.0);
    \draw[thick,red, double, ] (0,1.0) .. controls ++(0,.5) and ++(0,-.5) .. (0.6,2);
    \draw[thick, ] (1.2,0) .. controls ++(0,1) and ++(0,-1) .. (0,2);
    \node at (.6,-.2) {$\scs i$};
\end{tikzpicture}}
\;\; - \;\;
\hackcenter{\begin{tikzpicture}[scale=0.8]
    \draw[thick,  ] (0,0) .. controls ++(0,1) and ++(0,-1) .. (1.2,2);
    \draw[thick, red, double,] (.6,0) .. controls ++(0,.5) and ++(0,-.5) .. (1.2,1.0);
    \draw[thick,red, double, ] (1.2,1.0) .. controls ++(0,.5) and ++(0,-.5) .. (0.6,2.0);
    \draw[thick,  ] (1.2,0) .. controls ++(0,1) and ++(0,-1) .. (0,2.0);
    \node at (.6,-.2) {$\scs i$};
\end{tikzpicture}}
\;\; = \;\; \sum_{a+b+c=i-1}(-1)^c
\hackcenter{\begin{tikzpicture}[scale=0.8]
    \draw[thick, ] (0,0) -- (0,2)  node[pos=.5, shape=coordinate](DOT){};
    \draw[thick,red, double, ] (.6,0) --  (.6,2) node[pos=.75, shape=coordinate](DOT2){};
    \draw[thick, ] (1.2,0) -- (1.2,2)  node[pos=.5, shape=coordinate](DOT1){};
    \filldraw  (DOT) circle (2.5pt);
    \filldraw  (DOT1) circle (2.5pt);
    \filldraw[red]  (DOT2) circle (2.5pt);
 \node at (-.4,1) {$\scs a$};
    \node at (1.6,1) {$\scs b$};
    \node at (.6,-.2) {$\scs i$};
        \node at (.9,1.5) {$\scs {\sf e}_c$};
\end{tikzpicture}}
\end{equation}

\begin{lemma}\label{RthickBr2-lemma}

\begin{equation}\label{RthickBr2-rel}
\hackcenter{\begin{tikzpicture}[scale=0.8]
    \draw[thick, double] (0,0) .. controls ++(0,.5) and ++(0,-.5) .. (.75,1);
    \draw[thick, red, ] (.75,0) .. controls ++(0,.5) and ++(0,-.5) .. (0,1);
    \draw[thick,red,  ] (0,1 ) .. controls ++(0,.5) and ++(0,-.5) .. (.75,2);
    \draw[thick, double, ] (.75,1) .. controls ++(0,.5) and ++(0,-.5) .. (0,2);
    \node at (0,-.25) {$\;$};
    \node at (0,2.25) {$\;$};
    \node at (0,-.2) {$\scs k$};
\end{tikzpicture}}
 \;\; = \;\;
\sum_{d_1+d_2=k}(-1)^{d_2}
\hackcenter{\begin{tikzpicture}[scale=0.8]
    \draw[thick, double, ] (0,0) -- (0,1.5)  node[pos=.55, shape=coordinate](DOT){};
    \filldraw  (DOT) circle (2.5pt);
    \draw[red,thick,  ] (.75,0) -- (.75,1.5) node[pos=.55, shape=coordinate](DOT2){};
    \filldraw[red]  (DOT2) circle (2.75pt);
    \node at (0,-.2) {$\scs k$};
    \node at (-.3,.75) {$\scs {\sf e}_{d_1}$};
    \node at (1.15,.75) {$\scs d_2$};
\end{tikzpicture}}
%
\qquad \quad
\hackcenter{\begin{tikzpicture}[scale=0.8]
    \draw[thick, red] (0,0) .. controls ++(0,.5) and ++(0,-.5) .. (.75,1);
    \draw[thick, double ] (.75,0) .. controls ++(0,.5) and ++(0,-.5) .. (0,1);
    \draw[ thick, double, ] (0,1 ) .. controls ++(0,.5) and ++(0,-.5) .. (.75,2);
    \draw[thick,red,  ] (.75,1) .. controls ++(0,.5) and ++(0,-.5) .. (0,2);
    \node at (0,-.25) {$\;$};
    \node at (0,2.25) {$\;$};
    \node at (.8,-.2) {$\scs k$};
\end{tikzpicture}}
 \;\; = \;\;
\sum_{d_1+d_2=k}(-1)^{d_1}
\hackcenter{\begin{tikzpicture}[scale=0.8]
    \draw[thick, double, ] (.75,0) -- (.75,1.5)  node[pos=.55, shape=coordinate](DOT){};
    \filldraw  (DOT) circle (2.5pt);
    \draw[red,  ] (0,0) -- (0,1.5) node[pos=.55, shape=coordinate](DOT2){};
    \filldraw[red]  (DOT2) circle (2.75pt);
    \node at (.8,-.2) {$\scs k$};
    \node at (1.15,.75) {$\scs {\sf e}_{d_2}$};
    \node at (-.4,.75) {$\scs d_1$};
\end{tikzpicture}}
\end{equation}
\end{lemma}

\begin{proof}
This follows easily from the defining relations along with the definition of the thick black strand.
\end{proof}

\begin{lemma}
Recall the diagrams defined in \eqref{diagramsCk1}.  We have the following equality.
\begin{equation}\label{rb-cross-C}
\hackcenter{
\begin{tikzpicture}[scale=.75]
\draw[thick,] (-1.5,2) .. controls ++(0,.25) and ++(0,-.25) .. (-.9,2.5) -- (-.9,4) .. controls ++(0,.25) and ++(0,-.25) .. (-1.5,4.5);
\draw[thick,] (-.9,2) .. controls ++(0,.25) and ++(0,-.25) .. (-.3,2.5) -- (-.3,4.5);
\draw[thick,] (-.3,2) .. controls ++(0,.25) and ++(0,-.25) .. (.3,2.5) -- (.3,4.5);
\draw[thick,] (.9,2) -- (.9,4.5);
\draw[thick,red,double] (.3,2) .. controls ++(0,.25) and ++(0,-.25) .. (-1.5,2.5) -- (-1.5,4) .. controls ++(0,.25) and ++(0,-.25) .. (-.9,4.5);
\draw[fill=white!20,] (-1.1,2.75) rectangle (1.1,3.75);
\node at (0,3.25) {$C_k$};
    \node at (.3,1.8) {$\scs 1$};
\end{tikzpicture}}
=
\hackcenter{
\begin{tikzpicture}[scale=.75]
\draw[thick,] (-1.5,2) -- (-1.5,4.5);
\draw[thick,] (-.9,2) -- (-.9,4) .. controls ++(0,.25) and ++(0,-.25) ..  (-.3,4.5);
\draw[thick,] (-.3,2) -- (-.3,4) .. controls ++(0,.25) and ++(0,-.25) .. (.3,4.5);
\draw[thick,] (.9,2) .. controls ++(0,.25) and ++(0,-.25) .. (.3,2.5) -- (.3,4)  .. controls ++(0,.25) and ++(0,-.25) .. (.9,4.5);
\draw[thick,red,double] (.3,2) .. controls ++(0,.25) and ++(0,-.25) .. (.9,2.5) -- (.9,4) .. controls ++(0,.25) and ++(0,-.25) .. (-.9,4.5);
\draw[fill=white!20,] (-1.7,2.75) rectangle (.5,3.75);
\node at (-.6,3.25) {$C_k$};
    \node at (.3,1.8) {$\scs 1$};
\end{tikzpicture}}
+
\sum_{\ell=1}^{k-1}
\hackcenter{
\begin{tikzpicture}[scale=.75]
\draw[thick,] (-2.7,2) to (-2.7,4.5);
\draw[thick,] (-2.1,2) -- (-2.1,3) .. controls ++(0,.25) and ++(0,-.25) .. (-1.8,3.5) -- (-1.8,4.5);
\draw[thick,] (-1.2,2) -- (-1.2,3) .. controls ++(0,.25) and ++(0,-.25) .. (-.9,3.5) -- (-.9,4.5);
\draw[thick,] (-.6,2)--(-.6,3) .. controls ++(0,.25) and ++(0,-.25) .. (-.3,3.5) -- (-.3,4.5);
\draw[thick,] (.3,2)--(.3,3) .. controls ++(0,.25) and ++(0,-.25) .. (.6,3.5) -- (.6,4.5);
\draw[thick,] (1.2,2)--(1.2,3) .. controls ++(0,.25) and ++(0,-.25) .. (1.2,3.5) -- (1.2,4.5);
\draw[thick,red,double] (.75,2) -- (.75,3) .. controls ++(0,.25) and ++(0,-.25) .. (-2.45,3.5) -- (-2.45,4.5);
\draw[fill=white!20,] (-.5,3.5) rectangle (1.4,4.25);
\node at (.45,3.875) {$C_{k-\ell}$};
\draw[fill=white!20,] (-2.9,2.25) rectangle (-1,3);
\node at (-1.8,2.625) {$C_{\ell}$};
    \node at (.75,1.8) {$\scs 1$};
\node at (-.15,2.5) {$\cdots$};
\node at (-1.35,4) {$\cdots$};
\end{tikzpicture}}
\end{equation}

\end{lemma}
\proof
The lemma follows from repeated use of \eqref{r3-2}.

\endproof

\subsection{A faithful representation and basis}
We will construct a faithful (left) action of $W(\mathbf{s},n)$ on a direct sum of polynomial algebras.

For $1 \leq j \leq m+n$ let
\[
\#_{B,\mathbf{i}}(j)=|\{r | 1 \leq r \leq j-1, \mathbf{i}_r=\mathfrak{b}          \} |
\quad \quad
\#_{R,\mathbf{i}}(j) = j-1 - \#_{B,\mathbf{i}}(j).
\]
In other words, $\#_{B,\mathbf{i}}(j)$ is the number of the first $j-1$ entries of $\mathbf{i}$ which are $\mathfrak{b}$ and $\#_{R,\mathbf{i}}(j)$ is the number of the first $j-1$ entries which belong to $I_{\bf s}$.
When the sequence $\mathbf{i}$ is fixed we simply write $\#_B(j)$ and $\#_R(j)$.

Let $Y_{s_j}$ be an alphabet in $s_j$ variables and define
\[
R=\Sym(Y_{s_1}) \otimes \cdots \otimes \Sym(Y_{s_m})
\]
where $\Sym(Y_{s_j})$ is the algebra over $\Bbbk$ of symmetric functions in the alphabet $Y_{s_j}$.  Note that $R$ is determined by the sequence $\mathbf{s}$.

For each $\mathbf{i}=(i_1,...,i_{m+n}) \in  \mathrm{Seq}(\mathbf{s},n)$ in which $n$ of the letters are $\mf{b}$,  we associate a vector space
\begin{equation*}
V_{\mathbf{i}} =  R[X_{1, \mathbf{i}}, \ldots, X_{n, \mathbf{i}}]
\hspace{.5in}
V = \bigoplus_{\mathbf{i} \in \mathrm{Seq}(\mathbf{s},n)} V_{\mathbf{i}} .
\end{equation*}
Sometimes we write $f \in V_{\mathbf{i}}$ as $f(X_{\bf i})$ to emphasize the alphabet for which the variables of $f$ are taken from.

For $1 \leq r \leq n-1$ let
\[
\partial_r \colon V_{\mathbf{i}} \rightarrow V_{\mathbf{i}} \quad \quad
f(X_{1, \mathbf{i}}, \ldots, X_{n, \mathbf{i}}) \mapsto
\frac{f(X_{1, \mathbf{i}}, \ldots, X_{n, \mathbf{i}})-f(X_{\sigma_r(1), \mathbf{i}}, \ldots, X_{\sigma_r(n), \mathbf{i}})}
{X_{r, \mathbf{i}}-X_{r+1, \mathbf{i}}}
\]
be the $r$-th divided difference operator where each $\sigma_r \in S_n$ is the simple transposition fixing all elements except $r$ and $r+1$.

Each generator of $W(\mathbf{s},n)$ is declared to act trivially (by zero) on $V$ except in the following cases.

\begin{itemize}
\item $e(\mathbf{i}) \colon f \in V_{\mathbf{i}} \mapsto f \in V_{\mathbf{i}}$
\item $x_j e(\mathbf{i}) \colon f \in V_{\mathbf{i}} \mapsto X_{\#_B(j)+1, \mathbf{i}} f \in V_{\mathbf{i}}$,
if  $\mathbf{i}_j = \mathfrak{b}$
\item $E(d)_j e(\mathbf{i}) \colon f \in V_{\mathbf{i}} \mapsto {\sf e}_d(Y_{s_{\#_R(j)+1}}) f \in V_{\mathbf{i}}$,
if  $\mathbf{i}_j \neq \mathfrak{b}$
\item $\psi_j e(\mathbf{i}) \colon f \in V_{\mathbf{i}} \mapsto
\partial_{\#_B(j)+1} f \in V_{\mathbf{i}}$,
if  $\mathbf{i}_j =\mathbf{i}_{j+1}= \mathfrak{b}$
\item $\psi_j e(\mathbf{i}) \colon f(X_{\bf i}) \in V_{\mathbf{i}} \mapsto
f(X_{\sigma_j({\bf i})}) \in V_{\sigma_j(\mathbf{i})} $,
if $\mathbf{i}_j \neq \mathfrak{b}$ and $\mathbf{i}_{j+1}=\mathfrak{b}$
\item $\psi_j e(\mathbf{i}) \colon f(X_{\bf i}) \in V_{\mathbf{i}} \mapsto
\sum_{\alpha+\beta=\mathbf{i}_{j+1}} (-1)^{\beta}
X^{\alpha}_{ \#_B(j)+1, \sigma_j({\bf i}) }
{\sf e}_{\beta}(Y_{s_{\#_R(j+1)+1}} )
f(X_{\sigma_j({\bf i})}) \in V_{\sigma_j(\mathbf{i})} $,
if $\mathbf{i}_j = \mathfrak{b}$ and $\mathbf{i}_{j+1} \neq \mathfrak{b}$
\end{itemize}

\begin{proposition}
\label{proprep}
The action of the generators of $W(\mathbf{s},n)$ above extends to a  representation of $W(\mathbf{s},n)$ on $V$.
\end{proposition}

\begin{proof}
The fact that the generators of $W(\mathbf{s},n)$ prescribed above extend to a representation of the algebra on $V$ is a routine exercise.
\end{proof}

Let $\mathbf{i}, \mathbf{i}'  \in \mathrm{Seq}(\mathbf{s},n)$.  There is a partial order on
$\mathrm{Seq}(\mathbf{s},n)$ be declaring $\mathbf{i} < \mathbf{i}'$ if

\begin{equation}
\label{order}
\#_{B,\mathbf{i}}(j) > \#_{B,\mathbf{i}'}(j)
\quad \quad \text{ for } j=1,\ldots,m+n.
\end{equation}
Informally, the more to the left the black strands of $\mathbf{i}$ are, the smaller $\mathbf{i}$ is.

\begin{example}
Let $\mathbf{s}=(1,1)$ and $n=1$.  Then we order the elements of $\mathrm{Seq}(\mathbf{s},n)$ as follows:
\[
(\mathfrak{b},1,1) < (1,\mathfrak{b},1) < (1,1,\mathfrak{b}).
\]
\end{example}

Consider a diagram $D$ whose bottom boundary is determined by the sequence $\mathbf{i}$ and whose top boundary is determined by $\mathbf{i}'$.
Using the relations of the algebra we may assume that any pair of strands of $D$ intersect at most once and that all of the dots are at the bottom of the diagram.
Such a $D$ is determined by a minimal presentation $\tilde{w}=\sigma_{i_1} \cdots \sigma_{i_r}$ for an element $w \in S_{n+m}$ (along with the dots at the bottom).

Let ${}_{\mathbf{i}'} S_{\mathbf{i}} $ be the subset of $S_{n+m}$ consisting of permutations each of which takes $\mathbf{i}$ to $\mathbf{i}'$ by the standard action of $S_{n+m}$ on sequences.  For each $w \in {}_{\mathbf{i}'} S_{\mathbf{i}}$
we convert a minimal presentation $\tilde{w}$ into an element $\hat{w}$ of
$e(\mathbf{i}') W(\mathbf{s},n) e(\mathbf{i})$.
Let
\[
{}_{\mathbf{i}'} \hat{S}_{\mathbf{i}} =
\{ \hat{w} | w \in {}_{\mathbf{i}'} S_{\mathbf{i}}
\}.
\]
Let
\[
{}_{\mathbf{i}'} B {}_{\mathbf{i}}
=
\{
\hat{w}
\prod_{j \notin I_{\mathbf{s}}} x_j^{r_j}
\prod_{j \in I_{\mathbf{s}}}
\prod_{\gamma=1}^{\mathbf{i}_j}
E(\gamma)_j^{r_{\gamma_j}}
e(\mathbf{i})
| w \in {}_{\mathbf{i}'} S_{\mathbf{i}},
r_j, r_{\gamma_j} \in \Z_{\geq 0},
 \}.
\]
The diagram for each element in this set consists of crossings matching
$\mathbf{i}'$ and $\mathbf{i}$ on the top and dots on the bottom.  The red dots are labelled by elementary symmetric functions.

The proof of the following proposition is similar to the proof of \cite[Theorem 2.5]{KL1}.

\begin{proposition} \label{prop:basis}
$ e(\mathbf{i}') W(\mathbf{s},n) e(\mathbf{i})$ is a free graded abelian group with a homogeneous basis
${}_{\mathbf{i}'} B {}_{\mathbf{i}}$.
\end{proposition}

\begin{proof}
It is straightforward to check that ${}_{\mathbf{i}'} B {}_{\mathbf{i}}$ is a spanning set.

Let $ {}_{\mathbf{i}'} M_{\mathbf{i}}$ be the dotless diagram with fewest number of crossings such that the bottom boundary points correspond to $\mathbf{i}$ and the top boundary points correspond to $\mathbf{i}'$.
Note that strands with the same color do not intersect in $ {}_{\mathbf{i}'} M_{\mathbf{i}}$ and that $ {}_{\mathbf{i}} M_{\mathbf{i}}=e(\mathbf{i})$.
We have
\begin{equation}
\label{doublematchresolve}
({}_{\mathbf{i}} M_{\mathbf{i}'})
({}_{\mathbf{i}'} M_{\mathbf{i}})
=
\prod \sum_{\alpha+\beta=\mathbf{i}_b} (-1)^{\beta} x_a^{\alpha} E({\beta})_b
+
\prod \sum_{\alpha+\beta=\mathbf{i}_a} (-1)^{\alpha} E({\alpha})_a x_b^{\beta}
\end{equation}
where the first product in ~\eqref{doublematchresolve} is over pairs $(a,b)$ such that $ 1 \leq a < b \leq m+n$ and the arcs in ${}_{\mathbf{i}'} M_{\mathbf{i}}$ ending at the $a$-th and $b$-th bottom points counting from the left intersect and $ \mathbf{i}_a=\mathfrak{b}$.
The second product in ~\eqref{doublematchresolve} is over pairs $(a,b)$ such that  $ 1 \leq a < b \leq m+n$ and the arcs in ${}_{\mathbf{i}'} M_{\mathbf{i}}$ ending at the $a$th and $b$th bottom points counting from the left intersect and $ \mathbf{i}_a \neq \mathfrak{b}$.
Thus elements of the form $ {}_{\mathbf{i}'} M_{\mathbf{i}}$ are non-zero.

We prove that the elements of ${}_{\mathbf{i}'} B {}_{\mathbf{i}}$ are linearly independent by induction on $\mathbf{i}'$ using the partial order defined in ~\eqref{order}.  The minimal possible $\mathbf{i}'$ is of the form
\[
(\mathfrak{b}, \ldots, \mathfrak{b}, s_1, \ldots, s_m).
\]
Each $w \in {}_{\mathbf{i}'} S_{\mathbf{i}}$ could be written as
$w=w_1 w_0$ where $w_0$
is an element in ${}_{\mathbf{i}'} S_{\mathbf{i}}$ with no black strands crossing
and $w_1 \in S_n$.
Each minimal length representative $\tilde{w}_1 $ of $w_1$ determines the same element $\hat{w}_1$ of ${}_{\mathbf{i}'} \hat{S}_{\mathbf{i}}$.
Similarly,
each minimal length representative $\tilde{w}_0 $ of $w_0$ determines the same element $\hat{w}_0$ of ${}_{\mathbf{i}'} \hat{S}_{\mathbf{i}'}$.

The elements of ${}_{\mathbf{i}'} B {}_{\mathbf{i}}$ are of the form
$\hat{w}_1 \hat{w}_0 x^{u} e(\mathbf{i})$ where
$x^u$ is a product of dots.
Let $f$ be a monomial in $V_{\mathbf{i}}$.  By the action defined in Proposition ~\ref{proprep}, the element
$ \hat{w}_0 x^{u}$ takes $f$ to $x^{u} f \in V_{\mathbf{i}'}$ (suitably relabelled).
The element $\hat{w}_1$ acts on monomials as products of divided difference operators.  Since the action of the nilHecke algebra on the ring of polynomials is faithful, the elements $\hat{w}_1 \hat{w}_0 x^{u} e(\mathbf{i})$ are linearly independent.

Assume inductively that the elements of ${}_{\mathbf{i}'} B {}_{\mathbf{i}}$ are linearly independent.  Assume $\mathbf{i}_k'=\mathfrak{b}, \mathbf{i}_{k+1}' \neq \mathfrak{b}$ and that $\mathbf{i}''=\sigma_k (\mathbf{i}')$.
In order to show that ${}_{\mathbf{i}''} B {}_{\mathbf{i}}$ is a linearly independent set, we consider its image under the map
\[
\psi_k \colon e(\mathbf{i}'') W(\mathbf{s},n) e(\mathbf{i})
\rightarrow
e(\mathbf{i}') W(\mathbf{s},n) e(\mathbf{i}).
\]

Define a partial order on ${}_{\mathbf{i}'} B {}_{\mathbf{i}}$ by
$w_1 f_1 < w_2 f_2$ if $l(w_1)<l(w_2)$ or if $w_1=w_2$ and the dots on the first $j-1$ red strands counting from left to right are the same but the degree of $E(s_j)$ in $f_1$ is less than the degree of $E(s_j)$ in $f_2$ where the $j$-th red strand is labelled $s_j$.  Extend this partial order to a total order.

Let $\delta \colon {}_{\mathbf{i}''} B {}_{\mathbf{i}} \rightarrow {}_{\mathbf{i}'} B {}_{\mathbf{i}} $ be defined by $\delta(y)= \psi_k y$ if the strands with top boundary points at positions $k$ and $k+1$ are disjoint.

If the strands do intersect, set $\delta(y)=y' E(d)_{\ell}$
where $y'$ is obtained from $y$ by removing the crossing between the two strands and the red strand ending at position $k+1$ on the top is labelled by $d$ and the bottom boundary of that strand is at position $\ell$ from the left.

The map $\delta$ is clearly injective.
Note that $\psi_k y=\delta(y) $ plus lower terms.  Now the linear independence of
${}_{\mathbf{i}''} B {}_{\mathbf{i}}$ follows from the inductive hypothesis.
\end{proof}

\begin{corollary}
The action of $W(\mathbf{s},n)$ on $V$ is faithful.
\end{corollary}

\begin{example}
\label{basisexample}
Let $\mathbf{s}=(1,1)$ and $n=1$.
There are three elements in $\mathrm{Seq}(\mathbf{s},n)$:
\[
\mathbf{i}=(1,\mathfrak{b},1), \quad \mathbf{j}=(\mathfrak{b},1,1), \quad \mathbf{k}=(1,1,\mathfrak{b}).
\]

Elements of ${}_{\mathbf{j}} B {}_{\mathbf{i}}$, ${}_{\mathbf{i}} B {}_{\mathbf{i}}$, and
${}_{\mathbf{k}} B {}_{\mathbf{i}}$ respectively are of the form:
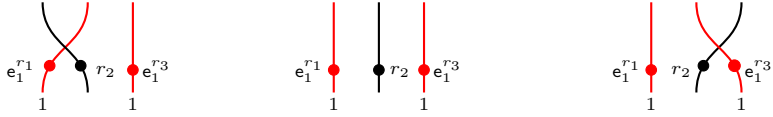
\[
\hackcenter{\begin{tikzpicture}[scale=0.8]
    \draw[red,thick, ] (0,0) .. controls (0,.75) and (.75,.75) .. (.75,1.5)
        node[pos=.25, shape=coordinate](DOT){};
            \filldraw[red]  (DOT) circle (2.5pt);
    \draw[thick, ] (.75,0) .. controls (.75,.75) and (0,.75) .. (0,1.5)
            node[pos=.25, shape=coordinate](DOT){};
    \filldraw[]  (DOT) circle (2.5pt);
        \draw[red,thick,  ] (1.5,0) -- (1.5,1.5) node[pos=.25, shape=coordinate](DOT2){};
    \filldraw[red]  (DOT2) circle (2.5pt);
    \node at (0,-.2) {$\scs 1$};
        \node at (1.5,-.2) {$\scs 1$};
    \node at (-.35,.35) {$\scs {\sf e}_{1}^{r_1}$};
                    \node at (1.9,.35) {$\scs {\sf e}_{1}^{r_3}$};
                                    \node at (1.05,.35) {$\scs r_2$};
\end{tikzpicture}}
\quad \quad \quad \quad
\hackcenter{\begin{tikzpicture}[scale=0.8]
        \draw[red,thick,  ] (-.75,0) -- (-.75,1.5) node[pos=.25, shape=coordinate](DOT2){};
        \filldraw[red]  (DOT2) circle (2.5pt);
             \draw[thick, ] (0,0) -- (0,1.5) node[pos=.25, shape=coordinate](DOT2){};
                     \filldraw[]  (DOT2) circle (2.5pt);
            \node at (-0.75,-.2) {$\scs 1$};
            \node at (0.75,-.2) {$\scs 1$};
               \draw[red,thick,  ] (.75,0) -- (.75,1.5) node[pos=.25, shape=coordinate](DOT2){};
                       \filldraw[red]  (DOT2) circle (2.5pt);
        \node at (-1.15,.35) {$\scs {\sf e}_{1}^{r_1}$};
                \node at (.35,.35) {$\scs r_2$};
                \node at (1.15,.35) {$\scs {\sf e}_{1}^{r_3}$};
\end{tikzpicture}}
\qquad \quad \quad \quad
\hackcenter{\begin{tikzpicture}[scale=0.8]
        \draw[red,thick,  ] (-.75,0) -- (-.75,1.5) node[pos=.25, shape=coordinate](DOT2){};
                \filldraw[red]  (DOT2) circle (2.5pt);
    \draw[thick, ] (0,0) .. controls (0,.75) and (.75,.75) .. (.75,1.5)node[pos=.25, shape=coordinate](DOT2){};
                    \filldraw[]  (DOT2) circle (2.5pt);
    \draw[red,thick, ] (.75,0) .. controls (.75,.75) and (0,.75) .. (0,1.5)
        node[pos=.25, shape=coordinate](DOT){};
    \filldraw[red]  (DOT) circle (2.75pt);
            \node at (-0.75,-.2) {$\scs 1$};
            \node at (0.75,-.2) {$\scs 1$};
            \node at (-1.15,.35) {$\scs {\sf e}_{1}^{r_1}$};
    \node at (1.05,.35) {$\scs {\sf e}_{1}^{r_3}$};
                    \node at (-.25,.35) {$\scs r_2$};
\end{tikzpicture}}
\]
where $r_1, r_2, r_3 \in \Z_{\geq 0}$.
\end{example}

\subsection{Connection to Webster's algebra}

Let $\tilde{T}$ be the algebra introduced by Webster in ~\cite[Section 4]{Web} for $\mathfrak{sl}_2$ for $\underline{\lambda}=(s_1,\ldots,s_m)$ and $(i_1,\ldots,i_n)=(1,\ldots,1)$.

Let $I(\mathbf{s},n)$ be the two-sided ideal which is generated by $ E(d)_j$, where $1\leq j\leq m+n$ and $1\leq d$.  Then $\tilde{T} \cong W(\mathbf{s},n)/I(\mathbf{s},n)$.




\subsection{Symmetries of the algebra}
There is an algebra automorphism  $\tau$ of the nilHecke algebra defined diagrammatically  by the symmetry of rescaling each crossing $\xy {\ar@{-} (2.5,-2.5)*{};(-2.5,2.5)*{}}; {\ar@{-} (-2.5,-2.5)*{};(2.5,2.5)*{} };
(4,0)*{};(-4,0)*{};\endxy
\mapsto -\xy {\ar@{-} (2.5,-2.5)*{};(-2.5,2.5)*{}}; {\ar@{-} (-2.5,-2.5)*{};(2.5,2.5)*{} };
(4,0)*{};(-4,0)*{};\endxy$ and reflecting diagrams across the vertical axis.  Algebraically this corresponds to the algebra homomorphism defined on generators by
\begin{equation}
\tau \maps x_{i} \mapsto x_{n+1-i}, \qquad \partial_{i} \mapsto - \partial_{n-i}.
\end{equation}
This symmetry immediately extends to the algebras $W(\mathbf{s},n)$ since the relations involving red strands remain invariant under this symmetry.

In \cite{Lau1} this symmetry (denoted by $\tilde{\sigma}$ there) is extended to define a 2-functor $\tau \maps \cal{U}(\mf{sl}_2) \to \cal{U}(\mf{sl}_2)$ define diagrammatically by rescaling each crossing, reflecting across the vertical axis, and sending weights $\lambda$ to $-\lambda$.   In \cite{KLMS} the symmetry $\tau$ is used to deduce new thick calculus relations from others by applying $\tau$.  Working with the symmetry $\tau$ in the Karoubi envelope $\dot{\cal{U}} = Kar(\cal{U})$ is somewhat subtle as the idempotents $e_a$ used to define divided powers $\cE^{(a)}$ are not stable under the symmetry $\tau$.  One can see that
\begin{equation}
 \tau
 \left( \; \;\; \hackcenter{\begin{tikzpicture}
    \draw[thick, ->] (-1.2,0) -- (-1.2,1.5);
    \draw[thick, ->] (-.6,0) -- (-.6,1.5);
    \draw[thick, -> ] (.6,0) -- (.6,1.5);
        \draw[thick, -> ] (1.2,0) -- (1.2,1.5);
    \node[draw, fill=white!20 ,rounded corners ] at (0,.75) {$ \qquad \quad e_a \quad \qquad$};
    \node at (0, .35) {$\cdots$};
    \node at (1.75,1.0) {$ \lambda$};
\end{tikzpicture}}
\; \right)
\quad := \quad (-1)^{\frac{a(a-1)}{2}}\;
   \hackcenter{\begin{tikzpicture}
    \draw[thick, ->] (-1.2,0) -- (-1.2,1.5);
    \draw[thick, ->] (-.6,0) -- (-.6,1.5);
    \draw[thick, -> ] (.6,0) -- (.6,1.5);
        \draw[thick, -> ] (1.2,0) -- (1.2,1.5);
    \node[draw, fill=white!20 ,rounded corners ] at (0,.45) {$ \qquad \quad D_a \quad \qquad$};
    \node at (0, .95) {$\cdots$};
    \node at (-.6,1.15) {$\bullet$};
    \node at (.6,1.15) {$\bullet$};
    \node at (.27,1.3) {$\scs  {a-2}$};
    \node at (1.2,1.15) {$\bullet$};
    \node at (1.5,1.3) {$\scs a-1$};
    \node at (-1.75,1.0) {$  -\lambda$};
\end{tikzpicture}}
\end{equation}
and that the idempotent $\oneml e_a$ is equivalent to the idempotent $\tau(e_a\onel)$ since
\[
\tau(e_a\onel) (\oneml e_a) = \tau(e_a\onel), \quad
\text{and} \quad
(\oneml e_a)\tau(e_a\onel) = (\oneml e_a).
\]
Since these idempotents are equivalent they give rise to isomorphic 1-morphisms in $\dot{\cal{U}}$.

In what follows we would like to apply the symmetry $\tau$ to thick calculus identities in $W(\mathbf{s},n)$ in order to obtain new thick calculus identities.  However, doing so is complicated by the above observation that $\tau( \cE^{(a)} \onel) := ( \oneml \cE^{a}, \tau(e_a\onel))$ so to obtain new thick calculus identities we must use the isomorphisms relating $\tau(e_a\onel)$ to $\oneml e_a$.

\subsubsection{Signed sequences}
Let $\uep=\epsilon_1^{a_1} \dots \varepsilon_m^{a_m}$ denote a {\it divided power signed sequence} with each $\varepsilon_i \in \{ +, -\}$ and each $a_i \in \Z_{\geq 0}$. To each such sequence we define a 1-morphism
\[
\cE_{\uep} \onel= \cE_{\varepsilon_1}^{(a_1)} \dots\cE_{\varepsilon_m}^{(a_m)} \onel
\]
in $\dot{\cal{U}}$, where $\cE_{+} := \cE$ and $\cE_{-} := \cF$.   Write
$
 e_{\uep}\onel
$
for the idempotent $e_{a_1} \otimes e_{a_2}\otimes \dots \otimes e_{a_m}\onel$ built from the horizontal composition of idempotents defining each $\cE_{\varepsilon_i}^{(a_i)}$, so that in $\dot{\cal{U}}$ we have
\[
 \cE_{\uep} \onel := (\cE_{\varepsilon_1}^{a_1} \dots \cE_{\varepsilon_m}^{a_m}\onel , e_{\uep}\onel).
\]

Given a divided power signed sequence $\uep$ let
\[
\tau( \uep) := \varepsilon_{m}^{a_m} \varepsilon_{m-1}^{a_{m-1}} \dots \varepsilon_{1}^{a_1}.
\]
This should not be confused with the result of applying the 2-isomorphism $\tau$ to the idempotent $e_{\uep}$.  These two possibilities are related by the following lemma.

\begin{lemma}
There is an isomorphism
\[
z_{\uep} \maps
\tau\left(\cE_{\uep} \onel, e_{\uep} \right)
    = \left(  \oneml\cE_{\varepsilon_m}^{a_m} \cE_{\varepsilon_{m-1}}^{a_{m-1}} \dots \cE_{\varepsilon_1}^{a_1}, \tau(e_{\uep}) \right) \longrightarrow \oneml\cE_{\tau(\uep)}
\]
between 1-morphisms in $\dot{\cal{U}}$.
\end{lemma}

\begin{proof}
Let $z_{\uep} = e_{\tau(\uep)}$ and $z^{-1}_{\uep} = \tau(e_{\uep})$.
\end{proof}

\begin{proposition}
Given an equality of 2-morphisms
\[
 X\onel = Y \onel \maps \cE_{\uep} \onel \to \cE_{\uep'}\onel
\]
in $\dot{\cal{U}}$ we have an equality
\[
\left(z_{\uep'}\tau(X\onel)z_{\uep}^{-1}\right) = \left(z_{\uep'}\tau(Y\onel)z_{\uep}^{-1}\right) \maps \oneml\cE_{\tau(\uep)} \to \oneml\cE_{\tau(\uep')}.
\]
\end{proposition}

The above proposition gives rise to a simple effective rule for obtaining new thick calculus identities by reflecting thick diagrams across the vertical axis and rescaling all splitters
\begin{equation}
\tau \left( \; \hackcenter{ \begin{tikzpicture} [scale=.75]
\draw[thick,  double, <-](0,2).. controls ++(0,-.75) and ++(0,.3) ..(.6,1);
\draw[thick,  double, <- ](1.2,2).. controls ++(0,-.75) and ++(0,.3) ..(.6,1) to (.6,0);;
 \node at (0,2.2) {$\scs a$};
 \node at (1.2,2.2) {$\scs b$};
\node at (.6,-.2) {$\scs a+b$};
\end{tikzpicture}}
\; \right)
\;\; = \;\;
(-1)^{ab}
\hackcenter{ \begin{tikzpicture} [scale=.75]
\draw[thick,  double, <-](0,2).. controls ++(0,-.75) and ++(0,.3) ..(.6,1);
\draw[thick,  double, <- ](1.2,2).. controls ++(0,-.75) and ++(0,.3) ..(.6,1) to (.6,0);;
 \node at (0,2.2) {$\scs b$};
 \node at (1.2,2.2) {$\scs a$};
\node at (.6,-.2) {$\scs a+b$};
\end{tikzpicture}}
\end{equation}
Note that this definition is consistent with the factorization of a thin crossing through a thickness two strand.
\[
\hackcenter{\begin{tikzpicture}[scale=0.75]
    \draw[thick, ->] (0,0) .. controls (0,.75) and (.75,.75) .. (.75,1.5);
    \draw[thick, ->] (.75,0) .. controls (.75,.75) and (0,.75) .. (0,1.5);
\end{tikzpicture}}
\;\; =
\;\;
\hackcenter{\begin{tikzpicture}[scale=0.75]
     \draw[thick,double, -] (0,.45) to (0,1.05);
     \draw[thick,->, ] (0,1.05) .. controls ++(.1,0) and ++(0,-.3) ..  (.35,1.5);
     \draw[thick,->, ] (0,1.05) .. controls ++(-.1,0) and ++(0,-.3) ..  (-.35,1.5);
     \draw[thick] (.35, 0).. controls ++(0,.3) and ++(.1,0) .. (0,.45);
     \draw[thick](-.35, 0) .. controls ++(0,.3) and ++(-.1,0).. (0,.45);
\end{tikzpicture}}
\]

\subsection{Inclusion map}

Let $\mathbf{s} = (s_{1}, s_{2}, \dots , s_{m})$ be a sequence of $m$ non-negative integers.
For $j=1,\ldots, m-1$,
denote by $\mathbf{s}^j$, the sequence of $m-1$ integers obtained from $\mathbf{s}$ by replacing the pair $(s_{j}, s_{j+1})$ with the singleton $s_{j}+s_{j+1}$. Denote by $\phi_{j,a}(\mathbf{s})$, where $j=1,...,m$ and $0\leq a\leq s_{j}$, the sequence of $m+1$ integers obtained from $\mathbf{s}$ by replacing the integer $s_j$ with the pair $(s_j-a,a)$:
\begin{eqnarray*}
\mathbf{s}^{j} &=& (s_{1}, \dots, s_{j-1},  s_{j}+  s_{j+1}, s_{j+2}, \dots, s_{m}),\\
\phi_{j,a}(\mathbf{s})&=&(s_1,\ldots,s_{j-1},s_{j}-a,a,s_{j+1},\ldots, s_m).
\end{eqnarray*}
Note that we have $\phi_{j,s_{j+1}}(\mathbf{s}^j)=\mathbf{s}$.

The map $\phi_{j,a}$ extends to a map from $\mathrm{Seq}(\mathbf{s},n)$ to $\mathrm{Seq}(\phi_{j,a}(\mathbf{s}),n)$ by replacing the integer $s_j$ in $\mathbf{i}$ with the pair $(s_j-a,a)$.
When $\mathbf{i}_\ell=s_j$,
\begin{eqnarray*}
\phi_{j,a}(\mathbf{i})&=& (...\mathbf{i}_{\ell-1},s_{j}-a,a,\mathbf{i}_{\ell+1},...)\in \mathrm{Seq}(\phi_{j,a}(\mathbf{s}),n).
\end{eqnarray*}

Define an inclusion map of algebras which is non-unital
\begin{equation}
 \Phi_{j,a} \maps W(\mathbf{s},n) \to W(\phi_{j,a}(\mathbf{s}), n)
\end{equation}
determined by sending idempotents $e(\mathbf{i})$ for $\mathrm{Seq}(\mathbf{s},n)$ by
\begin{equation*}
\Phi_{j,a}(e(\mathbf{i})) = e(\phi_{j,a}(\mathbf{i})).
\end{equation*}

Assume the $j$-th red entry occur in position $\ell'$ of $\mathbf{i}$.
The ``dot" generators $E(d)_\ell$ are mapped via
\begin{equation*}
\Phi_{j,a}(E(d)_\ell e(\mathbf{i})) =
\begin{cases}
E(d)_\ell e(\phi_{j,a}(\mathbf{i}))& \text{ if } \ell<\ell' \\
E(d)_{\ell+1} e(\phi_{j,a}(\mathbf{i}))& \text{ if } \ell'<\ell \\
\sum_{d_1+d_2=d}E(d_1)_{\ell}E(d_2)_{\ell+1} e(\phi_{j,a}(\mathbf{i}))& \text{ if } \ell'=\ell, 1\leq d\leq s_j
\end{cases}
\end{equation*}
On ``dot" generators $x_\ell$ define it by
\begin{equation*}
\Phi_{j,a}(x_\ell e(\mathbf{i})) =
\begin{cases}
x_\ell e(\phi_{j,a}(\mathbf{i}))& \text{ if } \ell<\ell' \\
x_{\ell+1} e(\phi_{j,a}(\mathbf{i}))& \text{ if } \ell'<\ell
\end{cases}
\end{equation*}
On ``crossing" generators $\psi_{\ell}$ define it by
\begin{equation*}
\Phi_{j,a}(\psi_\ell e(\mathbf{i})) =
\begin{cases}
\psi_{\ell} e(\phi_{j,a}(\mathbf{i})) & \text{ if } \ell< \ell'-1 \\
\psi_{\ell+1}\psi_{\ell} e(\phi_{j,a}(\mathbf{i})) & \text{ if } \mathbf{i}_{\ell'-1}=\mf{b}, \ell= \ell'-1 \\
\psi_{\ell}\psi_{\ell+1} e(\phi_{j,a}(\mathbf{i})) & \text{ if } \mathbf{i}_{\ell'+1}=\mf{b}, \ell= \ell' \\
\psi_{\ell+1} e(\phi_{j,a}(\mathbf{i})) & \text{ if } \ell> \ell'
\end{cases}
\end{equation*}

In particular, $\Phi_{j,a}$ is diagrammatically represented by
\begin{equation}
  \Phi_{j,a} \left(
\hackcenter{\begin{tikzpicture}[scale=0.8]
    \draw[thick,red, double, ] (0,0) to (0,1);
    \node   at (0,-.25) {$\scs s_j$};
    \filldraw[red]  (0,.5) circle (2.75pt);
    \node at (-.35,.5) {$\scs {\sf e}_d$};
\end{tikzpicture}}
  \right)
 \mapsto
  \sum_{d_1+d_2 = d}
  \hackcenter{\begin{tikzpicture}[scale=0.8]
    \draw[thick,red, double, ] (0,0) to (0,1);
    \draw[thick,red, double, ] (.8,0) to (.8,1);
    \node   at (0,-.25) {$\scs s_j-a  $};
    \node   at (.8,-.25) {$\scs   a$};
    \filldraw[red]  (0,.5) circle (2.75pt);
    \node at (-.35,.5) {$\scs {\sf e}_{d_1}$};
    \filldraw[red]  (.8,.5) circle (2.75pt);
    \node at (1.25,.5) {$\scs {\sf e}_{d_2}$};
\end{tikzpicture}}
\end{equation}

\begin{align}
 \Phi_{j,a}\left(  \;\;
\hackcenter{\begin{tikzpicture}[scale=0.8]
    \draw[thick, ] (0,0) .. controls (0,.75) and (.8,.75) .. (.8,1.5);
    \draw[thick,red, double, ] (.8,0) .. controls (.8,.75) and (0,.75) .. (0,1.5);
    \node   at (.8,-.25) {$\scs s_j$};
\end{tikzpicture}}
 \right)
 & \;\; \mapsto \;\;
 \hackcenter{\begin{tikzpicture}[scale=0.8]
    \draw[thick, ] (0,0) .. controls (0,.75) and (1.6,.75) .. (1.6,1.5);
    \draw[thick,red, double, ] (.8,0) .. controls (.8,.75) and (0,.75) .. (0,1.5);
    \draw[thick,red, double, ] (1.6,0) .. controls (1.6,.75) and (.8,.75) .. (.8,1.5);
    \node   at (.8,-.25) {$\scs s_j-a  $};
    \node   at (1.6,-.25) {$\scs  a$};
\end{tikzpicture}}
&
 \Phi_{j,a}\left(  \;
\hackcenter{\begin{tikzpicture}[scale=0.8]
    \draw[thick,red, double, ] (0,0) .. controls (0,.75) and (.8,.75) .. (.8,1.5);
    \draw[thick, ] (.8,0) .. controls (.8,.75) and (0,.75) .. (0,1.5);
    \node   at (0,-.25) {$\scs s_j$};
\end{tikzpicture}}
\; \right)
 \;\; \mapsto \;\;
 \hackcenter{\begin{tikzpicture}[scale=0.8]
    \draw[thick, ] (0,0) .. controls (0,.75) and (-1.6,.75) .. (-1.6,1.5);
    \draw[thick,red, double, ] (-.8,0) .. controls (-.8,.75) and (-0,.75) .. (-0,1.5);
    \draw[thick,red, double, ] (-1.6,0) .. controls (-1.6,.75) and (-.8,.75) .. (-.8,1.5);
    \node   at (-.8,-.25) {$\scs a$};
    \node   at (-1.6,-.25) {$\scs s_j-a$};
\end{tikzpicture}}
\end{align}

\begin{proposition}
The map
\begin{equation} \label{eq:inclusion}
\Phi_{j,s_{j+1}}:W(\mathbf{s}^j,n)\to W(\mathbf{s},n).
\end{equation}
is injective.
\end{proposition}

\begin{proof}
Injectivity of $\Phi_{j,s_{j+1}}$ is an immediate consequence of its definition on basis elements.
\end{proof}

\section{Reddotted bimodules}
In this section we will introduce certain bimodules $\mathsf{E}^{(a)}_i$ and
$\mathsf{F}^{(a)}_i$ over $W(\mathbf{s},n)$ for various sequences $\mathbf{s}$.
The Lie theoretic notation of these bimodules is chosen so that these bimodules will be the targets of the $1$-morphisms of a categorical quantum group action.

In order to describe these bimodules we will first introduce "halves" of these bimodules which we call {\it splitter} bimodules.  Then we will introduce
{\it ladder} bimodules.

\subsection{Splitter bimodules} \label{subsec:splitter}
The bimodules which we will first define algebraically are called splitter bimodules because of their graphical descriptions.

First note that the inclusion $\Phi_{j,s_{j+1}}$ determines the left action of $W(\mathbf{s}^j,n)$ on $W(\mathbf{s},n)$.
Define the $(W(\mathbf{s}^j,n) , W(\mathbf{s},n))$-bimodule $\bigtriangleup^j(\mathbf{s})$ as the bimodule
\begin{equation}
 \bigtriangleup^j(\mathbf{s}) := W(\mathbf{s}^j,n) \otimes_{W(\mathbf{s}^j,n)} W(\mathbf{s},n)\{-s_j \cdot s_{j+1}\}
\end{equation}
and the $(W(\mathbf{s},n),W(\mathbf{s}^j,n))$-bimodule $\bigtriangledown^j(\mathbf{s})$ as the bimodule
\begin{equation}
 \bigtriangledown^j(\mathbf{s}) := W(\mathbf{s},n) \otimes_{W(\mathbf{s}^j,n)} W(\mathbf{s}^j,n).
\end{equation}
We diagrammatically represent elements of these bimodules
as follows
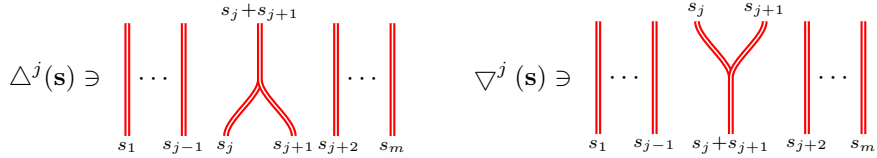
\[
\bigtriangleup^j(\mathbf{s}) \ni \hackcenter{ \begin{tikzpicture} [scale=.75]
\draw[thick,red, double, ] (-1.75,0)--(-1.75,2);
\draw[thick,red, double, ] (-.75,0)--(-.75,2);
\draw[thick, red, double] (0,0).. controls ++(0,.3) and ++(0,-.3) ..(.6,1);
\draw[thick,red, double, ] (1.2,0).. controls ++(0,.3) and ++(0,-.3) ..(.6,1) to (.6,2);
\draw[thick,red, double, ] (1.95,0)--(1.95,2);
\draw[thick,red, double, ] (2.95,0)--(2.95,2);
 \node at (-1.25,1) {$\cdots$};
 \node at (2.45,1) {$\cdots$};
 \node at (-1.75,-.2) {$\scs s_1$};
 \node at (-.75,-.2) {$\scs s_{j-1}$};
 \node at (2,-.2) {$\scs s_{j+2}$};
 \node at (2.95,-.2) {$\scs s_m$};
 \node at (0,-.2) {$\scs s_j$};
 \node at (1.2,-.2) {$\scs s_{j+1}$};
\node at (.6,2.2) {$\scs s_j+s_{j+1}$};
\end{tikzpicture}}
\qquad
\bigtriangledown^j(\mathbf{s}) \ni \hackcenter{ \begin{tikzpicture} [scale=.75]
\draw[thick,red, double, ] (-1.75,0)--(-1.75,2);
\draw[thick,red, double, ] (-.75,0)--(-.75,2);
\draw[thick,red, double, ](0,2).. controls ++(0,-.3) and ++(0,.3) ..(.6,1);
\draw[thick,red, double, ](1.2,2).. controls ++(0,-.3) and ++(0,.3) ..(.6,1) to (.6,0);
\draw[thick,red, double, ] (1.95,0)--(1.95,2);
\draw[thick,red, double, ] (2.95,0)--(2.95,2);
 \node at (-1.25,1) {$\cdots$};
 \node at (2.45,1) {$\cdots$};
 \node at (-1.75,-.2) {$\scs s_1$};
 \node at (-.75,-.2) {$\scs s_{j-1}$};
 \node at (1.95,-.2) {$\scs s_{j+2}$};
 \node at (2.95,-.2) {$\scs s_m$};
 \node at (0,2.2) {$\scs s_{j}$};
 \node at (1.2,2.2) {$\scs s_{j+1}$};
\node at (.6,-.2) {$\scs s_j+s_{j+1}$};
\end{tikzpicture}}
\]
with $n$ black strands intersecting the pictures in a manner similar to that described in Section ~\ref{sectiondiagdescription} along with red and black dots.

\begin{proposition}
\label{folk-rel}
There are relations
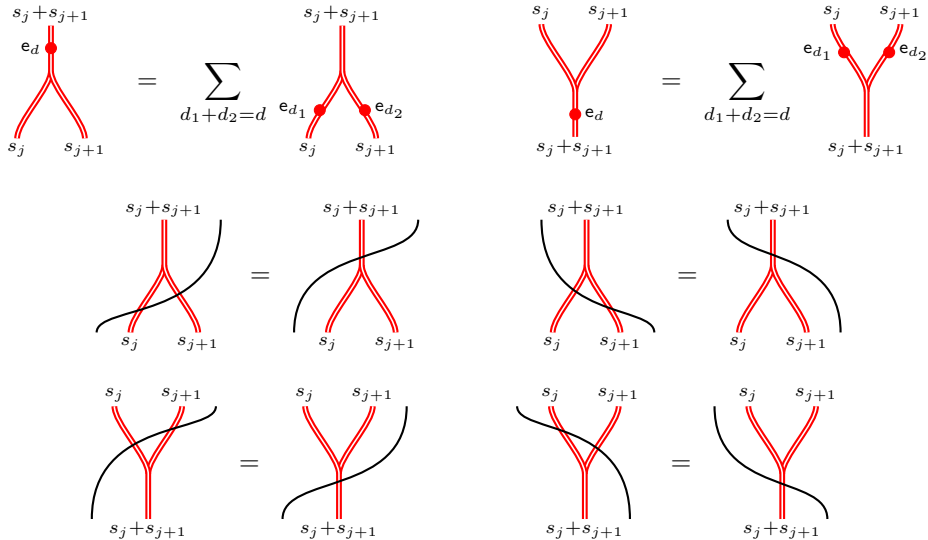
\[
\hackcenter{ \begin{tikzpicture} [scale=.75]
\draw[thick, red, double] (0,0).. controls ++(0,.3) and ++(0,-.3) ..(.6,1.2);
\draw[thick,red, double, ] (1.2,0).. controls ++(0,.3) and ++(0,-.3) ..(.6,1.2) to (.6,2);
\node at (0,-.2) {$\scs s_j$};
\node at (1.2,-.2) {$\scs s_{j+1}$};
\node at (.6,2.2) {$\scs s_j+s_{j+1}$};
\filldraw[red]  (0.6,1.6) circle (2.75pt);
\node at (.25,1.6) {$\scs {\sf e}_d$};
\end{tikzpicture}}
\;\; = \;\;
  \sum_{d_1+d_2 = d}
\hackcenter{ \begin{tikzpicture} [scale=.75]
\draw[thick, red, double] (0,0).. controls ++(0,.3) and ++(0,-.3) ..(.6,1.2);
\draw[thick,red, double, ] (1.2,0).. controls ++(0,.3) and ++(0,-.3) ..(.6,1.2) to (.6,2);
\node at (0,-.2) {$\scs s_j$};
\node at (1.2,-.2) {$\scs s_{j+1}$};
\node at (.6,2.2) {$\scs s_j+s_{j+1}$};
    \filldraw[red]  (0.2,.5) circle (2.75pt);
    \node at (-.25,.5) {$\scs {\sf e}_{d_1}$};
    \filldraw[red]  (1.0,.5) circle (2.75pt);
    \node at (1.45,.5) {$\scs {\sf e}_{d_2}$};
\end{tikzpicture}}
\qquad
\qquad
\hackcenter{ \begin{tikzpicture} [scale=.75, rotate=180]
\draw[thick, red, double] (0,0).. controls ++(0,.3) and ++(0,-.3) ..(.6,1.2);
\draw[thick,red, double, ] (1.2,0).. controls ++(0,.3) and ++(0,-.3) ..(.6,1.2) to (.6,2);
 \node at (0,-.2) {$\scs s_{j+1}$};
 \node at (1.2,-.2) {$\scs s_{j}$};
\node at (.6,2.2) {$\scs s_j+s_{j+1}$};
    \filldraw[red]  (0.6,1.6) circle (2.75pt);
    \node at (.25,1.6) {$\scs {\sf e}_d$};
\end{tikzpicture}}
\;\; = \;\;
  \sum_{d_1+d_2 = d}
\hackcenter{ \begin{tikzpicture} [scale=.75, rotate=180]
\draw[thick, red, double] (0,0).. controls ++(0,.3) and ++(0,-.3) ..(.6,1.2);
\draw[thick,red, double, ] (1.2,0).. controls ++(0,.3) and ++(0,-.3) ..(.6,1.2) to (.6,2);
 \node at (0,-.2) {$\scs s_{j+1}$};
 \node at (1.2,-.2) {$\scs s_{j}$};
\node at (.6,2.2) {$\scs s_j+s_{j+1}$};
    \filldraw[red]  (0.2,.5) circle (2.75pt);
    \node at (-.25,.5) {$\scs {\sf e}_{d_2}$};
    \filldraw[red]  (1.0,.5) circle (2.75pt);
    \node at (1.45,.5) {$\scs {\sf e}_{d_1}$};
\end{tikzpicture}}
\]
\[
\hackcenter{ \begin{tikzpicture} [scale=.75]
\draw[thick, red, double] (0,0).. controls ++(0,.3) and ++(0,-.3) ..(.6,1.2);
\draw[thick,red, double, ] (1.2,0).. controls ++(0,.3) and ++(0,-.3) ..(.6,1.2) to (.6,2);
 \node at (0,-.2) {$\scs s_j$};
 \node at (1.2,-.2) {$\scs s_{j+1}$};
\node at (.6,2.2) {$\scs s_j+s_{j+1}$};
\draw[thick,] (-.6,0) .. controls ++(0,.5) and ++(0,-1.8) .. (1.6,2);
\end{tikzpicture}}
\;\; = \;\;
\hackcenter{ \begin{tikzpicture} [scale=.75]
\draw[thick, red, double] (0,0).. controls ++(0,.3) and ++(0,-.3) ..(.6,1.2);
\draw[thick,red, double, ] (1.2,0).. controls ++(0,.3) and ++(0,-.3) ..(.6,1.2) to (.6,2);
 \node at (0,-.2) {$\scs s_j$};
 \node at (1.2,-.2) {$\scs s_{j+1}$};
\node at (.6,2.2) {$\scs s_j+s_{j+1}$};
\draw[thick,] (-.6,0) .. controls ++(0,1.6) and ++(0,-.8) .. (1.6,2);
\end{tikzpicture}}
\qquad
\qquad
\hackcenter{ \begin{tikzpicture} [scale=.75]
\draw[thick, red, double] (0,0).. controls ++(0,.3) and ++(0,-.3) ..(.6,1.2);
\draw[thick,red, double, ] (1.2,0).. controls ++(0,.3) and ++(0,-.3) ..(.6,1.2) to (.6,2);
 \node at (0,-.2) {$\scs s_j$};
 \node at (1.2,-.2) {$\scs s_{j+1}$};
\node at (.6,2.2) {$\scs s_j+s_{j+1}$};
\draw[thick,] (1.8,0) .. controls ++(0,.5) and ++(0,-1.8) .. (-.2,2);
\end{tikzpicture}}
\;\; = \;\;
\hackcenter{ \begin{tikzpicture} [scale=.75]
\draw[thick, red, double] (0,0).. controls ++(0,.3) and ++(0,-.3) ..(.6,1.2);
\draw[thick,red, double, ] (1.2,0).. controls ++(0,.3) and ++(0,-.3) ..(.6,1.2) to (.6,2);
 \node at (0,-.2) {$\scs s_j$};
 \node at (1.2,-.2) {$\scs s_{j+1}$};
\node at (.6,2.2) {$\scs s_j+s_{j+1}$};
\draw[thick,] (1.8,0) .. controls ++(0,1.6) and ++(0,-.8) .. (-.2,2);
\end{tikzpicture}}
\]
\[
\hackcenter{ \begin{tikzpicture} [scale=.75, rotate=180]
\draw[thick, red, double] (0,0).. controls ++(0,.3) and ++(0,-.3) ..(.6,1.2);
\draw[thick,red, double, ] (1.2,0).. controls ++(0,.3) and ++(0,-.3) ..(.6,1.2) to (.6,2);
 \node at (0,-.2) {$\scs s_{j+1}$};
 \node at (1.2,-.2) {$\scs s_{j}$};
\node at (.6,2.2) {$\scs s_j+s_{j+1}$};
\draw[thick,] (-.6,0) .. controls ++(0,.5) and ++(0,-1.8) .. (1.6,2);
\end{tikzpicture}}
\;\; = \;\;
\hackcenter{ \begin{tikzpicture} [scale=.75, rotate=180]
\draw[thick, red, double] (0,0).. controls ++(0,.3) and ++(0,-.3) ..(.6,1.2);
\draw[thick,red, double, ] (1.2,0).. controls ++(0,.3) and ++(0,-.3) ..(.6,1.2) to (.6,2);
 \node at (0,-.2) {$\scs s_{j+1}$};
 \node at (1.2,-.2) {$\scs s_{j}$};
\node at (.6,2.2) {$\scs s_j+s_{j+1}$};
\draw[thick,] (-.6,0) .. controls ++(0,1.6) and ++(0,-.8) .. (1.6,2);
\end{tikzpicture}}
\qquad
\qquad
\hackcenter{ \begin{tikzpicture} [scale=.75, rotate=180]
\draw[thick, red, double] (0,0).. controls ++(0,.3) and ++(0,-.3) ..(.6,1.2);
\draw[thick,red, double, ] (1.2,0).. controls ++(0,.3) and ++(0,-.3) ..(.6,1.2) to (.6,2);
 \node at (0,-.2) {$\scs s_{j+1}$};
 \node at (1.2,-.2) {$\scs s_{j}$};
\node at (.6,2.2) {$\scs s_j+s_{j+1}$};
\draw[thick,] (1.8,0) .. controls ++(0,.5) and ++(0,-1.8) .. (-.2,2);
\end{tikzpicture}}
\;\; = \;\;
\hackcenter{ \begin{tikzpicture} [scale=.75, rotate=180]
\draw[thick, red, double] (0,0).. controls ++(0,.3) and ++(0,-.3) ..(.6,1.2);
\draw[thick,red, double, ] (1.2,0).. controls ++(0,.3) and ++(0,-.3) ..(.6,1.2) to (.6,2);
 \node at (0,-.2) {$\scs s_{j+1}$};
 \node at (1.2,-.2) {$\scs s_{j}$};
\node at (.6,2.2) {$\scs s_j+s_{j+1}$};
\draw[thick,] (1.8,0) .. controls ++(0,1.6) and ++(0,-.8) .. (-.2,2);
\end{tikzpicture}}
\]
\end{proposition}

\begin{proof}
Use the definition of the inclusion \eqref{eq:inclusion} along with relations in the algebra.
\end{proof}

The first two equalities of Proposition \ref{folk-rel} imply the following.
\begin{corollary}
\label{reddot-slide}
\[
\hackcenter{ \begin{tikzpicture} [scale=.75]
\draw[thick, red, double] (0,0).. controls ++(0,.3) and ++(0,-.3) ..(.6,1.2);
\draw[thick,red, double, ] (1.2,0).. controls ++(0,.3) and ++(0,-.3) ..(.6,1.2) to (.6,2);
\node at (0,-.2) {$\scs s_j$};
\node at (1.2,-.2) {$\scs s_{j+1}$};
\node at (.6,2.2) {$\scs s_j+s_{j+1}$};
\filldraw[red]  (0.2,.5) circle (2.75pt);
\node at (-.25,.5) {$\scs {\sf e}_d$};
\end{tikzpicture}}
\;\; = \;\;
  \sum_{d_1+d_2 = d}(-1)^{d_2}
\hackcenter{ \begin{tikzpicture} [scale=.75]
\draw[thick, red, double] (0,0).. controls ++(0,.3) and ++(0,-.3) ..(.6,1.2);
\draw[thick,red, double, ] (1.2,0).. controls ++(0,.3) and ++(0,-.3) ..(.6,1.2) to (.6,2);
\node at (0,-.2) {$\scs s_j$};
\node at (1.2,-.2) {$\scs s_{j+1}$};
\node at (.6,2.2) {$\scs s_j+s_{j+1}$};
    \filldraw[red]  (0.6,1.6) circle (2.75pt);
    \node at (.25,1.6) {$\scs {\sf e}_{d_1}$};
    \filldraw[red]  (1.0,.5) circle (2.75pt);
    \node at (1.45,.5) {$\scs {\sf h}_{d_2}$};
\end{tikzpicture}}
\qquad
\qquad
\hackcenter{ \begin{tikzpicture} [scale=.75]
\draw[thick, red, double] (0,0).. controls ++(0,.3) and ++(0,-.3) ..(.6,1.2);
\draw[thick,red, double, ] (1.2,0).. controls ++(0,.3) and ++(0,-.3) ..(.6,1.2) to (.6,2);
\node at (0,-.2) {$\scs s_j$};
\node at (1.2,-.2) {$\scs s_{j+1}$};
\node at (.6,2.2) {$\scs s_j+s_{j+1}$};
    \filldraw[red]  (1.0,.5) circle (2.75pt);
    \node at (1.45,.5) {$\scs {\sf e}_{d}$};
\end{tikzpicture}}
\;\; = \;\;
  \sum_{d_1+d_2 = d}(-1)^{d_2}
\hackcenter{ \begin{tikzpicture} [scale=.75]
\draw[thick, red, double] (0,0).. controls ++(0,.3) and ++(0,-.3) ..(.6,1.2);
\draw[thick,red, double, ] (1.2,0).. controls ++(0,.3) and ++(0,-.3) ..(.6,1.2) to (.6,2);
\node at (0,-.2) {$\scs s_j$};
\node at (1.2,-.2) {$\scs s_{j+1}$};
\node at (.6,2.2) {$\scs s_j+s_{j+1}$};
    \filldraw[red]  (0.6,1.6) circle (2.75pt);
    \node at (.25,1.6) {$\scs {\sf e}_{d_1}$};
\filldraw[red]  (0.2,.5) circle (2.75pt);
\node at (-.25,.5) {$\scs {\sf h}_{d_2}$};
\end{tikzpicture}}
\]
\[
\hackcenter{ \begin{tikzpicture} [scale=.75, rotate=180]
\draw[thick, red, double] (0,0).. controls ++(0,.3) and ++(0,-.3) ..(.6,1.2);
\draw[thick,red, double, ] (1.2,0).. controls ++(0,.3) and ++(0,-.3) ..(.6,1.2) to (.6,2);
 \node at (0,-.2) {$\scs s_{j+1}$};
 \node at (1.2,-.2) {$\scs s_{j}$};
\node at (.6,2.2) {$\scs s_j+s_{j+1}$};
    \filldraw[red]  (1.0,.5) circle (2.75pt);
    \node at (1.4,.5) {$\scs {\sf e}_d$};
\end{tikzpicture}}
\;\; = \;\;
  \sum_{d_1+d_2 = d}(-1)^{d_2}
\hackcenter{ \begin{tikzpicture} [scale=.75, rotate=180]
\draw[thick, red, double] (0,0).. controls ++(0,.3) and ++(0,-.3) ..(.6,1.2);
\draw[thick,red, double, ] (1.2,0).. controls ++(0,.3) and ++(0,-.3) ..(.6,1.2) to (.6,2);
 \node at (0,-.2) {$\scs s_{j+1}$};
 \node at (1.2,-.2) {$\scs s_{j}$};
\node at (.6,2.2) {$\scs s_j+s_{j+1}$};
    \filldraw[red]  (0.6,1.6) circle (2.75pt);
    \node at (.15,1.6) {$\scs {\sf e}_{d_1}$};
\filldraw[red]  (0.2,.5) circle (2.75pt);
\node at (-.35,.5) {$\scs {\sf h}_{d_2}$};
\end{tikzpicture}}
\qquad
\qquad
\hackcenter{ \begin{tikzpicture} [scale=.75, rotate=180]
\draw[thick, red, double] (0,0).. controls ++(0,.3) and ++(0,-.3) ..(.6,1.2);
\draw[thick,red, double, ] (1.2,0).. controls ++(0,.3) and ++(0,-.3) ..(.6,1.2) to (.6,2);
 \node at (0,-.2) {$\scs s_{j+1}$};
 \node at (1.2,-.2) {$\scs s_{j}$};
\node at (.6,2.2) {$\scs s_j+s_{j+1}$};
\filldraw[red]  (0.2,.5) circle (2.75pt);
\node at (-.25,.5) {$\scs {\sf e}_d$};
\end{tikzpicture}}
\;\; = \;\;
  \sum_{d_1+d_2 = d}(-1)^{d_2}
\hackcenter{ \begin{tikzpicture} [scale=.75, rotate=180]
\draw[thick, red, double] (0,0).. controls ++(0,.3) and ++(0,-.3) ..(.6,1.2);
\draw[thick,red, double, ] (1.2,0).. controls ++(0,.3) and ++(0,-.3) ..(.6,1.2) to (.6,2);
 \node at (0,-.2) {$\scs s_{j+1}$};
 \node at (1.2,-.2) {$\scs s_{j}$};
\node at (.6,2.2) {$\scs s_j+s_{j+1}$};
    \filldraw[red]  (0.6,1.6) circle (2.75pt);
    \node at (.15,1.6) {$\scs {\sf e}_{d_1}$};
    \filldraw[red]  (1.0,.5) circle (2.75pt);
    \node at (1.45,.5) {$\scs {\sf h}_{d_2}$};
\end{tikzpicture}}
\]
\end{corollary}

\subsection{Ladder bimodules}
The main bimodules introduced in this section are called ladder bimodules because of their graphical depiction.

For a sequence $\mathbf{s} = (s_{1}, s_{2}, \dots , s_{m})$, we have maps
\begin{eqnarray*}
\phi_{j,a}(\mathbf{s})^{j+1}&=&(...,s_{j-1},s_{j}-a,s_{j+1}+a,s_{j+2},...),\\
\phi_{j+1,s_{j+1}-a}(\mathbf{s})^{j}&=&(...,s_{j-1},s_{j}+a,s_{j+1}-a,s_{j+2},...),
\end{eqnarray*}
We denote by $\alpha^{+j,a}(\mathbf{s})$ the sequence $\phi_{j+1,s_{j+1}-a}(\mathbf{s})^{j}$ and denote by $\alpha^{-j,a}(\mathbf{s})$ the sequence $\phi_{j,a}(\mathbf{s})^{j+1}$.
These maps $\alpha^{\pm j,a}$ extends into the set $\mathrm{Seq}(\mathbf{s},n)$ by
\begin{eqnarray*}
\alpha^{+j,a}(\mathbf{i})&=& (...,s_{j}+a,\mf{b},....,\mf{b},s_{j+1}-a,...),\\
\alpha^{-j,a}(\mathbf{i})&=& (...,s_{j}-a,\mf{b},....,\mf{b},s_{j+1}+a,...).
\end{eqnarray*}
$\alpha^{+j,a}$ maps the  elements $s_{j}$ and $s_{j+1}$ in $\mathbf{i}$ into $s_{j}+a$ and $s_{j+1}-a$ and $\alpha^{-j,a}$ maps $s_{j}$ and $s_{j+1}$ into $s_{j}-a$ and $s_{j+1}+a$.

We define the $(W(\alpha^{+i,a}({\bf s}),n),W(\mathbf{s},n))$-bimodule
$\mathsf{E}^{(a)}_i \mathsf{1}_{\mathbf{s}}$ as the bimodule
\begin{eqnarray*}
\mathsf{E}^{(a)}_i \mathsf{1}_{\mathbf{s}}
&=&
\bigtriangleup^i(\phi_{i+1,s_{i+1}-a}(\mathbf{s})) \otimes_{W(\phi_{i+1,s_{i+1}-a}(\mathbf{s}),n)}\bigtriangledown^{i+1}(\phi_{i+1,s_{i+1}-a}(\mathbf{s})) \\
&\cong&
W(\phi_{i+1,s_{i+1}-a}(\mathbf{s})^i,n) \otimes_{W(\phi_{i+1,s_{i+1}-a}(\mathbf{s}),n)}W(\phi_{i+1,s_{i+1}-a}(\mathbf{s}),n)\otimes_{W(\mathbf{s},n)} W(\mathbf{s},n)\{-s_i \cdot a\}
\end{eqnarray*}
and the $(W(\alpha^{-i,a}({\bf s}),n),W(\mathbf{s},n))$-bimodule $\mathsf{F}^{(a)}_i \mathsf{1}_{\mathbf{s}}$ as the bimodule
\begin{eqnarray*}
\mathsf{F}^{(a)}_i \mathsf{1}_{\mathbf{s}}&=&\bigtriangleup^{i+1}(\phi_{i,a}(\mathbf{s})) \otimes_{W(\phi_{i,a}(\mathbf{s}),n)}\bigtriangledown^i(\phi_{i,a}(\mathbf{s}))\\
&\cong&W(\phi_{i,a}(\mathbf{s})^{i+1},n)\otimes_{W(\phi_{i,a}(\mathbf{s})^{i+1},n)}W(\phi_{i,a}(\mathbf{s}),n)\otimes_{W(\mathbf{s},n)} W(\mathbf{s},n)\{-s_{i+1} \cdot a\}
\end{eqnarray*}
By definition, these bimodules have elements diagrammatically depicted by
\[
\mathsf{E}^{(a)}_i \mathsf{1}_{\mathbf{s}}
\ni \hackcenter{
\begin{tikzpicture}[scale=.4]
\node at (-5.5,0) { $\dots$};
\node at (5.5,0) { $\dots$};
    \draw [thick, red, double,, ] (-4,-2) to (-4,2);
    \draw [thick, red, double,, ] (4,-2) to (4,2);
	\draw [thick, red, double,, ] (2,-.5) to (2,2);
	\draw [thick, red, double,, ] (2,-.5) to (-2,.5);
	\draw [thick, red, double,, ] (-2,-2) to (-2,.5);
	\draw [thick, red, double,, ] (-2,.5) to (-2,2);
	\draw [thick, red, double,, ] (2,-2) to (2,-.5);
\node at (-4,-2.55) {\tiny $s_{i-1}$};	
\node at (-2,-2.55) {\tiny $s_i$};	
\node at (2,-2.55) {\tiny $s_{i+1}$};
\node at (4,-2.55) {\tiny $s_{i+2}$};
	\node at (-2,2.5) {\tiny $s_i+a$};
	\node at (2,2.5) {\tiny $s_{i+1}-a$};
	\node at (0,0.75) {\tiny $a$};
\node at (0,-1.5) { $\dots$};
\end{tikzpicture}}
\quad \quad
\mathsf{F}^{(a)}_i \mathsf{1}_{\mathbf{s}}
\ni \hackcenter{
\begin{tikzpicture}[scale=.4]
\node at (-5.5,0) { $\dots$};
\node at (5.5,0) { $\dots$};
    \draw [thick, red, double,, ] (-4,-2) to (-4,2);
    \draw [thick, red, double,, ] (4,-2) to (4,2);
	\draw [thick, red, double,, ] (2,.5) to (2,2);
	\draw [thick, red, double,, ] (-2,-.5) to (2,.5);
	\draw [thick, red, double,, ] (-2,-2) to (-2,-.5);
	\draw [thick, red, double,, ] (-2,-.5) to (-2,2);
	\draw [thick, red, double,, ] (2,-2) to (2,.5);
\node at (-4,-2.55) {\tiny $s_{i-1}$};	
\node at (-2,-2.55) {\tiny $s_i$};	
\node at (2,-2.55) {\tiny $s_{i+1}$};
\node at (4,-2.55) {\tiny $s_{i+2}$};
	\node at (-2,2.5) {\tiny $s_i- a$};
	\node at (2,2.5) {\tiny $s_{i+1}+a$};
	\node at (0,0.75) {\tiny $a$};
\node at (0,-1.5) { $\dots$};
\end{tikzpicture}}
\]
with $n$ black strands intersecting the pictures in a manner similar to that described in Section ~\ref{sectiondiagdescription} along with red and black dots.
We will refer to the red strand connecting vertical strands as a step.  When a step is not labelled, following earlier conventions we will assume that it has a label of $1$.

\begin{proposition}
\label{EiEi+1generators}
The bimodules
$\mathsf{E}^{}_{i} \mathsf{E}^{}_{i+1} \mathsf{1}_{\mathbf{s}} $ and
$\mathsf{E}^{}_{i+1} \mathsf{E}^{}_i \mathsf{1}_{\mathbf{s}} $ are generated by diagrams of the form
\begin{equation*}
\hackcenter{
\begin{tikzpicture}[scale=.4]
	\draw [thick,red, double, ] (6,-2) to (6,2);
	\draw [thick,red, double, ] (2,-2) to (2,2);
	\draw [thick,red, double, ] (-2,-2) to (-2,2);
	\draw [very thick, red ] (2,.5) to (-2,1.5);
	\draw [very thick, red ] (6,-1.5) to (2,-.5);
\node at (-2,-2.55) {\tiny $s_i$};	
\node at (2,-2.55) {\tiny $s_{i+1}$};
\node at (6,-2.55) {\tiny $s_{i+2}$};
	\node at (-2,2.5) {\tiny $s_i+1$};
	\node at (2,2.5) {\tiny $s_{i+1}$};
		\node at (6,2.5) {\tiny $s_{i+2}-1$};
\node at (0,-2.5) {\tiny $k_1$};
\draw [very thick,double,] (0,-2) to (0,2);
\node at (4,-2.5) {\tiny $k_2$};
\draw [very thick,double,] (4,-2) to (4,2);
\end{tikzpicture}}
\quad \quad \quad
\hackcenter{
\begin{tikzpicture}[scale=.4]
	\draw [thick,red, double, ] (6,-2) to (6,2);
	\draw [thick,red, double, ] (2,-2) to (2,2);
	\draw [thick,red, double, ] (-2,-2) to (-2,2);
	\draw [very thick, red ] (6,.5) to (2,1.5);
	\draw [very thick, red ] (2,-1.5) to (-2,-.5);
\node at (-2,-2.55) {\tiny $s_i$};	
\node at (2,-2.55) {\tiny $s_{i+1}$};
\node at (6,-2.55) {\tiny $s_{i+2}$};
	\node at (-2,2.5) {\tiny $s_i+1$};
	\node at (2,2.5) {\tiny $s_{i+1}$};
		\node at (6,2.5) {\tiny $s_{i+2}-1$};
\node at (0.05,-2.5) {\tiny $k_1$};
\draw [very thick,double,] (0,-2)  .. controls ++(0,1.6) and ++(0,-1.4) .. (-.6,2);
\node at (4.6,-2.5) {\tiny $k_3$};
\draw [very thick,double,] (4.55,-2) .. controls ++(0,1.6) and ++(0,-1.4) .. (4,2);
\draw[very thick,double,] (3.25,-2) .. controls ++(0,1.6) and ++(0,-1.4) .. (1,2);
\node at (3.3,-2.5) {\tiny $k_2$};
\end{tikzpicture}}
\end{equation*}
respectively with
$k_1, k_2, k_3 \geq 0$.
\end{proposition}

\begin{proof}
Using the first two equalities of Proposition \ref{folk-rel} and Corollary \ref{reddot-slide}, any red dot on the above diagram could be replaced by a sum of terms with red dots at the top and bottom of strands of the diagram.
For instance, a term with a red dot in the middle of the vertical strand labelled $s_{i+1}$ of the left diagram could be replaced by a sum of terms with red dots at the top of that middle vertical strand and dots on the upper step of the diagram.  A term with a dot on the upper step could be expressed as a sum of terms with dots on the top and bottom of the left-most vertical red strand using a complete-elementary symmetric function relation.
\end{proof}

\begin{proposition}
\label{EiEi+1Eigenerators}
The bimodules
$\mathsf{E}^{}_{i} \mathsf{E}^{}_{i+1} \mathsf{E}^{}_{i} \mathsf{1}_{\mathbf{s}} $
and
$\mathsf{E}^{}_{i+1} \mathsf{E}^{}_{i} \mathsf{E}^{}_{i+1} \mathsf{1}_{\mathbf{s}} $
are generated by elements of the form
\begin{equation}
\label{Eii+-1igen}
\hackcenter{
\begin{tikzpicture}[scale=.4]
	\draw [thick,red, double, ] (6,-2) to (6,4);
	\draw [thick,red, double, ] (2,-2) to (2,4);
	\draw [thick,red, double, ] (-2,-2) to (-2,4);
		\draw [very thick, red ] (2,2.5) to (-2,3.5);
		\draw [very thick, red ] (6,0) to (2,1);
		\draw [very thick, red ] (2,-1.5) to (-2,-.5);
\node at (-2,-2.55) {\tiny $s_i$};	
\node at (2,-2.55) {\tiny $s_{i+1}$};
\node at (6,-2.55) {\tiny $s_{i+2}$};
	\node at (-2,4.5) {\tiny $s_i+2$};
	\node at (2,4.5) {\tiny $s_{i+1}-1$};
	\node at (6,4.5) {\tiny $s_{i+2}-1$};
	\filldraw[red]  (0,-1) circle (5pt);
\node at (0,-1.6) {\tiny $f$};
\node at (-1,-2.5) {\tiny $k_1$};
\draw [very thick,double,] (-1,-2) to (-1,4);
\node at (4.65,-2.5) {\tiny $k_3$};
\draw [very thick,double,] (4.6,-2) to (4.6,4);
\draw[very thick,double,] (3.2,-2) .. controls ++ (0,1.2) and ++(0,-1.4) .. (.8,2) to (.8,4);
\node at (3.25,-2.5) {\tiny $k_2$};
    \end{tikzpicture}}
    \quad \quad \quad
\hackcenter{
\begin{tikzpicture}[scale=.4]
	\draw [thick,red, double, ] (6,-2) to (6,4);
	\draw [thick,red, double, ] (2,-2) to (2,4);
	\draw [thick,red, double, ] (-2,-2) to (-2,4);
		\draw [very thick, red ] (6,2.5) to (2,3.5);
		\draw [very thick, red ] (2,1) to (-2,2);
		\draw [very thick, red ] (6,-1.5) to (2,-.5);
\node at (-2,-2.55) {\tiny $s_i$};	
\node at (2,-2.55) {\tiny $s_{i+1}$};
\node at (6,-2.55) {\tiny $s_{i+2}$};
	\node at (-2,4.5) {\tiny $s_i+1$};
	\node at (2,4.5) {\tiny $s_{i+1}+1$};
	\node at (6,4.5) {\tiny $s_{i+2}-2$};
	\filldraw[red]  (4,3) circle (5pt);
\node at (4,2.4) {\tiny $f$};
\node at (-1,-2.5) {\tiny $k_1$};
\draw [very thick,double,] (-1,-2) to (-1,4);
\node at (4.75,-2.5) {\tiny $k_3$};
\draw [very thick,double,] (4.75,-2) to (4.75,4);
\draw[very thick,double,] (3.3,-2) to (3.3,.5) .. controls ++ (0,1.2) and ++(0,-1.6) ..  (.8,4);
\node at (3.3,-2.5) {\tiny $k_2$};
    \end{tikzpicture}}
\end{equation}

respectively with
$f, k_1, k_2, k_3 \geq 0$.
\end{proposition}

\begin{proof}
The proof of this is similar to the proof of Proposition ~\ref{EiEi+1generators}.
Note that in this case we can always write a term with a red dot on a step as a linear combination of terms with red dots at the top and bottom of the diagram and at the step of the ladder as indicated.
\end{proof}

\section{Bimodule homomorphisms}

\subsection{Morphisms $\partial_{(s_{j}-1,1)}^j$ and $\partial_{(1,s_{j}-1)}^j$}
We define the \textit{derivation} maps $\partial_{(s_{j}-1,1)}^j$ and $\partial_{(1,s_{j}-1)}^j$.

These are $(W({\bf s},n),W({\bf s},n))$-bimodule maps
\begin{align*}
\partial_{(s_{j}-1,1)}^j &\colon \bigtriangleup^j(\phi_{j,1}({\bf s})) \otimes_{W(\phi_{j,1}({\bf s}),n)} \bigtriangledown^j(\phi_{j,1}({\bf s})) \rightarrow W({\bf s},n) \\
\partial_{(1,s_{j}-1)}^j &\colon \bigtriangleup^j(\phi_{j,s_j-1}({\bf s})) \otimes_{W(\phi_{j,s_j-1}({\bf s}),n)} \bigtriangledown^j(\phi_{j,s_j-1}({\bf s})) \rightarrow W({\bf s},n)
\end{align*}
determined by
\begin{equation}
\label{bimodmap1}
\partial_{(s_{j}-1,1)}^j:\hackcenter{ \begin{tikzpicture} [scale=.75]
\draw[thick, red, double] (0,0).. controls ++(0,.3) and ++(0,-.3) ..(.6,.5) to (.6,1);
\draw[very thick,red,   ] (1.2,0).. controls ++(0,.3) and ++(0,-.3) .. (.6,.5) node[pos=.5, shape=coordinate](DOT){};
\draw[very thick,red,] (.6,-.5).. controls ++(0,.3) and ++(0,-.3) ..(1.2,0) ;
\draw[thick, red, double] (0,0).. controls ++(0,-.3) and ++(0,.3) ..(.6,-.5) to (.6,-1);
    \filldraw[red]  (DOT) circle (3.0pt);
\node at (-.5,0) {$\scs s_{j}-1$};
\node at (1.4,0) {$\scs 1$};
\node at (1.2,0.5) {$\scs m$};
\node at (.6,1.2) {$\scs s_{j}$};
\node at (.6,-1.2) {$\scs s_{j}$};
\end{tikzpicture}}
\mapsto
\hackcenter{ \begin{tikzpicture} [scale=.75]
\draw[thick,red, double, ] (0,-1) to (0,1) node[pos=.5, shape=coordinate](DOT){};
    \filldraw[red]  (DOT) circle (3.0pt);
\node at (.9,0) {$\scs {\sf h}_{m-s_{j}+1}$};
\node at (0,-1.2) {$\scs s_{j}$};
\end{tikzpicture}}
,
\qquad
\partial_{(1,s_{j}-1)}^j:\hackcenter{ \begin{tikzpicture} [scale=.75]
\draw[thick, red, double] (0,0).. controls ++(0,.3) and ++(0,-.3) ..(-.6,.5) to (-.6,1);
\draw[very thick,red,   ] (-1.2,0).. controls ++(0,.3) and ++(0,-.3) .. (-.6,.5) node[pos=.5, shape=coordinate](DOT){};
\draw[very thick,red,] (-.6,-.5).. controls ++(0,.3) and ++(0,-.3) ..(-1.2,0) ;
\draw[thick, red, double] (0,0).. controls ++(0,-.3) and ++(0,.3) ..(-.6,-.5) to (-.6,-1);
    \filldraw[red]  (DOT) circle (3.0pt);
\node at (.5,0) {$\scs s_{j}-1$};
\node at (-1.4,0) {$\scs 1$};
\node at (-1.2,0.5) {$\scs m$};
\node at (-.6,1.2) {$\scs s_{j}$};
\node at (-.6,-1.2) {$\scs s_{j}$};
\end{tikzpicture}}
\mapsto
(-1)^{s_j-1}
\hackcenter{ \begin{tikzpicture} [scale=.75]
\draw[thick,red, double, ] (0,-1) to (0,1) node[pos=.5, shape=coordinate](DOT){};
    \filldraw[red]  (DOT) circle (3.0pt);
\node at (.9,0) {$\scs {\sf h}_{m-s_{j}+1}$};
\node at (0,-1.2) {$\scs s_{j}$};
\end{tikzpicture}},
\end{equation}
where ${\sf h}_\ell$ is the $\ell$-th complete symmetric function in terms of the elementary symmetric functions ${\sf e}_1$,...,${\sf e}_{s_{j}}$.

The degrees of the morphisms $\partial_{(s_{j}-1,1)}^j$ and $\partial_{(1,s_{j}-1)}^j$ are both $1-s_j$.

\subsection{Morphisms $\iota_{(s_{j}-1,1)}^j$ and $\iota_{(1,s_{j}-1)}^j$}
We define the \textit{creation} maps $\iota_{(s_{j}-1,1)}^j$ and $\iota_{(1,s_{j}-1)}^j$.

There are $(W({\bf s},n),W({\bf s},n))$-bimodule maps $\iota_{(s_{j}-1,1)}^j: W({\bf s},n)\rightarrow \bigtriangleup^j(\phi_{j,1}({\bf s})) \otimes_{W(\phi_{j,1}({\bf s}),n)} \bigtriangledown^j(\phi_{j,1}({\bf s}))$ and $\iota_{(1,s_{j}-1)}^j:W({\bf s},n) \rightarrow \bigtriangleup^j(\phi_{j,s_j-1}({\bf s})) \otimes_{W(\phi_{j,s_j-1}({\bf s}),n)} \bigtriangledown^j(\phi_{j,s_j-1}({\bf s}))$ determined by
\begin{equation}
\label{bimodmap2}
\iota_{(s_{j}-1,1)}^j:
\hackcenter{ \begin{tikzpicture} [scale=.75]
\draw[thick,red, double, ] (0,-1) to (0,1);
\node at (0,-1.2) {$\scs s_{j}$};
\end{tikzpicture}}
\mapsto
\hackcenter{ \begin{tikzpicture} [scale=.75]
\draw[thick, red, double] (0,0).. controls ++(0,.3) and ++(0,-.3) ..(.6,.5) to (.6,1);
\draw[very thick,red,   ] (1.2,0).. controls ++(0,.3) and ++(0,-.3) .. (.6,.5) node[pos=.5, shape=coordinate](DOT){};
\draw[very thick,red,] (.6,-.5).. controls ++(0,.3) and ++(0,-.3) ..(1.2,0) ;
\draw[thick, red, double] (0,0).. controls ++(0,-.3) and ++(0,.3) ..(.6,-.5) to (.6,-1);
\node at (-.5,0) {$\scs s_{j}-1$};
\node at (1.4,0) {$\scs 1$};
\node at (.6,1.2) {$\scs s_{j}$};
\node at (.6,-1.2) {$\scs s_{j}$};
\end{tikzpicture}}
,
\qquad
\iota_{(1,s_{j}-1)}^j:
\hackcenter{ \begin{tikzpicture} [scale=.75]
\draw[thick,red, double, ] (0,-1) to (0,1);
\node at (0,-1.2) {$\scs s_{j}$};
\end{tikzpicture}}
\mapsto
\hackcenter{ \begin{tikzpicture} [scale=.75]
\draw[thick, red, double] (0,0).. controls ++(0,.3) and ++(0,-.3) ..(-.6,.5) to (-.6,1);
\draw[very thick,red,   ] (-1.2,0).. controls ++(0,.3) and ++(0,-.3) .. (-.6,.5) node[pos=.5, shape=coordinate](DOT){};
\draw[very thick,red,] (-.6,-.5).. controls ++(0,.3) and ++(0,-.3) ..(-1.2,0) ;
\draw[thick, red, double] (0,0).. controls ++(0,-.3) and ++(0,.3) ..(-.6,-.5) to (-.6,-1);
\node at (.5,0) {$\scs s_{j}-1$};
\node at (-1.4,0) {$\scs 1$};
\node at (-.6,1.2) {$\scs s_{j}$};
\node at (-.6,-1.2) {$\scs s_{j}$};
\end{tikzpicture}}
\end{equation}

The degrees of the morphisms $\iota_{(s_{j}-1,1)}^j$ and $\iota_{(1,s_{j}-1)}^j$ are $1-s_j$.

\subsection{Morphisms $\upsilon_{(s_j,1)}^j$ and $\upsilon_{(1,s_{j+1})}^j$}\label{sec:unzip}

Let ${\bf s}_{(j)}$ be a sequence $(s_1,s_2,...,s_m)$ of $m$ non-negative integers such that the $j$-th integer $s_j$ is one.

For a sequence $\mathbf{i}_1\in \mathrm{Seq}({\bf s}_{(j+1)},n)$ such that $s_{j}$ and $s_{j+1}$ are neighbors and a sequence $\mathbf{i}_2\in \mathrm{Seq}({\bf s}_{(j)},n)$ such that $s_{j}$ and $s_{j+1}$ are neighbors (i.e., $\mathbf{i}_1=(...,s_{j},1,...)$ and $\mathbf{i}_2=(...,1,s_{j+1},...)$), there is a map
\begin{equation*}
e(\mathbf{i}_1)(\bigtriangledown^j({\bf s}_{(j+1)}) \otimes_{W({\bf s}_{(j+1)}^j,n)} \bigtriangleup^j({\bf s}_{(j+1)}))e(\mathbf{i}_1) \rightarrow e(\mathbf{i}_1)W({\bf s}_{(j+1)},n)e(\mathbf{i}_1)
\end{equation*}
and a map
\begin{equation*}
e(\mathbf{i}_2)(\bigtriangledown^j({\bf s}_{(j)}) \otimes_{W({\bf s}_{(j)}^j,n)} \bigtriangleup^j({\bf s}_{(j)}))e(\mathbf{i}_2) \rightarrow e(\mathbf{i}_2)W({\bf s}_{(j)},n)e(\mathbf{i}_2)
\end{equation*}
determined by
\begin{equation}
\label{unzip1}
\hackcenter{
\begin{tikzpicture}[scale=.9]
\draw (0,0) -- (.5,.25)[thick,red, double, ];
\draw (1,0) -- (.5,.25)[very thick,red,   ];
\draw (0,1.5) -- (.5,1.25)[thick,red, double, ];
\draw (1,1.5) -- (.5,1.25)[very thick,red,  ];
\draw (.5,.25) -- (.5,1.25)[thick,red, double, ];
        \node at (0,-.2) {$\scs s_j$};
        \node at (1,-.2) {$\scs 1$};
        \node at (0.95,0.75) {$\scs s_j+1$};
        \node at (0,1.7) {$\scs s_j$};
        \node at (1,1.7) {$\scs 1$};
\end{tikzpicture}}
\mapsto
\hackcenter{
\begin{tikzpicture}[scale=.9]
\draw (0,0) -- (0,1.5)[thick,red, double, ];
\draw (1,0) -- (1,1.5)[very thick,red,  ];
        \node at (0,-.2) {$\scs $};
        \node at (1,-.2) {$\scs $};
        \node at (0,-.2) {$\scs s_j$};
        \node at (1,-.2) {$\scs 1$};
\end{tikzpicture}},
\qquad
\hackcenter{
\begin{tikzpicture}[scale=.9]
\draw (0,0) -- (.5,.25)[very thick,red,   ];
\draw (1,0) -- (.5,.25)[thick,red, double, ];
\draw (0,1.5) -- (.5,1.25)[very thick,red,  ];
\draw (1,1.5) -- (.5,1.25)[thick,red, double, ];
\draw (.5,.25) -- (.5,1.25)[thick,red, double, ];
        \node at (0,-.2) {$\scs 1$};
        \node at (1,-.2) {$\scs s_{j+1}$};
        \node at (1.25,0.75) {$\scs s_{j+1}+1$};
        \node at (0,1.7) {$\scs 1$};
        \node at (1,1.7) {$\scs s_{j+1}$};
\end{tikzpicture}}
\mapsto
\hackcenter{
\begin{tikzpicture}[scale=.9]
\draw (0,0) -- (0,1.5)[very thick,red,   ];
\draw (1,0) -- (1,1.5)[thick,red, double, ];
        \node at (0,-.2) {$\scs $};
        \node at (1,-.2) {$\scs $};
        \node at (0,-.2) {$\scs 1$};
        \node at (1,-.2) {$\scs s_{j+1}$};
\end{tikzpicture}}
\end{equation}

The degree of the first morphism is $s_j$ the degree of the second morphism is $s_{j+1}$.

This implies the morphisms
\begin{equation*}
e(\mathbf{i}_1)(\bigtriangledown^j({\bf s}_{(j+1)}) \otimes_{W({\bf s}_{(j+1)}^j,n)} \bigtriangleup^j({\bf s}_{(j+1)}))e(\mathbf{i}_1) \rightarrow e(\mathbf{i}_1)W({\bf s}_{(j+1)},n)e(\mathbf{i}_1)
\end{equation*}
\begin{equation*}
e(\mathbf{i}_2)(\bigtriangledown^j({\bf s}_{(j)}) \otimes_{W({\bf s}_{(j)}^j,n)} \bigtriangleup^j({\bf s}_{(j)}))e(\mathbf{i}_2) \rightarrow e(\mathbf{i}_2)W({\bf s}_{(j)},n)e(\mathbf{i}_2)
\end{equation*}
where $\mathbf{i}_1=(...,1,\underbrace{\mf{b},...,\mf{b}}_{k},s_{j+1},...)$ and $\mathbf{i}_2=(...,s_{j},\underbrace{\mf{b},...,\mf{b}}_{k},1,...)$ for $0\leq k\leq n$ are of degrees $s_j$ and $s_{j+1}$, respectively.

As a diagrammatic presentation, the first map is determined by the following mapping.
\begin{eqnarray}
\label{unzip}
&&
\hackcenter{
\begin{tikzpicture}[scale=.9]
\draw (-1,0) -- (.5,.25)[thick,red, double, ];
\draw (-1,1.5) -- (.5,1.25)[thick,red, double, ];
\draw (2,0) -- (.5,.25)[very thick,red,   ];
\draw (2,1.5) -- (.5,1.25)[very thick,red, ];
\draw (.5,.25) -- (.5,1.25)[thick,red, double, ];
\draw[thick,->] (.25,-.1) .. controls ++(0,.25) and ++(0,-.25) .. (.05,.5) -- (.05,1) .. controls ++(0,.25) and ++(0,-.25) .. (.25,1.6);
\draw[thick,->] (-.5,-.1) .. controls ++(0,.25) and ++(0,-.25) .. (-.7,.5) -- (-.7,1) .. controls ++(0,.25) and ++(0,-.25) .. (-.5,1.6);
\draw[thick,->] (.75,-.1) .. controls ++(0,.25) and ++(0,-.25) .. (.95,.5) -- (.95,1) .. controls ++(0,.25) and ++(0,-.25) .. (.75,1.6);
\draw[thick,->] (1.5,-.1) .. controls ++(0,.25) and ++(0,-.25) .. (1.7,.5) -- (1.7,1) .. controls ++(0,.25) and ++(0,-.25) .. (1.5,1.6);
        \node at (-1,-.2) {$\scs s_j$};
        \node at (2,-.2) {$\scs 1$};
        \node at (-1,1.7) {$\scs s_j$};
        \node at (2,1.7) {$\scs 1$};
        \node at (-.125,-.1) {$\cdots$};
        \node at (-.125,-.45) {$\underbrace{\hspace{.75cm}}_{k_1}$};
        \node at (1.125,-.1) {$\cdots$};
        \node at (1.125,-.45) {$\underbrace{\hspace{.75cm}}_{k_2}$};
\end{tikzpicture}}
\mapsto
\hackcenter{
\begin{tikzpicture}[scale=.9]
\draw[thick,red, double] (-1,0) .. controls ++(0,.25) and ++(0,-.25) .. (.25,.5) -- (.25,1) .. controls ++(0,.25) and ++(0,-.25) .. (-1,1.5);
\draw[very thick,red ] (2,0) .. controls ++(0,.25) and ++(0,-.25) .. (.75,.5) -- (.75,1) .. controls ++(0,.25) and ++(0,-.25) .. (2,1.5);
\draw[thick,->] (.25,-.1) .. controls ++(0,.25) and ++(0,-.25) .. (.05,.5) -- (.05,1) .. controls ++(0,.25) and ++(0,-.25) .. (.25,1.6);
\draw[thick,->] (-.5,-.1) .. controls ++(0,.25) and ++(0,-.25) .. (-.7,.5) -- (-.7,1) .. controls ++(0,.25) and ++(0,-.25) .. (-.5,1.6);
\draw[thick,->] (.75,-.1) .. controls ++(0,.25) and ++(0,-.25) .. (.95,.5) -- (.95,1) .. controls ++(0,.25) and ++(0,-.25) .. (.75,1.6);
\draw[thick,->] (1.5,-.1) .. controls ++(0,.25) and ++(0,-.25) .. (1.7,.5) -- (1.7,1) .. controls ++(0,.25) and ++(0,-.25) .. (1.5,1.6);
        \node at (-1,-.2) {$\scs s_j$};
        \node at (2,-.2) {$\scs 1$};
        \node at (-1,1.7) {$\scs s_j$};
        \node at (2,1.7) {$\scs 1$};
        \node at (-.125,-.1) {$\cdots$};
        \node at (-.125,-.45) {$\underbrace{\hspace{.75cm}}_{k_1}$};
        \node at (1.125,-.1) {$\cdots$};
        \node at (1.125,-.45) {$\underbrace{\hspace{.75cm}}_{k_2}$};
\end{tikzpicture}}
\end{eqnarray}
where $k_1+k_2=k$.

The second map is determined by
\begin{eqnarray}
\label{unzip}
&&
\hackcenter{
\begin{tikzpicture}[scale=.9]
\draw (-1,0) -- (.5,.25)[very thick,red,   ];
\draw (-1,1.5) -- (.5,1.25)[very thick,red,  ];
\draw (2,0) -- (.5,.25)[thick,red, double, ];
\draw (2,1.5) -- (.5,1.25)[thick,red, double, ];
\draw (.5,.25) -- (.5,1.25)[thick,red, double, ];
\draw[thick,->] (.25,-.1) .. controls ++(0,.25) and ++(0,-.25) .. (.05,.5) -- (.05,1) .. controls ++(0,.25) and ++(0,-.25) .. (.25,1.6);
\draw[thick,->] (-.5,-.1) .. controls ++(0,.25) and ++(0,-.25) .. (-.7,.5) -- (-.7,1) .. controls ++(0,.25) and ++(0,-.25) .. (-.5,1.6);
\draw[thick,->] (.75,-.1) .. controls ++(0,.25) and ++(0,-.25) .. (.95,.5) -- (.95,1) .. controls ++(0,.25) and ++(0,-.25) .. (.75,1.6);
\draw[thick,->] (1.5,-.1) .. controls ++(0,.25) and ++(0,-.25) .. (1.7,.5) -- (1.7,1) .. controls ++(0,.25) and ++(0,-.25) .. (1.5,1.6);
        \node at (-1,-.2) {$\scs 1$};
        \node at (2,-.2) {$\scs s_{j+1}$};
        \node at (-1,1.7) {$\scs 1$};
        \node at (2,1.7) {$\scs s_{j+1}$};
        \node at (-.125,-.1) {$\cdots$};
        \node at (-.125,-.45) {$\underbrace{\hspace{.75cm}}_{k_1}$};
        \node at (1.125,-.1) {$\cdots$};
        \node at (1.125,-.45) {$\underbrace{\hspace{.75cm}}_{k_2}$};
\end{tikzpicture}}
\mapsto
\hackcenter{
\begin{tikzpicture}[scale=.9]
\draw[very thick,red ] (-1,0) .. controls ++(0,.25) and ++(0,-.25) .. (.25,.5) -- (.25,1) .. controls ++(0,.25) and ++(0,-.25) .. (-1,1.5);
\draw[thick,red, double] (2,0) .. controls ++(0,.25) and ++(0,-.25) .. (.75,.5) -- (.75,1) .. controls ++(0,.25) and ++(0,-.25) .. (2,1.5);
\draw[thick,->] (.25,-.1) .. controls ++(0,.25) and ++(0,-.25) .. (.05,.5) -- (.05,1) .. controls ++(0,.25) and ++(0,-.25) .. (.25,1.6);
\draw[thick,->] (-.5,-.1) .. controls ++(0,.25) and ++(0,-.25) .. (-.7,.5) -- (-.7,1) .. controls ++(0,.25) and ++(0,-.25) .. (-.5,1.6);
\draw[thick,->] (.75,-.1) .. controls ++(0,.25) and ++(0,-.25) .. (.95,.5) -- (.95,1) .. controls ++(0,.25) and ++(0,-.25) .. (.75,1.6);
\draw[thick,->] (1.5,-.1) .. controls ++(0,.25) and ++(0,-.25) .. (1.7,.5) -- (1.7,1) .. controls ++(0,.25) and ++(0,-.25) .. (1.5,1.6);
        \node at (-1,-.2) {$\scs 1$};
        \node at (2,-.2) {$\scs s_{j+1}$};
        \node at (-1,1.7) {$\scs 1$};
        \node at (2,1.7) {$\scs s_{j+1}$};
        \node at (-.125,-.1) {$\cdots$};
        \node at (-.125,-.45) {$\underbrace{\hspace{.75cm}}_{k_1}$};
        \node at (1.125,-.1) {$\cdots$};
        \node at (1.125,-.45) {$\underbrace{\hspace{.75cm}}_{k_2}$};
\end{tikzpicture}}
\end{eqnarray}
where $k_1+k_2=k$.
Note that the diagram in the image of this map can be simplified using Relation \eqref{RthickBr2-rel}.

We define the bimodule morphisms $\upsilon_{(s_j,1)}^j$ and $\upsilon_{(1,s_{j+1})}^j$, called \textit{lollipop} (unzip) maps,
\[
\upsilon_{(s_j,1)}^j:\bigtriangledown^j({\bf s}_{(j+1)}) \otimes_{W({\bf s}_{(j+1)}^j,n)} \bigtriangleup^j({\bf s}_{(j+1)})
\rightarrow W({\bf s}_{(j+1)},n)
\]
and
\[
\upsilon_{(1,s_{j+1})}^j:\bigtriangledown^j({\bf s}_{(j)}) \otimes_{W({\bf s}_{(j)}^j,n)} \bigtriangleup^j({\bf s}_{(j)}) \rightarrow W({\bf s}_{(j)},n)
\]
determined by the morphisms in \eqref{unzip1}.

Moreover, composing the morphism $\upsilon_{(s_j,1)}^j$ and $\partial_{(1,s_{j+1}-1)}^{j+1}$ (resp. $\upsilon_{(1,s_j)}^j$ and $\partial_{(s_{j}-1,1)}^{j}$) using \eqref{RthickBr2-rel} we define the following morphism $\upsilon_{{\bf s},l}^j$ (resp. $\upsilon_{{\bf s},r}^j$)
\begin{eqnarray*}
\upsilon_{{\bf s},l}^j&:&
  \mathsf{F}_j \mathsf{E}_j \mathsf{1}_{\mathbf{s}}  \rightarrow W({\bf s},n),
\\
\upsilon_{{\bf s},r}^j&:&
    \mathsf{E}_j \mathsf{F}_j \mathsf{1}_{\mathbf{s}} \rightarrow W({\bf s},n).
\end{eqnarray*}
By definition, the degree of the morphism $\upsilon_{s,l}^j$ is $s_j-s_{j+1}+1$ and the degree of the morphism $\upsilon_{s,r}^j$ is $-s_j+s_{j+1}+1$.

As a diagrammatic presentation, the morphism $\upsilon_{{\bf s},l}^j$ is determined by the following mappings.
\begin{eqnarray*}
\hackcenter{
\begin{tikzpicture}[scale=.9]
\draw (0,0) -- (0,1.5)[thick,red, double, ];
\draw (1.15,0) -- (1.15,1.5)[very thick, double, ] ;
\draw (2.3,0) -- (2.3,1.5)[thick,red, double, ];
\draw  (0,.5) to (2.3,.25)[very thick, red ];
\draw [very thick, red ] (2.3,1.25) to (0,1);
\filldraw[red]  (1.65,.3) circle (2.5pt);
        \node at (1.65,.55) {$\scs \delta$};
        \node at (0,-.2) {$\scs s_j$};
        \node at (2.3,-.2) {$\scs s_{j+1}$};
        \node at (0,1.7) {$\scs s_j$};
        \node at (2.3,1.7) {$\scs s_{j+1}$};
        \node at (2.85,.75) {$\scs s_{j+1}-1$};
        \node at (-.45,.75) {$\scs s_j+1$};
        \node at (.75,.25) {$\scs 1$};
        \node at (.75,1.25) {$\scs 1$};
        \node at (1.15,-.2) {$\scs k$};
\end{tikzpicture}}
&\stackrel{\upsilon_{(s_j,1)}^j}{\mapsto}&
\hackcenter{
\begin{tikzpicture}[scale=.9]
\draw (0,0) -- (0,1.5)[thick,red, double, ];
\draw (1.15,0) -- (1.15,1.5)[very thick, double, ];
\draw (2.3,0) -- (2.3,1.5)[thick,red, double, ];
\draw [very thick, red, ] (2.3,.25) .. controls (-.2,.5) and (-.2,1) .. (2.3,1.25)node[pos=.5, shape=coordinate](DOT7){};
\filldraw[red]  (DOT7) circle (2.5pt);
        \node at (.4,1) {$\scs \delta$};
        \node at (0,-.2) {$\scs s_j$};
        \node at (2.3,-.2) {$\scs s_{j+1}$};
        \node at (0,1.7) {$\scs s_j$};
        \node at (2.3,1.7) {$\scs s_{j+1}$};
        \node at (2.85,.75) {$\scs s_{j+1}-1$};
        \node at (1.15,-.2) {$\scs k$};
\end{tikzpicture}}
\stackrel{\eqref{RthickBr2-rel}}{=}
\sum_{d=0}^{k}
(-1)^{k-d}
\hackcenter{
\begin{tikzpicture}[scale=.9]
\draw (0,0) -- (0,1.5)[thick,red, double, ];
\draw (2.3,0) -- (2.3,1.5)[thick,red, double];
\draw [very thick, red ] (2.3,.25) .. controls (1.9,.5) and (1.9,1) .. (2.3,1.25)node[pos=.75, shape=coordinate](DOT1){};
\draw (1.0,0) -- (1.0,1.5)[very thick, double, ] node[pos=.5, shape=coordinate](DOT2){};
        \node at (0,-.2) {$\scs s_j$};
        \node at (2.3,-.2) {$\scs s_{j+1}$};
        \node at (0,1.7) {$\scs s_j$};
        \node at (2.3,1.7) {$\scs s_{j+1}$};
        \node at (2.85,.75) {$\scs s_{j+1}-1$};
        \node at (1.0,-.2) {$\scs k$};
    \filldraw[red]  (DOT1) circle (2.5pt);
    \filldraw  (DOT2) circle (2.5pt);
        \node at (.7,.75) {$\scs {\sf e}_d$};
        \node at (1.7,1.25) {$\scs  k-d+\delta $};
\end{tikzpicture}}
\\
&\stackrel{\partial_{(1,s_{j+1}-1)}^{j+1}}{\mapsto}&
\sum_{d=0}^{k}
(-1)^{s_{j+1}-1+k-d}
\hackcenter{
\begin{tikzpicture}[scale=.9]
\draw (0,0) -- (0,1.5)[thick,red, double, ];
\draw (2.3,0) -- (2.3,1.5)[thick,red, double]node[pos=.75, shape=coordinate](DOT1){};
        \node at (0,-.2) {$\scs s_j$};
        \node at (2.3,-.2) {$\scs s_{j+1}$};
        \node at (0,1.7) {$\scs s_j$};
        \node at (2.3,1.7) {$\scs s_{j+1}$};
        \node at (1.15,-.2) {$\scs k$};
\draw (1.15,0) -- (1.15,1.5)[very thick, double, ] node[pos=.5, shape=coordinate](DOT2){};
    \filldraw[red]  (DOT1) circle (2.5pt);
    \filldraw  (DOT2) circle (2.5pt);
        \node at (.85,.75) {$\scs {\sf e}_d$};
        \node at (3.5,1.2) {$\scs {\sf h}_{k-d-s_{j+1}+1+\delta}$};
\end{tikzpicture}}
\end{eqnarray*}
The morphism $\upsilon_{{\bf s},r}^j$ is determined by the following mappings.
\begin{eqnarray*}
\hackcenter{
\begin{tikzpicture}[scale=.9]
\draw (0,0) -- (0,1.5)[thick,red, double, ];
\draw (2.3,0) -- (2.3,1.5)[thick,red, double, ];
\draw [very thick, red ] (0,.25) to (2.3,.5) ;
\draw [very thick, red ] (2.3,1) to (0,1.25);
\draw (1.15,0) -- (1.15,1.5)[very thick, double, ];
\filldraw[red]  (.65,.3) circle (2.5pt);
        \node at (.65,-.05) {$\scs \delta$};
        \node at (0,-.2) {$\scs s_j$};
        \node at (2.3,-.2) {$\scs s_{j+1}$};
        \node at (0,1.7) {$\scs s_j$};
        \node at (2.3,1.7) {$\scs s_{j+1}$};
        \node at (2.85,.75) {$\scs s_{j+1}+1$};
        \node at (-.45,.75) {$\scs s_j-1$};
        \node at (1.55,.25) {$\scs 1$};
        \node at (1.55,1.25) {$\scs 1$};
        \node at (1.15,-.2) {$\scs k$};
\end{tikzpicture}}
&\stackrel{\upsilon_{(1,s_j)}^j}{\mapsto}&
\hackcenter{
\begin{tikzpicture}[scale=.9]
\draw (0,0) -- (0,1.5)[thick,red, double, ];
\draw (2.3,0) -- (2.3,1.5)[thick,red, double, ];
\draw [very thick, red ] (0,.25) .. controls (2.5,.5) and (2.5,1) .. (0,1.25) node[pos=.5, shape=coordinate](DOT6){};
\filldraw[red]  (DOT6) circle (2.5pt);
\draw (1.15,0) -- (1.15,1.5)[very thick, double, ];
        \node at (1.9,1) {$\scs \delta$};
        \node at (0,-.2) {$\scs s_j$};
        \node at (2.3,-.2) {$\scs s_{j+1}$};
        \node at (0,1.7) {$\scs s_j$};
        \node at (2.3,1.7) {$\scs s_{j+1}$};
        \node at (-.45,.75) {$\scs s_j-1$};
        \node at (1.15,-.2) {$\scs k$};
\end{tikzpicture}}
\stackrel{\eqref{RthickBr2-rel}}{=}
\sum_{d=0}^{k}
(-1)^{k-d}
\hackcenter{
\begin{tikzpicture}[scale=.9]
\draw (0,0) -- (0,1.5)[thick,red, double, ];
\draw (2.3,0) -- (2.3,1.5)[thick,red, double, ];
\draw [very thick, red ] (0,.25) .. controls (.4,.5) and (.4,1) .. (0,1.25)node[pos=.75, shape=coordinate](DOT1){};
\draw (1.35,0) -- (1.35,1.5)[very thick, double, ] node[pos=.5, shape=coordinate](DOT2){};
        \node at (0,-.2) {$\scs s_j$};
        \node at (2.3,-.2) {$\scs s_{j+1}$};
        \node at (0,1.7) {$\scs s_j$};
        \node at (2.3,1.7) {$\scs s_{j+1}$};
        \node at (-.45,.75) {$\scs s_j-1$};
        \node at (1.35,-.2) {$\scs k$};
    \filldraw[red]  (DOT1) circle (2.5pt);
    \filldraw  (DOT2) circle (2.5pt);
        \node at (1.05,.75) {$\scs {\sf e}_d$};
        \node at (0.65,1.25) {$\scs  k-d+\delta$};
\end{tikzpicture}}
\\
&\stackrel{\partial_{(s_{j}-1,1)}^{j}}{\mapsto}&
\sum_{d=0}^{k}
(-1)^{k-d}
\hackcenter{
\begin{tikzpicture}[scale=.9]
\draw (0,0) -- (0,1.5)[thick,red, double, ]node[pos=.75, shape=coordinate](DOT1){};
\draw (2.3,0) -- (2.3,1.5)[thick,red, double, ];
\draw (1.15,0) -- (1.15,1.5)[very thick, double, ] node[pos=.5, shape=coordinate](DOT2){};
        \node at (0,-.2) {$\scs s_j$};
        \node at (2.3,-.2) {$\scs s_{j+1}$};
        \node at (0,1.7) {$\scs s_j$};
        \node at (2.3,1.7) {$\scs s_{j+1}$};
        \node at (1.15,-.2) {$\scs k$};
    \filldraw[red]  (DOT1) circle (2.5pt);
    \filldraw  (DOT2) circle (2.5pt);
        \node at (.85,.75) {$\scs {\sf e}_d$};
        \node at (-1.05,1.2) {$\scs {\sf h}_{k-d-s_j+1+\delta}$};
\end{tikzpicture}}
\end{eqnarray*}

\subsection{Morphisms $\zeta_{{\bf s},r}^j$ and $\zeta_{{\bf s},l}^j$}
First we define the \textit{zip} maps $\zeta_{(s_j,1)}^j$ and $\zeta_{(1,s_{j+1})}^j$.  Recall the sequence
${\bf s}_{(j)}$ defined in Section~\ref{sec:unzip}.

For  $\mathbf{i}_1=(...,s_{j},\underbrace{\mf{b},...,\mf{b}}_{k},1,...)\in \mathrm{Seq}({\bf s}_{(j+1)},n)$ and $\mathbf{i}_2=(...,1,\underbrace{\mf{b},...,\mf{b}}_{k},s_{j+1},...)\in \mathrm{Seq}({\bf s}_{(j)},n)$,
%
there is a map
\[
\zeta_{(s_j,1)}^j:W({\bf s}_{(j+1)},n) \rightarrow \bigtriangledown^j({\bf s}_{(j+1)}) \otimes_{W({\bf s}_{(j+1)}^{j},n)} \bigtriangleup^j({\bf s}_{(j+1)})
\]
of degree $s_j$ determined by
\begin{align}
\label{bimodmap7}
\zeta_{(s_j,1)}^j:
\hackcenter{
\begin{tikzpicture}[scale=.9]
\draw (0,0) -- (0,1.5)[thick,red, double, ];
\draw (1,0) -- (1,1.5)[very thick, double, ];
\draw (2,0) -- (2,1.5)[very thick,red,];
        \node at (0,-.2) {$\scs s_j$};
        \node at (2,-.2) {$\scs 1$};
        \node at (1,-.2) {$\scs k$};
        \node at (.5,2) {};
\end{tikzpicture}}
&\mapsto
(-1)^{k-1}
\hackcenter{
\begin{tikzpicture} [scale=.75]
\draw[thick, red, double] (-1,0) .. controls (-1,.5) .. (0,.75);
\draw[very thick, red] (1,0) .. controls (1,.5) ..  (0,.75);
\draw[thick, red, double] (0,.75) to (0,1.25);
\draw[thick, red, double] (0,1.25) .. controls (-1,1.5) .. (-1,2);
\draw[very thick, red] (0,1.25) .. controls (1,1.5) ..  (1,2);
\draw[very thick, double ] (0,0) to (0,.5);
\draw[thick] (0,.5) .. controls ++(-1.5,0) and ++(-1.5,0) .. (0,1.5);
\draw[very thick, double] (0,.5) .. controls ++(1.5,0) and ++(1.5,0) .. (0,1.5);
\draw[very thick, double,] (0,1.5) to (0,2);
 \node at (-1,-.2) {$\scs s_j$};
 \node at (1,-.2) {$\scs 1$};
 \node at (0,-.2) {$\scs k$};
 \node at (-1,2.2) {$\scs s_j$};
 \node at (1,2.2) {$\scs 1$};
    \node at (.45,1) {$\scs s_j+1$};
    \node at (1.65,1) {$\scs k-1$};
    \node at (-1.25,1) {$\scs 1$};
\end{tikzpicture}}
\;\; + \;\;
\sum_{d_1+d_2+d_3\atop =s_j-k}
(-1)^{s_j+d_1+d_2}
\hackcenter{
\begin{tikzpicture}[scale=.9]
\draw (0,0) -- (.5,.25)[thick,red, double, ]node[pos=.5, shape=coordinate](DOT2){};
\draw (2,0) -- (.5,.25)[very thick,red,   ];
\draw (0,1.5) -- (.5,1.25)[thick,red, double, ];
\draw (2,1.5) -- (.5,1.25)[very thick,red,   ]node[pos=.25, shape=coordinate](DOT1){};
\draw (.5,.25) -- (.5,1.25)[thick,red, double, ];
\draw[very thick,double ,] (1,-.1) .. controls ++(0,.25) and ++(0,-.25) .. (1.35,.5) -- (1.35,1) .. controls ++(0,.25) and ++(0,-.25) .. (1,1.6)node[pos=0, shape=coordinate](DOT3){};
    \filldraw[red]  (DOT1) circle (2.5pt);
    \filldraw[red]  (DOT2) circle (2.5pt);
    \filldraw  (DOT3) circle (2.5pt);
        \node at (1.6,1.7) {$\scs d_2$};
        \node at (.1,0.4) {$\scs {\sf e}_{d_3}$};
        \node at (0,-.2) {$\scs s_j$};
        \node at (2,-.2) {$\scs 1$};
        \node at (0.05,0.85) {$\scs s_j+1$};
        \node at (0,1.7) {$\scs s_j$};
        \node at (2,1.7) {$\scs 1$};
        \node at (1.75,1) {$\scs {\sf h}_{d_1}$};
\end{tikzpicture}},
\end{align}
and a map
\[
\zeta_{(1,s_{j+1})}^j:W({\bf s}_{(j)},n) \rightarrow \bigtriangledown^j({\bf s}_{(j)}) \otimes_{W({\bf s}_{(j)}^j,n)} \bigtriangleup^j({\bf s}_{(j)})
 \]
 of degree $s_{j+1}$ determined by
\begin{align}
\zeta_{(1,s_{j+1})}^j:
\hackcenter{
\begin{tikzpicture}[scale=.9]
\draw (0,0) -- (0,1.5)[very thick,red,  ];
\draw (1,0) -- (1,1.5)[very thick, double, ];
\draw (2,0) -- (2,1.5)[thick,red, double, ];
        \node at (0,-.2) {$\scs 1$};
        \node at (2,-.2) {$\scs s_{j+1}$};
        \node at (1,-.2) {$\scs k$};
        \node at (.5,2) {};
\end{tikzpicture}}
\mapsto&
(-1)^{s_{j+1}}
\hackcenter{
\begin{tikzpicture} [scale=.75]
\draw[very thick, red, ] (-1,0) .. controls (-1,.5) .. (0,.75);
\draw[thick, red, double] (1,0) .. controls (1,.5) ..  (0,.75);
\draw[thick, red, double] (0,.75) to (0,1.25);
\draw[very thick, red] (0,1.25) .. controls (-1,1.5) .. (-1,2);
\draw[thick, red, double] (0,1.25) .. controls (1,1.5) ..  (1,2);
\draw[very thick, double ] (0,0) to (0,.5);
\draw[very thick, double] (0,.5) .. controls ++(-1.35,0) and ++(-1.35,0) .. (0,1.5);
\draw[thick] (0,.5) .. controls ++(1.95,0) and ++(1.95,0) .. (0,1.5);
\draw[very thick, double  ,] (0,1.5) to (0,2);
 \node at (-1,-.2) {$\scs 1$};
 \node at (1,-.2) {$\scs s_{j+1}$};
 \node at (-1,2.2) {$\scs 1$};
 \node at (1,2.2) {$\scs s_{j+1}$};
    \node at (.75,1) {$\scs s_{j+1}+1$};
    \node at (1.45,.55) {$\scs 1$};
    \node at (-1.65,1) {$\scs k-1$};
\end{tikzpicture}} 
\;\; + \;\;
\sum_{d_1+d_2+d_3\atop =s_{j+1}-k}
(-1)^{s_{j+1}+d_3+k}
\hackcenter{
\begin{tikzpicture}[scale=.9]
\draw (-1,0) -- (.5,.25)[very thick,red,  ];
\draw (1,0) -- (.5,.25)[thick,red, double, ] node[pos=.5, shape=coordinate](DOT2){};
\draw (-1,1.5) -- (.5,1.25)[very thick,red,   ] node[pos=.25, shape=coordinate](DOT1){};
\draw (1,1.5) -- (.5,1.25)[thick,red, double, ];
\draw (.5,.25) -- (.5,1.25)[thick,red, double, ];
\draw[very thick,double ,] (0,-.1) .. controls ++(0,.25) and ++(0,-.25) .. (-.35,.5) -- (-.35,1) .. controls ++(0,.25) and ++(0,-.25) .. (0,1.6) node[pos=0, shape=coordinate](DOT3){};
    \filldraw[red]  (DOT1) circle (2.5pt);
    \filldraw[red]  (DOT2) circle (2.5pt);
    \filldraw  (DOT3) circle (2.5pt);
        \node at (-0.65,1.75) {$\scs d_2$};
        \node at (.95,0.3) {$\scs {\sf e}_{d_3}$};
        \node at (-1,-.2) {$\scs 1$};
        \node at (1,-.2) {$\scs s_{j+1}$};
        \node at (1.05,0.75) {$\scs s_{j+1}+1$};
        \node at (-1,1.7) {$\scs 1$};
        \node at (1,1.8) {$\scs s_{j+1}$};
        \node at (-.7,1) {$\scs {\sf h}_{d_1}$};
\end{tikzpicture}}.
\end{align}
Note that these maps are determined by the image of a thick black strand because by Lemma~\ref{thm_Ea} the identity on $k$ strands factors through the thickness $k$ thick strand. For example,
\[
\hackcenter{
\begin{tikzpicture} [scale=.85]
\draw[very thick, red  ] (0,0) to (0,2);
\draw[very thick, red  ] (2.4,0) to (2.4,2);
\draw[thick,  ] (.6,0) to (.6,2);
\draw[thick,  ] (1.2,0) to (1.2,2);
\draw[thick,  ] (1.8,0) to (1.8,2);
\draw[thick,  ] (.6,0) to (.6,2);
 \node at (1.5,1) {$\scs \cdots$};
 \node at (0,-.2) {$\scs 1$};
 \node at (.6,-.2) {$\scs 1$};
 \node at (1.2,-.2) {$\scs 1$};
  \node at (1.8,-.2) {$\scs 1$};
 \node at (2.4,-.2) {$\scs 1$};
\end{tikzpicture}}
\;\; = \;\;
\sum_{\und{\ell} \in \Sq(k)}(-1)^{|\hat{\und{\ell}}|}\;\;
\hackcenter{
\begin{tikzpicture} [scale=0.65]
\draw[very thick, red  ] (-2,-.5) to (-2,3.5);
\draw[very thick, red  ] (2,-.5) to (2,3.5);
\draw[thick,  ] (-.6,-.5) .. controls ++(0,.5) and ++(-.4,-.1) .. (0,1.15);
\draw[thick,  ] (.6,-.5) .. controls ++(0,.5) and ++(.4,-.1) .. (0,1.15);
\draw[thick,  ] (-1.2,-.5) .. controls ++(0,.7) and ++(-.4,-.1) .. (0,1.15);
\draw[thick,  ] (1.2,-.5) .. controls ++(0,.7) and ++(.4,-.1) .. (0,1.15);
\draw[thick,  double] (0,1.15) to (0,1.85);
\draw[thick,  ] (0,1.85) .. controls ++(.25,.1) and ++(0,-.5) .. (.6,3.5);
\draw[thick,  ] (0,1.85) .. controls ++(-.25,.1) and ++(0,-.5) .. (-.6,3.5);
\draw[thick,  ] (0,1.85) .. controls ++(.25,.1) and ++(0,-.7) .. (1.2,3.5);
\draw[thick,  ] (0,1.85) .. controls ++(-.25,.1) and ++(0,-.7) .. (-1.2,3.5);
\node[draw, fill=white!20 ,rounded corners ] at (0,2.7) {$ \qquad {\sf e}_{\und{\ell}} \qquad $};
\node[draw, fill=white!20 ,rounded corners ] at (0,.35) {$ \qquad   \und{x}^{\hat{\und{\ell}}} \qquad $};
%
 \node at (-1.2,-.7) {$\scs 1$};
 \node at (-.6,-.7) {$\scs 1$};
 \node at (.6,-.7) {$\scs 1$};
 \node at (1.2,-.7) {$\scs 1$};
  \node at (2,-.7) {$\scs 1$};
 \node at (-2,-.7) {$\scs 1$};
\end{tikzpicture}}
\;\; \mapsto \;\;
-\sum_{\und{\ell} \in \Sq(k)}(-1)^{|\hat{\und{\ell}}|}\;\;
\hackcenter{
\begin{tikzpicture} [scale=.65]
\draw[very thick, red , out=90, in=210 ] (-1.5,-1.65) to (0,1.25);
\draw[very thick, red , out=90, in=-30 ] (1.5,-1.65) to (0,1.25);
\draw[thick, red, double] (0,1.25) to (0,1.75);
\draw[very thick, red , out=150, in=-90 ] (0,1.75) to (-1.5,4.65);
\draw[very thick, red , out=30, in=-90 ] (0,1.75) to (1.5,4.65);
\draw[thick,   double ] (0,0) to (0,.75);
\draw[thick,   double, out =150, in=-90 ] (0,.75) to (-.6,1.5);
\draw[thick,    out =30, in=-90 ] (0,.75) to (.6,1.5);
\draw[thick,   double, out =90, in=200 ]   (-.6,1.5) to (0,2.25);
\draw[thick,    out =90, in=-20 ]   (.6,1.5) to (0,2.25);
\draw[thick,   double  ] (0,2.25) to (0,3);
\draw[thick,  ] (0,3) .. controls ++(.25,.1) and ++(0,-.5) .. (.3,4.65);
\draw[thick,  ] (0,3) .. controls ++(-.25,.1) and ++(0,-.5) .. (-.3,4.65);
\draw[thick,  ] (0,3) .. controls ++(.25,.1) and ++(0,-.7) .. (.9,4.65);
\draw[thick, ] (0,3) .. controls ++(-.25,.1) and ++(0,-.7) .. (-.9,4.65);
\node[draw, fill=white!20 ,rounded corners ] at (0,3.85) {$ \quad\; {\sf e}_{\und{\ell}} \quad\; $};
\draw[thick,  ] (-.3,-1.65) .. controls ++(0,.5) and ++(-.4,-.1) .. (0,0);
\draw[thick,  ] (.3,-1.65) .. controls ++(0,.5) and ++(.4,-.1) .. (0,0);
\draw[thick,  ] (-.9,-1.65) .. controls ++(0,.7) and ++(-.4,-.1) .. (0,0);
\draw[thick,  ] (.9,-1.65) .. controls ++(0,.7) and ++(.4,-.1) .. (0,0);
\node[draw, fill=white!20 ,rounded corners ] at (0,-.95) {$ \quad\;   \und{x}^{\hat{\und{\ell}}} \;\quad $};
 \node at (-1.5,-1.85) {$\scs 1$};
 \node at (-.9,-1.85) {$\scs 1$};
 \node at (-.3,-1.85) {$\scs 1$};
 \node at (.3,-1.85) {$\scs 1$};
 \node at (.9,-1.85) {$\scs 1$};
 \node at (1.5,-1.85) {$\scs 1$};
    \node at (.85,1.8) {$\scs 1$};
    \node at (-1.1,1.8) {$\scs k-1$};
\end{tikzpicture}}
\]
which can be generalized to thicker red strands in the obvious way.
Furthermore, this determines the image of any element  $x \in \nh_k$ by placing $x$ at the top or bottom of the diagram on the right.  This is well-defined by the following Lemmas \ref{lem:dot_through_onek} and \ref{lem:crossing_through_onek}.
Note that the proofs of each of the following two lemmas do not use maneuvers using the red strands.  However, the presence of the red strands complicates the proof (since both diagrams would be zero without them).

\begin{lemma}  \label{lem:dot_through_onek}
For arbitrary thickness of red strands $a$ and $b$, the following identity holds:
\begin{equation}\label{eq:push-lemma}
\sum_{\und{\ell} \in \Sq(k)}(-1)^{|\hat{\und{\ell}}|}\;\;
\hackcenter{
\begin{tikzpicture} [scale=.65]
\draw[very thick, red , out=90, in=210 ] (-1.5,-1.65) to (0,1.25);
\draw[very thick, red , out=90, in=-30 ] (1.5,-1.65) to (0,1.25);
\draw[thick, red, double] (0,1.25) to (0,1.75);
\draw[very thick, red , out=150, in=-90 ] (0,1.75) to (-1.5,4.65);
\draw[very thick, red , out=30, in=-90 ] (0,1.75) to (1.5,4.65);
\draw[thick,   double ] (0,.3) to (0,.75);
\draw[thick,   double, out =150, in=-90 ] (0,.75) to (-.6,1.5);
\draw[thick,    out =30, in=-90 ] (0,.75) to (.6,1.5);
\draw[thick,   double, out =90, in=200 ]   (-.6,1.5) to (0,2.25);
\draw[thick,    out =90, in=-20 ]   (.6,1.5) to (0,2.25);
\draw[thick,   double  ] (0,2.25) to (0,2.6);
\draw[thick,  ] (0,2.6) to (0,4.65);
\draw[thick,  ] (0,2.6) .. controls ++(.25,.1) and ++(0,-.7) .. (.9,4.65);
\draw[thick, ] (0,2.6) .. controls ++(-.25,.1) and ++(0,-.7) .. (-.9,4.65);
\node[draw, fill=white!20 ,rounded corners ] at (0,3.55) {$ \quad\; {\sf e}_{\und{\ell}} \quad\; $};
\draw[thick,  ] (0,-1.65) to (0,.3);
\draw[thick,  ] (-.9,-1.65) .. controls ++(0,.7) and ++(-.4,-.1) .. (0,.3);
\draw[thick,  ] (.9,-1.65) .. controls ++(0,.7) and ++(.4,-.1) .. (0,.3);
\node[draw, fill=white!20 ,rounded corners ] at (0,-.65) {$ \quad\;   \und{x}^{\hat{\und{\ell}}} \;\quad $};
\filldraw  (0,4.2) circle (2.5pt);
 \node at (-1.5,-1.85) {$\scs a$};
 \node at (-.9,-1.85) {$\scs 1$};
 \node at (0,-1.85) {$\scs 1$};
 \node at (.9,-1.85) {$\scs 1$};
 \node at (1.5,-1.85) {$\scs b$};
    \node at (.85,1.8) {$\scs 1$};
    \node at (-1.1,1.8) {$\scs k-1$};
    \node at (.45,-1.5) {$\scs \cdots$};
    \node at (-.45,-1.5) {$\scs \cdots$};
    \node at (.45,4.45) {$\scs \cdots$};
    \node at (-.45,4.45) {$\scs \cdots$};
\end{tikzpicture}}
\;\; = \;\;
\sum_{\und{\ell} \in \Sq(k)}(-1)^{|\hat{\und{\ell}}|}\;\;
\hackcenter{
\begin{tikzpicture} [scale=.65]
\draw[very thick, red , out=90, in=210 ] (-1.5,-1.65) to (0,1.5);
\draw[very thick, red , out=90, in=-30 ] (1.5,-1.65) to (0,1.5);
\draw[thick, red, double] (0,1.5) to (0,2);
\draw[very thick, red , out=150, in=-90 ] (0,2) to (-1.5,4.65);
\draw[very thick, red , out=30, in=-90 ] (0,2) to (1.5,4.65);
\draw[thick,   double ] (0,.5) to (0,1);
\draw[thick,   double, out =150, in=-90 ] (0,1) to (-.6,1.75);
\draw[thick,    out =30, in=-90 ] (0,1) to (.6,1.75);
\draw[thick,   double, out =90, in=200 ]   (-.6,1.75) to (0,2.5);
\draw[thick,    out =90, in=-20 ]   (.6,1.75) to (0,2.5);
\draw[thick,   double  ] (0,2.5) to (0,3);
\draw[thick,  ] (0,3) to(0,4.65);
\draw[thick,  ] (0,3) .. controls ++(.25,.1) and ++(0,-.7) .. (.9,4.65);
\draw[thick, ] (0,3) .. controls ++(-.25,.1) and ++(0,-.7) .. (-.9,4.65);
\node[draw, fill=white!20 ,rounded corners ] at (0,3.85) {$ \quad\; {\sf e}_{\und{\ell}} \quad\; $};
\draw[thick,  ] (0,-1.65) to (0,.5);
\draw[thick,  ] (-.9,-1.65) .. controls ++(0,.7) and ++(-.4,-.1) .. (0,.5);
\draw[thick,  ] (.9,-1.65) .. controls ++(0,.7) and ++(.4,-.1) .. (0,.5);
\node[draw, fill=white!20 ,rounded corners ] at (0,-.45) {$ \quad\;   \und{x}^{\hat{\und{\ell}}} \;\quad $};
\filldraw  (0,-1.2) circle (2.5pt);
 \node at (-1.5,-1.85) {$\scs a$};
 \node at (-.9,-1.85) {$\scs 1$};
 \node at (0,-1.85) {$\scs 1$};
 \node at (.9,-1.85) {$\scs 1$};
 \node at (1.5,-1.85) {$\scs b$};
    \node at (.85,1.8) {$\scs 1$};
    \node at (-1.1,1.8) {$\scs k-1$};
    \node at (.45,-1.5) {$\scs \cdots$};
    \node at (-.45,-1.5) {$\scs \cdots$};
    \node at (.45,4.45) {$\scs \cdots$};
    \node at (-.45,4.45) {$\scs \cdots$};
\end{tikzpicture}}
\end{equation}
\end{lemma}

\begin{proof}
Recall from Remark~\ref{rem:dual-bases} that the sets $\{{\sf e}_{\und{\ell}} \mid \und{\ell} \in \Sq(k) \}$ and $\{ \und{x}^{\hat{\und{\ell}}} \mid \und{\ell} \in \Sq(k)\}$ both form bases for the polynomial ring as a free module over the ring of symmetric functions $\Lambda_k$.   In particular, if we multiply any element in either of these bases by $x_i$, it is possible to rewrite the resulting polynomial as some $\Lambda_k$-combination of basis elements. The first part of Lemma~\ref{thm_Ea} shows that multiplying by a black dot on the $i$-th strand on top of the sum
\[
\sum_{\und{\ell} \in \Sq(k)}(-1)^{|\hat{\und{\ell}}|}\;\;
\hackcenter{
\begin{tikzpicture} [scale=0.75]
\draw[thick,  ] (-.6,-.5) .. controls ++(0,.5) and ++(-.4,-.1) .. (0,1.15);
\draw[thick,  ] (.6,-.5) .. controls ++(0,.5) and ++(.4,-.1) .. (0,1.15);
\draw[thick,  ] (-1.2,-.5) .. controls ++(0,.7) and ++(-.4,-.1) .. (0,1.15);
\draw[thick,  ] (1.2,-.5) .. controls ++(0,.7) and ++(.4,-.1) .. (0,1.15);
\draw[thick,  double] (0,1.15) to (0,1.85);
\draw[thick,  ] (0,1.85) .. controls ++(.25,.1) and ++(0,-.5) .. (.6,3.5);
\draw[thick,  ] (0,1.85) .. controls ++(-.25,.1) and ++(0,-.5) .. (-.6,3.5);
\draw[thick,  ] (0,1.85) .. controls ++(.25,.1) and ++(0,-.7) .. (1.2,3.5);
\draw[thick,  ] (0,1.85) .. controls ++(-.25,.1) and ++(0,-.7) .. (-1.2,3.5);
\node[draw, fill=white!20 ,rounded corners ] at (0,2.7) {$ \qquad {\sf e}_{\und{\ell}} \qquad $};
\node[draw, fill=white!20 ,rounded corners ] at (0,.35) {$ \qquad   \und{x}^{\hat{\und{\ell}}} \qquad $};
%
\end{tikzpicture}}
\]
gives the same result as multiplying by a black dot on the bottom of the sum.  Using \eqref{eq:splitters} all of the symmetric functions produced in rewriting either basis can be put on the thick $k$ labelled strand.  Hence, if we rewrite each diagram on the left-hand side of \eqref{eq:push-lemma} in the basis $\{{\sf e}_{\und{\ell}} \mid \und{\ell} \in \Sq(k) \}$ and push the symmetric functions onto the thick $k$-labelled strand, these symmetric functions can be pushed to the bottom of the diagram using \eqref{eq:splitters}.  By Lemma~\ref{thm_Ea} this gives the same result as rewriting each diagram on the right hand side of \eqref{eq:push-lemma} using the basis $\{ \und{x}^{\hat{\und{\ell}}} \mid \und{\ell} \in \Sq(k)\}$ establishing the lemma.
\end{proof}

\begin{lemma} \label{lem:crossing_through_onek} For arbitrary thickness of red strands $a$ and $b$, the identity
\begin{equation} \label{eq:crossing_through_onek}
\sum_{\und{\ell} \in \Sq(k)}(-1)^{|\hat{\und{\ell}}|}\;\;
\hackcenter{
\begin{tikzpicture} [scale=.65]
\draw[very thick, red , out=90, in=210 ] (-1.5,-1.65) to (0,1.25);
\draw[very thick, red , out=90, in=-30 ] (1.5,-1.65) to (0,1.25);
\draw[thick, red, double] (0,1.25) to (0,1.75);
\draw[very thick, red , out=150, in=-90 ] (0,1.75) to (-1.5,4.65);
\draw[very thick, red , out=30, in=-90 ] (0,1.75) to (1.5,4.65);
\draw[thick,   double ] (0,.3) to (0,.75);
\draw[thick,   double, out =150, in=-90 ] (0,.75) to (-.6,1.5);
\draw[thick,    out =30, in=-90 ] (0,.75) to (.6,1.5);
\draw[thick,   double, out =90, in=200 ]   (-.6,1.5) to (0,2.25);
\draw[thick,    out =90, in=-20 ]   (.6,1.5) to (0,2.25);
\draw[thick,   double  ] (0,2.25) to (0,2.6);
\draw[thick,   ](-.3,3.75) .. controls ++(0,.4) and ++(0,-.3) .. (.3,4.65);
\draw[thick,   ](.3,3.75) .. controls ++(0,.4) and ++(0,-.3) .. (-.3,4.65);
\draw[thick,  ] (0,2.6) .. controls ++(.15,.1) and ++(0,-.5) .. (.3,3.65);
\draw[thick,  ] (0,2.6) .. controls ++(-.15,.1) and ++(0,-.5) .. (-.3,3.65);
\draw[thick,  ] (0,2.6) .. controls ++(.25,.1) and ++(0,-.7) .. (1.1,4.65);
\draw[thick, ] (0,2.6) .. controls ++(-.25,.1) and ++(0,-.7) .. (-1.1,4.65);
\node[draw, fill=white!20 ,rounded corners ] at (0,3.55) {$ \quad\; {\sf e}_{\und{\ell}} \quad\; $};
\draw[thick,  ] (-.3,-1.65) .. controls ++(0,.5) and ++(-.4,-.1) .. (0,.3);
\draw[thick,  ] (.3,-1.65) .. controls ++(0,.5) and ++(.4,-.1) .. (0,.3);
\draw[thick,  ] (-1.1,-1.65) .. controls ++(0,.7) and ++(-.4,-.1) .. (0,.3);
\draw[thick,  ] (1.1,-1.65) .. controls ++(0,.7) and ++(.4,-.1) .. (0,.3);
\node[draw, fill=white!20 ,rounded corners ] at (0,-.65) {$ \quad\;   \und{x}^{\hat{\und{\ell}}} \;\quad $};
 \node at (-1.5,-1.85) {$\scs a$};
 \node at (-1.0,-1.85) {$\scs 1$};
 \node at (.3,-1.85) {$\scs 1$};
\node at (-.3,-1.85) {$\scs 1$};
 \node at (1.0,-1.85) {$\scs 1$};
 \node at (1.5,-1.85) {$\scs b$};
    \node at (.85,1.8) {$\scs 1$};
    \node at (-1,1.8) {$\scs k-1$};
    \node at (.6,-1.5) {$\scs \cdots$};
    \node at (-.55,-1.5) {$\scs \cdots$};
    \node at (.6,4.45) {$\scs \cdots$};
    \node at (-.6,4.45) {$\scs \cdots$};
\end{tikzpicture}}
\;\; = \;\;
\sum_{\und{\ell} \in \Sq(k)}(-1)^{|\hat{\und{\ell}}|}\;\;
\hackcenter{
\begin{tikzpicture} [scale=.65]
\draw[very thick, red , out=90, in=210 ] (-1.5,-1.65) to (0,1.5);
\draw[very thick, red , out=90, in=-30 ] (1.5,-1.65) to (0,1.5);
\draw[thick, red, double] (0,1.5) to (0,2);
\draw[very thick, red , out=150, in=-90 ] (0,2) to (-1.5,4.65);
\draw[very thick, red , out=30, in=-90 ] (0,2) to (1.5,4.65);
\draw[thick,   double ] (0,.5) to (0,1);
\draw[thick,   double, out =150, in=-90 ] (0,1) to (-.6,1.75);
\draw[thick,    out =30, in=-90 ] (0,1) to (.6,1.75);
\draw[thick,   double, out =90, in=200 ]   (-.6,1.75) to (0,2.5);
\draw[thick,    out =90, in=-20 ]   (.6,1.75) to (0,2.5);
\draw[thick,   double  ] (0,2.5) to (0,3);
\draw[thick,  ] (0,3) .. controls ++(.25,.1) and ++(0,-.5) .. (.3,4.65);
\draw[thick,  ] (0,3) .. controls ++(-.25,.1) and ++(0,-.5) .. (-.3,4.65);
\draw[thick,  ] (0,3) .. controls ++(.25,.1) and ++(0,-.7) .. (1.0,4.65);
\draw[thick, ] (0,3) .. controls ++(-.25,.1) and ++(0,-.7) .. (-1.0,4.65);
\node[draw, fill=white!20 ,rounded corners ] at (0,3.85) {$ \quad\; {\sf e}_{\und{\ell}} \quad\; $};
\draw[thick,   ](-.3,-1.65) .. controls ++(0,.3) and ++(0,-.3) .. (.3,-0.75);
\draw[thick,   ](.3,-1.65) .. controls ++(0,.3) and ++(0,-.3) .. (-.3,-0.75);
\draw[thick,  ] (-.3,-.55) .. controls ++(0,.5) and ++(-.2,-.1) .. (0,.5);
\draw[thick,  ] (.3,-.55) .. controls ++(0,.5) and ++(.2,-.1) .. (0, .5);
\draw[thick,  ] (-1.0,-1.65) .. controls ++(0,.7) and ++(-.4,-.1) .. (0,.5);
\draw[thick,  ] (1.0,-1.65) .. controls ++(0,.7) and ++(.4,-.1) .. (0,.5);
\node[draw, fill=white!20 ,rounded corners ] at (0,-.35) {$ \quad\;   \und{x}^{\hat{\und{\ell}}} \;\quad $};
 \node at (-1.5,-1.85) {$\scs a$};
 \node at (-1.0,-1.85) {$\scs 1$};
 \node at (.3,-1.85) {$\scs 1$};
\node at (-.3,-1.85) {$\scs 1$};
 \node at (1.0,-1.85) {$\scs 1$};
 \node at (1.5,-1.85) {$\scs b$};
    \node at (.85,1.8) {$\scs 1$};
    \node at (-1,1.8) {$\scs k-1$};
    \node at (.6,-1.5) {$\scs \cdots$};
    \node at (-.55,-1.5) {$\scs \cdots$};
    \node at (.6,4.45) {$\scs \cdots$};
    \node at (-.6,4.45) {$\scs \cdots$};
\end{tikzpicture}}
\end{equation}
holds.
\end{lemma}

\begin{proof}
Suppose we act by a black crossing on the $i$-th and $(i+1)$-st black strands.
It is helpful to note that the elementary symmetric function ${\sf e}_{\ell_{i}}^{(i)}$ on the first $i$ black strands in the standard elementary monomial ${\sf e}_{\und{\ell}}$ can be written as
\begin{equation} \label{eq:elem-reduction}
 {\sf e}_{\ell_{i}}^{(i)} = {\sf e}_{\ell_{i}}^{(i-1)} + x_{i} {\sf e}_{\ell_{i}-1}^{(i-1)}.
\end{equation}
Since the $(i,i+1)$ black crossing commutes with ${\sf e}_{\ell_a}^{(a)}$ for $a\neq i$, the left-hand side of \eqref{eq:crossing_through_onek} can be simplified using \eqref{eq:elem-reduction}, where the first term involving ${\sf e}_{\ell_{i}}^{(i-1)}$ vanishes as a result of the definition of the splitter and (1) in the definition of the nilHecke algebra ($\partial_i^2=0$).  The remaining term can be further simplified using (4) in the definition of the nilHecke algebra ($\partial_i x_i =  x_{i+1} \partial_i +1$) and again noting that the term where the black crossing does not resolve, again vanishes by the definition of the splitter and (1) in the definition of the nilHecke algebra.  Hence, the left-hand side reduces to
\begin{equation} \label{lem:crossing_goal}
\sum_{\und{\ell} \in \Sq(k)}(-1)^{|\hat{\und{\ell}}|}\;\;
\hackcenter{
\begin{tikzpicture} [scale=.65]
\draw[very thick, red , out=90, in=210 ] (-1.5,-1.65) to (0,1.25);
\draw[very thick, red , out=90, in=-30 ] (1.5,-1.65) to (0,1.25);
\draw[thick, red, double] (0,1.25) to (0,1.75);
\draw[very thick, red , out=150, in=-90 ] (0,1.75) to (-1.5,4.65);
\draw[very thick, red , out=30, in=-90 ] (0,1.75) to (1.5,4.65);
\draw[thick,   double ] (0,0) to (0,.75);
\draw[thick,   double, out =150, in=-90 ] (0,.75) to (-.6,1.5);
\draw[thick,    out =30, in=-90 ] (0,.75) to (.6,1.5);
\draw[thick,   double, out =90, in=200 ]   (-.6,1.5) to (0,2.25);
\draw[thick,    out =90, in=-20 ]   (.6,1.5) to (0,2.25);
\draw[thick,   double  ] (0,2.25) to (0,3);
\draw[thick,  ] (0,3) .. controls ++(.25,.1) and ++(0,-.5) .. (.3,4.65);
\draw[thick,  ] (0,3) .. controls ++(-.25,.1) and ++(0,-.5) .. (-.3,4.65);
\draw[thick,  ] (0,3) .. controls ++(.25,.1) and ++(0,-.7) .. (.9,4.65);
\draw[thick, ] (0,3) .. controls ++(-.25,.1) and ++(0,-.7) .. (-.9,4.65);
\node[draw, fill=white!20 ,rounded corners ] at (0,3.85) {$ \quad\; {\sf e}_{\und{\ell}'} \quad\; $};
\draw[thick,  ] (-.3,-1.65) .. controls ++(0,.5) and ++(-.4,-.1) .. (0,0);
\draw[thick,  ] (.3,-1.65) .. controls ++(0,.5) and ++(.4,-.1) .. (0,0);
\draw[thick,  ] (-.9,-1.65) .. controls ++(0,.7) and ++(-.4,-.1) .. (0,0);
\draw[thick,  ] (.9,-1.65) .. controls ++(0,.7) and ++(.4,-.1) .. (0,0);
\node[draw, fill=white!20 ,rounded corners ] at (0,-.95) {$ \quad\;   \und{x}^{\hat{\und{\ell}}} \;\quad $};
 \node at (-1.5,-1.85) {$\scs 1$};
 \node at (-.9,-1.85) {$\scs 1$};
 \node at (-.3,-1.85) {$\scs 1$};
 \node at (.3,-1.85) {$\scs 1$};
 \node at (.9,-1.85) {$\scs 1$};
 \node at (1.5,-1.85) {$\scs 1$};
    \node at (.85,1.8) {$\scs 1$};
    \node at (-1,1.8) {$\scs k-1$};
\end{tikzpicture}}
\end{equation}
where
\[
{\sf e}_{\und{\ell}'} =
{\sf e}_{\ell_1}^{(1)} \dots {\sf e}_{\ell_{i-1}}^{(i-1)} {\sf e}_{\ell_{i}-1}^{(i-1)} {\sf e}_{\ell_{i+1}}^{(i+1)} \dots
{\sf e}_{\ell_{k}}^{(k)},
\]
so that there are no elementary symmetric functions in the first $i$ variables.

For the right-hand side of \eqref{eq:crossing_through_onek} it is helpful to break the summation into two pieces depending on if $\ell_{i+1} < \ell_{i}+1$ (so that $\hat{\ell}_i < \hat{\ell}_{i+1}$) or $\ell_{i}+1< \ell_{i+1}$ (so that $\hat{\ell}_i > \hat{\ell}_{i+1}$.)  The case when $\hat{\ell}_i = \hat{\ell}_{i+1}$ vanishes.  Then using Lemma~\ref{diff1} we can simplify
\[
\hackcenter{\begin{tikzpicture}[scale=0.8]
    \draw[thick] (0,0) .. controls ++(0,.5) and ++(0,-.5) .. (.75,1) node[pos=1, shape=coordinate](DOT2){};
    \draw[thick] (.75,0) .. controls ++(0,.5) and ++(0,-.5) .. (0,1)  node[pos=1, shape=coordinate](DOT1){};
    \draw[thick,] (0,1 ) .. controls ++(0,.5) and ++(0,-.5) .. (.75,2);
    \draw[thick, ] (.75,1) .. controls ++(0,.5) and ++(0,-.5) .. (0,2);
    \node at (0,-.25) {$\;$};
    \node at (0,2.25) {$\;$};
    \filldraw  (DOT1) circle (2.5pt);
    \filldraw  (DOT2) circle (2.5pt);
    \node at (-.35,1) {$\scs \hat{\ell}_i$};
    \node at (1.25,1) {$\scs \hat{\ell}_{i+1}$};
\end{tikzpicture}}
 \;\; = \;\;
\left\{
  \begin{array}{ll}
    \displaystyle{\sum_{\overset{A+B=}{ \hat{\ell}_{i+1} -\hat{\ell}_i -1}  } }
\hackcenter{\begin{tikzpicture}[scale=0.8]
    \draw[thick,  ] (0,0) .. controls (0,.75) and (.75,.75) .. (.75,1.5) node[pos=.25, shape=coordinate](DOT1){};
    \draw[thick, ] (.75,0) .. controls (.75,.75) and (0,.75) .. (0,1.5) node[pos=.25, shape=coordinate](DOT2){};
      \filldraw  (DOT1) circle (2.5pt);
      \filldraw  (DOT2) circle (2.5pt);
    \node at (-.45,.5) {$\scs \hat{\ell}_i + A$};
    \node at (1.3,.5) {$\scs \hat{\ell_i} + B$};
\end{tikzpicture}}
 & \hbox{if $\hat{\ell}_i < \hat{\ell}_{i+1}$;} \\
 \displaystyle{\sum_{\overset{A+B=}{\hat{\ell}_i -\hat{\ell}_{i+1} -1}  } }
\hackcenter{\begin{tikzpicture}[scale=0.8]
    \draw[thick,  ] (0,0) .. controls (0,.75) and (.75,.75) .. (.75,1.5) node[pos=.25, shape=coordinate](DOT1){};
    \draw[thick, ] (.75,0) .. controls (.75,.75) and (0,.75) .. (0,1.5) node[pos=.25, shape=coordinate](DOT2){};
      \filldraw  (DOT1) circle (2.5pt);
      \filldraw  (DOT2) circle (2.5pt);
    \node at (-.65,.5) {$\scs \hat{\ell}_{i+1} + A$};
    \node at (1.4,.5) {$\scs \hat{\ell}_{i+1} + B$};
\end{tikzpicture}}
   & \hbox{if $\hat{\ell}_i > \hat{\ell}_{i+1}$.}
  \end{array}
\right.
\]
where $\hat{\ell}_i:=i-\ell_i$ and $\hat{\ell}_{i+1}=i+1 - \ell_{i+1}$.  Combining this fact with \eqref{eq:elem-reduction} one can show that the right-hand side of \eqref{eq:crossing_through_onek} is equal to the sum \eqref{lem:crossing_goal} completing the proof.
\end{proof}

\begin{proposition} \label{prop-bimodhom(1-a)}
The assignments $\zeta_{(a,1)}^j$ and $\zeta_{(1,a)}^j$ define bimodule homomorphisms.
\end{proposition}

\begin{proof}
Here, we show that $\zeta_{(1,a)}^j$ is a bimodule morphism.  The proof for $\zeta_{(a,1)}^j$ is similar.

\noindent
$\bullet$ Red dot action:   First, we check the equation
\begin{eqnarray}
\label{red-action}&&
\hackcenter{
}
\end{eqnarray*}
By \cite[Proposition 2.4.1]{KLMS}, 
this summation equals
\[
\hackcenter{
%
}
\]
Therefore, by using properties of elementary symmetric functions for the red dots, the difference between the left-hand side and right-hand side of \eqref{red-action} is zero.

Next, we check the   red dot action on the thick red strand.
For $1\leq i\leq a$, we check the equation
\begin{align} \label{eq:xx1}
\hackcenter{
%
}
\end{align}
By sliding the elementary symmetric function ${\sf e}_i$, the left-hand side of this equation is equal to
\begin{eqnarray*}
&&
\sum_{\alpha=0}^{i}(-1)^{\alpha}
\left(
\hackcenter{
%
}
\right)
\end{eqnarray*}
By \eqref{red-action}, the above  is equal to
\begin{eqnarray*}
&&
\sum_{\alpha=0}^{i}(-1)^{\alpha}
\left(
\hackcenter{
%
}
\right)
\end{eqnarray*}
This reduces to the right-hand side of \eqref{eq:xx1}.

\noindent
$\bullet$ Thin red-black crossing:
We check the equation
\begin{eqnarray}
\label{RB-cross-action}&&
\sum_{\und{\ell} \in \Sq(k)}(-1)^{|\hat{\und{\ell}}|}\;\;
\hackcenter{
%
}.
\end{eqnarray}

By definition, we have
\[
\hackcenter{
%
}
\]
Using \eqref{rb-cross-C}, \eqref{cross-D}, and \eqref{RBr2-rel}, this diagram simplifies to
\begin{eqnarray*}
\hackcenter{
%
}
\end{eqnarray*}
By Lemma \ref{diff}, this simplifies to
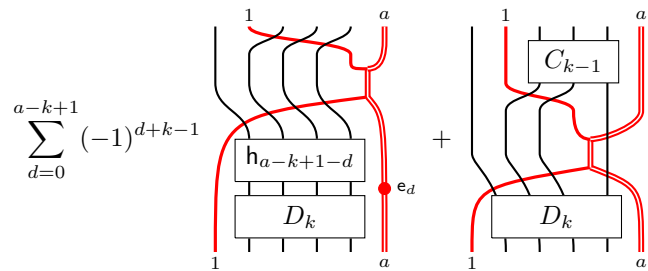
\[
\sum_{d=0}^{a-k+1}(-1)^{d+k-1}
\hackcenter{
%
}
\]

Therefore, by Lemma \ref{thm_Ea}, the first summation on the left-hand side of \eqref{RB-cross-action} is equal to
\[
\sum_{\und{\ell} \in \Sq(k)}(-1)^{|\hat{\und{\ell}}|}\;\;
\hackcenter{
%
}
\]

By Lemmas \ref{id-dots-C} and \ref{reduction-EX} and standard properties of complete symmetric functions, this is equal to
\begin{equation}\label{summand1}
\sum_{d_1+d_3=a-k+1}
\sum_{d=0}^{d_1}
(-1)^{d_3+k-1}
\hackcenter{
%
}
\end{equation}

On the other hand, we have
\begin{eqnarray*}
&&
\sum_{d_1+d_2+d_3\atop =a-k}
(-1)^{d_3+k}
\hackcenter{
%
}
\end{eqnarray*}
Therefore, we find that the left-hand side of \eqref{RB-cross-action} (the above summation and the summation \eqref{summand1}) is equal to the right-hand side.

\noindent
$\bullet$ Thick red-black crossing: the invariance under thick red line with a thin black strand can be established by a similar direct computation for which it is helpful to note the thick calculus identity  
\[
\hackcenter{
%
}
\]
together with \eqref{r3-2}.
\end{proof}

Composing the morphisms $\iota_{(1,s_{j+1}-1)}^{j+1}$ and $\zeta_{(s_j,1)}^j$ (resp. $\iota_{(s_{j}-1,1)}^{j}$ and $\zeta_{(1,s_{j+1})}^j$) we define the following bimodule homomorphisms
\begin{eqnarray*}
\zeta_{{\bf s},r}^j&:&
W({\bf s},n)
 \rightarrow    \mathsf{F}_j \mathsf{E}_j \mathsf{1}_{\mathbf{s}},
\\
\zeta_{{\bf s},l}^j&:&
W({\bf s},n)
    \rightarrow
     \mathsf{E}_j\mathsf{F}_j \mathsf{1}_{\mathbf{s}}.
\end{eqnarray*}
By definition, the degree of the morphism $\zeta_{{\bf s},r}^j$ is $s_j-s_{j+1}+1$ and the degree of the morphism $\zeta_{{\bf s},l}^j$ is $-s_j+s_{j+1}+1$.

As a diagrammatic presentation, the morphism $\zeta_{{\bf s},r}^j$ is determined by the following mapping.
\begin{align}
\zeta_{{\bf s},r}^j:
\hackcenter{
\begin{tikzpicture}[scale=.9]
\draw (0,0) -- (0,2)[thick,red, double, ];
\draw (1,0) -- (1,2)[very thick,double,];
\draw (2,0) -- (2,2)[thick,red, double];
        \node at (0,-.2) {$\scs s_j$};
        \node at (2,-.2) {$\scs s_{j+1}$};
        \node at (1,-.2) {$\scs k$};
        \node at (.5,2) {};
\end{tikzpicture}}
\stackrel{\iota_{(1,s_{j+1}-1)}^{j+1}}{\mapsto}
\hackcenter{
\begin{tikzpicture}[scale=.9]
\draw (0,0) -- (0,2)[thick,red, double, ];
\draw (1,0) -- (1,2)[very thick,double,];
        \node at (0,-.2) {$\scs s_j$};
        \node at (2.3,-.2) {$\scs s_{j+1}$};
        \node at (1,-.2) {$\scs k$};
        \node at (.5,2) {};
\draw[thick,red, double, ] (2.3,1).. controls ++(0,.3) and ++(0,-.3) ..(2,1.5) to (2,2);
\draw[very thick, red] (1.7,1).. controls ++(0,.3) and ++(0,-.3) ..(2,1.5);
\draw[thick,red, double, ] (2,0) to (2,.5).. controls ++(0,.3) and ++(0,-.3) ..(2.3,1) ;
\draw[very thick, red] (1.7,1).. controls ++(0,-.3) and ++(0,.3) ..(2,.5);
\node at (2.95,1) {$\scs s_{j+1}-1$};
\node at (1.5,1) {$\scs 1$};
\end{tikzpicture}} \hspace{2.8in}
\\\nonumber
\qquad \qquad \qquad
\stackrel{\zeta_{(s_j,1)}^j}{\mapsto}
(-1)^{k-1}
\hackcenter{
\begin{tikzpicture}[scale=.9]
\draw (0,0) -- (0,2)[thick,red, double, ];
\draw (1.15,.25) .. controls ++(0,.25) and ++(0,-.25) .. (-.2,.8) -- (-.2,1.2) .. controls ++(0,.25) and ++(0,-.25) .. (1.15,1.75)[thick];
\draw (1.15,0) -- (1.15,.25) .. controls ++(0,.25) and ++(0,-.25) .. (1.3,.8) -- (1.3,1.2) .. controls ++(0,.25) and ++(0,-.25) .. (1.15,1.75)--(1.15,2)[very thick,double,];
\draw (2.3,0) -- (2.3,2)[thick,red, double, ];
\draw [very thick, red] (0,.75) to (2.3,.5);
\draw [very thick, red] (2.3,1.5) to (0,1.25);
        \node at (0,-.2) {$\scs s_j$};
        \node at (2.6,-.2) {$\scs s_{j+1}$};
        \node at (0,2.2) {$\scs s_j$};
        \node at (2.6,2.2) {$\scs s_{j+1}$};
        \node at (2.95,1) {$\scs s_{j+1}-1$};
        \node at (.45,1) {$\scs s_j+1$};
        \node at (1.7,.4) {$\scs 1$};
        \node at (1.7,1.6) {$\scs 1$};
        \node at (-.4,1) {$\scs 1$};
        \node at (1.7,1) {$\scs k-1$};
        \node at (1.15,-.2) {$\scs k$};
\end{tikzpicture}}
+
\sum_{\overset{d_1+d_2+d_3}{=s_j-k}}
(-1)^{s_j+d_1+d_2}
\hackcenter{
\begin{tikzpicture}[scale=.9]
\draw (0,0) -- (0,2)[thick,red, double, ]node[pos=.1, shape=coordinate](DOT1){};
\draw (1.15,0) -- (1.15,2)[very thick,double,]node[pos=.5, shape=coordinate](DOT2){};
\draw (2.3,0) -- (2.3,2)[thick,red, double, ];
\draw [thick, red] (0,.5) to (2.3,.25);
\draw [thick, red] (2.3,1.75) to (0,1.5);
    \filldraw[red]  (DOT1) circle (2.5pt);
    \filldraw  (DOT2) circle (2.5pt);
    \filldraw[red]  (2.0125,1.71875) circle (2.5pt);
        \node at (0,-.2) {$\scs s_j$};
        \node at (2.6,-.2) {$\scs s_{j+1}$};
        \node at (0,2.2) {$\scs s_j$};
        \node at (2.6,2.2) {$\scs s_{j+1}$};
        \node at (2.95,1) {$\scs s_{j+1}-1$};
        \node at (-.45,1) {$\scs s_j+1$};
        \node at (.5,.65) {$\scs 1$};
        \node at (.5,1.7) {$\scs 1$};
        \node at (1.15,-.2) {$\scs k$};
        \node at (1.5,1) {$\scs {\sf h}_{d_1}$};
        \node at (-.4,0.2) {$\scs {\sf e}_{d_3}$};
        \node at (2.0125,1.95) {$\scs {d_2}$};
\end{tikzpicture}}
\end{align}

The morphism $\zeta_{{\bf s},l}^j$ is determined by the following mappings.
\begin{align}
\zeta_{{\bf s},l}^j:
\hackcenter{
\begin{tikzpicture}[scale=.9]
\draw (0,0) -- (0,2)[thick,red, double, ];
\draw (1,0) -- (1,2)[very thick,double,];
\draw (2,0) -- (2,2)[thick,red, double, ];
        \node at (0,-.2) {$\scs s_j$};
        \node at (2,-.2) {$\scs s_{j+1}$};
        \node at (1,-.2) {$\scs k$};
        \node at (.5,2) {};
\end{tikzpicture}}
\stackrel{\iota_{(s_{j}-1,1)}^{j}}{\mapsto}
\hackcenter{
\begin{tikzpicture}[scale=.9]
\draw[very thick,red] (.3,1).. controls ++(0,.3) and ++(0,-.3) ..(0,1.5);
\draw[thick,red, double] (-.3,1).. controls ++(0,.3) and ++(0,-.3) ..(0,1.5) to (0,2);
\draw[very thick,red] (0,.5).. controls ++(0,.3) and ++(0,-.3) ..(.3,1) ;
\draw[thick, red, double] (-.3,1).. controls ++(0,-.3) and ++(0,.3) ..(0,.5) to (0,0);
\draw (1,0) -- (1,2)[very thick,double,];
\draw (2,0) -- (2,2)[thick,red, double, ];
        \node at (0,-.2) {$\scs s_j$};
        \node at (2,-.2) {$\scs s_{j+1}$};
        \node at (1,-.2) {$\scs k$};
        \node at (.5,2) {};
\node at (-.7,1) {$\scs s_{j}-1$};
\node at (.5,1) {$\scs 1$};
\end{tikzpicture}} \hspace{3in}
\\\nonumber
\qquad \qquad \qquad
\stackrel{\zeta_{(1,s_{j+1})}^{j+1}}{\mapsto}
(-1)^{s_{j+1}}
\hackcenter{
\begin{tikzpicture}[scale=.9]
\draw (0,0) -- (0,2)[thick,red, double, ];
\draw (1.15,.25) .. controls ++(0,.25) and ++(0,-.25) .. (2.5,.8) -- (2.5,1.2) .. controls ++(0,.25) and ++(0,-.25) .. (1.15,1.75)[thick];
\draw (1.15,0) -- (1.15,.25) .. controls ++(0,.25) and ++(0,-.25) .. (1,.8) -- (1,1.2) .. controls ++(0,.25) and ++(0,-.25) .. (1.15,1.75)--(1.15,2)[very thick,double,];
\draw (2.3,0) -- (2.3,2)[thick,red, double, ];
\draw [very thick, red] (2.3,.75) to (0,.5);
\draw [very thick, red] (0,1.5) to (2.3,1.25);
        \node at (0,-.2) {$\scs s_j$};
        \node at (2.3,-.2) {$\scs s_{j+1}$};
        \node at (0,2.2) {$\scs s_j$};
        \node at (2.3,2.2) {$\scs s_{j+1}$};
        \node at (1.7,1) {$\scs s_{j+1}+1$};
        \node at (-.5,1) {$\scs s_j-1$};
        \node at (.6,.4) {$\scs 1$};
        \node at (.6,1.6) {$\scs 1$};
        \node at (2.7,1) {$\scs 1$};
        \node at (.6,1) {$\scs k-1$};
        \node at (1.15,-.2) {$\scs k$};
\end{tikzpicture}}
+
\sum_{\overset{d_1+d_2+d_3}{=s_{j+1}-k}}
(-1)^{s_{j+1}+d_3+k}
\hackcenter{
\begin{tikzpicture}[scale=.9]
\draw (0,0) -- (0,2)[thick,red, double, ];
\draw (1.15,0) -- (1.15,2)[very thick,double,]node[pos=.5, shape=coordinate](DOT2){};
\draw (2.3,0) -- (2.3,2)[thick,red, double, ]node[pos=.1, shape=coordinate](DOT1){};
\draw [very thick, red] (2.3,.5) to (0,.25);
\draw [very thick, red] (0,1.75) to (2.3,1.5);
    \filldraw[red]  (DOT1) circle (2.5pt);
    \filldraw  (DOT2) circle (2.5pt);
    \filldraw[red]  (.2875,1.71875) circle (2.5pt);
        \node at (0,-.2) {$\scs s_j$};
        \node at (2.6,-.2) {$\scs s_{j+1}$};
        \node at (0,2.2) {$\scs s_j$};
        \node at (2.6,2.2) {$\scs s_{j+1}$};
        \node at (2.95,1) {$\scs s_{j+1}+1$};
        \node at (-.45,1) {$\scs s_j-1$};
        \node at (1.5,.6) {$\scs 1$};
        \node at (1.5,1.75) {$\scs 1$};
        \node at (1.15,-.2) {$\scs k$};
        \node at (1.5,1) {$\scs {\sf h}_{d_1}$};
        \node at (2.7,0.2) {$\scs {\sf e}_{d_3}$};
        \node at (.35,1.95) {$\scs {d_2}$};
\end{tikzpicture}}
\end{align}

\subsection{Morphisms $\aleph_{i,j}$}
\allowdisplaybreaks
In this subsection we will prove that for $|i-j|=1$ there are bimodule homomorphisms
\begin{equation*}
\aleph_{i,j} \colon \mathsf{E}_i \mathsf{E}_j \mathsf{1}_{\mathbf{s}}
\rightarrow
\mathsf{E}_j \mathsf{E}_i \mathsf{1}_{\mathbf{s}}.
\end{equation*}
When $i$ and $j$ satisfy other conditions, we will also have maps
$\mathsf{E}_i \mathsf{E}_j \mathsf{1}_{\mathbf{s}}
\rightarrow
\mathsf{E}_j \mathsf{E}_i \mathsf{1}_{\mathbf{s}}$.
When $|i-j| \neq 1$, it is easier to prove these are bimodule homomorphisms, so we consider only the difficult cases in this subsection.

It is easier to define $\aleph_{i,j}$ when $j=i+1$.
\begin{equation}
\aleph_{i,i+1} \colon
\left(\hackcenter{
\begin{tikzpicture}[scale=.4]
	\draw [thick,red, double, ] (6,-2) to (6,2);
	\draw [thick,red, double, ] (2,-2) to (2,2);
	\draw [thick,red, double, ] (-2,-2) to (-2,2);
	\draw [very thick, red ] (2,.5) to (-2,1.5);
	\draw [very thick, red ] (6,-1.5) to (2,-.5);
\node at (-2,-2.55) {\tiny $s_i$};	
\node at (2,-2.55) {\tiny $s_{i+1}$};
\node at (6,-2.55) {\tiny $s_{i+2}$};
	\node at (-2,2.5) {\tiny $s_i+1$};
	\node at (2,2.5) {\tiny $s_{i+1}$};
		\node at (6,2.5) {\tiny $s_{i+2}-1$};
    \node at (-1,-2.5) {\tiny $k_1$};
\draw [very thick,double,] (-1,-2) to (-1,2);
\node at (5,-2.5) {\tiny $k_2$};
\draw [very thick,double,] (5,-2) to (5,2);
\end{tikzpicture}}
\mapsto
\hackcenter{
\begin{tikzpicture}[scale=.4]
	\draw [thick,red, double, ] (6,-2) to (6,2);
	\draw [thick,red, double, ] (2,-2) to (2,2);
	\draw [thick,red, double, ] (-2,-2) to (-2,2);
	\draw [very thick, red ] (6,.5) to (2,1.5);
	\draw [very thick, red ] (2,-1.5) to (-2,-.5);
\node at (-2,-2.55) {\tiny $s_i$};	
\node at (2,-2.55) {\tiny $s_{i+1}$};
\node at (6,-2.55) {\tiny $s_{i+2}$};
	\node at (-2,2.5) {\tiny $s_i+1$};
	\node at (2,2.5) {\tiny $s_{i+1}$};
		\node at (6,2.5) {\tiny $s_{i+2}-1$};
    \node at (-1,-2.5) {\tiny $k_1$};
\draw [very thick,double,] (-1,-2) to (-1,2);
\node at (5,-2.5) {\tiny $k_2$};
\draw [very thick,double,] (5,-2) to (5,2);
\end{tikzpicture}}\right)
\end{equation}
The other case is more involved since there is an extra set of generators given by a thick black strand going through the middle vertical red line.  The definition of $\aleph_{i+1,i}$ breaks down into three cases depending upon the thickness of this black strand.
\begin{align}
\label{i+1,icrossing}
&\left(\hackcenter{
\begin{tikzpicture}[scale=.4]
\draw [thick,red, double, ] (6,-2) to (6,2);
\draw [thick,red, double, ] (2,-2) to (2,2);
	\draw [thick,red, double, ] (-2,-2) to (-2,2);
	\draw [very thick, red ] (6,.5) to (2,1.5);
	\draw [very thick, red ] (2,-1.5) to (-2,-.5);
\node at (-2,-2.55) {\tiny $s_i$};	
\node at (2,-2.55) {\tiny $s_{i+1}$};
\node at (6,-2.55) {\tiny $s_{i+2}$};
	\node at (-2,2.5) {\tiny $s_i+1$};
	\node at (2,2.5) {\tiny $s_{i+1}$};
		\node at (6,2.5) {\tiny $s_{i+2}-1$};
    \node at (-1,-2.5) {\tiny $k_1$};
\draw [very thick,double,] (-1,-2) to (-1,2);
\node at (5,-2.5) {\tiny $k_2$};
\draw [very thick,double,] (5,-2) to (5,2);
\end{tikzpicture}}
\mapsto
\hackcenter{
\begin{tikzpicture}[scale=.4]
	\draw [thick,red, double, ] (6,-2) to (6,2);
	\draw [thick,red, double, ] (2,-2) to (2,2);
	\draw [thick,red, double, ] (-2,-2) to (-2,2);
	\draw [very thick, red ] (2,.5) to (-2,1.5);
\filldraw[red]  (3.8,-.95) circle (6pt);
\draw [very thick, red ] (6,-1.5) to (2,-.5);
\node at (-2,-2.55) {\tiny $s_i$};	
\node at (2,-2.55) {\tiny $s_{i+1}$};
\node at (6,-2.55) {\tiny $s_{i+2}$};
	\node at (-2,2.5) {\tiny $s_i+1$};
	\node at (2,2.5) {\tiny $s_{i+1}$};
		\node at (6,2.5) {\tiny $s_{i+2}-1$};
    \node at (-1,-2.5) {\tiny $k_1$};
\draw [very thick,double,] (-1,-2) to (-1,2);
\node at (4.7,-2.5) {\tiny $k_2$};
\draw [very thick,double,] (4.7,-2) to (4.7,2);
    \end{tikzpicture}}
	-
	\hackcenter{
    \begin{tikzpicture}[scale=.4]
    \draw [thick,red, double, ] (6,-2) to (6,2);
	\draw [thick,red, double, ] (2,-2) to (2,2);
	\draw [thick,red, double, ] (-2,-2) to (-2,2);
	\draw [very thick, red ] (2,.5) to (-2,1.5);
	\draw [very thick, red ] (6,-1.5) to (2,-.5);
	\filldraw[red]  (0,1) circle (6pt);
\node at (-2,-2.55) {\tiny $s_i$};	
\node at (2,-2.55) {\tiny $s_{i+1}$};
\node at (6,-2.55) {\tiny $s_{i+2}$};
	\node at (-2,2.5) {\tiny $s_i+1$};
	\node at (2,2.5) {\tiny $s_{i+1}$};
		\node at (6,2.5) {\tiny $s_{i+2}-1$};
    \node at (-1,-2.5) {\tiny $k_1$};
\draw [very thick,double,] (-1,-2) to (-1,2);
\node at (4.7,-2.5) {\tiny $k_2$};
\draw [very thick,double,] (4.7,-2) to (4.7,2);
    \end{tikzpicture}} \right)  \\
  &  \left(\hackcenter{
\begin{tikzpicture}[scale=.4]
	\draw [thick,red, double, ] (6,-2) to (6,2);
	\draw [thick,red, double, ] (2,-2) to (2,2);
	\draw [thick,red, double, ] (-2,-2) to (-2,2);
	\draw [very thick, red ] (6,.5) to (2,1.5);
	\draw [very thick, red ] (2,-1.5) to (-2,-.5);
\node at (-2,-2.55) {\tiny $s_i$};	
\node at (2,-2.55) {\tiny $s_{i+1}$};
\node at (6,-2.55) {\tiny $s_{i+2}$};
	\node at (-2,2.5) {\tiny $s_i+1$};
	\node at (2,2.5) {\tiny $s_{i+1}$};
		\node at (6,2.5) {\tiny $s_{i+2}-1$};
	\node at (0,-.65) {\tiny $1$};
	\node at (4,.65) {\tiny $1$};
    \node at (-1,-2.5) {\tiny $k_1$};
\draw [very thick,double,] (-1,-2) to (-1,2);
\node at (4.7,-2.5) {\tiny $k_2$};
\draw [very thick,double,] (4.7,-2) to (4.7,2);
\draw[very thick,] (3.3,-2) .. controls ++(0,1.6) and ++(0,-.8) .. (1,2);
\node at (3.3,-2.5) {\tiny $1$};
\end{tikzpicture}}
\mapsto
\hackcenter{
\begin{tikzpicture}[scale=.4]
	\draw [thick,red, double, ] (6,-2) to (6,2);
	\draw [thick,red, double, ] (2,-2) to (2,2);
	\draw [thick,red, double, ] (-2,-2) to (-2,2);
	\draw [very thick, red ] (2,.5) to (-2,1.5);
\draw [very thick, red ] (6,-1.5) to (2,-.5);
\node at (-2,-2.55) {\tiny $s_i$};	
\node at (2,-2.55) {\tiny $s_{i+1}$};
\node at (6,-2.55) {\tiny $s_{i+2}$};
	\node at (-2,2.5) {\tiny $s_i+1$};
	\node at (2,2.5) {\tiny $s_{i+1}$};
	\node at (6,2.5) {\tiny $s_{i+2}-1$};
	\node at (0,.65) {\tiny $1$};
		\node at (3,-.3) {\tiny $1$};
    \node at (-1,-2.5) {\tiny $k_1$};
\draw [very thick,double,] (-1,-2) to (-1,2);
\node at (4.75,-2.5) {\tiny $k_2$};
\draw [very thick,double,] (4.75,-2) to (4.75,2);
\draw[thick,] (3.3,-2) .. controls ++(0,.5) and ++(0,-1) .. (1,0) to  (1,2);
\node at (3.3,-2.5) {\tiny $1$};
    \end{tikzpicture}}
	-
	\hackcenter{
    \begin{tikzpicture}[scale=.4]
    \draw [thick,red, double, ] (6,-2) to (6,2);
	\draw [thick,red, double, ] (2,-2) to (2,2);
	\draw [thick,red, double, ] (-2,-2) to (-2,2);
	\draw [very thick, red ] (2,.5) to (-2,1.5);
	\draw [very thick, red ] (6,-1.5) to (2,-.5);
	\node at (-2,-2.55) {\tiny $s_i$};	
	\node at (2,-2.55) {\tiny $s_{i+1}$};
	\node at (6,-2.55) {\tiny $s_{i+2}$};
	\node at (-2,2.5) {\tiny $s_i+1$};
	\node at (2,2.5) {\tiny $s_{i+1}$};
	\node at (6,2.5) {\tiny $s_{i+2}-1$};
	\node at (1,.25) {\tiny $1$};
	\node at (4,-.65) {\tiny $1$};
	\node at (-1,-2.5) {\tiny $k_1$};
\draw [very thick,double,] (-1,-2) to (-1,2);
\node at (4.75,-2.5) {\tiny $k_2$};
\draw [very thick,double,] (4.75,-2) to (4.75,2);
\node at (3.3,-2.5) {\tiny $1$};
\draw[thick,] (3.3,-2) to (3.3,0) .. controls ++(0,1) and ++(0,-.5) .. (1,2);
    \end{tikzpicture}} \right) \\
&\left(\hackcenter{
\begin{tikzpicture}[scale=.4]
	\draw [thick,red, double, ] (6,-2) to (6,2);
	\draw [thick,red, double, ] (2,-2) to (2,2);
	\draw [thick,red, double, ] (-2,-2) to (-2,2);
	\draw [very thick, red ] (6,.5) to (2,1.5);
	\draw [very thick, red ] (2,-1.5) to (-2,-.5);
	\node at (-2,-2.55) {\tiny $s_i$};	
	\node at (2,-2.55) {\tiny $s_{i+1}$};
	\node at (6,-2.55) {\tiny $s_{i+2}$};
	\node at (-2,2.5) {\tiny $s_i+1$};
	\node at (2,2.5) {\tiny $s_{i+1}$};
	\node at (6,2.5) {\tiny $s_{i+2}-1$};
	\node at (0,-.65) {\tiny $1$};
	\node at (4,.65) {\tiny $1$};
	\node at (-1,-2.5) {\tiny $k_1$};
\draw [very thick,double,] (-1,-2) to (-1,2);
\node at (4.75,-2.5) {\tiny $k_3$};
\draw [very thick,double,] (4.75,-2) to (4.75,2);
\draw[very thick,double,] (3.3,-2) .. controls ++(0,1.6) and ++(0,-.8) .. (1,2);
\node at (3.3,-2.5) {\tiny $k_2$};
\end{tikzpicture}}
\mapsto
-
\hackcenter{
\begin{tikzpicture}[scale=.4]
	\draw [thick,red, double, ] (6,-2) to (6,2);
	\draw [thick,red, double, ] (2,-2) to (2,2);
	\draw [thick,red, double, ] (-2,-2) to (-2,2);
	\draw [very thick, red ] (2,.5) to (-2,1.5);
\draw [very thick, red ] (6,-1.5) to (2,-.5);
\node at (-2,-2.55) {\tiny $s_i$};	
\node at (2,-2.55) {\tiny $s_{i+1}$};
\node at (6,-2.55) {\tiny $s_{i+2}$};
	\node at (-2,2.5) {\tiny $s_i+1$};
	\node at (2,2.5) {\tiny $s_{i+1}$};
	\node at (6,2.5) {\tiny $s_{i+2}-1$};
	\node at (0,.65) {\tiny $1$};
	\node at (3,.75) {\tiny $1$};
	\node at (4.0,-.55) {\tiny $1$};
\node at (-1,-2.5) {\tiny $k_1$};
\draw [very thick,double,] (-1,-2) to (-1,2);
\node at (4.75,-2.5) {\tiny $k_3$};
\draw [very thick,double,] (4.75,-2) to (4.75,2);
\node at (3.3,-2.5) {\tiny $k_2$};
\draw[thick,] (3.3,-1.65) .. controls ++(.75,2.5) and ++(.5,-.5) .. (1,1.5);
\draw[thick,double,] (3.4,-2) .. controls ++(0,1) and ++(0,-1.5) ..(1,0) to  (1,2);
    \end{tikzpicture}} \right)
\end{align}
When $k=0$, \eqref{i+1,icrossing} is a bimodule homomorphism by Soergel calculus.
In order to check that all of the $\aleph_{i,j}$ are in fact bimodule homomorphisms, we must compare how $W({\bf s},n)$ acts on the top and bottom of the diagrams for these generators.

\noindent
$\bullet$ Black dot action: The compatibility of the actions of black dots with respect to these maps follows from the black dot sliding relation \eqref{blackdot}.

\noindent
$\bullet$ Red dot action:  The compatibility of the actions of red dots with respect to these maps follows from standard manipulations of elementary symmetric functions.

\noindent
$\bullet$ The action of red-black crossings:
We show the bimodule map structure of the $s_{i+1}$-red-black crossing action for the following generator.
We leave the other cases to the reader.  In what follows below we will often indicate a reindexing in a summation in a parenthesis under the new index.  Recall that

$$
\hackcenter{
}
$$

When $k=0$,  there is nothing to check.
%
When $k\geq 1$, acting by a crossing on the black strand and red strand labelled $s_{i+1}$ on the top boundary, we have
\begin{eqnarray}
\hackcenter{
%
}
\end{eqnarray}

Therefore, we show the following equalities.
\\
$\bullet$
For $k=1$, we show
\begin{eqnarray}
\label{bimod_str_k=1}&&
\hackcenter{
%
}
\right)
\end{eqnarray}

The left-hand side of this equality is equal to
\begin{eqnarray*}
&&
\sum_{A=1}^{s_{i+1}}
(-1)^{A}
\hackcenter{
%
}
\end{eqnarray*}

By sliding the dot labelled ${\sf e}_A$, this summation is equal to
\begin{eqnarray*}
&&
\sum_{A=1}^{s_{i+1}}
\sum_{B=0}^{A}
(-1)^{A+B}
\hackcenter{
%
}
\end{eqnarray*}

This is the same as the right-hand side of the equality \eqref{bimod_str_k=1}.

$\bullet$
For $k= 2$, we show
\begin{eqnarray}
\label{bimod_str_k=2}
&&
-\hackcenter{
%
}
\right)
\end{eqnarray}

Using relation \eqref{r3-2}, the left-hand side of this equality is equal to
\begin{eqnarray}
\label{bimod_str_k=2A}
-\hackcenter{
%
}
\end{eqnarray}

The first summation of \eqref{bimod_str_k=2A} is equal to

\begin{eqnarray*}
-
\sum_{A=0}^{s_{i+1}}
(-1)^{A}
\hackcenter{
%
}
\end{eqnarray*}

By sliding the dot labelled ${\sf e}_A$, the summation is

\begin{align}
\nonumber
&
-
\sum_{A+B+C\atop =s_{i+1}-1}
\sum_{D=0}^{A}
(-1)^{A+D}
\hackcenter{
%
}
\end{align}

Applying relation \eqref{RBr2-rel} to $A+1$ black dots of the second summation of \eqref{bimod_str_k=2A}, the second summation is equal to

\begin{align}
\nonumber
&
\sum_{A+B+C\atop=s_{i+1}-1}
(-1)^C
\left(
\hackcenter{
%
}
\end{align}

Therefore, adding the summation \eqref{bimod_str_k=2B} and the summation \eqref{bimod_str_k=2C}, we find that the summation \eqref{bimod_str_k=2A} is
\begin{eqnarray*}
\nonumber
&&
\sum_{A=0}^{s_{i+1}-1}
\sum_{C'=A\atop (\textcolor{blue}{C'=C+A})}^{s_{i+1}-1}
(-1)^{A+C'}
\hackcenter{
%
}
\end{eqnarray*}
Using relation \eqref{RBr2-rel}, we find that this summation is equal to the right-hand side of the summation~\eqref{bimod_str_k=2}.

$\bullet$
For $k\geq 3$, we show
\begin{eqnarray}
\label{bimod_str_k}
-
\hackcenter{
%
}
\end{eqnarray}
Using relation \eqref{r3-2}, the left-hand side of this equality is equal to
\begin{align}
&
\nonumber
-
\hackcenter{
%
}
\end{align}
Using \eqref{RBr2-rel} in the first term and sliding the $C$ dots in the second term up through the black crossing gives that the above is equal to
\begin{align}
\label{bimod_str_kA}
&
\sum_{A+C+D\atop =s_{i+1}-1}
(-1)^{A+1}
\hackcenter{
%
}
\end{align}
By ${\sf e}_A$ dot sliding, the first summation is %
\begin{eqnarray*}
\sum_{A+C+D\atop =s_{i+1}-1}
\sum_{E=0}^{A}
(-1)^{A+E+1}
\hackcenter{
%
}
\end{eqnarray*}
Applying relation \eqref{RBr2-rel} to the red dot on the left ladder of the second summation, this summation is equal to
\begin{align}
\label{bimod_str_kB}
&
\sum_{A+C+D\atop =s_{i+1}-1}
\sum_{E=0}^{A}
(-1)^{A+E+1}
\hackcenter{
%
}
\end{align}
%
The first and second summations simplify to
\begin{eqnarray}
&&
\label{bimod_str_kC}
\sum_{A=0}^{s_{i+1}-1}
\sum_{E=0}^{A}
(-1)^{A+E+1}
\hackcenter{
%
}
\end{eqnarray}

The third and fourth summations of \eqref{bimod_str_kB} are equal to
\begin{eqnarray*}
&&
\sum_{A=1}^{s_{i+1}-1}
\sum_{C+D\atop =s_{i+1}-1-A}
\sum_{E=0}^{A-1}
(-1)^{A+E}
\hackcenter{
%
}
\end{eqnarray*}

By relation \eqref{cross-D}, the second and third summations are zero and, applying Lemma \ref{diff1} to the first summation, it reduces to
\begin{eqnarray}
\label{bimod_str_kD}
\sum_{A'=0\atop (\textcolor{blue}{A'=A-1})}^{s_{i+1}-2}
\sum_{B+C+D=s_{i+1}-2-A'}
(-1)^{A'}
\hackcenter{
%
}
\end{eqnarray}
This summation cancels the second summation of \eqref{bimod_str_kA}.
Therefore, the left-hand side of \eqref{bimod_str_k} is equal to
\begin{eqnarray*}
\sum_{A=0}^{s_{i+1}-1}
\sum_{E=0}^{A}
(-1)^{A+E+1}
\hackcenter{
%
}
\end{eqnarray*}
This is the same as the right-hand side of the equality \eqref{bimod_str_k}.

\section{Categorical symmetric Howe duality}
\label{sectionconj2rep}

In this section we define a $2$-representation of the $2$-category $\cal{U}$ using ladder bimodules of $W({\bf s},n)$.  The proof that the assignments defined here give rise to a $2$-representation is checked in the next section where it is shown that all relations of the $2$-morphisms hold.

\subsection{The $2$-category $\Bim(n)$}
Here we define the target 2-category for our 2-representation of the 2-category $\cal{U}$.

\begin{definition}
Define a graded additive 2-category $\Bim(n)$ as the idempotent completion inside the 2-category of graded bimodules  with:
\begin{itemize}
  \item objects consisting of sequences ${\bf s}=(s_1,\ldots,s_m)$ such that each $s_i \in \Z_{\geq 0}$ and a symbol $\ast$.
  \item  The 1-morphisms are generated under tensor product of bimodules by identity bimodules $1_{\mathbf{s}}=W(\mathbf{s},n)$ and
\[
 \mathsf{E}_i^{(a)}1_{\mathbf{s}} \;\;   \;\; = \;\; \xy
(0,0)*{
\begin{tikzpicture}[scale=.4]
\node at (-5.5,0) { $\dots$};
\node at (5.5,0) { $\dots$};
    \draw [thick, red, double,, ] (-4,-2) to (-4,2);
    \draw [thick, red, double,, ] (4,-2) to (4,2);
	\draw [thick, red, double,, ] (2,-.5) to (2,2);
	\draw [thick, red, double,, ] (2,-.5) to (-2,.5);
	\draw [thick, red, double,, ] (-2,-2) to (-2,.5);
	\draw [thick, red, double,, ] (-2,.5) to (-2,2);
	\draw [thick, red, double,, ] (2,-2) to (2,-.5);
\node at (-4,-2.55) {\tiny $s_{i-1}$};	
\node at (-2,-2.55) {\tiny $s_i$};	
\node at (2,-2.55) {\tiny $s_{i+1}$};
\node at (4,-2.55) {\tiny $s_{i+2}$};
	\node at (-2,2.5) {\tiny $s_i+a$};
	\node at (2,2.5) {\tiny $s_{i+1}-a$};
	\node at (0,0.75) {\tiny $a$};
\node at (0,-1.5) { $\dots$};
    \draw [thick,] (1.5,-2) to (1.5,2);
    \draw [thick,] (-1.5,-2) to (-1.5,2);
\end{tikzpicture}
};
\endxy
\qquad   \mathsf{F}_i^{(a)}1_{\mathbf{s}} \;\;   \;\; = \;\; \xy
(0,0)*{
\begin{tikzpicture}[scale=.4]
\node at (-5.5,0) { $\dots$};
\node at (5.5,0) { $\dots$};
    \draw [thick, red, double,, ] (-4,-2) to (-4,2);
    \draw [thick, red, double,, ] (4,-2) to (4,2);
	\draw [thick, red, double,, ] (2,.5) to (2,2);
	\draw [thick, red, double,, ] (-2,-.5) to (2,.5);
	\draw [thick, red, double,, ] (-2,-2) to (-2,-.5);
	\draw [thick, red, double,, ] (-2,-.5) to (-2,2);
	\draw [thick, red, double,, ] (2,-2) to (2,.5);
\node at (-4,-2.55) {\tiny $s_{i-1}$};	
\node at (-2,-2.55) {\tiny $s_i$};	
\node at (2,-2.55) {\tiny $s_{i+1}$};
\node at (4,-2.55) {\tiny $s_{i+2}$};
	\node at (-2,2.5) {\tiny $s_i- a$};
	\node at (2,2.5) {\tiny $s_{i+1}+a$};
	\node at (0,0.75) {\tiny $a$};
\node at (0,-1.5) { $\dots$};
    \draw [thick,] (1.5,-2) to (1.5,2);
    \draw [thick,] (-1.5,-2) to (-1.5,2);
\end{tikzpicture}
};
\endxy
\]
together with their grading shifts. For the symbol $\ast$, the set of 1-morphisms is the empty set.
\[ \Hom(\ast.\ast)=\Hom(\ast.\mathbf{s})=\Hom(\mathbf{s},\ast)=\emptyset\]
Hence, an arbitrary 1-morphism is a summand of a direct sum of tensor products of these bimodules and their shifts.

\item The 2-morphisms of $\Bim(n)$ are the degree preserving bimodule homomorphisms between these bimodules.
\end{itemize}
\end{definition}

\subsection{The 2-representation}

There is a $2$-representation
\begin{align}
  \Psi \maps \Ucat \;\; &\to \;\;  \Bim(n) \\
  {\bf s} \;\; & \mapsto \left\{\begin{array}{ll}{\bf s}&\text{if } {\bf s}\in \Z_{\geq 0}^m \\ \ast &\text{if } {\bf s}\not\in \Z_{\geq 0}^m\end{array}\right.
\end{align}
By definition we have $\Hom(\ast,\ast)=\Hom(\ast,\mathbf{s})=\Hom(\mathbf{s},\ast)=0$.
Therefore, when ${\bf s}\not\in \Z^{m}_{\geq 0}$, the $1$-morphisms and $2$-morphisms get mapped to zero.
When ${\bf s}\in \Z^{m}_{\geq 0}$, we define
\begin{align}
  \mc{E}_i^{}1_{\mathbf{s}} \;\; &\mapsto \;\; \;\;  \mathsf{E}_i^{}1_{\mathbf{s}} \nn \\
   \mc{F}_i^{}1_{\mathbf{s}} \;\; &\mapsto \;\; \;\; \mathsf{F}_i^{}1_{\mathbf{s}} \nn
\end{align}
with 2-morphisms defined by the following assignments.
\\
\begin{eqnarray*}
\Psi\left(\;\; \hackcenter{
}
\right)\\
\end{eqnarray*}
The formula for $\partial(x_1^x x_2^y)$ above is given by
$\displaystyle\partial (x_1^x x_2^y)=\frac{x_1^x x_2^y-x_2^x x_1^y}{x_2-x_1}$.

The following fact could sometimes be used to simplify a calculation.
\begin{remark}
\label{thinkthickswapremark}
While the map associated to the crossing with the bottom boundary points labelled $i+1$ and $i$ seems arbitrary, note that
for $ s_i \geq 1$ and $ k>1$ there is an equality
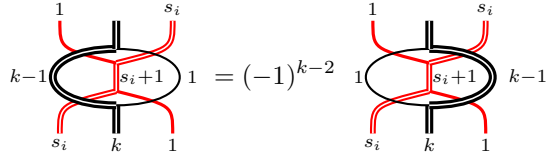
\begin{equation*}
\hackcenter{
%
}
\end{equation*}
\end{remark}

\subsection{Formulas implied by definition of $\Psi$}
In this section we compute the bimodule homomorphisms associated to non-generating 2-morphisms in $\cal{U}$.  These formulas will be useful for verifying the defining relations of $\cal{U}$ in Section~\ref{sectionrelationsproved}.

\subsubsection{Sideways crossings (same labelling)}
\label{subsectionsidewayssame}
The crossing formulas above along with cup and cap maps imply the following
formulas.

\begin{lemma} The $ii$-sideways crossings are given by the following formulas.
\begin{align}
&\Psi \left( \; \; \hackcenter{
} \nn
\end{align}
\end{lemma}

\begin{proof}
These follow by direct computations using the definitions from the previous subsection.
\end{proof}

\subsubsection{Sideways crossings (different labelling)}

\begin{lemma}
\label{Sidecrossingform1}
For $|i-j|>1$ we have the following assignments by $\Psi$ for sideways crossings:
\begin{equation*}
\hackcenter{%
}
\right)
\end{equation*}

\end{lemma}

\begin{proof}
These follow from the definitions.  The proof of Proposition \ref{isotopicprop} is used in this calculation.
\end{proof}

\begin{lemma}
\label{Sidecrossingform2}
For $|i-j|=1$ we have the following assignments for sideways crossings:

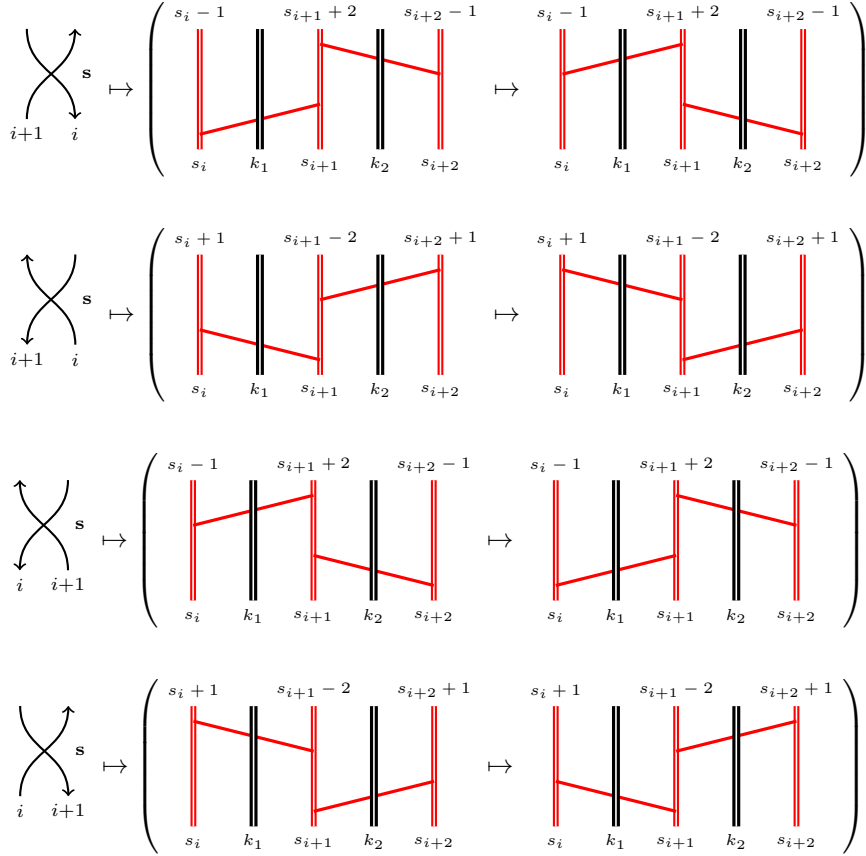
\begin{equation*}
\hackcenter{%
}
\right)
\end{equation*}
\end{lemma}

\begin{proof}
These are lengthy yet straightforward calculations using the definitions of cup, cap, and crossing bimodule homomorphisms.
\end{proof}

\subsubsection{Counter-clockwise bubble}
\label{sectioncounterclockwise}

It is convenient to compute the bimodule homomorphism for dotted bubbles.  For notational convenience in this section we assume that $s_i=a$ and $s_{i+1}=b$, so that $\bar{s}_i=a-b$.
\begin{align} \label{eq:cc-start}
& \Psi \left( \;\; \hackcenter{ %
}
  \nn
\end{align}

We now proceed to simplify ~\eqref{eq:cc-start}.  First assume $k>a$.  Then
$k-1-\alpha \geq a - \alpha$ and the second summation above vanishes and we get ~\eqref{eq:cc-start} is equal to
\begin{equation*}
\sum_{\beta=0}^{a}
\sum_{\alpha=0}^{\beta}
\sum_{\gamma=0}^{\beta-\alpha}
(-1)^{a+b+\beta+\gamma+1}
\hackcenter{
%
}
\end{equation*}
where the second expression follows from a relation between elementary and complete symmetric functions.

Now assume $k \leq a$.  Then we may write ~\eqref{eq:cc-start} as
\begin{align}
&\sum_{\beta=0}^{k-1}
\sum_{\alpha=0}^{\beta}
\sum_{\gamma=0}^{\beta-\alpha}
(-1)^{a+b+\beta+\gamma+1}
\hackcenter{
%
} \label{counterbub3}
 \end{align}

It is easy to see that ~\eqref{counterbub1} is equal to
\begin{equation*}
 \sum_{\alpha=0}^{k-1}
(-1)^{a+b+\alpha+1}
\hackcenter{
%
}
\end{equation*}

After replacing $\beta$ with $-\beta+a$ in ~\eqref{counterbub2} we get that
~\eqref{counterbub2} $+$ ~\eqref{counterbub3} is equal to

\begin{align}
&\sum_{\alpha=0}^{k-1}
\sum_{\beta=0}^{a-k}
\sum_{\gamma=0}^{k-1-\alpha}
(-1)^{b+\beta+\gamma+1}
\hackcenter{
%
} \label{counterbub5}
 \end{align}
It is straightforward to check that ~\eqref{counterbub4} $+$ ~\eqref{counterbub5} is equal to
\begin{align*}
&\sum_{\gamma=0}^{k}
\sum_{\beta=k}^{a}
\sum_{\alpha=0}^{\beta-\gamma}
(-1)^{a+b+\beta+\gamma+1}
\hackcenter{
%
} \\
\end{align*}
Now using a complete-elementary symmetric function relation we get that
~\eqref{counterbub4} $+$ ~\eqref{counterbub5} is equal to
\begin{equation*}
 \sum_{\beta=k}^{a}
(-1)^{a+b+\beta+1}
\hackcenter{
%
}
\end{equation*}
Thus ~\eqref{counterbub1}$+$~\eqref{counterbub2}$+$~\eqref{counterbub3} is equal to
\begin{equation*}
 \sum_{\beta=0}^{a}
(-1)^{a+b+\beta+1}
\hackcenter{
%
}
\end{equation*}

Therefore, we have shown
\begin{align}
\label{counterformula}
 \Psi \left( \;\; \hackcenter{ %
}
\end{align}

\subsubsection{Clockwise bubble}
\label{sectionclockwise}
Just as in Section ~\ref{sectioncounterclockwise} we calculate that
\begin{align}
\label{clockwiseformula}
 \Psi \left( \;\; \hackcenter{ %
}
\end{align}

\section{Proof of categorical action}
\label{sectionrelationsproved}
We now verify the defining relations of the $2$-morphisms in $\cal{U}$.

\subsection{Isotopic relations and cyclicity}

\begin{proposition}
\label{isotopicprop}
We have the following equalities of bimodule homomorphisms.
\begin{equation} \label{biadj1}
\Psi \left(
    \hackcenter{%
}
\right)
\end{equation}

\end{proposition}

\begin{proof}
We prove only the first equality in ~\eqref{biadj1}.  The other three equalities are similar.

The first homomorphism in ~\eqref{biadj1} is a composition of maps
$\mathsf{E}_i \rightarrow \mathsf{E}_i\mathsf{F}_i\mathsf{E}_i \rightarrow \mathsf{E}_i$ beginning with

\begin{equation}
\label{isotopyeq1}
\hackcenter{
%
}
\end{equation}
After moving the upper black fork above the top thin red line, ~\eqref{isotopyeq1} gets mapped to
\begin{equation}
\label{isotopyeq2}
(-1)^{k-1}
\hackcenter{
%
}
\end{equation}
The relation for a black-red double crossing \eqref{RBr2-rel} gives that ~\eqref{isotopyeq2} is equal to
\begin{equation}
\label{isotopyeq3}
\sum_{f=0}^{k-1} (-1)^{k-1+f}
\hackcenter{
%
}
\end{equation}
Applying a cap map to ~\eqref{isotopyeq3} yields
\begin{equation}
\label{isotopyeq4}
\sum_{f=0}^{k-1} (-1)^{k-1+f}
\hackcenter{
%
}
\end{equation}
Moving the left red fork past the thin black strand and then using a black-red double crossing relation \eqref{RBr2-rel} gives
\begin{equation}
\label{isotopyeq5}
\sum_{f=0}^{k-1}
\sum_{g=0}^{s_i}
(-1)^{k-1+f+g+s_i}
\hackcenter{
%
}
\end{equation}

Using thick calculus relations, note that the first term in ~\eqref{isotopyeq5} is non-zero only if $k \geq s_i +1$ and $f=g=s_i$ in which case it is equal to
\begin{equation} \label{eq:xx2}
\hackcenter{
%
}
\end{equation}
Similarly, the second term in ~\eqref{isotopyeq5} is non-zero only if $s_i \geq k$, $f=k$, $d_1=d_3=0$, and $d_2=s_i-k$
in which case it is also equal to \eqref{eq:xx2}.
Thus the composition of maps given by the first diagram in ~\eqref{biadj1} is the identity proving the first equality in ~\eqref{biadj1}.
\end{proof}

\begin{proposition} The assignment $\Psi$ preserves the dot cyclicity relation.
\begin{equation*} \label{eq_cyclic_dot-cyc}
\Psi \left(
\hackcenter{%
}
    \right)
\end{equation*}
\end{proposition}

\begin{proof}
This follows from the red dot action invariance in the proof of Proposition ~\ref{prop-bimodhom(1-a)}.
\end{proof}

\begin{proposition}
\label{crossingcyclicity}
There is an equality of bimodule homomorphisms
\begin{equation*}
\Psi \left(
\hackcenter{
%
}
  \right)
\end{equation*}
\end{proposition}

\begin{proof}
The case $|i-j|>1$ follows easily from the definitions.

When $i=j$, we will show that the middle homomorphism in the statement of the proposition is the map:
\begin{equation}
\label{downiimap}
\hackcenter{
%
}
\end{equation}

We first check that
\begin{equation}
\label{downdownmap}
  \Psi \left(
\hackcenter{%
}
  \right)
\end{equation}
annihilates the generator
\begin{equation}
\label{delta=0downdown}
\hackcenter{
%
}
\end{equation}
The nested cups in ~\eqref{downdownmap} send the generator ~\eqref{delta=0downdown} to
\begin{align}
\label{action1a}
&-\hackcenter{
%
} \nonumber
\end{align}
Applying the upward crossing map followed by the nested caps to all four terms above gives zero for degree reasons.
Note that applying ~\eqref{downdownmap} to a general generator
\begin{equation}
\label{deltageneraldowndown}
\hackcenter{
%
}
\end{equation}
is the same thing as applying
\begin{equation}
\label{downdownmapwithdot}
  \Psi \left(
\hackcenter{%
}
  \right)
\end{equation}
to ~\eqref{delta=0downdown}.
Using ~\eqref{eq_cyclic_dot-cyc}, nilHecke relations (to be proved independently later), and the fact that applying ~\eqref{downdownmap} to ~\eqref{delta=0downdown} is zero yield ~\eqref{downiimap}.
Similarly, the third homomorphism in the statement of the proposition acts on a generator as prescribed in ~\eqref{downiimap}.

When $j=i-1$ we get that the last two morphisms in the proposition map
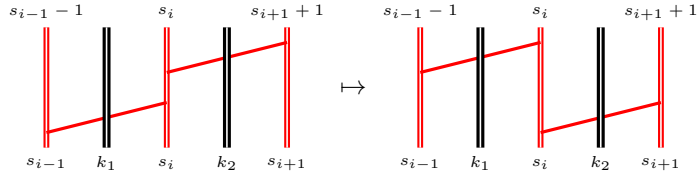
\begin{equation*}
\hackcenter{
%
}
\end{equation*}

When $j=i+1$ we get that the last two morphisms in the proposition map
\begin{equation*}
\hackcenter{
%
}
\end{equation*}

\end{proof}

\subsection{KLR relations}

\begin{proposition}
When the strands are labelled by adjacent nodes $|i-j|=1$, we have
\begin{align*}
\Psi \left(
 \;\; \hackcenter{
%
}
 \right).
\end{align*}
\end{proposition}

\begin{proof}
The two parts of the proposition are similar so we only sketch the second identity.
The generators for the bimodule
$ \mathsf{E}^{}_{i+1}  \mathsf{E}^{}_i \mathsf{1}_{\mathbf{s}} $
are given in ~\eqref{R2gen1}.
\begin{equation}
\label{R2gen1}
\hackcenter{
%
}
\end{equation}
There are three cases to consider: $k_2=0$, $k_2=1$, and $k_2 > 1$.

For $k_2=0$, the image of one of the crossing maps produces a sum of dots on the  red step strands labelled by $1$ along with a change in the relative heights of those red strands.  The other crossing map simply changes the relative heights of the red strands labelled by $1$.

The cases $k_2>0$ are verified in a straightforward manner from the definitions using Lemma ~\ref{RthickBr2-lemma}.
For $k_2>1$, one must use \cite[Proposition 2.5.3]{KLMS}
to simplify the diagrams.
\end{proof}

\begin{proposition}
If $|i-j|>1$ then

\begin{equation*}
\Psi \left( \;\;\hackcenter{
\begin{tikzpicture}[scale=0.8]
    \draw[thick, ->] (0,0) .. controls ++(0,.5) and ++(0,-.4) .. (.75,.8) .. controls ++(0,.4) and ++(0,-.5) .. (0,1.6);
    \draw[thick, ->] (.75,0) .. controls ++(0,.5) and ++(0,-.4) .. (0,.8) .. controls ++(0,.4) and ++(0,-.5) .. (.75,1.6);
    \node at (1.1,1.25) { ${\bf s}$};
    \node at (-.2,.1) {\tiny $i$};
    \node at (.95,.1) {\tiny $j$};
\end{tikzpicture}} \right)
  \quad = \quad
  \Psi \left( \;\; \hackcenter{
\begin{tikzpicture}[scale=0.8]
    \draw[thick, ->] (0,0) to (0,1.6);
    \draw[thick, ->] (.75,0) to (.75,1.6);
    \node at (1.1,1.25) { ${\bf s}$};
    \node at (-.2,.1) {\tiny $i$};
    \node at (.95,.1) {\tiny $j$};
\end{tikzpicture}} \right).
\end{equation*}
\end{proposition}

\begin{proof}
This follows immediately from the definitions of the bimodule homomorphisms since each one just corresponds to an isotopy of diagrams of elements in the bimodule.
\end{proof}

\begin{proposition}
\label{nilheckeactionE^a}
There is an action of the nilHecke algebra $\nh_a$ on the bimodule $\mathsf{E}^{a}_i \mathsf{1}_{\mathbf{s}} $.  That is, we have equalities of the following bimodule homomorphisms
\begin{equation*}
\Psi \left( \;\; \hackcenter{
\begin{tikzpicture}[scale=0.8]
    \draw[thick, ->] (0,0) .. controls ++(0,.5) and ++(0,-.4) .. (.75,.8) .. controls ++(0,.4) and ++(0,-.5) .. (0,1.6);
    \draw[thick, ->] (.75,0) .. controls ++(0,.5) and ++(0,-.4) .. (0,.8) .. controls ++(0,.4) and ++(0,-.5) .. (.75,1.6);
    \node at (1.1,1.25) { ${\bf s}$};
    \node at (-.2,.1) {\tiny $i$};
    \node at (.95,.1) {\tiny $i$};
\end{tikzpicture}} \right)
  \quad = \quad
0
\end{equation*}

\begin{equation*}
\Psi \left(
\hackcenter{\begin{tikzpicture}[scale=0.8]
    \draw[thick, ->] (0,0) .. controls ++(0,.55) and ++(0,-.5) .. (.75,1)
        node[pos=.25, shape=coordinate](DOT){};
    \draw[thick, ->] (.75,0) .. controls ++(0,.5) and ++(0,-.5) .. (0,1);
    \filldraw  (DOT) circle (2.5pt);
    \node at (-.2,.15) {\tiny $i$};
    \node at (.95,.15) {\tiny $i$};
\node at (1.1, .75) { ${\bf s}$};
\end{tikzpicture}}
\quad-\quad
\hackcenter{\begin{tikzpicture}[scale=0.8]
    \draw[thick, ->] (0,0) .. controls ++(0,.55) and ++(0,-.5) .. (.75,1)
        node[pos=.75, shape=coordinate](DOT){};
    \draw[thick, ->] (.75,0) .. controls ++(0,.5) and ++(0,-.5) .. (0,1);
    \filldraw  (DOT) circle (2.5pt);
    \node at (-.2,.15) {\tiny $i$};
    \node at (.95,.15) {\tiny $i$};
\node at (1.1, .75) { ${\bf s}$};
\end{tikzpicture}} \right)
\;\; = \Psi \left(\hackcenter{\begin{tikzpicture}[scale=0.8]
    \draw[thick, ->] (0,0) .. controls ++(0,.55) and ++(0,-.5) .. (.75,1);
    \draw[thick, ->] (.75,0) .. controls ++(0,.5) and ++(0,-.5) .. (0,1) node[pos=.75, shape=coordinate](DOT){};
    \filldraw  (DOT) circle (2.5pt);
    \node at (-.2,.15) {\tiny $i$};
    \node at (.95,.15) {\tiny $i$};
\node at (1.1, .75) { ${\bf s}$};
\end{tikzpicture}}
\quad-\quad
\hackcenter{\begin{tikzpicture}[scale=0.8]
    \draw[thick, ->] (0,0) .. controls ++(0,.55) and ++(0,-.5) .. (.75,1);
    \draw[thick, ->] (.75,0) .. controls ++(0,.5) and ++(0,-.5) .. (0,1) node[pos=.25, shape=coordinate](DOT){};
    \filldraw  (DOT) circle (2.5pt);
    \node at (-.2,.15) {\tiny $i$};
    \node at (.95,.15) {\tiny $i$};
\node at (1.1, .75) { ${\bf s}$};
\end{tikzpicture}} \right)
  \;\; = \;\;
 \Psi  \left(  \hackcenter{\begin{tikzpicture}[scale=0.8]
    \draw[thick, ->] (0,0) to  (0,1);
    \draw[thick, ->] (.75,0)to (.75,1) ;
    \node at (-.2,.15) {\tiny $i$};
    \node at (.95,.15) {\tiny $i$};
\node at (1.1, .75) { ${\bf s}$};
\end{tikzpicture}} \right)
\end{equation*}

\begin{equation*}
\Psi \left(
 \hackcenter{\begin{tikzpicture}[scale=0.8]
    \draw[thick, ->] (0,0) .. controls ++(0,1) and ++(0,-1) .. (1.5,2);
    \draw[thick, ] (.75,0) .. controls ++(0,.5) and ++(0,-.5) .. (0,1);
    \draw[thick, ->] (0,1) .. controls ++(0,.5) and ++(0,-.5) .. (0.75,2);
    \draw[thick, ->] (1.5,0) .. controls ++(0,1) and ++(0,-1) .. (0,2);
    \node at (-.2,.15) {\tiny $i$};
    \node at (.95,.15) {\tiny $i$};
    \node at (1.75,.15) {\tiny $i$};
\node at (1.9, 1.25) { ${\bf s}$};
\end{tikzpicture}} \right)
 \;\; =\;\;
 \Psi \left(
\hackcenter{\begin{tikzpicture}[scale=0.8]
    \draw[thick, ->] (0,0) .. controls ++(0,1) and ++(0,-1) .. (1.5,2);
    \draw[thick, ] (.75,0) .. controls ++(0,.5) and ++(0,-.5) .. (1.5,1);
    \draw[thick, ->] (1.5,1) .. controls ++(0,.5) and ++(0,-.5) .. (0.75,2);
    \draw[thick, ->] (1.5,0) .. controls ++(0,1) and ++(0,-1) .. (0,2);
    \node at (-.2,.15) {\tiny $i$};
    \node at (.95,.15) {\tiny $i$};
    \node at (1.75,.15) {\tiny $i$};
\node at (1.9, 1.25) { ${\bf s}$};
\end{tikzpicture}}\right)
\end{equation*}
along with identities arising from $\Psi$ for far away strands.
\end{proposition}

\begin{proof}
Recall that the bimodule $\mathsf{E}^{a}_i \mathsf{1}_{\mathbf{s}} $ is spanned by diagrams as in ~\eqref{bimodE^ainproof} with black strands in the background and all strands may carry dots.

\begin{equation}
\label{bimodE^ainproof}
\hackcenter{
\begin{tikzpicture}[scale=.4]
	\draw [thick,red, double, ] (2,-4) to (2,2);
	\draw [thick,red, double, ] (-2,-4) to (-2,2);
	\draw[very thick, red ] (2,-0) to (-2,1)  node[pos=.25, shape=coordinate](DOT){};
	\draw[very thick, red ] (2,-3) to (-2,-2) node[pos=.75, shape=coordinate](DOT2){};
\node at (-2,-4.55) {\tiny $s_i$};	
\node at (2,-4.55) {\tiny $s_{i+1}$};
	\node at (-2,2.5) {\tiny $s_i+a$};
	 \node[rotate=90 ] at (0,-1) { $\dots$};
	\node at (2,2.5) {\tiny $s_{i+1}-a$};
\end{tikzpicture}}
\end{equation}

A nilHecke dot on the $i$-th strand starting from the right corresponds to a dot on the $i$-th red strand labelled by $1$ starting from the bottom.
A nilHecke crossing on the $i$-th and $(i+1)$-st strands starting from the right corresponds to the divided difference operator applied to dots on the $i$-th and $(i+1)$-st red strands starting from the bottom.
It is then immediate that the nilHecke relations are satisfied.

\end{proof}

\begin{proposition}
For $i \neq j$ the dot sliding relations
\begin{equation*}
\Psi \left(
\hackcenter{\begin{tikzpicture}[scale=0.8]
    \draw[thick, ->] (0,0) .. controls ++(0,.55) and ++(0,-.5) .. (.75,1)
        node[pos=.75, shape=coordinate](DOT){};
    \draw[thick, ->] (.75,0) .. controls ++(0,.5) and ++(0,-.5) .. (0,1);
    \filldraw  (DOT) circle (2.5pt);
    \node at (-.2,.15) {\tiny $i$};
    \node at (.95,.15) {\tiny $j$};
\node at (1.1, .75) { ${\bf s}$};
\end{tikzpicture}}\right)
 \;\; =
\Psi \left(\hackcenter{\begin{tikzpicture}[scale=0.8]
    \draw[thick, ->] (0,0) .. controls ++(0,.55) and ++(0,-.5) .. (.75,1)
        node[pos=.25, shape=coordinate](DOT){};
    \draw[thick, ->] (.75,0) .. controls ++(0,.5) and ++(0,-.5) .. (0,1);
    \filldraw  (DOT) circle (2.5pt);
    \node at (-.2,.15) {\tiny $i$};
    \node at (.95,.15) {\tiny $j$};
\node at (1.1, .75) { ${\bf s}$};
\end{tikzpicture}} \right)
\qquad \Psi \left( \hackcenter{\begin{tikzpicture}[scale=0.8]
    \draw[thick, ->] (0,0) .. controls ++(0,.55) and ++(0,-.5) .. (.75,1);
    \draw[thick, ->] (.75,0) .. controls ++(0,.5) and ++(0,-.5) .. (0,1) node[pos=.75, shape=coordinate](DOT){};
    \filldraw  (DOT) circle (2.5pt);
    \node at (-.2,.15) {\tiny $i$};
    \node at (.95,.15) {\tiny $j$};
\node at (1.1, .75) { ${\bf s}$};
\end{tikzpicture}} \right)
\;\;  =
\Psi \left(  \hackcenter{\begin{tikzpicture}[scale=0.8]
    \draw[thick, ->] (0,0) .. controls ++(0,.55) and ++(0,-.5) .. (.75,1);
    \draw[thick, ->] (.75,0) .. controls ++(0,.5) and ++(0,-.5) .. (0,1) node[pos=.25, shape=coordinate](DOT){};
    \filldraw  (DOT) circle (2.5pt);
    \node at (-.2,.15) {\tiny $i$};
    \node at (.95,.15) {\tiny $j$};
\node at (1.1, .75) { ${\bf s}$};
\end{tikzpicture}} \right)
\end{equation*}
hold.
\end{proposition}

\begin{proof}
This follows easily from the definition of $\Psi$.
\end{proof}

\begin{proposition}
Unless $i = k$ and $|i-j|=1$, the relation
\begin{equation*}
\Psi \left(
 \hackcenter{\begin{tikzpicture}[scale=0.8]
    \draw[thick, ->] (0,0) .. controls ++(0,1) and ++(0,-1) .. (1.5,2);
    \draw[thick, ] (.75,0) .. controls ++(0,.5) and ++(0,-.5) .. (0,1);
    \draw[thick, ->] (0,1) .. controls ++(0,.5) and ++(0,-.5) .. (0.75,2);
    \draw[thick, ->] (1.5,0) .. controls ++(0,1) and ++(0,-1) .. (0,2);
    \node at (-.2,.15) {\tiny $i$};
    \node at (.95,.15) {\tiny $j$};
    \node at (1.75,.15) {\tiny $k$};
\node at (1.9, 1.25) { ${\bf s}$};
\end{tikzpicture}} \right)
 \;\; =\;\;
 \Psi \left(
\hackcenter{\begin{tikzpicture}[scale=0.8]
    \draw[thick, ->] (0,0) .. controls ++(0,1) and ++(0,-1) .. (1.5,2);
    \draw[thick, ] (.75,0) .. controls ++(0,.5) and ++(0,-.5) .. (1.5,1);
    \draw[thick, ->] (1.5,1) .. controls ++(0,.5) and ++(0,-.5) .. (0.75,2);
    \draw[thick, ->] (1.5,0) .. controls ++(0,1) and ++(0,-1) .. (0,2);
    \node at (-.2,.15) {\tiny $i$};
    \node at (.95,.15) {\tiny $j$};
    \node at (1.75,.15) {\tiny $k$};
\node at (1.9, 1.25) { ${\bf s}$};
\end{tikzpicture}}\right)
\end{equation*}
holds. Otherwise, $|i-j|=1$ and
\begin{equation*}
\Psi \left(
 \hackcenter{\begin{tikzpicture}[scale=0.8]
    \draw[thick, ->] (0,0) .. controls ++(0,1) and ++(0,-1) .. (1.5,2);
    \draw[thick, ] (.65,0) .. controls ++(0,.5) and ++(0,-.5) .. (0,1);
    \draw[thick, ->] (0,1) .. controls ++(0,.5) and ++(0,-.5) .. (0.75,2);
    \draw[thick, ->] (1.5,0) .. controls ++(0,1) and ++(0,-1) .. (0,2);
    \node at (-.2,.15) {\tiny $i$};
    \node at (1.05,.15) {\tiny $i\pm 1$};
    \node at (1.75,.15) {\tiny $i$};
\node at (1.9, 1.25) { ${\bf s}$};
\end{tikzpicture}}
 \;\; - \;\;
\hackcenter{\begin{tikzpicture}[scale=0.8]
    \draw[thick, ->] (0,0) .. controls ++(0,1) and ++(0,-1) .. (1.5,2);
    \draw[thick, ] (.65,0) .. controls ++(0,.65) and ++(0,-.5) .. (1.5,1);
    \draw[thick, ->] (1.5,1) .. controls ++(0,.5) and ++(0,-.5) .. (0.75,2);
    \draw[thick, ->] (1.5,0) .. controls ++(0,1) and ++(0,-1) .. (0,2);
    \node at (-.2,.15) {\tiny $i$};
    \node at (1.05,.15) {\tiny $i\pm 1$};
    \node at (1.75,.15) {\tiny $i$};
\node at (1.9, 1.25) { ${\bf s}$};
\end{tikzpicture}}\right)
\;\; = \;\;
\mp
\Psi \left(
\hackcenter{\begin{tikzpicture}[scale=0.8]
    \draw[thick, ->] (0,0) to (0,2);
    \draw[thick, ->] (.65,0) to (.65,2);
    \draw[thick, ->] (1.5,0) to (1.5,2);
    \node at (-.2,.15) {\tiny $i$};
    \node at (1.05,.15) {\tiny $i\pm 1$};
    \node at (1.75,.15) {\tiny $i$};
\node at (1.9, 1.25) { ${\bf s}$};
\end{tikzpicture}}
\right)
\end{equation*}
\end{proposition}

\begin{proof}
We provide most of the details for the case $j=i+1$ and $k=i$.
The bimodule is generated by
\begin{equation}
\label{ii+1irel1}
\hackcenter{
\begin{tikzpicture}[scale=.4]
	\draw [thick,red, double, ] (6,-2) to (6,4);
	\draw [thick,red, double, ] (2,-2) to (2,4);
	\draw [thick,red, double, ] (-2,-2) to (-2,4);
		\draw [very thick, red ] (2,2.5) to (-2,3.5);
		\draw [very thick, red ] (6,0) to (2,1);
		\draw [very thick, red ] (2,-1.5) to (-2,-.5);
\node at (-2,-2.55) {\tiny $s_i$};	
\node at (2,-2.55) {\tiny $s_{i+1}$};
\node at (6,-2.55) {\tiny $s_{i+2}$};
\node at (-1,-2.55) {\tiny $k_1$};	
\node at (3.3,-2.55) {\tiny $k$};	
\node at (4.7,-2.55) {\tiny $k_2$};	
\node at (-2,4.5) {\tiny $s_i+2$};
	\node at (2,4.5) {\tiny $s_{i+1}-1$};
	\node at (6,4.5) {\tiny $s_{i+2}-1$};
	\filldraw[red]  (0,-1) circle (6pt);
\node at (0,-1.6) {\tiny $f$};
\draw[very thick,double,] (3.3,-2) .. controls ++(0,.75) and ++(0,-2) .. (1,2) to (1,4);
\draw [very thick,double,] (4.7,-2) to (4.7,4);
\draw [very thick,double,] (-1,-2) to (-1,4);
    \end{tikzpicture}}
\end{equation}
For simplicity we may assume $k_1=k_2=0$.    There are three cases to consider separately: $k=0$, $k=1$, and $k>1$.

First we let $k=0$.
Then the generator \eqref{ii+1irel1} gets mapped under
\begin{equation*}
  \Psi\left(\hackcenter{\begin{tikzpicture}[scale=0.8]
    \draw[thick, ->] (0,0) .. controls ++(0,.55) and ++(0,-.5) .. (.75,1)
        node[pos=.25, shape=coordinate](DOT){};
    \draw[thick, ->] (.75,0) .. controls ++(0,.5) and ++(0,-.5) .. (0,1);
    \draw[thick, ->] (1.5,0)to (1.5,1);
    \node at (-.2,.15) {\tiny $i$};
    \node at (.95,.15) {\tiny $j$};
    \node at (1.7,.15) {\tiny $i$};
\end{tikzpicture}} \right)
\end{equation*}
to the element
\begin{equation}
\label{ii+1irel2}
\hackcenter{
\begin{tikzpicture}[scale=.4]
	\draw [thick,red, double, ] (6,-2) to (6,4);
	\draw [thick,red, double, ] (2,-2) to (2,4);
	\draw [thick,red, double, ] (-2,-2) to (-2,4);
		\draw [very thick, red ] (2,1.5) to (-2,2.5);
		\draw [very thick, red ] (6,2.5) to (2,3.5);
		\draw [very thick, red ] (2,-1.5) to (-2,-.5);
\node at (-2,-2.55) {\tiny $s_i$};	
\node at (2,-2.55) {\tiny $s_{i+1}$};
\node at (6,-2.55) {\tiny $s_{i+2}$};
	\node at (-2,4.5) {\tiny $s_i+2$};
	\node at (2,4.5) {\tiny $s_{i+1}-1$};
	\node at (6,4.5) {\tiny $s_{i+2}-1$};
	\filldraw[red]  (0,-1) circle (6pt);
\node at (0,-1.6) {\tiny $f$};
    \end{tikzpicture}}
\end{equation}
Since the map
\begin{equation*}
  \Psi\left(\hackcenter{\begin{tikzpicture}[scale=0.8]
    \draw[thick, ->] (0,0) .. controls ++(0,.55) and ++(0,-.5) .. (.75,1)
        node[pos=.25, shape=coordinate](DOT){};
    \draw[thick, ->] (.75,0) .. controls ++(0,.5) and ++(0,-.5) .. (0,1);
    \draw[thick, ->] (-.75,0)to (-.75,1);
    \node at (-.2,.15) {\tiny $i$};
    \node at (.95,.15) {\tiny $i$};
    \node at (-.95,.15) {\tiny $j$};
\end{tikzpicture}} \right)
\end{equation*}
is a divided difference operator applied to dots on the lowermost red strands labelled by $1$, it sends the element in ~\eqref{ii+1irel2} to
\begin{equation}
\label{ii+1irel3}
-\sum_{\gamma=0}^{f-1}
\hackcenter{
\begin{tikzpicture}[scale=.4]
	\draw [thick,red, double, ] (6,-2) to (6,4);
	\draw [thick,red, double, ] (2,-2) to (2,4);
	\draw [thick,red, double, ] (-2,-2) to (-2,4);
		\draw [very thick, red ] (2,1.5) to (-2,2.5);
		\draw [very thick, red ] (6,2.5) to (2,3.5);
		\draw [very thick, red ] (2,-1.5) to (-2,-.5);
\node at (-2,-2.55) {\tiny $s_i$};	
\node at (2,-2.55) {\tiny $s_{i+1}$};
\node at (6,-2.55) {\tiny $s_{i+2}$};
	\node at (-2,4.5) {\tiny $s_i+2$};
	\node at (2,4.5) {\tiny $s_{i+1}-1$};
	\node at (6,4.5) {\tiny $s_{i+2}-1$};
	\filldraw[red]  (0,-1) circle (6pt);
\node at (0,-1.6) {\tiny $\gamma$};
	\filldraw[red]  (0,2) circle (6pt);
\node at (0,1.4) {\tiny $f-1-\gamma$};
    \end{tikzpicture}}
\end{equation}
Finally, the map
\begin{equation*}
  \Psi\left(\hackcenter{\begin{tikzpicture}[scale=0.8]
    \draw[thick, ->] (0,0) .. controls ++(0,.55) and ++(0,-.5) .. (.75,1)
        node[pos=.25, shape=coordinate](DOT){};
    \draw[thick, ->] (.75,0) .. controls ++(0,.5) and ++(0,-.5) .. (0,1);
    \draw[thick, ->] (1.5,0)to (1.5,1);
    \node at (-.2,.15) {\tiny $j$};
    \node at (.95,.15) {\tiny $i$};
    \node at (1.7,.15) {\tiny $i$};
\end{tikzpicture}} \right)
\end{equation*}
sends the element ~\eqref{ii+1irel3} to
\begin{equation}
\label{ii+1irel4}
\sum_{\gamma=0}^{f-1}
\hackcenter{
\begin{tikzpicture}[scale=.4]
	\draw [thick,red, double, ] (6,-2) to (6,4);
	\draw [thick,red, double, ] (2,-2) to (2,4);
	\draw [thick,red, double, ] (-2,-2) to (-2,4);
		\draw [very thick, red ] (2,2.5) to (-2,3.5);
		\draw [very thick, red ] (6,0) to (2,1);
		\draw [very thick, red ] (2,-1.5) to (-2,-.5);
\node at (-2,-2.55) {\tiny $s_i$};	
\node at (2,-2.55) {\tiny $s_{i+1}$};
\node at (6,-2.55) {\tiny $s_{i+2}$};
	\node at (-2,4.5) {\tiny $s_i+2$};
	\node at (2,4.5) {\tiny $s_{i+1}-1$};
	\node at (6,4.5) {\tiny $s_{i+2}-1$};
	\filldraw[red]  (0,-1) circle (6pt);
\node at (0,-1.6) {\tiny $\gamma$};
	\filldraw[red]  (0,3) circle (6pt);
\node at (0,3.5) {\tiny $f-\gamma$};
    \end{tikzpicture}}
    -
    \sum_{\gamma=0}^{f-1}
\hackcenter{
\begin{tikzpicture}[scale=.4]
	\draw [thick,red, double, ] (6,-2) to (6,4);
	\draw [thick,red, double, ] (2,-2) to (2,4);
	\draw [thick,red, double, ] (-2,-2) to (-2,4);
		\draw [very thick, red ] (2,2.5) to (-2,3.5);
		\draw [very thick, red ] (6,0) to (2,1);
		\draw [very thick, red ] (2,-1.5) to (-2,-.5);
\node at (-2,-2.55) {\tiny $s_i$};	
\node at (2,-2.55) {\tiny $s_{i+1}$};
\node at (6,-2.55) {\tiny $s_{i+2}$};
	\node at (-2,4.5) {\tiny $s_i+2$};
	\node at (2,4.5) {\tiny $s_{i+1}-1$};
	\node at (6,4.5) {\tiny $s_{i+2}-1$};
	\filldraw[red]  (0,-1) circle (6pt);
\node at (0,-1.6) {\tiny $\gamma$};
	\filldraw[red]  (0,3) circle (6pt);
\node at (0,3.5) {\tiny $f-\gamma-1$};
	\filldraw[red]  (4,.5) circle (6pt);
    \end{tikzpicture}}
\end{equation}
Similarly one computes that
\begin{equation*}
  \Psi \left(
\hackcenter{\begin{tikzpicture}[scale=0.8]
    \draw[thick, ->] (0,0) .. controls ++(0,1) and ++(0,-1) .. (1.5,2);
    \draw[thick, ] (.75,0) .. controls ++(0,.5) and ++(0,-.5) .. (1.5,1);
    \draw[thick, ->] (1.5,1) .. controls ++(0,.5) and ++(0,-.5) .. (0.75,2);
    \draw[thick, ->] (1.5,0) .. controls ++(0,1) and ++(0,-1) .. (0,2);
    \node at (-.2,.15) {\tiny $i$};
    \node at (.95,.15) {\tiny $j$};
    \node at (1.75,.15) {\tiny $k$};
\node at (1.9, 1.25) { ${\bf s}$};
\end{tikzpicture}}\right)
\end{equation*}
applied to the element ~\eqref{ii+1irel1} is
\begin{equation}
\label{ii+1irel5}
\sum_{\gamma=0}^{f}
\hackcenter{
\begin{tikzpicture}[scale=.4]
	\draw [thick,red, double, ] (6,-2) to (6,4);
	\draw [thick,red, double, ] (2,-2) to (2,4);
	\draw [thick,red, double, ] (-2,-2) to (-2,4);
		\draw [very thick, red ] (2,2.5) to (-2,3.5);
		\draw [very thick, red ] (6,0) to (2,1);
		\draw [very thick, red ] (2,-1.5) to (-2,-.5);
\node at (-2,-2.55) {\tiny $s_i$};	
\node at (2,-2.55) {\tiny $s_{i+1}$};
\node at (6,-2.55) {\tiny $s_{i+2}$};
	\node at (-2,4.5) {\tiny $s_i+2$};
	\node at (2,4.5) {\tiny $s_{i+1}-1$};
	\node at (6,4.5) {\tiny $s_{i+2}-1$};
	\filldraw[red]  (0,-1) circle (6pt);
\node at (0,-1.6) {\tiny $\gamma$};
	\filldraw[red]  (0,3) circle (6pt);
\node at (0,3.5) {\tiny $f-\gamma$};
    \end{tikzpicture}}
    -
    \sum_{\gamma=0}^{f-1}
\hackcenter{
\begin{tikzpicture}[scale=.4]
	\draw [thick,red, double, ] (6,-2) to (6,4);
	\draw [thick,red, double, ] (2,-2) to (2,4);
	\draw [thick,red, double, ] (-2,-2) to (-2,4);
		\draw [very thick, red ] (2,2.5) to (-2,3.5);
		\draw [very thick, red ] (6,0) to (2,1);
		\draw [very thick, red ] (2,-1.5) to (-2,-.5);
\node at (-2,-2.55) {\tiny $s_i$};	
\node at (2,-2.55) {\tiny $s_{i+1}$};
\node at (6,-2.55) {\tiny $s_{i+2}$};
	\node at (-2,4.5) {\tiny $s_i+2$};
	\node at (2,4.5) {\tiny $s_{i+1}-1$};
	\node at (6,4.5) {\tiny $s_{i+2}-1$};
	\filldraw[red]  (0,-1) circle (6pt);
\node at (0,-1.6) {\tiny $\gamma$};
	\filldraw[red]  (0,3) circle (6pt);
\node at (0,3.5) {\tiny $f-\gamma-1$};
	\filldraw[red]  (4,.5) circle (6pt);
    \end{tikzpicture}}
\end{equation}
Subtracting the element ~\eqref{ii+1irel5} from the element ~\eqref{ii+1irel4} produces
\begin{equation*}
- \hackcenter{
\begin{tikzpicture}[scale=.4]
	\draw [thick,red, double, ] (6,-2) to (6,4);
	\draw [thick,red, double, ] (2,-2) to (2,4);
	\draw [thick,red, double, ] (-2,-2) to (-2,4);
		\draw [very thick, red ] (2,2.5) to (-2,3.5);
		\draw [very thick, red ] (6,0) to (2,1);
		\draw [very thick, red ] (2,-1.5) to (-2,-.5);
\node at (-2,-2.55) {\tiny $s_i$};	
\node at (2,-2.55) {\tiny $s_{i+1}$};
\node at (6,-2.55) {\tiny $s_{i+2}$};
	\node at (-2,4.5) {\tiny $s_i+2$};
	\node at (2,4.5) {\tiny $s_{i+1}-1$};
	\node at (6,4.5) {\tiny $s_{i+2}-1$};
	\filldraw[red]  (0,-1) circle (6pt);
\node at (0,-1.6) {\tiny $f$};
    \end{tikzpicture}}
\end{equation*}
which verifies the proposition in this case.

Next we let $k=1$.  Again recall that we assume $k_1=k_2=0$.  Then a straightforward calculation gives that the map
\begin{equation*}
 \Psi \left(
 \hackcenter{\begin{tikzpicture}[scale=0.8]
    \draw[thick, ->] (0,0) .. controls ++(0,1) and ++(0,-1) .. (1.5,2);
    \draw[thick, ] (.65,0) .. controls ++(0,.5) and ++(0,-.5) .. (0,1);
    \draw[thick, ->] (0,1) .. controls ++(0,.5) and ++(0,-.5) .. (0.75,2);
    \draw[thick, ->] (1.5,0) .. controls ++(0,1) and ++(0,-1) .. (0,2);
    \node at (-.2,.15) {\tiny $i$};
    \node at (1.05,.15) {\tiny $i\pm 1$};
    \node at (1.75,.15) {\tiny $i$};
\node at (1.9, 1.25) { ${\bf s}$};
\end{tikzpicture}} \right)
\end{equation*}
sends the generator ~\eqref{ii+1irel1} to
\begin{equation}
\label{ii+1irel6}
\hackcenter{
\begin{tikzpicture}[scale=.4]
	\draw [thick,red, double, ] (6,-2) to (6,4);
	\draw [thick,red, double, ] (2,-2) to (2,4);
	\draw [thick,red, double, ] (-2,-2) to (-2,4);
		\draw [very thick, red ] (2,2.5) to (-2,3.5);
		\draw [very thick, red ] (6,0) to (2,1);
		\draw [very thick, red ] (2,-1.5) to (-2,-.5);
\node at (-2,-2.55) {\tiny $s_i$};	
\node at (2,-2.55) {\tiny $s_{i+1}$};
\node at (6,-2.55) {\tiny $s_{i+2}$};
\node at (-2,4.5) {\tiny $s_i+2$};
	\node at (2,4.5) {\tiny $s_{i+1}-1$};
	\node at (6,4.5) {\tiny $s_{i+2}-1$};
\filldraw[]  (1.2,1) circle (6pt);
\draw[thick,]  (3,-2) .. controls ++(0,.65) and ++(0,-2) .. (1,2) to  (1,4);
\node[draw, fill=white!20 ,rounded corners, rotate=90 ] at (-.8,1) {\tiny $\hspace{.25in} \partial_1(x_1^f) \hspace{.25in}$};
    \end{tikzpicture}}
    -\hackcenter{
\begin{tikzpicture}[scale=.4]
	\draw [thick,red, double, ] (6,-2) to (6,4);
	\draw [thick,red, double, ] (2,-2) to (2,4);
	\draw [thick,red, double, ] (-2,-2) to (-2,4);
		\draw [very thick, red ] (2,2.5) to (-2,3.5);
		\draw [very thick, red ] (6,0) to (2,1);
		\draw [very thick, red ] (2,-1.5) to (-2,-.5);
\node at (-2,-2.55) {\tiny $s_i$};	
\node at (2,-2.55) {\tiny $s_{i+1}$};
\node at (6,-2.55) {\tiny $s_{i+2}$};
\node at (-2,4.5) {\tiny $s_i+2$};
	\node at (2,4.5) {\tiny $s_{i+1}-1$};
	\node at (6,4.5) {\tiny $s_{i+2}-1$};
\draw[thick,]  (3,-2) .. controls ++(0,.65) and ++(0,-2) .. (1,2) to (1,4);
\node[draw, fill=white!20 ,rounded corners, rotate=90 ] at (-.8,1) {\tiny $\hspace{.25in} \partial_1(x_1^f x_2) \hspace{.25in}$};
    \end{tikzpicture}}
    -\hackcenter{
\begin{tikzpicture}[scale=.4]
	\draw [thick,red, double, ] (6,-2) to (6,4);
	\draw [thick,red, double, ] (2,-2) to (2,4);
	\draw [thick,red, double, ] (-2,-2) to (-2,4);
		\draw [very thick, red ] (2,2.5) to (-2,3.5);
		\draw [very thick, red ] (6,0) to (2,1);
		\draw [very thick, red ] (2,-1.5) to (-2,-.5);
\node at (-2,-2.55) {\tiny $s_i$};	
\node at (2,-2.55) {\tiny $s_{i+1}$};
\node at (6,-2.55) {\tiny $s_{i+2}$};
\node at (-2,4.5) {\tiny $s_i+2$};
	\node at (2,4.5) {\tiny $s_{i+1}-1$};
	\node at (6,4.5) {\tiny $s_{i+2}-1$};
\filldraw[]  (3,1.5) circle (6pt);
\draw[thick,]   (3,-2) to (3,2) .. controls ++(0,1) and ++(0,-.5) .. (1,4);
\node[draw, fill=white!20 ,rounded corners, rotate=90 ] at (-.8,1) {\tiny $\hspace{.25in} \partial_1(x_1^f) \hspace{.25in}$};
    \end{tikzpicture}}
  +  \hackcenter{
\begin{tikzpicture}[scale=.4]
	\draw [thick,red, double, ] (6,-2) to (6,4);
	\draw [thick,red, double, ] (2,-2) to (2,4);
	\draw [thick,red, double, ] (-2,-2) to (-2,4);
		\draw [very thick, red ] (2,2.5) to (-2,3.5);
		\draw [very thick, red ] (6,0) to (2,1);
		\draw [very thick, red ] (2,-1.5) to (-2,-.5);
\node at (-2,-2.55) {\tiny $s_i$};	
\node at (2,-2.55) {\tiny $s_{i+1}$};
\node at (6,-2.55) {\tiny $s_{i+2}$};
\node at (-2,4.5) {\tiny $s_i+2$};
	\node at (2,4.5) {\tiny $s_{i+1}-1$};
	\node at (6,4.5) {\tiny $s_{i+2}-1$};
\draw[thick,]  (3,-2) to (3,2) .. controls ++(0,1) and ++(0,-.5) .. (1,4);
\node[draw, fill=white!20 ,rounded corners, rotate=90 ] at (-.8,1) {\tiny $\hspace{.25in} \partial_1(x_1^f x_2) \hspace{.25in}$};
    \end{tikzpicture}}
\end{equation}

Under the map
\begin{equation*}
 \Psi\left( \hackcenter{\begin{tikzpicture}[scale=0.8]
    \draw[thick, ->] (0,0) .. controls ++(0,1) and ++(0,-1) .. (1.5,2);
    \draw[thick, ] (.65,0) .. controls ++(0,.65) and ++(0,-.5) .. (1.5,1);
    \draw[thick, ->] (1.5,1) .. controls ++(0,.5) and ++(0,-.5) .. (0.75,2);
    \draw[thick, ->] (1.5,0) .. controls ++(0,1) and ++(0,-1) .. (0,2);
    \node at (-.2,.15) {\tiny $i$};
    \node at (1.05,.15) {\tiny $i\pm 1$};
    \node at (1.75,.15) {\tiny $i$};
\node at (1.9, 1.25) { ${\bf s}$};
\end{tikzpicture}} \right)
\end{equation*}
the generator ~\eqref{ii+1irel1} gets sent to
\begin{equation}
\label{ii+1irel7}
\hackcenter{
\begin{tikzpicture}[scale=.4]
	\draw [thick,red, double, ] (6,-2) to (6,4);
	\draw [thick,red, double, ] (2,-2) to (2,4);
	\draw [thick,red, double, ] (-2,-2) to (-2,4);
		\draw [very thick, red ] (2,2.5) to (-2,3.5);
		\draw [very thick, red ] (6,0) to (2,1);
		\draw [very thick, red ] (2,-1.5) to (-2,-.5);
\node at (-2,-2.55) {\tiny $s_i$};	
\node at (2,-2.55) {\tiny $s_{i+1}$};
\node at (6,-2.55) {\tiny $s_{i+2}$};
\node at (-2,4.5) {\tiny $s_i+2$};
	\node at (2,4.5) {\tiny $s_{i+1}-1$};
	\node at (6,4.5) {\tiny $s_{i+2}-1$};
\filldraw[]  (1.2,1) circle (6pt);
\draw[thick,]  (3,-2) .. controls ++(0,.65) and ++(0,-2) .. (1,2) to  (1,4);
\node[draw, fill=white!20 ,rounded corners, rotate=90 ] at (-.8,1) {\tiny $\hspace{.25in} \partial_1(x_1^f) \hspace{.25in}$};
    \end{tikzpicture}}
    -\hackcenter{
\begin{tikzpicture}[scale=.4]
	\draw [thick,red, double, ] (6,-2) to (6,4);
	\draw [thick,red, double, ] (2,-2) to (2,4);
	\draw [thick,red, double, ] (-2,-2) to (-2,4);
		\draw [very thick, red ] (2,2.5) to (-2,3.5);
		\draw [very thick, red ] (6,0) to (2,1);
		\draw [very thick, red ] (2,-1.5) to (-2,-.5);
\node at (-2,-2.55) {\tiny $s_i$};	
\node at (2,-2.55) {\tiny $s_{i+1}$};
\node at (6,-2.55) {\tiny $s_{i+2}$};
\node at (-2,4.5) {\tiny $s_i+2$};
	\node at (2,4.5) {\tiny $s_{i+1}-1$};
	\node at (6,4.5) {\tiny $s_{i+2}-1$};
\filldraw[red]  (1.5,-1.3) circle (6pt);
\draw[thick,]  (3,-2) .. controls ++(0,.65) and ++(0,-2) .. (1,2) to  (1,4);
\node[draw, fill=white!20 ,rounded corners, rotate=90 ] at (-.8,1) {\tiny $\hspace{.25in} \partial_1(x_1^f ) \hspace{.25in}$};
    \end{tikzpicture}}
    -\hackcenter{
\begin{tikzpicture}[scale=.4]
	\draw [thick,red, double, ] (6,-2) to (6,4);
	\draw [thick,red, double, ] (2,-2) to (2,4);
	\draw [thick,red, double, ] (-2,-2) to (-2,4);
		\draw [very thick, red ] (2,2.5) to (-2,3.5);
		\draw [very thick, red ] (6,0) to (2,1);
		\draw [very thick, red ] (2,-1.5) to (-2,-.5);
\node at (-2,-2.55) {\tiny $s_i$};	
\node at (2,-2.55) {\tiny $s_{i+1}$};
\node at (6,-2.55) {\tiny $s_{i+2}$};
\node at (-2,4.5) {\tiny $s_i+2$};
	\node at (2,4.5) {\tiny $s_{i+1}-1$};
	\node at (6,4.5) {\tiny $s_{i+2}-1$};
\filldraw[]  (3,1.5) circle (6pt);
\draw[thick,]   (3,-2) to (3,2) .. controls ++(0,1) and ++(0,-.5) .. (1,4);
\node[draw, fill=white!20 ,rounded corners, rotate=90 ] at (-.8,1) {\tiny $\hspace{.25in} \partial_1(x_1^f) \hspace{.25in}$};
    \end{tikzpicture}}
  +  \hackcenter{
\begin{tikzpicture}[scale=.4]
	\draw [thick,red, double, ] (6,-2) to (6,4);
	\draw [thick,red, double, ] (2,-2) to (2,4);
	\draw [thick,red, double, ] (-2,-2) to (-2,4);
		\draw [very thick, red ] (2,2.5) to (-2,3.5);
		\draw [very thick, red ] (6,0) to (2,1);
		\draw [very thick, red ] (2,-1.5) to (-2,-.5);
\node at (-2,-2.55) {\tiny $s_i$};	
\node at (2,-2.55) {\tiny $s_{i+1}$};
\node at (6,-2.55) {\tiny $s_{i+2}$};
\node at (-2,4.5) {\tiny $s_i+2$};
	\node at (2,4.5) {\tiny $s_{i+1}-1$};
	\node at (6,4.5) {\tiny $s_{i+2}-1$};
\draw[thick,]   (3,-2) to (3,2) .. controls ++(0,1) and ++(0,-.5) .. (1,4);
\node[draw, fill=white!20 ,rounded corners, rotate=90 ] at (-.8,1) {\tiny $\hspace{.25in} \partial_1(x_1^f x_2) \hspace{.25in}$};
    \end{tikzpicture}}
\end{equation}
Subtracting \eqref{ii+1irel7} from \eqref{ii+1irel6} yields
\begin{equation*}
-\hackcenter{
\begin{tikzpicture}[scale=.4]
	\draw [thick,red, double, ] (6,-2) to (6,4);
	\draw [thick,red, double, ] (2,-2) to (2,4);
	\draw [thick,red, double, ] (-2,-2) to (-2,4);
		\draw [very thick, red ] (2,2.5) to (-2,3.5);
		\draw [very thick, red ] (6,0) to (2,1);
		\draw [very thick, red ] (2,-1.5) to (-2,-.5);
\node at (-2,-2.55) {\tiny $s_i$};	
\node at (2,-2.55) {\tiny $s_{i+1}$};
\node at (6,-2.55) {\tiny $s_{i+2}$};
\node at (-2,4.5) {\tiny $s_i+2$};
	\node at (2,4.5) {\tiny $s_{i+1}-1$};
	\node at (6,4.5) {\tiny $s_{i+2}-1$};
	\filldraw[red]  (0,-1) circle (6pt);
\node at (0,-1.6) {\tiny $f$};
\draw[thick,]  (3,-2) .. controls ++(0,.65) and ++(0,-2) .. (1,2) to  (1,4);
    \end{tikzpicture}}
\end{equation*}
which verifies the proposition in this case.

Now assume $k>1$.  Then the map
\begin{equation*}
  \Psi \left(\hackcenter{\begin{tikzpicture}[scale=0.8]
    \draw[thick, ->] (0,0) .. controls ++(0,1) and ++(0,-1) .. (1.5,2);
    \draw[thick, ] (.65,0) .. controls ++(0,.5) and ++(0,-.5) .. (0,1);
    \draw[thick, ->] (0,1) .. controls ++(0,.5) and ++(0,-.5) .. (0.75,2);
    \draw[thick, ->] (1.5,0) .. controls ++(0,1) and ++(0,-1) .. (0,2);
    \node at (-.2,.15) {\tiny $i$};
    \node at (1.05,.15) {\tiny $i\pm 1$};
    \node at (1.75,.15) {\tiny $i$};
\node at (1.9, 1.25) { ${\bf s}$};
\end{tikzpicture}} \right)
\end{equation*}
sends the generator ~\eqref{ii+1irel1} to
\begin{equation}
\label{ii+1irel8}
\sum_{g=1}^k (-1)^{g+1}\hspace{-.1in}
\hackcenter{
\begin{tikzpicture}[scale=.4]
	\draw [thick,red, double, ] (5.75,-2) to (5.75,4);
	\draw [thick,red, double, ] (2,-2) to (2,4);
	\draw [thick,red, double, ] (-2,-2) to (-2,4);
		\draw [very thick, red ] (2,2.5) to (-2,3.5);
		\draw [very thick, red ] (5.75,0) to (2,1);
		\draw [very thick, red ] (2,-1.5) to (-2,-.5);
\draw[thick]  (3.3,-1.5) .. controls ++(1.5,.5) and ++(1.5,.5) .. (1,3.5);
\draw[very thick, double, ]  (3.4,-2) to (3.4,-1.7) .. controls ++(0,.5) and ++(0,-2) .. (1,2) to  (1,4);
\filldraw[]  (1.4,.55) circle (6pt);
\node at (-2,-2.55) {\tiny $s_i$};	
\node at (2,-2.55) {\tiny $s_{i+1}$};
\node at (5.75,-2.55) {\tiny $s_{i+2}$};
\node at (-2,4.5) {\tiny $s_i+2$};
	\node at (2,4.5) {\tiny $s_{i+1}-1$};
	\node at (5.6,4.5) {\tiny $s_{i+2}-1$};
\node at (.8,-.3) {\tiny ${\sf e}_{k-g}$};
\node[draw, fill=white!20 ,rounded corners, rotate=90 ] at (-.9,1) {\tiny $\hspace{.25in} \partial_1(x_1^f x_2^g) \hspace{.25in}$};
    \end{tikzpicture}}
  \hspace{-.1in}  +\sum_{g=0}^{k-1} (-1)^{g+1} \hspace{-.1in}
\hackcenter{
\begin{tikzpicture}[scale=.4]
	\draw [thick,red, double, ] (5.75,-2) to (5.75,4);
	\draw [thick,red, double, ] (2,-2) to (2,4);
	\draw [thick,red, double, ] (-2,-2) to (-2,4);
		\draw [very thick, red ] (2,2.5) to (-2,3.5);
		\draw [very thick, red ] (5.75,0) to (2,1);
		\draw [very thick, red ] (2,-1.5) to (-2,-.5);
		\filldraw[red]  (.2,2.9) circle (6pt);
\node at (-2,-2.55) {\tiny $s_i$};	
\node at (2,-2.55) {\tiny $s_{i+1}$};
\node at (5.75,-2.55) {\tiny $s_{i+2}$};
\node at (-2,4.5) {\tiny $s_i+2$};
	\node at (2,4.5) {\tiny $s_{i+1}-1$};
	\node at (5.6,4.5) {\tiny $s_{i+2}-1$};
\node at (.75,-.45) {\tiny ${\sf e}_{\overset{k-g}{-1}}$};
\draw[thick]  (3.3,-1.5) .. controls ++(1.5,.5) and ++(1.5,.5) .. (1,3.5);
\draw[very thick, double, ]  (3.4,-2) to (3.4,-1.7) .. controls ++(0,.5) and ++(0,-2) .. (1,2) to  (1,4);
\filldraw[]  (1.4,.55) circle (6pt);
\node[draw, fill=white!20 ,rounded corners, rotate=90 ] at (-.9,1) {\tiny $\hspace{.25in} \partial_1(x_1^f x_2^g) \hspace{.25in}$};
    \end{tikzpicture}}
 \hspace{-.1in}   +\sum_{g=0}^{k-1} (-1)^{g+1} \hspace{-.1in}
\hackcenter{
\begin{tikzpicture}[scale=.4]
	\draw [thick,red, double, ] (5.75,-2) to (5.75,4);
	\draw [thick,red, double, ] (2,-2) to (2,4);
	\draw [thick,red, double, ] (-2,-2) to (-2,4);
		\draw [very thick, red ] (2,2.5) to (-2,3.5);
		\draw [very thick, red ] (5.75,0) to (2,1);
		\draw [very thick, red ] (2,-1.5) to (-2,-.5);
\node at (-2,-2.55) {\tiny $s_i$};	
\node at (2,-2.55) {\tiny $s_{i+1}$};
\node at (5.75,-2.55) {\tiny $s_{i+2}$};
\node at (-2,4.5) {\tiny $s_i+2$};
	\node at (2,4.5) {\tiny $s_{i+1}-1$};
	\node at (5.6,4.5) {\tiny $s_{i+2}-1$};
\node at (.75,-.5) {\tiny ${\sf e}_{\overset{k-g}{-1}}$};
\draw[thick]  (3.35,-1.5) .. controls ++(2,4) and ++(2,0) .. (1.6,.3);
\draw[very thick, double, ]  (3.4,-2) .. controls ++(0,.85) and ++(-0,-.8) .. (1,1) to  (1,4);
\filldraw[]  (2.55,-.85) circle (6pt);
\node[draw, fill=white!20 ,rounded corners, rotate=90 ] at (-.9,1) {\tiny $\hspace{.25in} \partial_1(x_1^f x_2^g) \hspace{.25in}$};
    \end{tikzpicture}}
\end{equation}

Under the map
 \begin{equation*}
 \Psi\left( \hackcenter{\begin{tikzpicture}[scale=0.8]
    \draw[thick, ->] (0,0) .. controls ++(0,1) and ++(0,-1) .. (1.5,2);
    \draw[thick, ] (.65,0) .. controls ++(0,.65) and ++(0,-.5) .. (1.5,1);
    \draw[thick, ->] (1.5,1) .. controls ++(0,.5) and ++(0,-.5) .. (0.75,2);
    \draw[thick, ->] (1.5,0) .. controls ++(0,1) and ++(0,-1) .. (0,2);
    \node at (-.2,.15) {\tiny $i$};
    \node at (1.05,.15) {\tiny $i\pm 1$};
    \node at (1.75,.15) {\tiny $i$};
\node at (1.9, 1.25) { ${\bf s}$};
\end{tikzpicture}} \right)
\end{equation*}
the generator ~\eqref{ii+1irel1} gets sent to
\begin{equation}
\label{ii+1irel9}
\sum_{g=0}^{k-1} (-1)^{g}\hspace{-.1in}
\hackcenter{
\begin{tikzpicture}[scale=.4]
	\draw [thick,red, double, ] (5.75,-2) to (5.75,4);
	\draw [thick,red, double, ] (2,-2) to (2,4);
	\draw [thick,red, double, ] (-2,-2) to (-2,4);
		\draw [very thick, red ] (2,2.5) to (-2,3.5);
		\draw [very thick, red ] (5.75,0) to (2,1);
		\draw [very thick, red ] (2,-1.5) to (-2,-.5);
		\filldraw[red]  (1,-1.25) circle (6pt);
		\node at (1,-1.8) {\tiny $g$};	
\node at (-2,-2.55) {\tiny $s_i$};	
\node at (2,-2.55) {\tiny $s_{i+1}$};
\node at (5.75,-2.55) {\tiny $s_{i+2}$};
\node at (-2,4.5) {\tiny $s_i+2$};
	\node at (2,4.5) {\tiny $s_{i+1}-1$};
	\node at (5.6,4.5) {\tiny $s_{i+2}-1$};
\node at (.75,-.5) {\tiny ${\sf e}_{k-g-1}$};
\draw[thick]  (3.3,-1.5) .. controls ++(1.5,.5) and ++(1.5,.5) .. (1,3.5);
\draw[very thick, double, ]  (3.4,-2) to (3.4,-1.7) .. controls ++(0,.5) and ++(0,-2) .. (1,2) to  (1,4);
\filldraw[]  (1.4,.55) circle (6pt);
\node[draw, fill=white!20 ,rounded corners, rotate=90 ] at (-1.1,1) {\tiny $\hspace{.25in} \partial_1(x_1^f x_2) \hspace{.25in}$};
    \end{tikzpicture}}
   \hspace{-.1in} -\sum_{g=0}^{k-1} (-1)^{g}\hspace{-.1in}
\hackcenter{
\begin{tikzpicture}[scale=.4]
	\draw [thick,red, double, ] (5.75,-2) to (5.75,4);
	\draw [thick,red, double, ] (2,-2) to (2,4);
	\draw [thick,red, double, ] (-2,-2) to (-2,4);
		\draw [very thick, red ] (2,2.5) to (-2,3.5);
		\draw [very thick, red ] (5.75,0) to (2,1);
		\draw [very thick, red ] (2,-1.5) to (-2,-.5);
		\filldraw[red]  (0,3) circle (6pt);
			\filldraw[red]  (1,-1.25) circle (6pt);
		\node at (1,-1.8) {\tiny $g$};	
\node at (-2,-2.55) {\tiny $s_i$};	
\node at (2,-2.55) {\tiny $s_{i+1}$};
\node at (5.75,-2.55) {\tiny $s_{i+2}$};
\node at (-2,4.5) {\tiny $s_i+2$};
	\node at (2,4.5) {\tiny $s_{i+1}-1$};
	\node at (5.6,4.5) {\tiny $s_{i+2}-1$};
\node at (.75,-.5) {\tiny ${\sf e}_{k-g-1}$};
\draw[thick]  (3.3,-1.5) .. controls ++(1.5,.5) and ++(1.5,.5) .. (1,3.5);
\draw[very thick, double, ]  (3.4,-2) to (3.4,-1.7) .. controls ++(0,.5) and ++(0,-2) .. (1,2) to  (1,4);
\filldraw[]  (1.4,.55) circle (6pt);
\node[draw, fill=white!20 ,rounded corners, rotate=90 ] at (-1.1,1) {\tiny $\hspace{.25in} \partial_1(x_1^f) \hspace{.25in}$};
    \end{tikzpicture}}
    \hspace{-.1in}-\sum_{g=0}^{k-1} (-1)^{g}\hspace{-.1in}
\hackcenter{
\begin{tikzpicture}[scale=.4]
	\draw [thick,red, double, ] (5.75,-2) to (5.75,4);
	\draw [thick,red, double, ] (2,-2) to (2,4);
	\draw [thick,red, double, ] (-2,-2) to (-2,4);
		\draw [very thick, red ] (2,2.5) to (-2,3.5);
		\draw [very thick, red ] (5.75,0) to (2,1);
		\draw [very thick, red ] (2,-1.5) to (-2,-.5);
\node at (-2,-2.55) {\tiny $s_i$};	
\node at (2,-2.55) {\tiny $s_{i+1}$};
\node at (5.75,-2.55) {\tiny $s_{i+2}$};
\node at (-2,4.5) {\tiny $s_i+2$};
	\node at (2,4.5) {\tiny $s_{i+1}-1$};
	\node at (5.6,4.5) {\tiny $s_{i+2}-1$};
\node at (.75,-.3) {\tiny ${\sf e}_{\overset{k-g}{-1}}$};
		\filldraw[red]  (1,-1.25) circle (6pt);
		\node at (1,-1.8) {\tiny $g$};	
\draw[thick]  (3.35,-1.5) .. controls ++(2,4) and ++(2,0) .. (1.6,.3);
\draw[very thick, double, ]  (3.4,-2) .. controls ++(0,.85) and ++(-0,-.8) .. (1,1) to  (1,4);
\filldraw[]  (2.55,-.85) circle (6pt);
\node[draw, fill=white!20 ,rounded corners, rotate=90 ] at (-1.1,1) {\tiny $\hspace{.25in} \partial_1(x_1^f) \hspace{.25in}$};
    \end{tikzpicture}}
\end{equation}
Subtracting ~\eqref{ii+1irel9} from ~\eqref{ii+1irel8} yields
\begin{equation*}
-\hackcenter{
\begin{tikzpicture}[scale=.4]
	\draw [thick,red, double, ] (6,-2) to (6,4);
	\draw [thick,red, double, ] (2,-2) to (2,4);
	\draw [thick,red, double, ] (-2,-2) to (-2,4);
		\draw [very thick, red ] (2,2.5) to (-2,3.5);
		\draw [very thick, red ] (6,0) to (2,1);
		\draw [very thick, red ] (2,-1.5) to (-2,-.5);
\node at (-2,-2.55) {\tiny $s_i$};	
\node at (2,-2.55) {\tiny $s_{i+1}$};
\node at (6,-2.55) {\tiny $s_{i+2}$};
\node at (-2,4.5) {\tiny $s_i+2$};
	\node at (2,4.5) {\tiny $s_{i+1}-1$};
	\node at (6,4.5) {\tiny $s_{i+2}-1$};
	\filldraw[red]  (0,-1) circle (6pt);
\node at (0,-1.6) {\tiny $f$};
\draw[very thick,double,] (3.2,-2) .. controls ++(0,.5) and ++(0,-2) ..(.8,2) to  (.8,4);
    \end{tikzpicture}}
\end{equation*}
which verifies the proposition in this case.

The case $j=i-1$ and $k=i$ is similar.

The proposition follows immediately from the definitions when $i=k$ and $|i-j|>1$.

The case $i=j=k$ is covered by Proposition ~\ref{nilheckeactionE^a}.

The remaining cases are straightforward calculations.
\end{proof}

\subsection{Infinite Grassmannian relations}
For notational convenience in this section we assume that $s_i=a$ and $s_{i+1}=b$, so that $\bar{s}_i=a-b$.  We show that the 2-functor $\Psi$ preserves the infinite Grassmannian relations.
\begin{align}
&\Psi \left( \;\; \sum_{x+y=\alpha} \hackcenter{ \begin{tikzpicture}
 \draw (-.15,.3) node {$\scs i$};
 \draw  (0,0) arc (180:360:0.5cm) [thick];
 \draw[,<-](1,0) arc (0:180:0.5cm) [thick];
\filldraw  [black] (.1,-.25) circle (2.5pt);
 \node at (-.35,-.45) {\tiny $\overset{\bar{s}_i-1}{+x}$};
 \node at (.85,1) { ${\bf s}$};
\end{tikzpicture}  \begin{tikzpicture}
  \draw (-.15,.3) node {$\scs i$};
 \draw  (0,0) arc (180:360:0.5cm) [thick];
 \draw[->](1,0) arc (0:180:0.5cm) [thick];
\filldraw  [black] (.9,-.25) circle (2.5pt);
 \node at (1.25,-.5) {\tiny $\overset{-\bar{s}_i-1}{+y}$};
\end{tikzpicture}  }
\;\; \right)
\;\; \maps \;\;
\hackcenter{
\begin{tikzpicture}[scale=.4]
	\draw [thick,red, double, ] (2,-2) to (2,2);
	\draw [thick,red, double, ] (-2,-2) to (-2,2);
\node at (-2,-2.55) {\tiny $a$};	
\node at (2,-2.55) {\tiny $b$};
	\node at (-2,2.5) {\tiny $a$};
    \node at (0,2.5) {\tiny $k$};
    \node at (0,-2.5) {\tiny $k$};
	\node at (2,2.5) {\tiny $b$};
    \draw [very thick, double, ] (0,-2) to (0,2);
\end{tikzpicture}} \hspace{1.9in}\\
&\qquad \;\; \mapsto \;\;
-\sum_{x+y=\alpha}
\sum_{f=0}^y
\sum_{g=0}^x
(-1)^{f+g}
\hackcenter{
\begin{tikzpicture}[scale=.4]
\draw [thick,red, double, ] (2,-2) to (2,2);
\draw [thick,red, double, ] (-2,-2) to (-2,2);
\filldraw[thick, red]  (2,1) circle (5pt);
\filldraw[thick, red]  (-2,1) circle (5pt);
    \node at (2.9,1) {\tiny ${\sf e}_{g}$};
    \node at (-3.2,1) {\tiny ${\sf h}_{x-g}$};
\node at (-2,-2.55) {\tiny $a$};	
\node at (2,-2.55) {\tiny $b$};
	\node at (-2,2.5) {\tiny $a$};
    \node at (0,2.5) {\tiny $k$};
    \node at (0,-2.5) {\tiny $k$};
	\node at (2,2.5) {\tiny $b$};
    \draw [very thick, double, ] (0,-2) to (0,2);
\filldraw[thick, red]  (-2,-.5) circle (5pt);
\filldraw[thick, red]  (2,-.5) circle (5pt);
    \node at (-2.9,-.5) {\tiny ${\sf e}_{f}$};
    \node at (3.2,-.5) {\tiny ${\sf h}_{y-f}$};
\end{tikzpicture}}
\end{align}
For a term in this triple sum to be nonzero we must have $x-g \geq 0$ and $y-f=\alpha-x-f \geq 0$ \\(or $g \leq x$ and $x \leq \alpha-f$) so we have
\begin{align}
-\sum_{f=0}^a
\sum_{g=0}^b
  \sum_{x=g}^{\alpha-f}
(-1)^{f+g}
\hackcenter{
\begin{tikzpicture}[scale=.4]
	\draw [thick,red, double, ] (2,-2) to (2,2);
\draw [thick,red, double, ] (-2,-2) to (-2,2);
\filldraw[thick, red]  (2,1) circle (5pt);
\filldraw[thick, red]  (-2,1) circle (5pt);
    \node at (2.9,1) {\tiny ${\sf e}_{g}$};
    \node at (-3.2,1) {\tiny ${\sf h}_{x-g}$};
\node at (-2,-2.55) {\tiny $a$};	
\node at (2,-2.55) {\tiny $b$};
	\node at (-2,2.5) {\tiny $a$};
    \node at (0,2.5) {\tiny $k$};
    \node at (0,-2.5) {\tiny $k$};
	\node at (2,2.5) {\tiny $b$};
    \draw [very thick, double, ] (0,-2) to (0,2);
\filldraw[thick, red]  (-2,-.5) circle (5pt);
\filldraw[thick, red]  (2,-.5) circle (5pt);
    \node at (-2.9,-.5) {\tiny ${\sf e}_{f}$};
    \node at (3.2,-.5) {\tiny ${\sf h}_{\overset{\alpha-x}{-f}}$};
\end{tikzpicture}}
= -
\sum_{f=0}^a
\sum_{g=0}^b
 \sum_{x'=0}^{\alpha-f-g}
(-1)^{f+g}
\hackcenter{
\begin{tikzpicture}[scale=.4]
	\draw [thick,red, double, ] (2,-2) to (2,2);
\draw [thick,red, double, ] (-2,-2) to (-2,2);
\filldraw[thick, red]  (2,1) circle (5pt);
\filldraw[thick, red]  (-2,1) circle (5pt);
    \node at (2.9,1) {\tiny ${\sf e}_{g}$};
    \node at (-3.2,1) {\tiny ${\sf h}_{x'}$};
\node at (-2,-2.55) {\tiny $a$};	
\node at (2,-2.55) {\tiny $b$};
	\node at (-2,2.5) {\tiny $a$};
    \node at (0,2.5) {\tiny $k$};
    \node at (0,-2.5) {\tiny $k$};
	\node at (2,2.5) {\tiny $b$};
    \draw [very thick, double, ] (0,-2) to (0,2);
\filldraw[thick, red]  (-2,-.5) circle (5pt);
\filldraw[thick, red]  (2,-.5) circle (5pt);
    \node at (-2.9,-.5) {\tiny ${\sf e}_{f}$};
    \node at (3.2,-.5) {\tiny ${\sf h}_{\overset{\alpha-x'-}{g-f}}$};
\end{tikzpicture}}
\end{align}
Now observe that if $b> \alpha-f$ then the $g$ summation may as well stop at $\alpha-f$ since we have ${\sf h}_{(\alpha-f)-g-x'}$.  However, if $b \leq \alpha-f$ then we may as well keep on summing $g$ all the way to $\alpha-f$ since all these terms are zero by ${\sf e}_g$ on thickness $b$ strand.  Either way, we can take the $g$ summation to $\alpha-f$, so the above is
\begin{align}
&= -
\sum_{f=0}^a
\sum_{g=0}^{\alpha-f}\;
 \sum_{x'=0}^{\alpha-f-g}
(-1)^{f+g}
\hackcenter{
\begin{tikzpicture}[scale=.4]
\draw [thick,red, double, ] (2,-2) to (2,2);
\draw [thick,red, double, ] (-2,-2) to (-2,2);
\filldraw[thick, red]  (2,1) circle (5pt);
\filldraw[thick, red]  (-2,1) circle (5pt);
    \node at (2.9,1) {\tiny ${\sf e}_{g}$};
    \node at (-3.2,1) {\tiny ${\sf h}_{x'}$};
\node at (-2,-2.55) {\tiny $a$};	
\node at (2,-2.55) {\tiny $b$};
	\node at (-2,2.5) {\tiny $a$};
    \node at (0,2.5) {\tiny $k$};
    \node at (0,-2.5) {\tiny $k$};
	\node at (2,2.5) {\tiny $b$};
    \draw [very thick, double, ] (0,-2) to (0,2);
\filldraw[thick, red]  (-2,-.5) circle (5pt);
\filldraw[thick, red]  (2,-.5) circle (5pt);
    \node at (-2.9,-.5) {\tiny ${\sf e}_{f}$};
    \node at (3.3,-.5) {\tiny ${\sf h}_{\overset{\alpha-x'-}{g-f}}$};
\end{tikzpicture}}
= -
\sum_{f=0}^a
 \sum_{x'=0}^{\alpha-f}
 \sum_{g=0}^{\alpha-f-x'}\;
(-1)^{f+g}
\hackcenter{
\begin{tikzpicture}[scale=.4]
	\draw [thick,red, double, ] (2,-2) to (2,2);
\draw [thick,red, double, ] (-2,-2) to (-2,2);
\filldraw[thick, red]  (2,1) circle (5pt);
\filldraw[thick, red]  (-2,1) circle (5pt);
    \node at (2.9,1) {\tiny ${\sf e}_{g}$};
    \node at (-3.2,1) {\tiny ${\sf h}_{x'}$};
\node at (-2,-2.55) {\tiny $a$};	
\node at (2,-2.55) {\tiny $b$};
	\node at (-2,2.5) {\tiny $a$};
    \node at (0,2.5) {\tiny $k$};
    \node at (0,-2.5) {\tiny $k$};
	\node at (2,2.5) {\tiny $b$};
    \draw [very thick, double, ] (0,-2) to (0,2);
\filldraw[thick, red]  (-2,-.5) circle (5pt);
\filldraw[thick, red]  (2,-.5) circle (5pt);
    \node at (-2.9,-.5) {\tiny ${\sf e}_{f}$};
    \node at (3.3,-.5) {\tiny ${\sf h}_{\overset{\alpha-x'-}{f-g}}$};
\end{tikzpicture}}
\\
&=
 -
\sum_{f=0}^a
 \sum_{x'=0}^{\alpha-f}
(-1)^{f}
\delta_{\alpha-f,x'}
\hackcenter{
\begin{tikzpicture}[scale=.4]
	\draw [thick,red, double, ] (2,-2) to (2,2);
\draw [thick,red, double, ] (-2,-2) to (-2,2);
\filldraw[thick, red]  (-2,1) circle (5pt);
    \node at (-3.2,1) {\tiny ${\sf h}_{x'}$};
\node at (-2,-2.55) {\tiny $a$};	
\node at (2,-2.55) {\tiny $b$};
	\node at (-2,2.5) {\tiny $a$};
    \node at (0,2.5) {\tiny $k$};
    \node at (0,-2.5) {\tiny $k$};
	\node at (2,2.5) {\tiny $b$};
    \draw [very thick, double, ] (0,-2) to (0,2);
\filldraw[thick, red]  (-2,-.5) circle (5pt);
    \node at (-2.9,-.5) {\tiny ${\sf e}_{f}$};
\end{tikzpicture}}
=
 -
\sum_{f=0}^a
(-1)^{f}
\hackcenter{
\begin{tikzpicture}[scale=.4]
	\draw [thick,red, double, ] (2,-2) to (2,2);
\draw [thick,red, double, ] (-2,-2) to (-2,2);
\filldraw[thick, red]  (-2,1) circle (5pt);
    \node at (-3.2,1) {\tiny ${\sf h}_{\alpha-f}$};
\node at (-2,-2.55) {\tiny $a$};	
\node at (2,-2.55) {\tiny $b$};
	\node at (-2,2.5) {\tiny $a$};
    \node at (0,2.5) {\tiny $k$};
    \node at (0,-2.5) {\tiny $k$};
	\node at (2,2.5) {\tiny $b$};
    \draw [very thick, double, ] (0,-2) to (0,2);
\filldraw[thick, red]  (-2,-.5) circle (5pt);
    \node at (-2.9,-.5) {\tiny ${\sf e}_{f}$};
\end{tikzpicture}}
\end{align}
Again, if $\alpha >a$ then we can take the $f$ summation all the way to $\alpha$ because of the ${\sf e}_f$ on the thickness $a$ strand.   Otherwise, $\alpha \leq a$ and we have to stop the $f$ summation at $\alpha$ because of the ${\sf h}_{\alpha-f}$.  Either way we get
\begin{align}
  =
 -
\sum_{f=0}^{\alpha}
(-1)^{f}
\hackcenter{
\begin{tikzpicture}[scale=.4]
	\draw [thick,red, double, ] (2,-2) to (2,2);
\draw [thick,red, double, ] (-2,-2) to (-2,2);
\filldraw[thick, red]  (-2,1) circle (5pt);
    \node at (-3.2,1) {\tiny ${\sf h}_{\alpha-f}$};
\node at (-2,-2.55) {\tiny $a$};	
\node at (2,-2.55) {\tiny $b$};
	\node at (-2,2.5) {\tiny $a$};
    \node at (0,2.5) {\tiny $k$};
    \node at (0,-2.5) {\tiny $k$};
	\node at (2,2.5) {\tiny $b$};
    \draw [very thick, double, ] (0,-2) to (0,2);
\filldraw[thick, red]  (-2,-.5) circle (5pt);
    \node at (-2.9,-.5) {\tiny ${\sf e}_{f}$};
\end{tikzpicture}}
= -
\delta_{\alpha,0}
\hackcenter{
\begin{tikzpicture}[scale=.4]
	\draw [thick,red, double, ] (2,-2) to (2,2);
\draw [thick,red, double, ] (-2,-2) to (-2,2);
\node at (-2,-2.55) {\tiny $a$};	
\node at (2,-2.55) {\tiny $b$};
	\node at (-2,2.5) {\tiny $a$};
    \node at (0,2.5) {\tiny $k$};
    \node at (0,-2.5) {\tiny $k$};
	\node at (2,2.5) {\tiny $b$};
    \draw [very thick, double, ] (0,-2) to (0,2);
\end{tikzpicture}}
\end{align}
so that $\Psi$ preserves the infinite Grassmannian relation.

\subsection{Mixed relation}

\begin{proposition}
\label{mixedprop}
For $i \neq j$ we have the following equalities.
\begin{equation*}
\Psi \left(   \hackcenter{\begin{tikzpicture}[scale=0.8]
    \draw[thick] (0,0) .. controls ++(0,.5) and ++(0,-.5) .. (.75,1);
    \draw[thick, <-] (.75,0) .. controls ++(0,.5) and ++(0,-.5) .. (0,1);
    \draw[thick] (0,1 ) .. controls ++(0,.5) and ++(0,-.5) .. (.75,2);
    \draw[thick, ->] (.75,1) .. controls ++(0,.5) and ++(0,-.5) .. (0,2);
        \node at (-.2,.15) {\tiny $i$};
    \node at (.95,.15) {\tiny $j$};
\node at (1.1,.85) {  ${\bf s}$};
\end{tikzpicture}}
 \right)
 \;\; = \;\;
\Psi \left( \hackcenter{\begin{tikzpicture}[scale=0.8]
    \draw[thick, ->] (0,0) -- (0,2);
    \draw[thick, <-] (.75,0) -- (.75,2);
     \node at (-.2,.2) {\tiny $i$};
    \node at (.95,.2) {\tiny $j$};
\node at (1.1,.85) {  ${\bf s}$};
\end{tikzpicture}}
 \right)
\qquad \quad
 \Psi \left(   \hackcenter{\begin{tikzpicture}[scale=0.8]
    \draw[thick,<-] (0,0) .. controls ++(0,.5) and ++(0,-.5) .. (.75,1);
    \draw[thick] (.75,0) .. controls ++(0,.5) and ++(0,-.5) .. (0,1);
    \draw[thick, ->] (0,1 ) .. controls ++(0,.5) and ++(0,-.5) .. (.75,2);
    \draw[thick] (.75,1) .. controls ++(0,.5) and ++(0,-.5) .. (0,2);
        \node at (-.2,.15) {\tiny $j$};
    \node at (.95,.15) {\tiny $i$};
\node at (1.1,.85) {  ${\bf s}$};
\end{tikzpicture}}
 \right)
 \;\;=\;\;
\Psi \left( \hackcenter{\begin{tikzpicture}[scale=0.8]
    \draw[thick, <-] (0,0) -- (0,2);
    \draw[thick, ->] (.75,0) -- (.75,2);
     \node at (-.2,.2) {\tiny $j$};
    \node at (.95,.2) {\tiny $i$};
\node at (1.1,.85) {  ${\bf s}$};
\end{tikzpicture}}
 \right)
\end{equation*}
\end{proposition}

\begin{proof}
This follows from Lemmas ~\ref{Sidecrossingform1} and ~\ref{Sidecrossingform2}.
\end{proof}

\subsection{$\mathcal{E} \mathcal{F}$ and $\mathcal{F} \mathcal{E}$ decompositions} \label{sec:EF}
In this subsection we will prove the $\mathcal{E} \mathcal{F}$ and $\mathcal{F} \mathcal{E}$ decompositions using a key argument provided to us by Sabin Cautis.  It requires proving the decompositions in some easy cases along with other relations already proved.

For notational convenience in this section we assume that the sequence ${\bf s}$ has $s_i=a$ and $s_{i+1}=b$, so that $\bar{s}_i=a-b$.

\begin{proposition}
\begin{align}
\Psi\left(
\sum_{\overset{f_1+f_2+f_3}{=a-b-1}}\hackcenter{
 \begin{tikzpicture}[scale=0.8]
 \draw[thick,->] (0,-1.0) .. controls ++(0,.5) and ++ (0,.5) .. (.8,-1.0) node[pos=.75, shape=coordinate](DOT1){};
  \draw[thick,<-] (0,1.0) .. controls ++(0,-.5) and ++ (0,-.5) .. (.8,1.0) node[pos=.75, shape=coordinate](DOT3){};
 \draw[thick,->] (0,0) .. controls ++(0,-.45) and ++ (0,-.45) .. (.8,0)node[pos=.25, shape=coordinate](DOT2){};
 \draw[thick] (0,0) .. controls ++(0,.45) and ++ (0,.45) .. (.8,0);
 \draw (-.15,.7) node { $\scs i$};
\draw (1.05,0) node { $\scs i$};
\draw (-.15,-.7) node { $\scs i$};
 \node at (.95,.65) {\tiny $f_3$};
 \node at (-.55,-.05) {\tiny $\overset{-\bar{s}_i-1}{+f_2}$};
  \node at (.95,-.65) {\tiny $f_1$};
 \node at (1.6,.3) { ${\bf s}$};
 \filldraw[thick]  (DOT3) circle (2.5pt);
  \filldraw[thick]  (DOT2) circle (2.5pt);
  \filldraw[thick]  (DOT1) circle (2.5pt);
\end{tikzpicture} }
\right)
\;\; \maps \;\;
\hackcenter{
\begin{tikzpicture}[scale=.4]
	\draw [thick,red, double, ] (2,-2) to (2,2);
	\draw [thick,red, double, ] (-2,-2) to (-2,2);
	\draw [very thick, red ] (-2,-1.5) to (2,-.5);
	\draw [very thick, red ] (-2,1.5) to (2,.5);
	\filldraw[thick, red]  (-1,-1.25) circle (4pt);
		\node at (-1,-1.8) {\tiny $\delta$};	
		\node at (-2,-2.55) {\tiny $a$};	
\node at (2,-2.55) {\tiny $b$};
	\node at (-2,2.5) {\tiny $a$};
    \node at (0,2.5) {\tiny $k$};
    \node at (0,-2.5) {\tiny $k$};
	\node at (2,2.5) {\tiny $b$};
    \draw [very thick, double, ] (0,-2) to (0,2);
\end{tikzpicture}} \mapsto
\sum_{\overset{f_1+f_2+f_3}{=a-b-1}}
\sum_{j=0}^k
 \sum_{\ell=0}^a
 (-1)^{j+\ell+1}
\hackcenter{
\begin{tikzpicture}[scale=.5]
	\draw [thick,red, double, ] (-2,-3) to (-2,3);
	\draw [thick,red, double, ] (2,-3) to (2,3);
	\draw [very thick, red ] (-2,-1.25) to (2,-.25);
	\draw [very thick, red ] (-2,1.75) to (2,.75);
    \filldraw[thick, red]  (-2,-2.45) circle (4pt);
            \draw [very thick, double, ] (0,-3) to (0,-1.95) .. controls ++(-1.4,.5) and ++(-1.4,-.5) .. (0,2.25) to (0,3);;
            \draw [very thick,   ]  (0,-1.95) .. controls ++(3.5,.5) and ++(3.5,-.5) .. (0,2.25) to (0,3) node[pos=.85, shape=coordinate](DOT3){};
    	\filldraw[thick, red]  (-2,-1.75) circle (4pt);
		\filldraw[thick, red]  (1,1) circle (4pt);
        \filldraw[thick, red]  (2,-1.75) circle (4pt);
            \filldraw[thick]  (0,-2.55) circle (4pt);
        \node at (-3.35,-2.6) {\tiny ${\sf h}_{\overset{\delta+f_1+j}{-a+1}}$};
        \node at (-2.6,-1.7) {\tiny ${\sf e}_{\ell}$};
        \node at (1,-2.5) {\tiny$ {\sf e}_{k-j}$};
        \node at (1,1.4) {\tiny$ f_3$};
        \node at (3,-1.7) {\tiny ${\sf h}_{f_2-\ell}$};
\node at (-2,-3.55) {\tiny $a$};	
\node at (2,-3.55) {\tiny $b$};
	\node at (-2,3.5) {\tiny $a$};
    \node at (0,3.5) {\tiny $k$};
    \node at (0,-3.5) {\tiny $k$};
	\node at (2,3.5) {\tiny $b$};
\end{tikzpicture}} \nn \\
\qquad \qquad +
\sum_{\overset{f_1+f_2+f_3}{=a-b-1}}
\sum_{j=0}^k
 \sum_{\ell=0}^a
 (-1)^{j+\ell+1}
 \sum_{\overset{d_1+d_2+d_3}{=b-k}}(-1)^{d_3+k}
 \hackcenter{
\begin{tikzpicture}[scale=.5]
	\draw [thick,red, double, ] (-2,-3) to (-2,3);
	\draw [thick,red, double, ] (2,-3) to (2,3);
	\draw [very thick, red ] (-2,-1.25) to (2,-.25);
	\draw [very thick, red ] (-2,1.75) to (2,.75);
            \draw [very thick, double, ] (-.5,-3) to (-.5,3);;
        \filldraw[thick, red]  (-2,-2.45) circle (4pt);
    	\filldraw[thick, red]  (-2,-1.75) circle (4pt);
		\filldraw[thick, red]  (1,1) circle (4pt);
        \filldraw[thick, red]  (2,-1.75) circle (4pt);
        \filldraw[thick, red]  (2,-2.45) circle (4pt);
            \filldraw[thick]  (-.5,-2.55) circle (4pt);
            \filldraw[thick]  (-.5,-1.75) circle (4pt);
        \node at (-3.35,-2.6) {\tiny ${\sf h}_{\overset{\delta+f_1+j}{-a+1}}$};
        \node at (2.7,-2.6) {\tiny ${\sf e}_{d_3}$};
        \node at (-2.6,-1.7) {\tiny ${\sf e}_{\ell}$};
        \node at (.5,-2.5) {\tiny$ {\sf e}_{k-j}$};
        \node at (.35,-1.7) {\tiny$ {\sf h}_{d_1}$};
        \node at (1,1.4) {\tiny$ f_3+d_2$};
        \node at (3,-1.7) {\tiny ${\sf h}_{f_2-\ell}$};
\node at (-2,-3.55) {\tiny $a$};	
\node at (2,-3.55) {\tiny $b$};
	\node at (-2,3.5) {\tiny $a$};
    \node at (0,3.5) {\tiny $k$};
    \node at (0,-3.5) {\tiny $k$};
	\node at (2,3.5) {\tiny $b$};
\end{tikzpicture}}
\end{align}
\end{proposition}

\begin{proof}
This follows from the definitions of the generating $2$-morphisms.
\end{proof}

\begin{proposition}
\begin{align}
\Psi\left(
\sum_{\overset{f_1+f_2+f_3}{=-\bar{s}_i-1}}\hackcenter{
 \begin{tikzpicture}[scale=0.8]
 \draw[thick,<-] (0,-1.0) .. controls ++(0,.5) and ++ (0,.5) .. (.8,-1.0) node[pos=.75, shape=coordinate](DOT1){};
  \draw[thick,->] (0,1.0) .. controls ++(0,-.5) and ++ (0,-.5) .. (.8,1.0) node[pos=.75, shape=coordinate](DOT3){};
 \draw[thick ] (0,0) .. controls ++(0,-.45) and ++ (0,-.45) .. (.8,0)node[pos=.25, shape=coordinate](DOT2){};
 \draw[thick, ->] (0,0) .. controls ++(0,.45) and ++ (0,.45) .. (.8,0);
 \draw (-.15,.7) node { $\scs i$};
\draw (1.05,0) node { $\scs i$};
\draw (-.15,-.7) node { $\scs i$};
 \node at (.95,.65) {\tiny $f_3$};
 \node at (-.55,-.05) {\tiny $\overset{\bar{s}_i-1}{+f_2}$};
  \node at (.95,-.65) {\tiny $f_1$};
 \node at (1.6,.3) { ${\bf s}$};
 \filldraw[thick]  (DOT3) circle (2.5pt);
  \filldraw[thick]  (DOT2) circle (2.5pt);
  \filldraw[thick]  (DOT1) circle (2.5pt);
\end{tikzpicture} }
\right)
\;\; \maps \;\;
\hackcenter{
\begin{tikzpicture}[scale=.4]
	\draw [thick,red, double, ] (2,-2) to (2,2);
	\draw [thick,red, double, ] (-2,-2) to (-2,2);
	\draw [very thick, red ] (-2,-.5) to (2,-1.5);
	\draw [very thick, red ] (-2,.5) to (2,1.5);
	\filldraw[thick, red]  (-1,.75) circle (4pt);
		\node at (-1,1.5) {\tiny $\delta$};	
		\node at (-2,-2.55) {\tiny $a$};	
\node at (2,-2.55) {\tiny $b$};
	\node at (-2,2.5) {\tiny $a$};
    \node at (0,2.5) {\tiny $k$};
    \node at (0,-2.5) {\tiny $k$};
	\node at (2,2.5) {\tiny $b$};
    \draw [very thick, double, ] (0,-2) to (0,2);
\end{tikzpicture}} \mapsto
\sum_{\overset{f_1+f_2+f_3}{=b-a-1}}
\sum_{j=0}^k
 \sum_{\ell=0}^b
 (-1)^{j+\ell+k}
\hackcenter{
\begin{tikzpicture}[scale=.5]
	\draw [thick,red, double, ] (-2,-3) to (-2,3);
	\draw [thick,red, double, ] (2,-3) to (2,3);
	\draw [very thick, red ] (-2,-.25) to (2,-1.25);
	\draw [very thick, red ] (-2,.75) to (2,1.75);
    \filldraw[thick, red]  (-2,-2.45) circle (4pt);
        \draw [very thick, double, ] (0,-3) to (0,-2.25) .. controls ++(1.4,.5) and ++(1.4,-.5) .. (0,2.25) to (0,3) node[pos=.5, shape=coordinate](DOT3){};;
    \draw [very thick, ]  (0,-2.25) .. controls ++(-3.5,.5) and ++(-3.5,-.5) .. (0,2.25) to (0,3);
    	\filldraw[thick, red]  (2,-2.75) circle (4pt);
		\filldraw[thick, red]  (-1,1) circle (4pt);
        \filldraw[thick, red]  (2,-1.75) circle (4pt);
            \filldraw[thick]  (0,-2.55) circle (4pt);
        \node at (-3,-2.6) {\tiny ${\sf h}_{f_2-\ell}$};
        \node at (2.5,-2.75) {\tiny ${\sf e}_{\ell}$};
        \node at (1,-2.5) {\tiny$ {\sf e}_{k-j}$};
        \node at (-1,1.4) {\tiny$ f_3$};
        \node at (3.1,-1.7) {\tiny  ${\sf h}_{\overset{\delta+f_1+j}{-b+1}}$};
\node at (-2,-3.55) {\tiny $a$};	
\node at (2,-3.55) {\tiny $b$};
	\node at (-2,3.5) {\tiny $a$};
    \node at (0,3.5) {\tiny $k$};
    \node at (0,-3.5) {\tiny $k$};
	\node at (2,3.5) {\tiny $b$};
\end{tikzpicture}} \nn \\
\qquad \qquad +
\sum_{\overset{f_1+f_2+f_3}{=b-a-1}}
\sum_{j=0}^k
 \sum_{\ell=0}^b
 (-1)^{j+\ell+1}
 \sum_{\overset{d_1+d_2+d_3}{=a-k}}(-1)^{a+d_1+d_2}
 \hackcenter{
\begin{tikzpicture}[scale=.5]
	\draw [thick,red, double, ] (-2,-3) to (-2,3);
	\draw [thick,red, double, ] (2,-3) to (2,3);
	\draw [very thick, red ] (-2,-.25) to (2,-1.25);
	\draw [very thick, red ] (-2,.75) to (2,1.75);
            \draw [very thick, double, ] (-.5,-3) to (-.5,3);;
        \filldraw[thick, red]  (-2,-2.45) circle (4pt);
    	\filldraw[thick, red]  (-2,-1.75) circle (4pt);
		\filldraw[thick, red]  (1,1.5) circle (4pt);
        \filldraw[thick, red]  (2,-1.75) circle (4pt);
        \filldraw[thick, red]  (2,-2.45) circle (4pt);
            \filldraw[thick]  (-.5,-2.55) circle (4pt);
            \filldraw[thick]  (-.5,-1.75) circle (4pt);
        \node at (-3.2,-2.6) {\tiny ${\sf h}_{f_2-\ell}$};
        \node at (2.7,-2.6) {\tiny ${\sf e}_{\ell}$};
        \node at (-2.6,-1.7) {\tiny ${\sf e}_{d_3}$};
        \node at (.5,-2.5) {\tiny$ {\sf e}_{k-j}$};
        \node at (.35,-1.7) {\tiny$ {\sf h}_{d_1}$};
        \node at (1,1.9) {\tiny$ f_3+d_2$};
        \node at (3.1,-1.7) {\tiny ${\sf h}_{\overset{\delta+f_1+j}{-b+1}}$};
\node at (-2,-3.55) {\tiny $a$};	
\node at (2,-3.55) {\tiny $b$};
	\node at (-2,3.5) {\tiny $a$};
    \node at (0,3.5) {\tiny $k$};
    \node at (0,-3.5) {\tiny $k$};
	\node at (2,3.5) {\tiny $b$};
\end{tikzpicture}}
\end{align}
\end{proposition}

\begin{proof}
This too follows from the definitions of the generating $2$-morphisms.
\end{proof}

We begin with the decompositions which imply that
\begin{equation*}
\mathsf{E}_i \mathsf{F}_i 1_{\mathbf{s}} \cong [a] 1_{\mathbf{s}}; \quad \quad  s_i=a, s_{i+1}=0
\end{equation*}
\begin{equation*}
\mathsf{F}_i \mathsf{E}_i 1_{\mathbf{s}} \cong [b] 1_{\mathbf{s}}; \quad \quad  s_i=0, s_{i+1}=b
\end{equation*}

\begin{proposition}
\label{EFdecomspecialcasea00b} \hfill
\begin{enumerate}
\item Suppose ${\bf s}=(s_1,\ldots,s_m)$ with $s_i=a,s_{i+1}=0$.  Then we have the following equality of bimodule homomorphisms.
\begin{equation*}
\Psi\left(
 \hackcenter{\begin{tikzpicture}[scale=0.8]
    \draw[thick] (0,0) .. controls ++(0,.5) and ++(0,-.5) .. (.75,1);
    \draw[thick,<-] (.75,0) .. controls ++(0,.5) and ++(0,-.5) .. (0,1);
    \draw[thick] (0,1 ) .. controls ++(0,.5) and ++(0,-.5) .. (.75,2);
    \draw[thick, ->] (.75,1) .. controls ++(0,.5) and ++(0,-.5) .. (0,2);
        \node at (-.2,.15) {\tiny $i$};
    \node at (.95,.15) {\tiny $i$};
    \node at (1.1,1.1) {${\bf s}$};
\end{tikzpicture}}
\right)
\;\; = \;\;
\Psi\left(
\hackcenter{\begin{tikzpicture}[scale=0.8]
    \draw[thick, ->] (0,0) -- (0,2);
    \draw[thick, <-] (.75,0) -- (.75,2);
     \node at (-.2,.2) {\tiny $i$};
    \node at (.95,.2) {\tiny $i$};
\node at (1.1,1.1) {${\bf s}$};
\end{tikzpicture}}
\right)
\;\; +\;\;
\Psi\left(
\sum_{\overset{f_1+f_2+f_3}{=a-1}}\hackcenter{
 \begin{tikzpicture}[scale=0.8]
 \draw[thick,->] (0,-1.0) .. controls ++(0,.5) and ++ (0,.5) .. (.8,-1.0) node[pos=.75, shape=coordinate](DOT1){};
  \draw[thick,<-] (0,1.0) .. controls ++(0,-.5) and ++ (0,-.5) .. (.8,1.0) node[pos=.75, shape=coordinate](DOT3){};
 \draw[thick,->] (0,0) .. controls ++(0,-.45) and ++ (0,-.45) .. (.8,0)node[pos=.25, shape=coordinate](DOT2){};
 \draw[thick] (0,0) .. controls ++(0,.45) and ++ (0,.45) .. (.8,0);
 \draw (-.15,.7) node { $\scs i$};
\draw (1.05,0) node { $\scs i$};
\draw (-.15,-.7) node { $\scs i$};
 \node at (.95,.65) {\tiny $f_3$};
 \node at (-1,-.05) {\tiny $ -a-1+f_2$};
  \node at (.95,-.65) {\tiny $f_1$};
 \node at (1.6,.3) { ${\bf s}$};
 \filldraw[thick]  (DOT3) circle (2.5pt);
  \filldraw[thick]  (DOT2) circle (2.5pt);
  \filldraw[thick]  (DOT1) circle (2.5pt);
\end{tikzpicture} }
\right)
\end{equation*}
\item Suppose ${\bf s}=(s_1,\ldots,s_m)$ with $s_i=0,s_{i+1}=b$.  Then we have the following equality of bimodule homomorphisms.
\begin{equation*}
 \Psi\left(
 \hackcenter{\begin{tikzpicture}[scale=0.8]
    \draw[thick,<-] (0,0) .. controls ++(0,.5) and ++(0,-.5) .. (.75,1);
    \draw[thick] (.75,0) .. controls ++(0,.5) and ++(0,-.5) .. (0,1);
    \draw[thick, ->] (0,1 ) .. controls ++(0,.5) and ++(0,-.5) .. (.75,2);
    \draw[thick] (.75,1) .. controls ++(0,.5) and ++(0,-.5) .. (0,2);
        \node at (-.2,.15) {\tiny $i$};
    \node at (.95,.15) {\tiny $i$};
\node at (1.1,1.1) {${\bf s}$};
\end{tikzpicture}}
\right)
\;\; = \;\;
\Psi\left(
\hackcenter{\begin{tikzpicture}[scale=0.8]
    \draw[thick, <-] (0,0) -- (0,2);
    \draw[thick, ->] (.75,0) -- (.75,2);
     \node at (-.2,.2) {\tiny $i$};
    \node at (.95,.2) {\tiny $i$};
\node at (1.1,1.1) {${\bf s}$};
\end{tikzpicture}}
\right)
\;\; + \;\;
\Psi\left(
\sum_{\overset{f_1+f_2+f_3}{=b-1}}\hackcenter{
 \begin{tikzpicture}[scale=0.8]
 \draw[thick,<-] (0,-1.0) .. controls ++(0,.5) and ++ (0,.5) .. (.8,-1.0) node[pos=.75, shape=coordinate](DOT1){};
  \draw[thick,->] (0,1.0) .. controls ++(0,-.5) and ++ (0,-.5) .. (.8,1.0) node[pos=.75, shape=coordinate](DOT3){};
 \draw[thick ] (0,0) .. controls ++(0,-.45) and ++ (0,-.45) .. (.8,0)node[pos=.25, shape=coordinate](DOT2){};
 \draw[thick, ->] (0,0) .. controls ++(0,.45) and ++ (0,.45) .. (.8,0);
 \draw (-.15,.7) node { $\scs i$};
\draw (1.05,0) node { $\scs i$};
\draw (-.15,-.7) node { $\scs i$};
 \node at (.95,.65) {\tiny $f_3$};
 \node at (-1,-.05) {\tiny $-b-1+f_2 $};
  \node at (.95,-.65) {\tiny $f_1$};
 \node at (1.6,.3) { ${\bf s}$};
 \filldraw[thick]  (DOT3) circle (2.5pt);
  \filldraw[thick]  (DOT2) circle (2.5pt);
  \filldraw[thick]  (DOT1) circle (2.5pt);
\end{tikzpicture} }
\right)
\end{equation*}
\end{enumerate}
\end{proposition}

\begin{proof}
Both items are proved in the same way.  In fact, the second part follows from the first part by symmetry so we only sketch the proof of the first equation.

By definition the term on the left side of the equation vanishes.  It is a routine calculation to check that the second term on the right side is negative of the identity morphism.  Note that in the course of the computation one sees that $f_1=f_2=0$ and $f_3=a-1$.
\end{proof}

The next crucial result is due to Cautis.
Let ${\bf s}=(s_1, \ldots, s_m)$ and then set
${\bf s'}=(s_1, \ldots, s_i,0,s_{i+1},\ldots,s_m)$ to be a sequence of $m+1$ integers.
\begin{proposition}
\label{sabinprop}
For any sequence ${\bf s}$ there are isomorphisms
\begin{enumerate}
\item If $s_i \geq s_{i+1}$ then
$ \mathsf{E}_i \mathsf{F}_i 1_{\mathbf{s}} \cong
\mathsf{F}_i \mathsf{E}_i 1_{\mathbf{s}} \oplus
[s_i-s_{i+1}] 1_{\mathbf{s}}$,
\item If $s_i \leq s_{i+1}$ then
$ \mathsf{F}_i \mathsf{E}_i 1_{\mathbf{s}} \cong
\mathsf{E}_i \mathsf{F}_i 1_{\mathbf{s}} \oplus
[s_{i+1}-s_{i}] 1_{\mathbf{s}}$.
\end{enumerate}
\end{proposition}

\begin{proof}
We will only establish the first isomorphism as the second one is proved in a similar way.  We proceed by induction on $s_i+s_{i+1}$.
The isomorphism exists for the special case that $s_{i+1}=0$ by Proposition ~\ref{EFdecomspecialcasea00b}.

Decomposing $ \mathsf{E}_i \mathsf{F}_i 1_{\mathbf{s}}$ is the same as decomposing
$ \mathsf{E}_i \mathsf{E}_{i+1} \mathsf{F}_{i+1}  \mathsf{F}_i 1_{\mathbf{s'}}$ under the canonical identification of $W({\bf s},n)$ with $W({\bf s}',n)$.
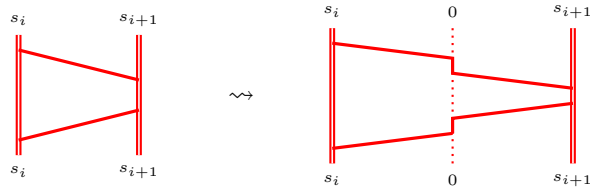
\[
\hackcenter{
\begin{tikzpicture}[scale=.4]
	\draw [thick,red, double, ] (2,-2) to (2,2);
	\draw [thick,red, double, ] (-2,-2) to (-2,2);
	\draw [very thick, red ] (-2,-1.5) to (2,-.5);
	\draw [very thick, red ] (-2,1.5) to (2,.5);
		\node at (-2,-2.55) {\tiny $s_i$};	
    \node at (2,-2.55) {\tiny $s_{i+1}$};
	\node at (-2,2.5) {\tiny $s_i$};
	\node at (2,2.5) {\tiny $s_{i+1}$};
\end{tikzpicture}}
\qquad \rightsquigarrow\qquad
\hackcenter{
\begin{tikzpicture}[scale=.4]
\draw [thick,red, double, ] (6,-2) to (6,2.5);
\draw [thick,red, dotted ] (2,-2) to (2,2.5);
	\draw [thick,red, double, ] (-2,-2) to (-2,2.5);
	\draw [very thick, red ] (2,-1) to (-2,-1.5);
    \draw [very thick, red ] (2,-1) to (2,-.5) to (6,0);
    \draw [very thick, red ]  (6,.5) to (2,1) to (2,1.5) to (-2,2);
\node at (-2,-2.55) {\tiny $s_i$};	
\node at (2,-2.55) {\tiny $0$};
\node at (6,-2.55) {\tiny $s_{i+1}$};
	\node at (-2,3.05) {\tiny $s_i$};
	\node at (2,3.05) {\tiny $0$};
		\node at (6,3.05) {\tiny $s_{i+1}$};
\end{tikzpicture}}
\]

We begin by considering
$\mathsf{E}_{i} \mathsf{F}_{i+1} \mathsf{E}_{i+1} \mathsf{F}_{i} 1_{\mathbf{s'}}$.
Note that
\begin{align*}
\mathsf{E}_{i} \mathsf{F}_{i+1} \mathsf{E}_{i+1} \mathsf{F}_{i} 1_{\mathbf{s'}}
&\cong
\mathsf{F}_{i+1} \mathsf{E}_{i} \mathsf{F}_{i} \mathsf{E}_{i+1} 1_{\mathbf{s'}} \\
&\cong
\mathsf{F}_{i+1} \mathsf{F}_{i} \mathsf{E}_{i} \mathsf{E}_{i+1} 1_{\mathbf{s'}}
\oplus [s_i-1] \mathsf{F}_{i+1} \mathsf{E}_{i+1} 1_{\mathbf{s'}} \\
&\cong
\mathsf{F}_{i+1} \mathsf{F}_{i} \mathsf{E}_{i} \mathsf{E}_{i+1} 1_{\mathbf{s'}}
\oplus [s_{i+1}][s_i-1] 1_{\mathbf{s'}}
\end{align*}
where the first isomorphism is a consequence of Proposition ~\ref{mixedprop}, the second isomorphism follows from induction and the third isomorphism follows from ~\ref{EFdecomspecialcasea00b}.
On the other hand
\begin{align*}
\mathsf{E}_{i} \mathsf{F}_{i+1} \mathsf{E}_{i+1} \mathsf{F}_{i} 1_{\mathbf{s'}}
&\cong
\mathsf{E}_{i} \mathsf{E}_{i+1} \mathsf{F}_{i+1} \mathsf{F}_{i} 1_{\mathbf{s'}}
\oplus [s_{i+1}-1] \mathsf{E}_{i} \mathsf{F}_{i} 1_{\mathbf{s'}} \\
&\cong
\mathsf{E}_{i} \mathsf{E}_{i+1} \mathsf{F}_{i+1} \mathsf{F}_{i} 1_{\mathbf{s'}}
\oplus [s_i][s_{i+1}-1]  1_{\mathbf{s'}}
\end{align*}
where the first isomorphism follows from induction and the second isomorphism is a consequence of Proposition ~\ref{EFdecomspecialcasea00b}.  Thus
\begin{equation}
\label{preisoEFdecomp}
\mathsf{E}_{i} \mathsf{E}_{i+1} \mathsf{F}_{i+1} \mathsf{F}_{i} 1_{\mathbf{s'}}
\oplus [s_i][s_{i+1}-1]  1_{\mathbf{s'}}
\cong
\mathsf{F}_{i+1} \mathsf{F}_{i} \mathsf{E}_{i} \mathsf{E}_{i+1} 1_{\mathbf{s'}}
\oplus [s_{i+1}][s_i-1] 1_{\mathbf{s'}}.
\end{equation}
Since the graded homomorphism spaces between these bimodules are finite dimensional, we may apply the Krull-Schmidt theorem to ~\eqref{preisoEFdecomp} and obtain
\begin{equation*}
\mathsf{E}_{i} \mathsf{E}_{i+1} \mathsf{F}_{i+1} \mathsf{F}_{i} 1_{\mathbf{s'}}
\cong
\mathsf{F}_{i+1} \mathsf{F}_{i} \mathsf{E}_{i} \mathsf{E}_{i+1} 1_{\mathbf{s'}}
\oplus [s_i-s_{i+1}] 1_{\mathbf{s'}}
\end{equation*}
which proves the proposition.
\end{proof}

\begin{proposition}
\label{EFdecompprop}
Let ${\bf s}=(s_1,\ldots,s_m)$.  Then we have the following equalities of bimodule homomorphisms
\begin{enumerate}
\item
\begin{equation*}
\Psi\left(
 \hackcenter{\begin{tikzpicture}[scale=0.8]
    \draw[thick] (0,0) .. controls ++(0,.5) and ++(0,-.5) .. (.75,1);
    \draw[thick,<-] (.75,0) .. controls ++(0,.5) and ++(0,-.5) .. (0,1);
    \draw[thick] (0,1 ) .. controls ++(0,.5) and ++(0,-.5) .. (.75,2);
    \draw[thick, ->] (.75,1) .. controls ++(0,.5) and ++(0,-.5) .. (0,2);
        \node at (-.2,.15) {\tiny $i$};
    \node at (.95,.15) {\tiny $i$};
    \node at (1.1,1.1) {${\bf s}$};
\end{tikzpicture}}
\right)
\;\; = \;\;
\Psi\left(
\hackcenter{\begin{tikzpicture}[scale=0.8]
    \draw[thick, ->] (0,0) -- (0,2);
    \draw[thick, <-] (.75,0) -- (.75,2);
     \node at (-.2,.2) {\tiny $i$};
    \node at (.95,.2) {\tiny $i$};
\node at (1.1,1.1) {${\bf s}$};
\end{tikzpicture}}
\right)
\;\; +\;\;
\Psi\left(
\sum_{\overset{f_1+f_2+f_3}{=a-b-1}}\hackcenter{
 \begin{tikzpicture}[scale=0.8]
 \draw[thick,->] (0,-1.0) .. controls ++(0,.5) and ++ (0,.5) .. (.8,-1.0) node[pos=.75, shape=coordinate](DOT1){};
  \draw[thick,<-] (0,1.0) .. controls ++(0,-.5) and ++ (0,-.5) .. (.8,1.0) node[pos=.75, shape=coordinate](DOT3){};
 \draw[thick,->] (0,0) .. controls ++(0,-.45) and ++ (0,-.45) .. (.8,0)node[pos=.25, shape=coordinate](DOT2){};
 \draw[thick] (0,0) .. controls ++(0,.45) and ++ (0,.45) .. (.8,0);
 \draw (-.15,.7) node { $\scs i$};
\draw (1.05,0) node { $\scs i$};
\draw (-.15,-.7) node { $\scs i$};
 \node at (.95,.65) {\tiny $f_3$};
 \node at (-1.2,-.05) {\tiny $ -a+b-1+f_2$};
  \node at (.95,-.65) {\tiny $f_1$};
 \node at (1.6,.3) { ${\bf s}$};
 \filldraw[thick]  (DOT3) circle (2.5pt);
  \filldraw[thick]  (DOT2) circle (2.5pt);
  \filldraw[thick]  (DOT1) circle (2.5pt);
\end{tikzpicture} }
\right)
\end{equation*}
\item
\begin{equation*}
 \Psi\left(
 \hackcenter{\begin{tikzpicture}[scale=0.8]
    \draw[thick,<-] (0,0) .. controls ++(0,.5) and ++(0,-.5) .. (.75,1);
    \draw[thick] (.75,0) .. controls ++(0,.5) and ++(0,-.5) .. (0,1);
    \draw[thick, ->] (0,1 ) .. controls ++(0,.5) and ++(0,-.5) .. (.75,2);
    \draw[thick] (.75,1) .. controls ++(0,.5) and ++(0,-.5) .. (0,2);
        \node at (-.2,.15) {\tiny $i$};
    \node at (.95,.15) {\tiny $i$};
\node at (1.1,1.1) {${\bf s}$};
\end{tikzpicture}}
\right)
\;\; = \;\;
\Psi\left(
\hackcenter{\begin{tikzpicture}[scale=0.8]
    \draw[thick, <-] (0,0) -- (0,2);
    \draw[thick, ->] (.75,0) -- (.75,2);
     \node at (-.2,.2) {\tiny $i$};
    \node at (.95,.2) {\tiny $i$};
\node at (1.1,1.1) {${\bf s}$};
\end{tikzpicture}}
\right)
\;\; + \;\;
\Psi\left(
\sum_{\overset{f_1+f_2+f_3}{=b-a-1}}\hackcenter{
 \begin{tikzpicture}[scale=0.8]
 \draw[thick,<-] (0,-1.0) .. controls ++(0,.5) and ++ (0,.5) .. (.8,-1.0) node[pos=.75, shape=coordinate](DOT1){};
  \draw[thick,->] (0,1.0) .. controls ++(0,-.5) and ++ (0,-.5) .. (.8,1.0) node[pos=.75, shape=coordinate](DOT3){};
 \draw[thick ] (0,0) .. controls ++(0,-.45) and ++ (0,-.45) .. (.8,0)node[pos=.25, shape=coordinate](DOT2){};
 \draw[thick, ->] (0,0) .. controls ++(0,.45) and ++ (0,.45) .. (.8,0);
 \draw (-.15,.7) node { $\scs i$};
\draw (1.05,0) node { $\scs i$};
\draw (-.15,-.7) node { $\scs i$};
 \node at (.95,.65) {\tiny $f_3$};
 \node at (-1.2,-.05) {\tiny $-b+a-1+f_2 $};
  \node at (.95,-.65) {\tiny $f_1$};
 \node at (1.6,.3) { ${\bf s}$};
 \filldraw[thick]  (DOT3) circle (2.5pt);
  \filldraw[thick]  (DOT2) circle (2.5pt);
  \filldraw[thick]  (DOT1) circle (2.5pt);
\end{tikzpicture} }
\right)
\end{equation*}
\end{enumerate}
\end{proposition}

\begin{proof}
In \cite[Theorem 1.1]{CLau} (see also \cite{Brundan2}), it is proved that if one has all $2$-categorical relations
along with abstract isomorphisms (from Proposition ~\ref{sabinprop})
\begin{align*}
\mathsf{E}_i \mathsf{F}_i 1_{\mathbf{s}} &\cong
\mathsf{F}_i \mathsf{E}_i 1_{\mathbf{s}} \oplus
[s_i-s_{i+1}] 1_{\mathbf{s}} \\
\mathsf{F}_i \mathsf{E}_i 1_{\mathbf{s}} &\cong
\mathsf{E}_i \mathsf{F}_i 1_{\mathbf{s}} \oplus
[s_{i+1}-s_{i}] 1_{\mathbf{s}}
\end{align*}
for $s_i \geq s_{i+1}$ and $s_{i+1} \geq s_i$ respectively that these isomorphisms
could be written in terms of generators for the $2$-category and they satisfy the relations of the proposition.
\end{proof}

\section{Main Theorems}
In this section we state the  main results of this paper.

\begin{theorem}
\label{theorem2rep}
The assignment $\Psi$ extends to a $2$-functor
$\Psi \maps \cal{U} \to \Bim(n)$.
\end{theorem}

\begin{proof}
This follows from the check of the relations in Section \ref{sectionrelationsproved}.
\end{proof}

\begin{definition}
Let $\Br_m$ be the braid group of type $A_m$.  That is, it is generated by elements
$T_i, T_i'$ for $i=1,\ldots,m-1$ subject to relations
\begin{itemize}
\item $T_i T_i'=1=T_i'T_i$,
\item $T_i T_j = T_j T_i$ if $|i-j|>1$,
\item $T_i T_j T_i = T_j T_i T_j$ if $|i-j|=1$.
\end{itemize}
\end{definition}

It was shown in ~\cite{CautisKam} that a categorical $\cal{U}$ action gives rise to a categorical $\Br_m$ action.   Cautis and Kamnitzer's work extended the foundational work of Chuang and Rouquier \cite{CR} where $\mathfrak{sl}_2$ categorification was developed.

We will now recall the main result of \cite{CautisKam} in the context of a particular weight space of a representation of $\mathfrak{gl}_m$.

Let ${\bf s}=(s_1,\ldots,s_m)$.  There are complexes

\begin{align*}
\mathsf{T}_i 1_{\bf s}&=
\mathsf{E}_i^{(s_{i+1}-s_i)} 1_{\bf s}
\rightarrow
\mathsf{E}_i^{(s_{i+1}-s_i+1)} \mathsf{F}_i^{(1)} 1_{\bf s} \langle 1 \rangle
\rightarrow
\cdots
\rightarrow
\mathsf{E}_i^{(s_{i+1}-s_i+j)} \mathsf{F}_i^{(j)} 1_{\bf s} \langle j \rangle
\rightarrow
\cdots
\quad \quad \text{ for } s_{i+1} \geq s_i \\
&=
\mathsf{F}_i^{(s_{i}-s_{i+1})} 1_{\bf s}
\rightarrow
\mathsf{F}_i^{(s_{i}-s_{i+1}+1)} \mathsf{E}_i^{(1)} 1_{\bf s} \langle 1 \rangle
\rightarrow
\cdots
\rightarrow
\mathsf{F}_i^{(s_{i}-s_{i+1}+j)} \mathsf{E}_i^{(j)} 1_{\bf s} \langle j \rangle
\rightarrow
\cdots
\quad \quad \text{ for } s_{i} \geq s_{i+1}
\end{align*}

\begin{align*}
1_{\bf s} \mathsf{T}'_i &=
\cdots
\rightarrow
1_{\bf s} \mathsf{E}_i^{(j)}  \mathsf{F}_i^{(s_{i+1}-s_i+j)}  \langle -j \rangle
\rightarrow
\cdots
\rightarrow
1_{\bf s} \mathsf{E}_i^{(1)} \mathsf{F}_i^{(s_{i+1}-s_i+1)}   \langle -1 \rangle
\rightarrow
1_{\bf s} \mathsf{F}_i^{(s_{i+1}-s_i)}
\quad \quad \text{ for } s_{i+1} \geq s_i \\
&=
\cdots
\rightarrow
1_{\bf s} \mathsf{F}_i^{(j)}  \mathsf{E}_i^{(s_{i}-s_{i+1}+j)}  \langle -j \rangle
\rightarrow
\cdots
\rightarrow
1_{\bf s} \mathsf{F}_i^{(1)} \mathsf{E}_i^{(s_{i}-s_{i+1}+1)}   \langle -1 \rangle
\rightarrow
1_{\bf s} \mathsf{E}_i^{(s_{i}-s_{i+1})}
\quad \quad \text{ for } s_{i} \geq s_{i+1}
\end{align*}
where the differentials are given by explicit bimodule homomorphisms.
See for example~\cite[Section 2.2]{LQR}.

Note that when ${\bf s}=(1,\ldots,1)$, these complexes simplify to
\begin{equation}
\label{complexes1}
\mathsf{T}_i' =
\xymatrix{
\mathsf{E}_i \mathsf{F}_i 1_{\bf s} \langle -1 \rangle \ar[r]^-{d} & \mathsf{\Id}
}
\quad \quad \quad \quad
\mathsf{T}_i =
\xymatrix{
\mathsf{\Id}  \ar[r]^-{d'} & \mathsf{E}_i \mathsf{F}_i 1_{\bf s} \langle 1 \rangle
}
\end{equation}
with the differentials given by
\begin{equation}
\label{complexes2}
d\; =\;\;
\hackcenter{\begin{tikzpicture}[scale=0.8]
   \draw[thick, ->] (0,0) .. controls (0,1.0) and (.8,1.0) .. (.8,0);
    \node   at (0.8,-.25) {$\scs i$};
        \node at (1,0.75) {$\scs s$};
\end{tikzpicture}}
\quad \quad \quad \quad
d' \; =\;\;
\hackcenter{\begin{tikzpicture}[scale=0.8]
    \draw[thick, <-] (0,1.5) .. controls (0,.5) and (.8,.5) .. (.8,1.5);
   \node   at (0.8,1.75) {$\scs i$};
        \node at (1,0.75) {$\scs s$};
\end{tikzpicture}}
\end{equation}

\begin{theorem}
\label{mainbraidthm}
As functors on the homotopy category of $W({\bf s},n)$-modules, the complexes
$\mathsf{T}_i, \mathsf{T}_i'$ satisfy braid group relations.  That is, there are isomorphisms
\begin{itemize}
\item $ \mathsf{T}_i'\mathsf{T}_i \cong \Id$,
\item $\mathsf{T}_i \mathsf{T}_i' \cong \Id $,
\item $\mathsf{T}_i \mathsf{T}_j  \cong \mathsf{T}_j \mathsf{T}_i $ if $|i-j|>1$,
\item $\mathsf{T}_i \mathsf{T}_j \mathsf{T}_i \cong \mathsf{T}_j \mathsf{T}_i \mathsf{T}_j $ if $|i-j|=1$.
\end{itemize}
Furthermore, as endofunctors on $K^b(W(1^m,n) \dmod)$, the complexes $\mathsf{T}_i$ and $\mathsf{T}_i'$ satisfy strong braid group relations.
\end{theorem}

\begin{proof}
These complexes satisfy braid group relations by Theorem \ref{theorem2rep} and
\cite[Theorem 6.3]{CautisKam}.

Restricting to the weight space $(1^m)$ of $\Bim(n)$, the complexes simplify to ~\eqref{complexes1} with the maps given in ~\eqref{complexes2}.
The work of Elias and Krasner~\cite{EKras} combined with the connection between Soergel calculus and the 2-category $\cal{U}$ from \cite[Lemma 6.5]{MSV} shows that this is a strong braid group action.
\end{proof}

%


\providecommand{\bysame}{\leavevmode\hbox to3em{\hrulefill}\thinspace}
\providecommand{\MR}{\relax\ifhmode\unskip\space\fi MR }
\providecommand{\MRhref}[2]{%
  \href{http://www.ams.org/mathscinet-getitem?mr=#1}{#2}
}
\providecommand{\href}[2]{#2}

%
%

\end{document}